%% file: main.tex
\documentclass[oneside]{amsart}
\usepackage[lmargin=1in,rmargin=1in, 
bmargin=1in, tmargin=1in]{geometry}
\usepackage[dvipsnames]{xcolor}

\usepackage{amsmath, amssymb,tikz}
\usepackage{tikz-cd}
\usepackage{tikz}
\usepackage{todonotes}
\usepackage{mathrsfs}
\usepackage{extarrows}
\usepackage{graphicx}
\usepackage[breaklinks, 
pagebackref]{hyperref}
\usepackage{IEEEtrantools}
\usepackage{mathrsfs}
\usepackage[utf8]{inputenc}
\usepackage[english]{babel}
\usepackage{nicematrix}   

\usepackage{comment}
\usepackage{csquotes}

\usepackage[shortlabels]{enumitem}

\def\mmath#1{\text{\scalebox{1.09}{$#1$}}}
\def\smath#1{\text{\scalebox{0.9}{$#1$}}}

\def\mfrac#1#2{\mmath{\frac{#1}{#2}}}
\def\sfrac#1#2{\smath{\frac{#1}{#2}}}

\usepackage{makebox}  

\hypersetup{
    colorlinks=true,
    linkcolor=blue,
    filecolor=red, 
    citecolor=red,          
    urlcolor=gray,
    pdftitle={zetaelements},
    pdfpagemode=FullScreen,
    }

\usepackage{tikz}

\makeatletter
\newcommand*{\encircled}[1]{\relax\ifmmode\mathpalette\@encircled@math{#1}\else\@encircled{#1}\fi}
\newcommand*{\@encircled@math}[2]{\@encircled{$\m@th#1#2$}}
\newcommand*{\@encircled}[1]{%
  \tikz[baseline,anchor=base]{\node[draw,circle,outer sep=0pt,inner sep=.2ex] {#1};}}
\makeatother

\DeclareFontEncoding{LS1}{}{}
\DeclareFontSubstitution{LS1}{stix}{m}{n}
\newcommand*\kay{%
  \text{%
  \fontencoding{LS1}%
  \fontfamily{stixscr}%
  \fontseries{\textmathversion}%
  \fontshape{n}%
  \selectfont\symbol{"6B}}}
\makeatletter
  \newcommand*\textmathversion{\csname textmv@\math@version\endcsname}
  \newcommand*\textmv@normal{m}
  \newcommand*\textmv@bold{b}
\makeatother
\DeclareFontEncoding{LS1}{}{}
\DeclareFontSubstitution{LS1}{stix}{m}{n}

\makeatletter
\DeclareFontEncoding{LS1}{}{}
\DeclareFontSubstitution{LS1}{stix}{m}{n}

\makeatletter

\DeclareFontEncoding{LS1}{}{}
\DeclareFontSubstitution{LS1}{stix}{m}{n}

\makeatletter

\DeclareFontEncoding{LS1}{}{}
\DeclareFontSubstitution{LS1}{stix}{m}{n}
\newcommand*\ess{%
  \text{%
  \fontencoding{LS1}%
  \fontfamily{stixscr}%
  \fontseries{\textmathversion}%
  \fontshape{n}%
  \selectfont\symbol{"73}}}
\makeatletter

\usepackage{etoolbox}

\usepackage{dynkin-diagrams}  

\usepackage{nicematrix}
\tikzset{ 
    table/.style={
        matrix of math nodes,
        row sep=-\pgflinewidth,
        column sep=-\pgflinewidth,
        nodes={rectangle,text width=3em,align=center},
        text depth=1.25ex,
        text height=2.5ex,
        nodes in empty cells,
        left delimiter=[,
        right delimiter={]},
        ampersand replacement=\&
    }
}
\usetikzlibrary{decorations.pathreplacing, calligraphy}

\usepackage{setspace}\setdisplayskipstretch{}

\input{macros}

\title{On constructing zeta elements for Shimura varieties}     
\author{Syed Waqar Ali    Shah}    
\date{}

\begin{document}

\begin{abstract}We present a novel axiomatic framework for establishing horizontal norm relations in Euler systems that are built from pushforwards of classes in the  motivic cohomology of Shimura varieties. This framework  is uniformly applicable to  the Euler systems of  both algebraic cycles and Eisenstein classes. It also applies to non-spherical pairs of groups that   fail  to  satisfy  a  local multiplicity one hypothesis, and thus lie beyond the reach of existing methods. 
A key application of this work is the construction of an  Euler system for the spinor Galois representations arising in the cohomology of Siegel modular varieties of genus three,  which is  undertaken  in two companion articles.
\end{abstract}    

\maketitle

\tableofcontents

\input{Introduction.tex}   
\input{Preliminaries.tex}

\input{AbstractZeta.tex}

\input{HeckePolynomials.tex}

\input{Decompositions.tex}

\input{Examples.tex}

\bibliographystyle{amsalpha}  
\bibliography{refs}

\end{document}

%% file: macros.tex
\usepackage{amsmath,amsthm,amssymb}
\usepackage{colonequals}

\newcommand{\QQ}{\mathbb{Q}}

\newcommand{\ZZ}{\mathbb{Z}}

\newcounter{dummypart}

\makeatletter
\newcommand{\partdivider}[2]{%
  \thispagestyle{headings}
  \null\vspace{1 em}
    {\noindent #1}
    {%
      \refstepcounter{dummypart}%
      \label{dummypart:\thedummypart}%
      \hypertarget{dummypart:\thedummypart}{\textbf{#2}}%
    }
  \addtocontents{toc}{\vspace{0.8em} \protect\hyperlink{dummypart:\thedummypart}{\textbf{#2}}\protect\par}
  \vspace{1em}\null%
}
\makeatother
 
\usepackage{capt-of}    

\usepackage{pict2e,picture}

\makeatletter
\newcommand{\pnrelbar}{%
  \linethickness{\dimen2}%
  \sbox\z@{$\m@th\prec$}%
  \dimen@=1.1\ht\z@
  \begin{picture}(\dimen@,.4ex)
  \roundcap
  \put(0,.2ex){\line(1,0){\dimen@}}
  \put(\dimexpr 0.5\dimen@-.2ex\relax,0){\line(1,1){.4ex}}
  \end{picture}%
}

\newcommand{\precneq}{\mathrel{\vcenter{\hbox{\text{\prec@neq}}}}}
\newcommand{\prec@neq}{%
  \dimen2=\f@size\dimexpr.04pt\relax
  \oalign{%
    \noalign{\kern\dimexpr.2ex-.5\dimen2\relax}
    $\m@th\prec$\cr
    \noalign{\kern-.5\dimen2}
    \hidewidth\pnrelbar\hidewidth\cr
  }%
}
\makeatother

\newcommand{\CC}{\mathbb{C}}

\newcommand{\res}{\operatorname{res}}

\newcommand{\ch}{\mathrm{ch}}

\newcommand{\delT}{\mathbb{S}}   
\newcommand{\Gscr}{\mathscr{G}}    
\newcommand{\Sh}{\mathrm{Sh}}    
\newcommand{\Ab}{\mathbb{A}}   
    
\newcommand{\pr}{\mathrm{pr}}

\DeclareMathOperator{\supp}{Supp}

\DeclareMathOperator{\Gal}{Gal}   
\DeclareMathOperator{\Oscr}{\mathscr{O}}

\numberwithin{equation}{subsection}

\newtheorem{theorem}[equation]{Theorem}
\newtheorem{proposition}[equation]{Proposition}
\newtheorem{lemma}[equation]{Lemma}
\newtheorem{corollary}[equation]{Corollary}
\newtheorem*{corollary*}{Corollary}

\newtheorem{problem}[equation]{Problem} 

\newtheorem{theoremx}{Theorem}

\theoremstyle{definition}
\newtheorem{note}[equation]{Note}   

\newtheorem{question}[equation]{Problem}    
\newtheorem*{notation*}{Notation}

\theoremstyle{definition}

\newtheorem{definition}[equation]{Definition}

\theoremstyle{remark}
\newtheorem{remark}[equation]{Remark}
\newtheorem{example}{Example}[section]   

\newtheorem{convention}[equation]{Convention}

\newtheorem{notation}{Notation}[section]

\tikzcdset{scale cd/.style={every label/.append style={scale=#1},
    cells={nodes={scale=#1}}}}

\usepackage{tikz-cd}
\usepackage{mathrsfs}

\DeclareFontFamily{U}{wncy}{}
\DeclareFontShape{U}{wncy}{m}{n}{<->wncyr10}{}
\DeclareSymbolFont{mcy}{U}{wncy}{m}{n}
\DeclareMathSymbol{\sha}{\mathord}{mcy}{"58}

\makeatletter
\renewcommand*\env@matrix[1][\arraystretch]{%
  \edef\arraystretch{#1}%
  \hskip -\arraycolsep
  \let\@ifnextchar\new@ifnextchar
  \array{*\c@MaxMatrixCols c}}
\makeatother

\newcommand{\Gb}{\mathbf{G}}   
\newcommand{\Tb}{\mathbf{T}}

\newcommand{\Hb}{\mathbf{H}}

\usepackage{mathtools,bm}

\providecommand\given{} % ensure it exists
\newcommand\givensymbol[1]{%
  \nonscript\;\delimsize#1\allowbreak\nonscript\;\mathopen{}%
}
\DeclarePairedDelimiterX\Set[1]\{\}{%
  \renewcommand\given{\givensymbol{\vert}}%
  #1%
}

\usepackage{stmaryrd}

\newcommand{\Xcal}{\mathcal{X}}

\newcommand{\GG}{\mathbb{G}}
\newcommand{\OO}{\mathcal{O}}
\newcommand{\GL}{\mathrm{GL}}

\newcommand{\et}{\mathrm{\acute{e}t}}

\newcommand{\RR}{\mathbb{R}}

%% file: Introduction.tex
\section{Introduction}

Euler  systems are objects of  an arithmetic-algebraic-geometric nature that are designed   to provide a  handle on  Selmer groups   of $ p $-adic     Galois representations    
and  play a crucial role in linking these  arithmetic groups to special values of  $ L $-functions.    Though Euler systems are Galois theoretic objects,  the tools involved in their   construction     are often  of  an   automorphic nature.  A typical setup of its kind    
starts  by  identifying the Galois  representation in the cohomology of a  Shimura  variety  of a    reductive    group  $ G   $. The Galois representation is 
required to be automorphic, i.e.,    its $ L$-function  
matches that of a corresponding  automorphic   representation.   The class at the bottom of such a hypothetical system is   taken to be the pushforward  of a  special element  that  lives in the   motives of a sub-Shimura variety arising from a  reductive  subgroup $ H $  of     $ G $. Two common types of special elements are  
fundamental cycles  and Eisenstein classes, and their  respective Euler systems are   often   distinguished based on this dichotomy.   The desire to construct an Euler system via such pushforwards is  also  motivated by a corresponding period integral ($p$-adic or complex)  which     provides a link between   $ L$-values and the bottom class of this hypothetical system.  For this reason, the classes in an Euler system are  sometimes  also  referred to as `zeta elements' (\cite{kkato}).  Once such a setup is identified,  the problem of constructing the deeper (horizontal) layers of  zeta elements  is   often      tackled by  judiciously picking   special elements of the same type   in  the   motives at  deeper levels of $  H  $ and pushing them  into motives of $ G  $ along   conjugated   embeddings.     
   This turns out to be a  rather  challenging problem in  general.    At the moment, there is no  known   general    method   that illuminates  what levels and conjugated embeddings   would yield the desired  norm relations in any  particular  setting.

The primary   goal of this article is to describe an  axiomatic machinery  that specifies   a   precise criteria  for   constructing the   deeper  layers  of zeta elements  in   the       aforementioned    settings. It is also designed to handle the potential failure of the so-called \emph{multiplicity one} hypothesis,   which is a crucial requirement for the technique of local zeta integrals introduced in \cite{LSZ}.  This failure does arise in practice, 
most notably in the situation studied in \cite{CJR},  where the relevant period integral  unfolds to non-unique models. An immediate new application of our work is the construction of a  full Euler system for $ \mathrm{GSp}_{6}$ envisioned in \emph{loc.cit.}, which is carried out in \cite{Siegel2}, \cite{Siegel1} using the framework presented here.  Other forthcoming applications include \cite{EulerGU22} and \cite{EulerG2}.  

\subsection{Main  results}   To describe our  main  results,    it is convenient to work in the   abstract setup of locally profinite groups as used in  \cite{Anticyclo} (cf.\ \cite{loe}). 
Let $ H = \prod_{v \in I}' H_{v} $, $ G = \prod_{v \in I}'G_{v} $, 
be respectively the restricted products of locally profinite groups     $ H_{v} $, $G_{v} $ taken with respect to compact open subgroups $U_{v} \subset H_{v} $, $  K_{v }   \subset G_{v} $. We assume that $ H_{v} $ is a   closed 
subgroup of $ G_{v} $ and that $ U_{v} = H \cap K_{v} $. 
Let $ \Upsilon_{H}$, $ \Upsilon_{G} $ be suitable      non-empty     collections of compact  open subgroups of  $ H $, $ G $.   
Let $$ N : \Upsilon_{H} \to \ZZ_{p}  \text{-Mod}, \quad \quad M : \Upsilon_{G} \to \ZZ_{p}\text{-Mod}$$ be mappings  that associate to each compact open subgroup a $ \ZZ_{p} $-module  in a           functorial         manner  mimicking   the   abstract        properties of cohomology of Shimura varieties over varying levels.  More precisely, it is assumed that for each $ K_{1} \subset K_{2} $ in $ \Upsilon_{G} $ and $ g \in G $, there exist  three maps;   $ \pr^{*} : M(K_{2}) \to M(K_{1}) $ referred to as \emph{restriction}, $ \pr_{*} : M(K_{1} ) \to M(K_{2}) $ referred to as \emph{induction}  and $ [g]^{*} : M(K_{1}) \to M(gK_{1}g^{-1}) $ referred to as \emph{conjugation} and  that these maps   satisfy     certain  compatibility conditions.     Similarly for $  N     $.   We also require that there are   maps  $   \iota_{* } :  N(K_{1} \cap  H) \to M(K_{1}) $  for all $ K_{1}  \in   \Upsilon_{G}   $.    
These model the   behaviour of pushforwards induced by embeddings of Shimura varieties.

Fix for each $ v \in  I $  a     compact open normal subgroup $ L_{v} $ of $ K_{v} $.  By $ K$, $L $, $U$, we denote   respectively  the products of $ K_{v} $, $L_{v} $, $ U_{v} $ over all $  v  $ and assume that $ L, K \in \Upsilon_{G} $ and $ U \in \Upsilon_{H} $.      
Let $ \mathcal{N}$ denote the set of all finite subsets of $ I $. For $ \nu \in  \mathcal{N }$,  we denote $ G_{\nu} = \prod_{v \in \nu} G_{v} $, $ G^{   \nu   }      =   G / G_{\nu} $ and use similar notations for $  H $, $ U $,    $ K $, $ L $.  Set   $ K[\nu] = K^{\nu} L_{\nu} $ for $ \nu \in \mathcal{N}  $.   Thus $ K[\mu]  \subset    K[\nu] $ whenever  $ \nu \subset \mu   $  and  we  denote by   $ \pr_{\mu,\nu,*} $ the induction map $ M(K[\mu]) \to M(K[\nu]) $.    
For each $ v \in I $, let $ \mathfrak{H}_{v} $ be a finite $ \ZZ_{p} $-linear combination  
of characteristic functions of double cosets in $ K_{v} \backslash G_{v} / K_{v} $.  Then  for any   pair of disjoint  $ \mu, \nu \in \mathcal{N} $, there are   linear   maps $$\mathfrak{H}_{\mu,*} : M(K[\nu]) \to M(   K [ \nu ] )    $$
given essentially by sums of tensor products of Hecke correspondences in  $ \mathfrak{H}_{v} $ for    $ v \in \mu $.     For  $ v \in I $, let $ g_{v,1}, \ldots, g_{v,r_{v}} \in G $ be an arbitrary but fixed  set of representatives  for  $$ H_{v} \backslash H_{v} \cdot \supp( \mathfrak{H}_{v}) /K_{v} . $$    For $ i = 1, \ldots, r_{v}    $, let $ H_{v, i} : = H_{v} \cap g_{v,i} K_{v} g_{v,i}^{-1}  $, $  V_{v,i} = H_{v}  \cap  g_{v,i} L_{v} g_{v,i} ^{-1}   \subset H_{v,i}  $ and   let     $ \mathfrak{h}_{v,i}  =  \mathfrak{h}_{g_{v , i } }   $ be the function $  h  \mapsto  \mathfrak{H}_{v}((-)g_{v,i}) $ for $ h \in  H_{v}   $.     Then we have induced $ \ZZ_{p}$-linear maps $   \mathfrak{h}_{v,i,*} : N(U) \to N(H_{v,i} U^{v}   ) $ for each $ i $   given by  Hecke  correspondences. Given $ x_{U} \in N(U) $, our goal is to be able to construct      classes $ y_{\nu} \in  
M (K[\nu]) $ such that $ y_{\varnothing}  =  \iota_{*}(x_{U}) \in M(K)  $ and 
\begin{equation} 
   \label{marketingrelation} \mathfrak{H}_{\mu  \setminus 
   \nu  ,  *  } ( y_{\nu} )   =  \pr_{ \mu ,  \nu,  *} ( y_{\mu} )
\end{equation}
for all  $ \nu , \mu \in \mathcal{N} $  satisfying  $ \nu \subset \mu  $.   A classical example of norm relation in this   format is  \cite[Proposition   2.4]{kkato}. 
See   
\cite[Theorem   2.25]{Norm} for an exposition of Heegner point scenario  in  a  similar spirit.

\begin{theoremx}[Theorem  \ref{gluingzeta}]   Let $ x_{U} \in N(U) $. Assume that $   N  $ equals a restricted tensor product $ \otimes'_{v} N_{v} $ with respect to $ x_{U_{v}} \in N_{v}   ( U _{ v} )   $ (see below) and $ x_{U} =  \otimes'_{v}  x_{U_{v}}  $. Suppose that  for each $ v \in I $ and $ 1 \leq i \leq r_{v} $, there exists $ x_{v, i} \in N_{v}(V_{v,i}) $ such that \begin{equation}   
\label{derivedmarketing} \mathfrak{h}_{v,i ,  *  } (x_{U_{v}}) = \pr_{V_{v,i}, H_{v,i},*}(x_{v,i})
\end{equation}    Then there exist classes $ y_{\nu} \in  M
(K[\nu]) $  for  all  $ \nu \in \mathcal{N} $ such that $ y_{\varnothing} =  \iota_{*}(x_{U}) $ and  \normalfont{    (\ref{marketingrelation})}   holds for all $ \nu , \mu \in \mathcal{N}   $   satisfying   
     $ \nu \subset \mu  $.   
     \label{marketing1}    
\end{theoremx}

That $ N = \otimes'_{v} N_{v} $ 
means the following. 
For each $ v \in I $, there are  functorial $ \ZZ_{p}$-Mod valued mappings $ N_{v} $ on compact open subgroups of $H _{v} $   
and  there  are    elements $ x_{U_{v}} \in N_{v}(U_{v}) $ such that for any  compact open subgroup   $ W = \prod_{v } W_{v}    \in  \Upsilon_{H} $ that satisfies $ W_{v} = U_{v} $ for all but finitely  many $ v $,  $ N(W) $ equals  the restricted tensor product $ \otimes'_{v} N_{v}(W_{v}) $ 
taken with respect to $ x_{U_{v}} $.

In the case where special elements are taken to be fundamental cycles, the situation can be modelled by taking $ N(W) = \ZZ_{p} \cdot 1_{W} $ where $ 1_{W} $ denotes the   fundamental  class of the Shimura variety of level $ W  \in  \Upsilon_{H} $. Then $ N $ is trivially a restricted tensor product. 
In this case,  Theorem \ref{marketing1}  reduces to  verifying   certain      congruence     conditions between degrees of Hecke operators. Here we define 
the degree  of a double coset operator  $ T =  \ch(W h W')  $  to be $ |  Wh W'/ W' | $  and that of $ T_{*} $ to be $ | W\backslash W h W' | $, and extend this notion to linear combinations of such operators in the obvious way. Set  $ d_{v,i} = [ H_{v,i} :  V_{v,i} ]  $. Note that $ d_{v,i} $ divides the index $    [K_{v} : L_{v}] $. 

\begin{theoremx}[Corollary \ref{easyzeta1}]   Let $N$ be as above and  $ x_{U} = 1_{U} \in N(U)$. If for each $ v \in I$ and $ 1  \leq i  \leq  r_{v}  $, the degree  of   $ \mathfrak{h}_{v,i,*}     $ lies in  $ 
 d_{v,i}   \cdot      \ZZ_{p} $,  
 there exist classes  $ y_{\nu} \in M( K [\nu] )  $ for each $ \nu \in  \mathcal{N} $   such that $ y_{\varnothing} =  \iota_{*}(x_{U}) $ and  {\normalfont(\ref{marketingrelation})}  is  satisfied for all $ \nu \subset \mu $ in $ \mathcal{N}  $.    \label{marketing2}    
\end{theoremx}

There is a generalization of such congruence criteria that applies  
to   Eisenstein classes. 
For each $ v \in I $, let $ X_{v}$ be a locally compact Hausdorff totally disconnected topological  space endowed with a continuous right  action   $ X_{v}  \times  H_{v} \to X_{v} $. Let $ Y_{v} \subset X_{v} $ be a compact open subset invariant under $ U_{v} $.  Let $ X = \prod_{v } ' X_{v } $ be  the  restricted topological product of $ X_{v} $ taken  with  respect to $ Y_{v} $.  Then we get a  smooth  left action of $ H $ on the  so-called Schwartz  space $ \mathcal{S}_{X}   $ of all locally constant  compactly  supported  $ \ZZ_{p} $-valued  functions on $ X  $. For our next result, we assume that for each    $ W $ of the form $  \prod_{v \in I} W_{v} \in \Upsilon_{G} $,   we have   $ N(W) $ equals $   \mathcal{S}_{X}(W) $, the $ \ZZ_{p}$-module   of all $ W $-invariant functions in $ \mathcal{S}_{X}  $.  Then $ N $ is a restricted tensor product of $N_{v} $ with respect to $ \phi_{U_{v}} = \ch(Y_{v}) \in N_{v}(U_{v}) $  where $ N_{v}(W_{v}) $ for  a compact open subgroup $ W_{v} \subset H_{v}$   is  the set of all  $ W_{v} $-invariant  Schwartz  functions on $ X_{v} $.  
Given compact open subgroups  $ V_{v}, W_{v}  \subset H_{v} $  such that $ V_{v}  \subset W_{v}  $ 
and $ x \in X $, we denote by $ V_{v}(x,W_{v})   $   
the subgroup of $ W_{v}  $    generated by $ V_{v} $ and the stabilizer $ \mathrm{Stab}_{W_{v}}(x) $ of $ x $ in $ W_{v} $.   
\begin{theoremx}[Theorem 
 \ref{traceriteria}]        \label{marketing3} Let $  \phi_{U} = \otimes ' \phi_{U_{v}} \in N(U) = \mathcal{S}_{X}(U)  $. 
Suppose that for each $ v \in I $ and $ 1 \leq i \leq r_{v} $,    
\begin{equation} \label{marketing2cor} \left (\mathfrak{h}_{v,i,*}( \phi_{U_{v}}  ) \right   ) (x)  \in  [ V_{v,i}(x,H_{g,i}): V_{v,i}] \cdot \ZZ_{p} 
\end{equation}    for all $ x \in \supp    ( \mathfrak{h}_{v,i, *} (\phi_{U_{v}}  )       )     $.  Then there exist $  y_{\nu} \in M(K[\nu]) $ for all $ \nu \in  \mathcal{N} $  such that $ y_{\varnothing} = \iota_{*} ( \phi_{U} ) $     and   {\normalfont (\ref{marketingrelation})}  is satisfied for all $ \nu \subset \mu $ in $ \mathcal{N}  $. 
\end{theoremx}

If  $ X  $ is  reduced to a point $ \left \{ \mathrm{pt} \right \} $, 
one recovers   Theorem   
\ref{marketing2}  since for all $ v $ and $ i $,    $ V_{v,i}(\mathrm{pt},H_{v,i}) = H_{v,i} $ and the action of $ \mathfrak{h}_{v,i,*}   $ is via multiplication by  its degree. 
 
While it is conceivable to prove our main result in a more direct fashion (see Remark \ref{nozeta}), we have chosen to develop our approach from the point of view of specifying a  ``best possible test vector" that yields a solution to (\ref{marketingrelation}). Let us explain this. It is possible to recast  the  relations
 (\ref{marketingrelation}) in terms of  intertwining maps of  smooth representations of $ H \times G $ by passing to the inductive limit over all levels. More precisely,  let $ \widehat{\makebox*{$N$}{$N$}} $, 
$\widehat{\makebox*{$N$}{$M$}} $     denote the inductive limits of $ N(V) \otimes_{\ZZ_{p}} \QQ_{p} $, $ M(K') \otimes_{\ZZ_{p}}     \QQ_{p} $ over all levels $ V \in \Upsilon_{H} $, $ K' \in \Upsilon_{G} 
 $  with   respect  to  restrictions. 
These are naturally   smooth representations of $ H$, $ G $ respectively. Let $ \mathcal{H}( G )$ denote the $ \QQ_{p}$-valued  Hecke algebra of $ G $ with respect to a suitable Haar measure on $G$. We can  construct  a map $$ \hat{\iota}_{*} : \widehat{\makebox*{$N$}{$N$}} \otimes_{ \QQ_{p}  } \mathcal{H}(
G 
) \to   
\widehat{\makebox*{$N$}{$M$}}$$ 
of $ H \times G  $ representations with suitably 
 defined actions on the source and the target.     
One can then take an arbitrary finite sum of twisted pushforwards from $ N $ to $ M $ of classes at arbitrary  `local' levels of $ H_{v} $ and ask whether the element given by this sum satisfies the norm relation  (\ref{marketingrelation}), say for $ \nu = \varnothing $, $ \mu = \left \{ v \right \}$.  
In terms of the map $   \hat{\iota}_{*} $, this becomes a problem of specifying  a ``test vector'' in $ \widehat{N} \otimes_{\QQ_{p}} \mathcal{H}(G) $  that satisfies certain integrality properties and whose image under $ \hat{\iota}_{*}  $ equals $   \hat{ \iota   }      _{*}(x_{U} \otimes \mathfrak{H}_{v} ) $. This leads to a notion of \emph{integral test vector} given for instance in \cite[Definition 3.2.1]{gu21}, analogues of which also appear in several other recent    works.  If such an integral test vector lies in the $H_{v}$-coinvariant class of $ \hat{\iota}_{*}(x_{U} \otimes \mathfrak{H}_{v}) $, we refer to it as an  \emph{abstract zeta element}.      
\begin{theoremx}[Theorem  \ref{zetacriteria}]  An abstract  zeta element at $ v $ exists if and only if  the  norm relations {\normalfont(\ref{derivedmarketing})} hold for  
$ 1 \leq i \leq   r_{v}  $ up to $ \ZZ_{p}$-torsion.  
\end{theoremx} 
This result connects our approach to the one pursued in \cite{LSZ} (cf.\ \cite{gu21}), which seeks such integral test vectors by means of local zeta integrals. However, it   provides no mechanism on how one may find them in the first place. Our approach, on the other hand, pinpoints an essentially unique test vector in terms of the Hecke polynomial to check the norm relations with. Another key advantage of our approach over theirs is that ours is inherently integral, as no volume factor normalizations show up in the criteria above.  Crucially, it also has broader applicability, since  it remains  effective even in cases where the so-called ‘multiplicity one’ hypothesis fails to hold.

\subsection{Auxiliary results}  The  execution of our approach  hinges on  explicit  description of   the Hecke polynomials $ \mathfrak{H}_{v} $ and   their    twisted restrictions  $ \mathfrak{h}_{v,i} $. This requires among other things a description of left or right cosets contained in  double cosets of parahoric subgroups.   It is possible to exploit the   affine  cell decompositions   of flag varieties  to  specify  a    ``geometric"    set of representatives which makes  double coset manipulations a much more  pleasant task.       In \cite{Lansky},  Lansky derives such a  decomposition recipe for double cosets of parahoric subgroups of split Chevalley groups by studying the structure of the underlying Iwahori  Hecke  algebras.    Though the class of groups we are interested in is not covered by Lansky's results, the ideas therein are completely adaptable. 
We generalize Lansky's recipe by  axiomatizing it in the language of  generalized   Tits systems as follows. 

Let $ \mathcal{T} =  (G,B,N) $ be a Tits system and let $ \varphi : G  \to \tilde{G} $ be  a  $(B,N)$-adapted inclusion. Let $ W = N / B $ be the Weyl group of $     \mathcal{T}   $  and   $  S $ the generating  set of reflections in $ W $ determined by $ \mathcal{T} $. We assume that $ B s B / B $ is finite for each $ s \in S $. Then $ B w B / B $ is finite for each $ w \in W $. For each $ s \in S $, let $ \kay_{s} \subset G $ denote a set of representatives for $ BsB/B $.  For $ X  \subset  S $, let $ W_{X} $ denote the subgroup of $ W $ generated by $ X $ and $ K_{X} = B W_{X} B \subset G  $ the corresponding parabolic subgroup.  Let $ \hat{B} $ be the normalizer of $ B $ in $ \tilde{G} $, $ \Omega = \hat{B} / B $ and $ \tilde{W}  =  W \rtimes \Omega $ be the extended Weyl group. 
For any $ X $, $ Y \subset S $,  let $ [W_{X}   \backslash  \tilde{W}   /  W_{Y} ]  $ denote the set of all $ w \in \tilde{W} $ whose length  among elements of $ W_{X} w W_{Y} $ is  minimal. For   a  reduced  decomposition  $ w = s_{1} \ldots s_{m} \rho $ where $ s_{i} \in W $, $ \rho \in \Omega $,   let  $ \kay _{w} = \kay_{s_{1}} \times \ldots \kay_{s_{m}}  $, $\tilde{\rho} \in \hat{B}  $ a lift of $ \rho $  and $ \mathcal{X}_{w} : \kay_{w} \to G $ the map which sends $ \vec{\kappa} = (\kappa_{1}, \ldots, \kappa_{m})  \in \kay_{w} $ to the product $ \kappa_{1} \ldots \kappa_{m} \tilde{\rho} $.  
Then the image of $ \mathcal{X}_{w}  $ modulo $ B $ only depends on the element $  w  $.     
   
\begin{theoremx}[Theorem     \ref{BNrecipe}] 
For any $ X , Y \subset S $ and $ w \in [ W_{X}  \backslash  \tilde{W}  /  W_{Y}   ]  $, we have $$ K_{X} w K_{Y} =   \bigsqcup _{\tau}     \bigsqcup_{  \vec{\kappa} \in \kay_{\tau  w } }  \mathcal{X}_{ \tau w}  (  \vec{\kappa}  ) K_{Y}   $$
where $ \tau \in W_{X}  $ overs minimal length representatives of $ W_{X}  /   (  W_{X} \cap w W_{Y} w^{-1}   )     $.
\label{BNmarketing}    
\end{theoremx}

The images of the maps $ \mathcal{X}_{w} $ defined   above  can  be  viewed  as  
 affine  generalizations  of   the more  familiar \emph{Schubert cells} one encounters in the geometry of  flag varieties. See  \S   \ref{titsmotivation} for a discussion.        
In practice, the recipe is applied by taking  $  B $ to be the Iwahori subgroup of the reductive group at hand, $ W $ the affine Weyl group  and $ \tilde{W} $  the Iwahori Weyl group. The recursive nature of Schubert cells $ \mathcal{X}_{w}$ proves to be particularly advantageous    for  computing the twsited restrictions of Hecke polynomials.

\subsection{Other approaches} The framework presented here focuses on ‘pushforward-style’ constructions in cohomology, motivated by period integrals where cusp forms on a larger group are integrated against (some gadget on)  a smaller group. Recently, a new ‘pullback-style’ approach has been proposed by Skinner and Vincentelli (\cite{Vincentelli}), opening up the possibility of using ``potentially motivic" classes such as the Siegel Eisenstein class constructed in   \cite{FaltingsEisenstein}.   
Another approach, developed by Eric Urban, uses congruences between Eisenstein series to intrinsically construct Euler system classes in Galois cohomology \cite{UrbanII}, \cite{Urban}.   Both of these approaches differ fundamentally from ours and do not seem applicable to the  various  settings that can be explored using our method, e.g., \cite{Siegel2}.

For Euler systems of fundamental cycles, an earlier approach developed by Cornut and his collaborators also aims to prove norm relations in the style of (\ref{marketingrelation}). This approach involves studying the Hecke action on the corresponding  Bruhat-Tits buildings, e.g., see \cite{Cornut}, \cite{JetchevBoum}, \cite{Reda}. However, it was observed  in \cite{explicitdescent} that the Hecke action used in these works is not compatible with the geometric one. Cornut has informed us however  that this issue can be resolved. It is our expectation  that insights from studying actions on Bruhat-Tits buildings  may  provide a more conceptual explanation for  computations in our  own work.

\subsection{Organization}   This  article  is  divided into two  parts, where the first develops our approach  abstractly   and the second executes it in concrete situations.  We briefly outline the contents of each  section  within both.  

In  \S \ref{Setup}, we revisit and expand upon the abstract formalism   of   functors      developed  in \cite[\S 2]{Anticyclo}. Our motivation  here is  partly  to develop a  framework for Hecke operators that works well in the absence of Galois descent.  We prove several basic results that normally require a passage to inductive limits.    We also introduce the notion of \emph{mixed Hecke correspondences} that allow us to relate   double  coset  operators   of a  locally  profinite group  to those of a closed subgroup.  These play a crucial role in establishing the aforementioned   norm  relation     criteria. 
We  end  the  section by outlining how the formalism applies to Shimura varieties even in the absence of Milne's (SV5) axiom, which was assumed in \cite{Anticyclo}.   

In \S \ref{abstractsec}, we develop our machinery from the point of view of  abstract zeta elements.  We  strive for maximum possible generality in defining these and establish a structural result  in Theorem \ref{zetacriteria}. This  
allows    us to focus attention on a specific type of such elements; the one given by  twisted  restrictions of Hecke polynomials. This comes with the added benefit of eliminating all normalizations by volume factors, giving a highly canonical crtieria that we are  able to upgrade in \S \ref{handling torsion} to finite levels.  To handle Euler systems of Eisenstein classes, we have included an axiomatic study of traces in arbitrary Schwartz spaces of totally disconnected spaces   in \S \ref{Schwartz}. 
Finally,  a toy example of CM points on modular curves is included  to illustrate our machinery in a simple case.    

In \S \ref{LfactorHeckesection}, we collect several  facts about  Satake transforms and Hecke polynomials.  Everything  here   can  be   considered     well-known   to   experts        and  we make   no  claim      of     originality. 
We have  however  chosen to   include  
proofs of a few results, partly because  we could not find a  satisfactory reference that covers the generality we wish to work in  and partly because   conventions    seem to differ from one reference to  another. Some of these results 
play a crucial role in our  computations. A few results are included (without proofs) to provide a check on our computations. For instance, certain congruence properties of Kazhdan-Lusztig polynomials are  not necessary for the computations done in this article  but are invoked  in \cite{Siegel1}.

In \S \ref{titsdecosec},  we develop from scratch  another important   ingredient    of our approach. After    justifying all the necessary facts we need on (generalized) Tits systems, we prove a recipe for decomposing  certain    double cosets following  the method   of    Lansky.   We briefly review some facts from Bruhat-Tits   theory        that allow us to apply this recipe in   practice.   The results of this section also complement the content of \S \ref{LfactorHeckesection} in the sense that we can often  use the decomposition recipe 
to efficiently invert  various Satake transforms for Hecke polynomials, though we note that this step can often be skipped.

Part II of this article  is devoted to examples. Its primary purpose is to provide concrete evidence that the abstract criteria proposed in Theorem \ref{marketing1} does hold in practice. We study the split case of unitary Shimura varieties $ \mathrm{GU}_{1,2m-1} $ for arbitrary $ m $ in \S \ref{GLnLfactorsection} and the inert case for $ m = 2  $ in \S \ref{GU4Lfactorsection}. The inert case for general $   m    $ is the subject of a later work. In both these scenarios, we  show that along anticyclotomic towers, the criteria of Theorem \ref{marketing2} holds for pushforwards of fundamental cycles of products of two sub-Shimura varieties. The split case of our results for these  Shimura varieties   strengthens  \cite[Theorem 7.1]{Anticyclo} and also applies  to   certain     CM versions of these varieties.    An interesting  observation in the split case is that our criteria fails to hold if one considers the full abelian tower   (Remark  \ref{anticyclobehavior}). This is consistent with the  well-documented   observation that Heegner points do not ``go up"   cyclotomic extensions.   Another  interesting  observation  is that  the degrees of the various restrictions of the Hecke polynomials turn out to be $ q $-analogues of the binomial expansion $ (1-1)^{k} $ for $ k $ a positive integer. This alludes to an intimate relationship between  twisted   restrictions of Hecke polynomials and    factors of  Satake  polynomials.       

In \S \ref{gsp4section}, we study the case  of genus  two Siegel modular varieties. Here we establish that the criteria of Theorem  \ref{marketing2cor} holds for pushforwards of cup products of Eisenstein classes for modular     curves.  This yields the  ``ideal"   version of the horizontal norm relations alluded to  in  \cite[p.671]{LSZ}. An interesting observation here is that of the two twisted restrictions of the  spinor   Hecke polynomial, one is essentially the standard $ \GL_{2} $-Hecke polynomial  for a diagonally embedded copy of $ \GL_{2} $ and the corresponding trace computation is reminiscent of \cite[Proposition 1.10]{Colmez2002-2003} for Kato's Euler system.

\subsection{Acknowledgements}This article is based on the author’s thesis work done at Harvard University. The author is sincerely grateful to Barry Mazur for all his advice and encouragement; Christophe Cornut for valuable feedback; Antonio Cauchi for bringing the author’s attention to several new applications of this work and for carefully explaining an unfolding calculation; and Wei Zhang for several useful conversations. Many ideas of this article have their roots in an earlier joint work, and the author is thankful to his collaborator Andrew Graham for their continued discussions. The softwares MATLAB\textsuperscript{®} and SageMath proved particularly helpful in carrying out and verifying numerous computations that arose in the course of this project.  A part of this work was also completed when the author was affiliated with the University of California, Santa Barbara. The author extends his sincere gratitude to Francesc Castella, Zheng Liu, and Adebisi Agboola for their mentorship and support.

%% file: Preliminaries.tex
  
\partdivider{}{Part 1. General theory}  
\section{Preliminaries}     \label{Setup}
In this section, we recall   and   expand   upon     the  abstract  formalism of  functors  on compact  open   subgroups of locally profinite groups  as  introduced in  \cite[\S 2]{Anticyclo}\footnote{which in turn was inspired by \cite{loe})} which we will use in \S \ref{abstractsec} to study norm relation problems encountered in the settings of Shimura varieties.      We   note     that   a   few   conditions  of \cite{Anticyclo}   have  been  relaxed for generality while others related to vertical norm relations have been dropped  completely since they do not pertain to the questions addressed in this article\footnote{though they are needed again in \cite{Siegel1})}.    Note also that   the terminology in a few  places   has been modified to   match  what seems to be the standard in pre-existing   literature on such  functors  e.g.,    \cite{Ulc}, \cite{Bley}, \cite{Webb}, \cite{may}, \cite{Dress}, \cite{Green}.  The material in this section can however be read independently of all of these  sources.

One of the reasons    for developing this formalism further   (besides convenience and generality)  is to address the failure of Galois descent in the  cohomology of Shimura varieties with integral coefficients.  This failure  in particular means that the usual approach to Hecke operators   as seen in  the theory of smooth representations is no longer available as  cohomology at finite level can no longer be recovered after passage to limit by taking invariants. See  \cite[\S 4.2.2] {highercoleman} for a discussion of a similar issue  that arises when defining cohomology with support conditions.    
In our development, the role of Galois descent is  primarily  played by what is known in literature as  \emph{Mackey's decomposition formula} which was used in \cite{loe} to study vertical norm relations. This formula  can serves as a replacement for Galois descent and allow us to derived many results that hold when coefficients are taken  in a field.        One may thus  view this  formalism as  an integral counterpart of the ordinary theory of     abstract     smooth representations.

\subsection{RIC functors} \label{basicsetup} For $ G $ a locally profinite group,    let $\Upsilon = \Upsilon_{G}  $ be a non-empty    collection of compact open subgroups of $ G $ satisfying the following conditions
\begin{itemize} [before = \vspace{\smallskipamount}, after =  \vspace{\smallskipamount}]    \setlength\itemsep{0.1em}  
\item[(T1)] For all $g \in G $, $K \in \Upsilon$, $gKg^{-1}  \in   \Upsilon$.
\item[(T2)] For all $K, L \in \Upsilon$,   there exists a $ K ' \in \Upsilon $ such that $ K ' \triangleleft K $, $ K'  \subset L   $.     
\item [(T3)] For all $K, L \in \Upsilon$,  $K \cap L \in \Upsilon$.
\end{itemize} 
Clearly the set of all compact open subgroups of $ G $ satisfies these properties.  Let  $ \mathcal{F} $ be any collection of compact open subgroups of $ G $.  We let $ \Upsilon(\mathcal{F} ) $ denote the collection of all compact open subgroups of $ G $ that are obtained as finite intersections of conjugates of elements in $ \mathcal{F} $. We refer to $ \Upsilon(\mathcal{F})  $ as the     collection \emph{generated by $ \mathcal{F} $}. 
\begin{lemma}   \label{UpsilonLemma}     For any $ \mathcal{F} $  as  above,  the collection   $ \Upsilon( \mathcal{F}  )      $ satisfies satisfies {\normalfont(T1)-(T3)}. In particular, any  collection that satisfies {\normalfont(T1)-(T3)} and contains $ \mathcal{F} $ must contain $ 
   \Upsilon(\mathcal{F} )$.       
\end{lemma}

\begin{proof}   Axioms (T1) and (T3) are automatic for $ \Upsilon(\mathcal{F}) $  and we need to  verify (T2). Let $ K , L \in   \Upsilon (\mathcal{F} )   $.  Pick a (necessarily finite)  decomposition $ K L  =  \bigsqcup_{\gamma} \gamma  L 
$ and define  $ K ' : = K \cap \bigcap   \nolimits    _{\gamma} \gamma  L  \gamma^{-1} \in    \Upsilon(F)  $    Then $ K ' \triangleleft K $ and $ K' \subset L $ and so (T2) is satisfied.    
\end{proof}    
To any $ \Upsilon $ as above, we associate a \emph{category of compact opens}   $   \mathcal{P}(G)    =    \mathcal{P}(G, \Upsilon)$ whose objects are elements of $\Upsilon$ and whose morphisms are given by  $ \mathrm{Hom}_{\mathcal{P}(G)} ( L, K ) =  \left \{ g \in G \, | \,  g^{-1} L g  \subset  K   \right \}  $ for $L , K \in \Upsilon $ with 
 composition  given by  $$ (L \xrightarrow{ g } K  ) \circ  (  L ' \xrightarrow {h } L ) =  (L' \xrightarrow{ h } L \xrightarrow{g} K )  =  ( L '  \xrightarrow{hg} K ) . $$  A morphism  $ (L \xrightarrow{ g } K ) $ will be denoted by $ [g]_{L,K} $, and if $  e   $ denotes the identity of $ G $, the inclusion $ (L\xrightarrow{e}K) $ will also denoted by $ \pr_{L,K} $. Throughout this section, 
 let  $ R $ denote  a commutative ring with identity.   \begin{definition}   
\label{RICdefinition}     An   \emph{RIC functor $ M $ on $ (G, \Upsilon ) $ valued in    $ R $-\text{Mod}}  is a pair of covariant functors $$  M ^ { * } :  \mathcal{P}(G) ^ { \text{op} }  \to   R\text{-Mod}  \quad  \quad   \quad        M  _{  * } :  \mathcal{P}(G)  \to  R\text{-Mod}  $$
satisfying   the   following   three   conditions:         
\begin{itemize}  
\item [(C1)] $ M^{*}(K) = M_{*}(K)$ for all $ K  \in   \Upsilon $. We will denote the common $ R $-module  by $ M(K) $.
\item [(C2)]For all $ L, K   \in \Upsilon $ such that $ g ^{-1} L g = K $, $$ (L \xrightarrow{g} K)^{*} =  ( K \xrightarrow{g^{-1} } L )_{*}  \in  \mathrm{Hom}(M(K),M(L)) . $$
Here, for $ \phi  \in   \mathcal{P} (   G   )  $ a morphism, we denote $ \phi_{*}  : = M_{*}(\phi) $, $ \phi^{*}  : = M^{*}(\phi) $. 
\item [(C3)] $  [ \gamma ]_{K,K,*} : M(K) \to M(K) $ is the identity map  for all $ K \in \Upsilon $, $ \gamma \in K $.
\end{itemize}    
We refer to the maps $ \phi^{*} $ (resp.,   $ \phi_{*} $)    in (C2) above as the \emph{pullbacks}  (resp., \emph{pushforwards}) induced by $ \phi $. If moreover $ \phi =  [e]   $, we also refer to $ \phi^{*} = \pr^{*} $ (resp.,     $ \phi_{*} = \pr_{*} $) as  \emph{restrictions} (resp.,  \emph{inductions}).    We  say that a functor $ M $  is   $ \ZZ $-\emph{torsion  free}  if $ M(K) $ is   $  \ZZ   $-torsion free  for all $ K \in \Upsilon $. Moreover, we  say that  $ M $ is   
\begin{itemize}  
\item [(G)] $ \emph{Galois} $ if for all $ L , K \in \Upsilon $, $ L \triangleleft K $, $$ \pr_{L,K}^{*} : M (K) \xrightarrow{\sim} M(L)^{K/L} . $$
Here the (left) action $ K/L \times M(L) \to M(L) $ is given by $ (\gamma, x) \mapsto  [ \gamma ]^{*}_{L,L} (x)   $.
\item [(Co)]  \emph{cohomological} if for all $ L , K \in  \Upsilon $  with   $  L   \subset K $,  
 $$ ( L  \xrightarrow {  e}  K )   _ { * }   \circ   (  L  \xrightarrow{ e  }  K    ) ^ { *  }  =  [K:L] \cdot   (  K  \xrightarrow {  e  }       K   )^{*} .         $$
That is,   the composition is multiplication by index $ [K:L] $ on  $ M(K) $.   
\item [(M)] \emph{Mackey}   
if for all $K, L, L' \in \Upsilon$ with $L ,L' \subset K$, we have a commutative diagram 
\begin{equation}   \label{Mackeydiagram}    
 \begin{tikzcd}[column   sep = large]
 \bigoplus_{\gamma} M(L_{\gamma}) \arrow[r,"{\sum \pr_{*}}"] & M(L) \\
 M(L') \arrow[r,"{\pr_{*}}",swap] \arrow[u,"{\bigoplus  [\gamma]^{*}}"] & M(K) \arrow[u,"{\pr^{*}}",swap] 
\end{tikzcd}
\end{equation}   
where the direct sum in the top left corner is over a fixed choice of coset representatives $\gamma \in K $ of  the  double   quotient     $ L \backslash K / L'$ and $L_{\gamma} =  L  \cap  \gamma L' \gamma^{-1}    \in    \Upsilon       $. The  condition  is then satisfied by  any such  choice   of     representatives of $ L \backslash  K /  L '  $.     
\end{itemize}   
If $ M $ satisfies both (M) and (Co), we will say that $M  $ is \emph{CoMack}.  If $ S $ is  an   $ R $-algebra, the mapping $ K \mapsto M(K) \otimes _{R}  S  $ is a $ S $-valued  RIC functor, which is cohomological or Mackey if $ M $  is   so.  
\end{definition}

In what follows, we will often say that $ M : G \to  R $-Mod is a functor when we mean to say that $ M $ is  a RIC functor on $ (G, \Upsilon) $ and suppress $ \Upsilon $ if it  is clear from  context.
\begin{remark} The acronym  RIC stands  for restriction, induction, conjugation     and   the terminology is  borrowed from  \cite{Ulc}. Cf.\  \cite[Definition 1.5.10]{Neu} and \cite{Dress}.    
\end{remark}    

\begin{definition} A \emph{morphism} $ \varphi : N  \to   M   $ of RIC   functors  is  a  pair  of  natural  transformations $  \varphi    ^ { * } :  N ^ { *  } \to  M^ {*  }  ,   \varphi_{*} : N _{*}  \to  M    _{*} $
such that $ \varphi_{*} ( K )  = \varphi^{*} ( K ) $ for all $ K \in  \Upsilon  $.  We   denote  this   common  morphism   by  $ \varphi(K) $.  The category of $ R $-Mod valued RIC functors on $ (G,\Upsilon) $ is denoted $ \text{RIC}_{R}(G, \Upsilon) $ and the category of CoMack functors by $ \mathrm{CoMack}_{R}(G, \Upsilon) $. 
\end{definition}

We record some  straightforward   implications.    Let $ M : \mathcal{P}(G,\Upsilon) \to R\text{-Mod} $ be  a RIC functor. 

\begin{lemma}   \label{Mackeyprime}     The functor  $ M $ is  Mackey if and if only  if   for all $K, L, L' \in \Upsilon$ with $L ,L' \subset K$, we have a commutative diagram  \begin{equation}   \label{Mackeyprimediagram}        \begin{tikzcd}  [column sep = large]  \bigoplus_ { \delta } M ( L _ {\delta } '  )   \arrow[r,  "{ \sum [\delta ]_{*}}"]   &  
    M ( L  )  \\ 
    M ( L' )  \arrow[r, "{\pr_{*}}"]   \arrow[u, "{\bigoplus\,\pr^{*}}"] &  M (K)   \arrow[u,"{\pr^{*}}", swap]  
    \end{tikzcd}
\end{equation} where the direct sum in the top left corner is over a fixed choice of coset representatives $  \delta  \in K $ of $  L ' \backslash K / L $ and $L_{\delta } ' = L' \cap  \delta  L   \delta    ^{-1}   \in    \Upsilon       $. 
\end{lemma}    

\begin{proof}  If $ \gamma \in K $ runs     over representatives of $  L \backslash K / L ' $,  then  $ \delta  =  \gamma ^{-1} $ runs      over a set of representatives    for $ L ' \backslash K / L $. For each $ \gamma \in K $, we have a commutative diagram
\begin{center} 
\begin{tikzcd}  & M(L_{\delta} ' )   \arrow[r, "{[\delta]_{*}}"]     &  M(L) \\ 
M(L')   \arrow[ru,   "\pr_{*}" ]      \arrow[r, "{[\gamma]^{*} }",   swap]   &    M(L_{\gamma} )   \arrow[ru, swap,  "\pr_{*}"]    
\end{tikzcd} 
\end{center}  
where $ \delta =   \gamma   ^{-1}   $.  Indeed, the two triangles obtained by sticking the arrow $ M(L_{\gamma}) \xrightarrow{  [  \gamma   ]_{*}   } M(L'_{\delta})  $ in the diagram above are  commutative.  From this,  it is straightforward to see that  diagram (\ref{Mackeydiagram}) commutes if and only if  diagram  (\ref{Mackeyprimediagram}) does.      
\end{proof}

\begin{remark}  We will refer to the commutativity of the diagram     (\ref{Mackeyprimediagram})    as axiom (M$'$).   
\end{remark}

\begin{definition}   \label{injectiverestriction} We say that $ M $ has    \emph{injective restrictions}       if   $ \pr_{L,K} ^ {* }  :  M ( K )  \to   M ( L )   $       are    injective  for  all  $ L , K  \in   \Upsilon $, $ L  \subset    K   $.    
\end{definition}   

\begin{lemma}   Suppose $ M $ is either Galois or cohomological and $ \ZZ $-torsion free. Then $ M $ has injective restrictions. 
\end{lemma}

\begin{proof} Let $ L , K \in \Upsilon $ with $ L \subset K $.  If $ M $ is Galois, pick $ K ' $ such that $ K' \triangleleft K $, $ K' \subset L $ (using axiom (T2)). Then $  \pr_{K', K}^{*}   =     \pr_{K', L } ^ {  * }  \circ  \pr_{L, K} ^ { *  } : M ( K )  \to  M (  K '  )     $ is injective by definition which implies the same for $ \pr_{L,K}  ^ { * }  $. If $ M $ is cohomological, then  $  \pr_{L,K,*}  \circ  \pr_{L,K}^{*} = [K:L]  $ which is injective if $ M(K) $ is $ \ZZ $-torsion free which again implies the same for $ \pr_{L,K}^{*} $.    
 \end{proof}

\begin{lemma}         \label{RICtracemakey}    Suppose $ M $ is Mackey.   Let $ L, K \in  \Upsilon $  with $ L \subset K $ and let $ K' \in \Upsilon $ be such that  $ K'  \triangleleft K $ and $  K' \subset L $\footnote{such a $K'$ exists by (T2)}.    Then 
$    \pr_{K' , K } ^ { *   }  \circ \pr  _ { L , K , * }    =    
\sum_{ \gamma }  [\gamma]_{K'  ,     L }^{*}      
$ where $ \gamma $ runs over $ K / L   $.  

\end{lemma}

\begin{proof} Since $ K ' \triangleleft K $ and $ K' \subset L $,     the  right  multiplication action of $ K' $  on $  L  \backslash  K  $ is trivial i.e.\    $    L  \backslash K /  K '  = L   \backslash K   $.  By  axiom  (M$'$) obtained in Lemma   \ref{Mackeyprime},   we   see    that    
\begin{center} 
\begin{tikzcd}  [column sep = large]  \bigoplus_ {  \delta    }  M   (   K' 
)   \arrow[r,  "{ \sum [  \delta   ]_{*}}"]   &  
    M (  K '   )  \\ 
    M ( L )  \arrow[r, "{\pr_{*}}"]   \arrow[u, "{\bigoplus\,\pr^{*}}"] &  M (K)   \arrow[u,"{\pr^{*}}", swap]  
    \end{tikzcd}
    \end{center} 
    where $ \delta $ runs over   $  L  \backslash  K   $.  Since $ [\delta]_{K', K ' , *}  = [\delta^{-1}]_{K', K'}  ^ { *   }    $, we may replace $ \delta $ with $ \delta = \gamma ^{-1} $. Then $ \gamma   \in K  $ runs over $ K / L $  as $ \delta $ runs over $ L   \backslash  K  $ and  the  claim follows.   
\end{proof}    

\begin{corollary}     \label{RICComackGalois}     Suppose $ R $ is a $ \QQ $-algebra. Then $ M $ is Galois if it is CoMack.  
\end{corollary}

\begin{proof}    For any $ L , K, \Upsilon $ with $ L \triangleleft K $, $ \pr_{L,K}^{*} : M(K) \to M(L) $ is injective by Lemma  \ref{injectiverestriction}. If $ x \in M(L) $ is $ K $-invariant, then $  \pr_{L,K}^{*} \circ \pr_{L,K,*}(x)  =   \sum_{ \gamma \in K / L} [\gamma]^{*}_{L,K} (x) = [K:L] x $ by Lemma \ref{RICtracemakey} and so $ \pr_{L,K}^{*} (y) = x $ if $ y = [K : L] ^{-1} \pr_{L,K,*}(x) \in M(K) $ . Thus    $ \pr_{L,K}^{*} $ surjects onto $ M(L)^{K/L} $.    
\end{proof}

It is clear how to define the direct sum and tensor product of functors on  finitely  many  groups $ G_{1} , \ldots, G_{n} $ to obtain a functor on $ G_{1} \times \cdots \times G_{n} $. A more involved construction is that of restricted tensor products which we elaborate on  now. Say for the rest of this subsection only   that $ G = \prod_{ v \in I}' G_{v } $ is a restricted   direct 
product of locally profinite groups $ G_{ v } $ with respect to compact open subgroups $ K_{ v } $ given for each $ v  \in    I    $.     For $   \nu    $ a finite subset of $ I $, we denote $ G_{\nu}  : = \prod_{v \in  \nu } $, $ G^{\nu} :   = G/G_{\nu} $ and similarly for $ K_{\nu} $, $K ^{\nu} $.  For each $ v  \in I $, let $ \Upsilon_{ v } $ be a collection of compact open subgroups of $ G _  { v } $ that satisfies (T1)-(T3) and which contains $ K_{v } $.  Let $ \Upsilon_{I} \subset  \prod_  {v  \in I    } \Upsilon  _  {v  } $ be the collection of all subgroups of the form $ L_{\nu} K^{\nu} $ where  $ \nu  $  is   a   finite   subset of $ I $ and    $ L_{\nu} \in \prod_{ v \in \nu} \Upsilon_{v} $.   Then $ \Upsilon_{I} $ satisfies (T1)-(T3) and contains $ K $.  If  $   L \in \prod_{v \in I} \Upsilon $, we denote $ L_{v} $ its component group at $ v $.    

\begin{definition} Let $ N_{v} : \mathcal{P}( G_{v}   ,  \Upsilon_{\ell} )    \to R\text{-Mod} $ be a RIC functor  and let $ \phi_{K_{v}} \in N_{v}(K_{v}) $ for each $ v \in I $.  The  \emph{restricted     tensor product}  $ N = \otimes_{v}' N_{v}$ with respect to $ \phi_{K_{v}} $  is the RIC functor $ M : \mathcal{P}(G , \Upsilon_{I}) \to  R\text{-Mod} $ given by $ L \mapsto \otimes'_{v} N(L_{v}) $ where $ \otimes '_{v} $ denotes the restricted tensor product of $ R $-modules $ N(L_{v}) $ with respect to $ \phi_{K_{v}} $. 
\end{definition}  
We elaborate on the definition above. Fix $ L \in \Upsilon_{I} $ and  write $ L = L_{\nu} K^{\nu} $. For each finite subset $ \mu $ of $ I $ with $ \mu \supset \nu $, denote $ N_{\mu} : =  \otimes_{v \in \mu} N_{v} (L_{v} ) $ the usual tensor product of $ R $-modules.   If $ \mu_{1} \subset \mu_{2} $ are two such sets, there is an induced map $ N_{\mu_{1}}    \to N_{\mu_{2}}  $ of $ R $-modules that sends $ x \in N_{\mu_{1}} $ to $ x \otimes \bigotimes_{v \in \mu_{2} \setminus \mu_{1}} \phi_{K_{v}} $.  Then $ N(L) =  \varinjlim _{\mu}   N    _{\mu}   $   where  the  inductive  limit is  over the directed  set of all finite subsets $ \mu $ of $ I $ that contain $  \nu    $.

\subsection{Inductive Completions}    
Let $ (G, \Upsilon) $ be as in   \S \ref{basicsetup}. The category $  \mathrm{CoMack}_{R}(G ,\Upsilon) $ is closely related to the category of smooth $ G $-representations. We show that
when $ R $ is a field and $ \Upsilon $ is the collection of all compact open subgroups of $ G $, there is an equivalence between the two. When $ R $ is not a field however, axiom (G) can fail and the former category requires a more careful treatment. 

\begin{definition}  \label{Mtrivdefinition} Let $ \pi $ be a left module over $ R[G] $. We say that $ \pi $ is a  \emph{smooth   representation} of $ G $ if for any   $ x \in \pi $, there is a compact open 
 subgroup   $ K \subset G $  such that $ x $ is fixed under the (left) action of $ K$. A  \emph{morphism} of smooth representations is a $ R $-linear map respecting the $ G $-actions. The category of smooth representations of $ G $ is denoted $ \mathrm{SmthRep}_{R}(G) $. 
\end{definition}  

Suppose $ \pi \in \mathrm{SmthRep}_{R}(G) $. For $ K \in \Upsilon $, let $ M_{\pi} : G  \to R\text{-Mod} $ be the functor given by $ K \mapsto \pi^{K} $. For $ g \in G $ and $ (L \xrightarrow{g} K ) \in \mathcal{P}(G,\Upsilon) $,    let 
\begin{alignat*}{4}  [g]  ^ { * } : M(K) &   \to M(L)  &\hspace{0.8in}  [g]_{ *} :  M(L)   &  \to M(K) \\  
 x  & \mapsto g \cdot x  &\hspace{0.7in}   x  &  \mapsto \sum  \nolimits _{\gamma \in K / g^{-1} L g }  \gamma g ^{-1} \cdot x 
\end{alignat*}
Here, $ g \cdot x \in \pi $ in the mapping on the left  above is indeed a well-defined   element  of   $ M(L) $ as it is invariant under $ L \subset gKg^{-1} $ and similar remarks apply to the expression on the right above. In particular, the map $ [e]^{*}_{L,K} : M(K) \to M(L) $ is the   inclusion $  \pi^{K}  \hookrightarrow \pi^{L} $.  The following is then straightforward.    
\begin{lemma} The mapping $ M_{\pi} $ is a RIC functor that is CoMack and Galois. \end{lemma}

\begin{definition}    We refer to  $ M_{\pi} $  as  the \emph{RIC functor associated to $ \pi $}. If $ \pi = R $ is the trivial representation, we denote the associated functor by $  M_{\mathrm{triv}} $ and  refer to  it as the  \emph{trivial functor}.      
\end{definition}   

\begin{definition}  \label{Inductivecompletion} Let $ M : G  \to  R $-Mod  be a functor.  
The   \emph{inductive  completion} $ \widehat{M} $ is defined to be  the  limit  $ \varinjlim_{K \in \Upsilon}  M ( K )  $ where the limit is taken over all    restriction  maps.     We let $ j_{K}  : M ( K )  \to   \widehat{M} $ denote the natural map.   
\end{definition}    There is an induced smooth   action  $  G \times \widehat{M}  \to  \widehat{M} ,  (g,x) \mapsto g\cdot x $ where $ g \cdot x $ is defined as follows. Let  $ K   \in   \Upsilon    $, $ x_{K} \in M(K) $ be such $ j_{K}(x_{K}) = x $. Then $  g \cdot x $  is defined to be the image of $ x_{K} $ under the composition  $$  M(K)  \xrightarrow{ [g]^{*} }   M (  g K  g ^ { - 1} )  \to   \widehat{M} .  $$     It is a routine check that is well-defined. The action so-defined is smooth  as the image of  $ j_{K}  :  M(K)  \to  \widehat{M}(K)  $ is contained in the $ K $-invariants $ \widehat{M}^{K} $.   If   $ M$ is  also Galois, $ j_{K} $  identifies $ M(K) $ with  $  \widehat{M}^{K} $.
Moreover    if $ \varphi :  M \to N $ is a  morphism of functors, the induced map $ \widehat{\varphi} : \widehat{M} \to \widehat{N}     $ is $ G $-equivariant.

\begin{lemma}   \label{cohoinj}     

Suppose $ M $ is   cohomological.  Then  $  \ker( j_{K}  ) $ is contained in $   M(K)_{\ZZ\text{-}\mathrm{tors}}     $. In particular, if $ R $ is a field of  characteristic  zero,  $ j_{K} $ is injective.     
\end{lemma}

\begin{proof}   Let  $  x \in   \ker (   j_{K} ) $.  By  definition,  there exists $ L \in \Upsilon $, $ L \subset  K $ such that $ \pr_{L,K}^{*} (x) = 0  $.  Since $ \pr_{L,K,*} \circ \pr _{ L, K }^{*}  = [ K : L ] $, we must have $ [K:L] \cdot x = 0 $.  
\end{proof}    

The following result  seems originally due to \cite{Yoshida} for finite groups.         

\begin{proposition}   Let $ R $ be a $ \QQ $-algebra and $ \Upsilon $ the collection of all compact open subgroups of $ G  $.     Then the  functor $   \mathrm{SmthRep}_{R}(G)   \to   \mathrm{CoMack}_{R}(G  ,  \Upsilon    ) 
$ given by $ \pi \mapsto  M_{\pi} $ 
induces  an  equivalence of categories with (quasi)    inverse given by $ M \mapsto \widehat{M} $. 
\end{proposition}    

\begin{proof}  By Lemma  \ref{RICComackGalois}, any CoMack functor valued in a $ \QQ $-algebra  is Galois and therefore one can recover a functor $ M $ from the representation $  \widehat{M} $.  Similarly, $ \varinjlim_{ K \subset G}  \pi^{K} = \pi $ by smoothness of $ \pi  $.     
\end{proof}    
\subsection{Hecke operators} 
\label{heckeoperatorsubsection}    

A smooth representation comes equipped with an action of algebra of measures known as \emph{Hecke algebra}. In this subsection, we briefly review the properties of this action and fix conventions.   For  background material on Haar measures and further reading,  the reader may consult \cite[Ch.\  1, \S 3]{Vigneras}.     
Let $ (G, \Upsilon)$ be as in \S   \ref{basicsetup}.

\begin{definition}   \label{Heckealgebradefinition}     Let $  \mu $ be a left invariant Haar measure on $ G $  valued in $ R $\footnote{e.g.,  if $ \mu $ is $ \QQ $-valued and $ R $ is a $ \QQ $-algebra}  and let $ K \in \Upsilon $.    
The   \emph{Hecke algebra} $ \mathcal{H}_{R}( K  \backslash  G /  K  ) $ \emph{of level $ K$} is defined to be the convolution algebra locally constant $ K $-bi-invariant functions valued in $ R   $.  The convolution product is denoted by $ * $.  
The  \emph{Hecke algebra of $ G $} over $ \Upsilon $  is   defined  to  be     $ \mathcal{H}_{R}(G)  =  \mathcal{H}_{R}(G, \Upsilon) =   \bigcup _ { K \in \Upsilon }   \mathcal{H}_{R} ( K \backslash G /  K  ) $. The \emph{transposition}  on $ \mathcal{H}_{R} ( G , \Upsilon ) $ is the mapping $ \xi \mapsto \xi^{t} = ( g \mapsto \xi(g^{-1})  )  $, $ \xi \in  \mathcal{H}_{R}(G) $.    
\end{definition}  
The convolution    $  \xi_{1} * \xi_{2} $ where $ \xi_{1} , \xi_{2} \in \mathcal{H}_{R}(G, \Upsilon) $ is given by    $$  \left (      \xi_{1}  * \xi_{2} \right  ) ( g ) = \int_{x \in G } \xi_{1}(x)  \xi_{2}(x^{-1} g)  \,   d \mu(g) .  $$  In particular, if $ \xi_{1} = \ch (  \alpha K  ) $ for $ \alpha \in G $, $ K \in  \Upsilon   $     and $ \xi_{2} $ is right $ K $-invariant, then $ \xi_{1} * \xi_{2} = \mu(K) \xi_{2} ( \alpha ^{-1} ( - ) )  $.   
If $ G $ is unimodular, then one  also has $ (  \xi _{1} *  \xi_{2} ) ( g)  =  \int_{G}  \xi_{1} ( g y^{-1} ) \xi_{2}(y)   \,   d \mu(y) $ obtained by  substituting $ x $ with $ g y ^{-1} $.      The transposition map is an anti-involution of $ \mathcal{H}_{R}(G) $ i.e.\ $ ( \xi_{1} * \xi_{2} ) ^ { t }   = \xi_{2}   ^  {  t  }      * \xi_{1}  ^{t} $ for all $ \xi _{1} , \xi_{2} \in  \mathcal{H}_{R}(G) $. It stabilizes    $ \mathcal{H}_{R}(K \backslash G / K ) $ for any $ K \in  \Upsilon $.

The  Hecke  algebra  $ \mathcal{H}_{R} ( K  \backslash  G / K   )  $ has   an   $ R $-basis  given by the characteristic functions of double cosets $ K \sigma K $ for $ \sigma \in  K \backslash G /  K    $   denoted  $ \ch ( K \sigma  K ) $ and  referred    to   as     \emph{Hecke operators}. The \emph{degree}     of $ \ch( K \sigma K ) $ is defined to be $ | K \sigma K / K | $ or equivalently, the index $ [ K : K \cap \sigma K \sigma^{-1} $].       The  product    $ \ch(K\sigma K) * \ch(K  \tau   K ) $   is  supported on $ K  \sigma K \tau K $ and can  be described  explicitly as a function on $ G / K $    as  follows: if  $ K \sigma K = \bigsqcup_{i}  \alpha _{ i } K $, $ K  \tau  K =  \bigsqcup_{j} \beta_{j}   K  $, then 
\begin{align}   \label{convo} \mathrm{ch}(K \sigma K) *  \mathrm{ch}(K \tau   K)   &      = \mu ( K ) \cdot \sum   \nolimits    _ { i } \ch ( \alpha_{i} K \tau K )   =   \mu(K)  \cdot      \sum   \nolimits     _{ i , j }  \mathrm{ch} ( \alpha_{i}  \beta _{j} K )  
\end{align}   
On the other hand, the value of the convolution at $ \upsilon \in G $ equals $ \mu ( K \sigma K \cap \upsilon K \tau^{-1} K  ) $. Thus the convolution above can be written as   $   \mu(K)  \cdot    \sum_{  \upsilon   } c_{\sigma, \tau } ^ { \upsilon }  \ch ( K \upsilon K ) $ where $   \upsilon     \in    K \backslash K \sigma K \tau K / K $  and    
\begin{equation}    \label{Heckeformula} 
c_{\sigma , \tau } ^{\upsilon }  = |  (  K  \sigma K  \cap  \upsilon K \tau ^{-1} K  ) / K  |   
\end{equation}    
If  $ \mu( K ) = 1 $, then $ \mathcal{H}_{R}(K \backslash G / K ) $ is unital and   the  mapping $ \mathcal{H}_{R}( K \backslash G  / K )  \to R $ given by  $ \ch ( K \sigma K ) \mapsto | K \sigma K / K  |     $ is a homormorphism of  unital    rings. 

Any smooth left representation $  \pi  \in      \mathrm{SmthRep}_{R}(G) $  inherits a  left  action of the Hecke algebra $  \mathcal{H}_{R}(G, \Upsilon) $. 
The   action of $  \mathrm{ch}(K' \sigma K  ) \in  \mathcal{H}_{R} ( G,   \Upsilon   )  $ on an element $ x \in \pi $ invariant under $ K $  is given by $ \mathrm{ch}(K ' \sigma  K) \cdot x =  \mu( K )  \sum _{ \alpha \in K  '  \sigma    K / K } \alpha \cdot x   $.   
Similarly, if $ K \in \Upsilon $, the $ R $-module $ \pi ^{K} $ is stable under the action of $ \mathcal{H}_{R}(K \backslash G / K ) $ and is therefore a module over it. 
In particular, if   $ M $ is a  RIC functor, then  $ \widehat{ M} $ is  a module over $ \mathcal{H}_{R}(G, \Upsilon) $ and if $ M $ is Galois, $ M(K) =  \widehat{M}^{K}  $ is naturally a module over  $ \mathcal{H}_{R}(   K   \backslash  G  /  K    )    $.

We note that $ \mathcal{H}_{R}(G , \Upsilon ) $ is   itself   a smooth  left representation   of     $ G $ under both right and left translation actions.  It   is therefore a  (left) module   over   itself in two   distinct ways.  Let   
\begin{alignat*}{4}  \lambda : G \times \mathcal{H}_{R} ( G , \Upsilon )  &  \to   \mathcal{H}_{R} ( G , \Upsilon )  &  \hspace{0.8in}    \rho :  G \times   \mathcal{H}_{R}(G, \Upsilon )  &   \to  \mathcal{H}_{R}(G ,  \Upsilon)    \\
(g, \xi )   &  \mapsto \xi ( g^{-1} (-) )   &  \hspace{0.8in}    (g , \xi )   &     \mapsto  \xi ( (-) g )
\end{alignat*}
When $ \mathcal{H}_{R}(G  ,   \Upsilon    ) $ is considered as a $ G $-representation under $ \lambda $, the induced action of $ \mathcal{H}_{R}(G,  \Upsilon ) $ on itself is that of the convolution product $ *  $.  When $ \mathcal{H}_{R}(G  ,  \Upsilon   ) $ is considered as a $ G $-representation under $ \rho $, the induced action of $ \mathcal{H}_{R}(G,  \Upsilon ) $ will be denoted by  $ *_{\rho} $. 
There is a relation between $ * $ and $ *_{\rho} $ that is   useful    to   record.

\begin{lemma}   \label{rhoconvolution}     For  $  \xi_{1} ,  \xi_{2}    \in \mathcal{H}_{R}(G, \Upsilon ) $,    $  \xi_{1}  *  _ { \rho }  \xi_{2}  = \xi_{2}  * \xi _{1} ^ { t }  $.      
\end{lemma}    \begin{proof}   By definition,  we have  for all $ g \in G $   \begin{align*}  \left ( \xi_{1}  * _ {\rho } \xi_{2}    \right  )  ( g)  
& = \int_{G} \xi_{1}( x )  \,  \xi_{2}(gx)  \,  d \mu(x)   \\
 & = \int_{G}   \xi_{2}(y) \,  \xi_{1}( g^{-1} y )    \,  d \mu(y )  =  \int_{G} \xi_{2}(y) \, \xi_{1}^{t}( y^{-1} g )  \, d \mu(y)  \\
 &  =  \left  (  \xi_{2} * \xi_{1}^{t}  \right       )  ( g) 
\end{align*} 
where in the second equality, we used the change of variables $ x = g^{-1} y   $.    
\end{proof}    
        
\subsection{Hecke correspondences}    
On RIC functors, one may abstractly  define correspondences  in the same manner as one does for  the  cohomology  of  Shimura varieties. We explore the relationship between such correspondences and the action of Hecke algebra defined in \S   \ref{heckeoperatorsubsection}. More crucially, we need to  establish the usual properties of Hecke correspondences in the absence of axiom (G).      
\begin{definition}   \label{Heckecorrdefinition}
Let $M \colon  G   \to R\text{-Mod}$ be a  functor.    For every $K, K' \in \Upsilon$ and $\sigma \in   G    $, the \emph{Hecke  correspondence}  $ [ K '  \sigma K ]   $    is defined to be the composition
\[
[K ' \sigma K] \colon M(K) \xrightarrow {\pr^{*}} M(K \cap \sigma^{-1} K ' \sigma) \xrightarrow {[\sigma]^{*}} M(\sigma K \sigma^{-1} \cap K ') \xrightarrow {\pr_{*}} M(K') .
\]
If $ \mathcal{C}_{R} (K  '  \backslash G / K ) $ denotes   the free $ R $-module on functions $ \ch ( K ' \sigma K ) $, $ \sigma \in K' \backslash G / K $, there  is a $ R $-linear  mapping $ \mathcal{C}_{R} ( K '  \backslash G / K  ) \to  \mathrm{Hom}_{R}( M(K) , M(K') )  $ given by $ \ch ( K ' \sigma K ) \mapsto [K' \sigma K ] $. The \emph{transpose} of $  [ K ' \sigma  K ]  $ is defined to be the correspondence $$ [K' \sigma K]_{*} =  [K \sigma ^{-1} K ' ]   : M ( K' ) \to  M( K )      $$ which we also refer to as the \emph{covariant action} of $ [K \sigma K' ] $.    The  \emph{degree} of $ [K' \sigma K ] $ is defined to be the cardinality of $ K ' \sigma K / K   $ or equivalently, the index $ [ K'  : K' \cap \sigma K \sigma^{-1} ] $. The degree of $ [K' \sigma K]_{*} $ is the degree of $ [K \sigma^{-1} K'] $.  
\end{definition}

\begin{lemma}   \label{pullHecke}  Let $ M  : G \to R $-\text{Mod} be a  Mackey functor and let $ K , K ', L \in \Upsilon $ with $ L \subset K    $.  Suppose  that   $ K 
'  \sigma K   =  \bigsqcup _{  i }  L \sigma_{i} K $.  Then  $   \pr_{L, K' } ^ { * }  \circ  [ K '  \sigma K  ]   = \sum _{ i }   [L \sigma _ { i }  K   ] $.   
\end{lemma} 
\begin{proof} Denote $  L  '    : =K ' \cap \sigma K \sigma^{-1}   \in   \Upsilon    $.      As     $  K' / L' \to K ' \sigma K / K  $, $ \gamma   L '    \mapsto \gamma \sigma K  $ is a bijection, so is the induced map $ L \backslash K' / L ' 
   \to 
 L  \backslash K'  \sigma K / K $   and  we   may  therefore    assume that $  \sigma = \gamma_{i} \sigma  $   where  $ \gamma _{i} $  form  a    set of representatives for  $  L \backslash K ' /  L ' $.  Set $ L_{i}   :  =  L  \cap    \gamma_{i} L '  \gamma_{i}  ^ { - 1}     $.    Since $ M $ is Mackey, we see that the square in the diagram
\begin{center}
\begin{tikzcd}[sep = large,  /tikz/column 3/.style={column sep=0.1em}]  & 
 \bigoplus _{  i }      
 M ( L_{i} )   \arrow[r, " \sum  
 \pr_{*}"]           &     M(L)                  \\
 M( K ) \arrow[r, "{[\sigma]^{*}}"']   \arrow[ur ,   "{[\sigma_{i}]^{*}}"]  
 \arrow[rr, "{[K'  \sigma   K]}"', bend right, swap]         & M ( L ' )  \arrow[r, "\pr_{*}"'] \arrow[u, "{\oplus 
 [\gamma   _ { i }     ]^{*}}"'] & M(K') \arrow[u, "{\pr^{*}}"']
\end{tikzcd}
\end{center}    
commutes and  therefore  so does the whole diagram.  Noting that $ L_{i} =  L \cap \sigma_{i} K \sigma_{i}^{-1}    $, the 
 claim  follows  from the  commutativity of the diagram  above.       
\end{proof}

\begin{corollary}   \label{HeckeAgree}     Let $ M : G \to  R$-Mod be a Mackey functor  and  $ \mu $ be a Haar measure on $ G $. Let $ K , K ' \in  \Upsilon $  be  such that $\mu (  K  ) \in R   $.  Then for any $  \sigma \in G $,  the actions    of $ [K  '  \sigma  K ] $ and $  \ch(K  ' \tau K) $ on   $ \widehat{M} $ agree up to $ \mu(K) $. That is,      for all $ x \in M(K) $,  $$ \mu(K) \cdot  j_{K}\circ [K  '  \sigma K](x) =  \mathrm{ch}(K '  \sigma     K ) \cdot  j_{K} (  x    )    .  $$
In particular  if $ M(K) \to \widehat{M} $ is injective  and $ \mu(K)  = 1 $,  the $ R $-linear mapping $ \mathcal{H}_{R}(K \backslash G / K )  \to  \mathrm{End} _{R} M(K) $   given   by     $ \ch( K\tau K ) \mapsto [K \tau K ] $  is  an  $ R $-algebra  homomorphism.   
\end{corollary}   

\begin{proof} 
By (T2), there exist $ L  \in \Upsilon $ such that $ L \subset \sigma   K  \sigma^{-1}  $, $ L \triangleleft K ' $. Then $ K ' \sigma K / K  =   L 
 \backslash K ' \sigma K / K  $ and   $ [  L \gamma K   ]  = [ \gamma ] _{ L , K } ^ { *  } $ for any $ \gamma K \subset K '  \sigma K/ K $.  So we get the first claim by    Lemma  \ref{pullHecke}.  
 The second claim    then      follows  by  the first and  eq.\ (\ref{convo}).      
\end{proof}

\begin{remark} See Corollary  \ref{Heckecomposition} where the map $ \mathcal{H}_{R}(K \backslash G  / K ) \to  \mathrm{End}_{R}  M(K) $ is shown to be an algebra  homomorphism under the assumption that  $ M $ is  CoMack.           
\end{remark}

\begin{lemma}   \label{Hecketensor}       Suppose that $ G = G_{1} \times G_{2} $, $ \sigma_{i} \in G_{i} $ and  $ K_{i}, L_{i} \subset G_{i} $ are compact open subgroups such that  $ K_{1} K_{2} $, $ L_{1}   L_{2 }$, $ K_{1}L_{2} $ and $ L_{1} K_{2}   $  are all  in $ \Upsilon  $. Let   $ M : G \to R $-Mod  be a Mackey functor.  Denoting  $  \tau_{1} = (\sigma_{1} , 1) $, $  \tau_{2} =   ( 1,  \sigma_{2} ) $,  we  have   $$  [ (L_{1} L_{2})  \tau_{1} (K_{1} L_{2}) ] \circ [(K_{1} L_{2})  \tau_{2 }  (  K_{1}  K_{2}  )    ] =  [(L_{1}L_{2})\textbf{} \tau_{1}\tau_{2} ( K_{1}K_{2}) ] =   [ (  L_{1} L_{2} )  \tau_{2 }  (  L_{1} K_{2}  ) ] \circ[  ( L_{1}  K_{2}  )  \tau_{1}    (K_{1}K_{2}  )  ]    $$     
as morphisms   $ M(K_{1} K_{2} ) \to  M    (L_{1} L_{2}    ) $.  We also denote this morphism  as $ [L_{1} \sigma_{1} K _{1} ]  \otimes  [L_{2}  \sigma_{2} K _{2}  ]   $ and refer to it as the tensor product.     A similar fact holds for tensor products of  a finite number of    Hecke correspondences  in  restricted topological product of groups. 
\end{lemma}  

\begin{proof} For $ i = 1, 2 $, denote    $ P_{i} = \sigma_{i} K_{i} \sigma_{i}^{-1}  \cap  L_{i}  $.   Then 
\begin{align*} \tau_{1} (  K_{1}K_{2} )  \tau_{1}^{-1}  \cap (L_{1}K_{2})   & = P_{1}K_{2} , \\ 
\tau_{2}  (L_{1} K_{2}) \tau_{2}^{-1} \cap (L_{1}L_{2}) & =L_{1} P_{2}  , \\   
\tau_{1}\tau_{2}  (  K_{1}K_{2} 
 )  (\tau_{1} \tau_{2} ) ^{-1}  \cap  (L_{1}L_{2})   &   = P_{1} P_{2}  
 \end{align*} and all of these  groups  are in $ \Upsilon  $.  Since  $   \tau_{2}^{-1}  (L_{1}P_{2}   )   \tau_{2} \backslash   L_{1} K_{2}  / P_{1} K_{2}   =   \left  \{  1_{K}    \right  \}    $ and  $  M $ is Mackey,  we get a   commutative  diagram
\begin{center}
\begin{tikzcd}[row sep = large, column sep = tiny]
&    & M(P_{1}P_{2}) \arrow[rd, "\mathrm{pr}_{*}"]               &                                                    &      \\
& M(P_{1}K_{2}) \arrow[rd, "\mathrm{pr}_{*}"] \arrow[ru, "{[\tau_{2}]^{*}}"] &                             & M(L_{1}P_{2}) \arrow[rd, "\mathrm{pr}_{*}"] &      \\
M(K_{1}K_{2}) \arrow[rr, "{[(L_{1}K_{2})\tau_{1}(K_{1}K_{2})]}"'] \arrow[ru, "{[\tau_{1}]^{*}}"] &                                        & M(L_{1}K_{2}) \arrow[rr, "{[(L_{1}L_{2})\tau_{2}(L_{1}K_{2})]}"'] \arrow[ru, "{[\tau_{2}]^{*}}"] &                                      & M(L_{1}L_{2})
\end{tikzcd}\end{center}    
which implies  that     $ [   L_{1} L _{2}    \tau_{2}     L_{1}  K_{2}   ] \circ [  L_{1} K_{2}  \tau_{1} K_{1}K_{2}] = [L_{1}L_{2} \tau_{1} \tau_{2}  K_{1}K_{2}] $. By interchanging  the roles of $ \tau_{1} $,  $ \tau_{2}   $, we get the second  equality.  
\end{proof}

\subsection{Mixed  Hecke     correspondences}   
\label{mixedHeckecorrsec}    
In  the   situations that we are going to  consider, the classes used for constructing Euler systems  are  pushforwarded from a functor  associated with a  smaller  (closed)  subgroup. Here we study this scenario  abstractly and introduce some terminology that will be used extensively in the next section.     Let $ \iota  : H \hookrightarrow  G $  be a closed  subgroup, and $ \Upsilon_{H} , \Upsilon_{G} $ be  a collection of compact open subgroups of $ H $, $ G $ respectively  satisfying (T1)-(T3) and such that the collection $ \iota^{-1} ( \Upsilon_{G}) :    = \left \{ K \cap H \, | \, K \in  \Upsilon_{G}  \right \} $ is contained in  $    \Upsilon_{H}  $.   Note that $  \iota^{-1}( \Upsilon_{G} ) $ itself satisfies (T1)-(T3) for $ H $ and we refer to it as the \emph{pullback} of $ \Upsilon_{G} $ to $ H $.        
\begin{definition}    \label{pushforwarddefinition}     
We say that
$ (  U, K ) \in \Upsilon_{H} \times \Upsilon_{G} $ forms a \emph{compatible pair} if $ U \subset K $.  A \emph{morphism of compatible pairs} $ h : (V,L) \to  (U,K) $ is a pair of morphisms 
$ (V   \xrightarrow{h}  U )  $, $  (L  \xrightarrow{h}  K )   $ for some $ h \in H  $.  Let  $  M   _   {H} $, $ M_{G } $ be $ R  
 $-Mod     valued functors on $  H $, $    G     $ respectively. A  \emph{pushforward} $ M _{ H } \to  M_{G} $  is a family of morphisms $  \iota_{U, K, *} : M_{H}(U) \to M_{G}(K) $ for all compatible   pairs $ (U,K) \in \Upsilon_{H} \times \Upsilon_{G} $  such that $ \iota_{U,K,*} $, $ \iota_{V,L,*} $ commute with the pushforwards $ [ h ] _ { * } $ on $ M_{H} $, $ M_{G} $   induced by  any   morphism $ h : (V,L)  \to (U,K) $ of compatible pairs.   We  say that $ \iota_{*}   $  is   \emph{Mackey} if for all $ U \in \Upsilon_{H} $, $ L , K   \in   \Upsilon _ { G}   $   satisfying   $     U , L \subset K $, we have a commutative diagram 
\begin{center}
    \begin{tikzcd}  [column sep = large]  \bigoplus_ { \gamma } M_{H} ( U _ {\gamma}   )   \arrow[r,  "{ \sum [\gamma]_{*}}"]   &  
    M_{G} ( L  )  \\ 
    M_{H} ( U )  \arrow[r, "{\iota_{*}}"]   \arrow[u, "{\oplus\,\pr^{*}}"] &  M_{G}(K)   \arrow[u,"{\pr^{*}}", swap]  
    \end{tikzcd}
\end{center}
where $ \gamma \in U \backslash K  /  L $ is a fixed  set of representatives, $ U_{\gamma} = U \cap \gamma L \gamma^{-1} $ and  $ [\gamma]_{*} : M _ { H } (  U  _ { \gamma }  )  \to  M_{G} ( L )  $  denotes  the    composition $ M_{H} ( U_{\gamma })  \xrightarrow {  \iota_{* } } M _{G} (  \gamma L \gamma ^{-1} )   \xrightarrow { [ \gamma ] _ { * }  }  M_{G }   (   L   )        $.
\end{definition} 
If $ \varphi_{G} : N_{G} \to M_{G} $ is a morphism of functors, then it may be viewed as a pushforward in the sense of Definition  \ref{pushforwarddefinition}. We will say that $ \varphi $ is \emph{Mackey} if it is so as a pushforward. 
\begin{lemma} If $ M_{G} $ is Mackey, then so is any morphism $ \varphi_{G} : N_{G} \to M_{G} $. 
\end{lemma}
\begin{proof} If $ M $ is Mackey, then $ M $ satisfies the axiom (M$'$) given in Lemma \ref{Mackeyprime}. Using its notation, the commutativity of \ref{Mackeyprimediagram} implies that     
\begin{equation*}   
    \begin{tikzcd}  [column sep = large]  \bigoplus _ { \delta}  N _ { G } ( L' _  { \delta }  )  \arrow[r,  "\varphi_{G}"]  & \bigoplus_ { \delta } M _ { G }  ( L _ {\delta } '  )  
    \\ 
   N  _ { G  }   (L') \arrow[r, "\varphi_{G}" ,swap]  \arrow[ u , "{\bigoplus\pr^{*}}" ] &  M  _ { G  }     ( L' )  
    \arrow[u, "{\bigoplus\pr^{*}}"'] 
    \end{tikzcd}
\end{equation*}
is   commutative  as   well.      
\end{proof}   

\begin{definition}   \label{mixedHeckedefinition}   Let $ \iota _ { * }  : M _{  H } \to M_{G }$ be a pushforward.    For   $ U \in \Upsilon _{H} $, $ K \in \Upsilon_{G} $ and $ \sigma \in G $,    the \emph{mixed Hecke correspondence}  $ [ U \sigma K ] _ {  * } $ is defined as $$ [U \sigma K ]_{*} :  M_{H} (U)  \xrightarrow {  \pr ^ { * }   } M_{H } ( U \cap \sigma K \sigma ^ {  - 1 } )   \xrightarrow{\iota_{*}} M_{G}( \sigma K \sigma ^{-1} ) \xrightarrow{ [\sigma]_{*} }  M_{G}(K) .  $$ One can verify that  $ [U\sigma K]_{ * }$ depends only on the double coset $ U \sigma K $. The  \emph{degree} of $ [U\sigma K]_{*} $ is defined to be the index $  [H \cap \sigma K \sigma^{-1} : U \cap \sigma K  \sigma ^{-1}  ]    $.       
\end{definition}

\begin{remark}   Suppose that $ H = G  $, $ \iota = \mathrm{id} $ and  $ \iota_{*} : M_{G} \to M_{G}  $ is the  identity  map.   Then one can verify that $ \iota_{*} $ is Mackey iff $ M_{G} $ is. Moreover if $ U , K \in \Upsilon_{G}$   and $ \sigma \in G $,   we  have      $ [U \sigma K]_{*} =  [ U  \sigma  K ] ^{t} =  [K \sigma^{-1} U   ] $ agrees with the covariant action introduced before  and the degrees   of $ [ U \sigma K ]_{*} $, $ [K \sigma^{-1} U] $ also agree.     The `$ * $' in the notation  of mixed Hecke  correspondence  is meant to emphasize its `pushforward nature' and its dependence  on  $ \iota_{*} $.    We    note  that $ \deg  \,  [U \sigma K   ]  _ { * }     $  is    however      independent  of  $ \iota_{*}  $.        \label{MixedHeckegeneralize} 
\end{remark}          

\begin{lemma} \label{mixedhecketwisting}  Let $ \iota_{*} : M_{H} \to M_{G} $ be a pushforward and let  $ \sigma \in G $,  $ U \in \Upsilon_{H} $, $  K  \in   \Upsilon_{G} $.  For $ h \in H $, $ g \in  G   $,   denote   $  U^{h} : = h U h ^{-1} $, $ K^{g   }  :  =  g K  g ^ { - 1 }  $.   Then   $$    [ U \sigma K ] _ { * }   =  [ U^{h}  h \sigma      K    ]   \circ [ h ] ^{ * } _ { U   ^{h}  , U    }  =  [g]_{K^{g},  K   ,    * }   \circ      [ U \sigma      g ^ {- 1 }    K^{g}   ]   _{*}  .   $$   
Moreover  $ \deg  [ U  \sigma  K ] _{*}  = \deg [ U ^{h} h \sigma   K] _{ * } = \deg  [ U  \sigma g^{-1 }  K^{g} ]_{*} $.  
\end{lemma}

\begin{proof}  Let $ V  :  = U \cap \sigma K \sigma  ^{-1} $, $  V '   :  = U^{h} \cap h \sigma K ( h \sigma )^{-1} $,  $ L : 
 =  \sigma K  \sigma ^{-1} $  and  $  L' :   = h \sigma K \sigma ^{-1}  h^{-1 }$.  By  definition, $ h V h ^{-1} = V ' $, $ h  L h ^{-1} = L' $, $ V \subset L $ and $  V '   \subset  L  '  $.    One  easily verifies that the  diagram   
\begin{center}  
\begin{tikzcd}[sep = large] 
M_{H}(U) \arrow[d, "{[h]^{*}}"'] \arrow[r, "\mathrm{pr}^{*}"] \arrow[rrr, "{[U\sigma K]_{*}}", bend left = 22] & M_{H}(V) \arrow[d, "{[h]^{*}}"] \arrow[r, "\iota_{*}"] & M_{G}(L) \arrow[d, "{[h]^{*}}"'] \arrow[r, "{[g]_{*}}"] & M_{G}(K) \\
M_{H}(U^{h}) \arrow[r, "\mathrm{pr}^{*}"]                                                                  & M_{H}(V') \arrow[r, "\iota_{*}"]                         & M_{G}(L') \arrow[ru, "{[hg]_{*}}"']                   &         
\end{tikzcd}
\end{center}    
is commutative which implies   $  [ U \sigma K ] _ { * }   =  [ U^{h} h \sigma  K    ]   \circ [ h ] ^{ * } _ { U   ^{h}  , U    } $.   By  definition,  $ \deg [ U \sigma K ]_{*} = [ H \cap L  : V ] $ and $  \deg [ U ^{h} h \sigma   K ] _{*}   =[  H \cap L ' : V' ]   $.   Since  $ L $, $ L' $ and $ V $, $ V'$ are conjugates under $ h $, $ [ H \cap L : V ]  = [ H \cap L' : V ' ] $ and so $  \deg [ U \sigma K ]_{*}  =  \deg [ U ^{h} h \sigma    K ]  $.  The proof for the second set of equalities  is  similar.    
\end{proof}

\begin{lemma}   \label{mixedHeckecompose}        Let    $ \iota_{*} : M_{H} \to M_{G} $ be a pushforward and let  $ \sigma \in G $, $ U,  V \in \Upsilon_{H} $, $ K, L \in \Upsilon_{G} $ be such that $ V \subset   U  \subset \sigma K \sigma^{-1} $, $ V \subset  \sigma L \sigma^{-1} $ and $  L  \subset K $. Then  $$  [V \sigma K ]_{*}   =  [U \sigma K ] _ { * } \circ \pr_{V , U , * }  =   \pr_{L,K , * } \circ  [ V  \sigma L  ]   .    $$ 
\end{lemma}

\begin{proof}    Since $ U \subset \sigma K \sigma^{-1} $, $ [U \sigma K ]_{*} $ is the composition $ M_{H}(U) \xrightarrow{ \iota_{*} }  M_{G} ( \sigma K  \sigma ^{-1} )   \xrightarrow {  [ \sigma_{*}    ] }  M_{G}(K) $. Similarly $ [ V \sigma L ] _{*}  $ is the composition $ M_{H}(V)  \xrightarrow{  \iota_{*} }  M_{G} ( \sigma L \sigma^{-1})   \xrightarrow{   [\sigma]_{*} }  M_{G}(L) $.  Since pushforwards commute with each other, the claim  follows.    
\end{proof}    

The following result  
is an  analogue of Lemma \ref{pullHecke} for pushforwards.

\begin{lemma}  \label{MixedHeckecompose}   Let  $  \iota_{*}   :  M_{ H }  \to   M_{G}  $ be a Mackey pushforward and let $ \sigma \in G $, $ U \in \Upsilon_{H} $, $   K,  K'  \in  \Upsilon_{G}  $ with $  U  \subset   K    $.    Suppose that $ K  \sigma K '  = \bigsqcup_{i}  U \sigma _{i} K ' $.   Then $  [ K  \sigma K  ' ]   _ {  *   }     \circ  \iota_{U , K ,  * }    =  \sum _{ i }  [ U \sigma _{ i } K '] _ { * }  $.
\end{lemma} 

\begin{proof}  Let $ L  := K \cap   \sigma K'  \sigma^{-1}  $. As     $ K / L \to K \sigma K ' / K ' $, $ \gamma L \mapsto \gamma \sigma K' $ is a bijection, so is the induced map $ U \backslash K /  L \to U \backslash  K \sigma K' / L' $ and    we may thus  assume that $ \sigma_{i} = \gamma_{i} \sigma $  where   $ \gamma _{i} \in K $ forms    a     set of representatives of $   U  \backslash K  / L $.  Let    $$ K_{i}   '   :  =  \sigma _{i } K '  \sigma_{i} ^{-1}  ,   \quad    L_{i}   :   = \gamma_{i} L \gamma_{i} ^{-1}   ,   \quad       U_{i } : = U  \cap K_{i}  '  .    $$     Then   $  L_{i} = K \cap \sigma_{i} K '  \sigma_{i}^{-1} =  K \cap K_{i} ' $ and   therefore   $     U _{i} =   U  \cap L_{i}    $.       As $ \iota_{*} $ is Mackey,     we see that     
\begin{equation}    \label{mixedmecke}     \pr_{L , K } ^{ * }  \circ  \iota _ { U , K , * }  =  \sum  _{ i }       [  \gamma _{ i }  ] _ { U _{ i }  ,  L   ,  *  } \circ  \pr_{U_{i} , U  } ^ {  *  }  
\end{equation} where $ [\gamma_{i} ] _{ U_{i}, L , * } : =  [  \gamma_{i} ] _ { L_{i}, L , * }  \circ      \iota_{ U_{i}, L_{i} , * }  = [U_{i}  \gamma_{i } L ]_{*}  $ (see the  diagram on the left below). 
\begin{center}
\begin {tikzcd}[sep = large, /tikz/column 3/.style={column sep=0.1em}]   

\bigoplus _ { i } M_{H} ( U_{i} ) 
\arrow[r, "{\sum [\gamma_{i}]_{*}}"]  &   M_{G}  (    L)   \arrow[dr, "{[\sigma]_{*}}"]  
&    &   M_{H} ( U_{i }  )  \arrow[dr,  "\iota_{*}", swap]  \arrow[r, "{[\gamma_{i}]_{*}}"] & M_{G}(L)   
\arrow[dr,  "{[\sigma]_{*}}"]      \\
M_{H} ( U )   \arrow[r, "{\iota_{*}}", swap]   \arrow[u,"{\oplus \, \pr^{*}}"]  &  M_{G}(K)   \arrow[u, "{\pr^{*}}" ]   \arrow[r,"{[K\sigma K']_{*}}" , swap ]  &  M_{G}(K ')   &   &  M_{G} ( K _{i}' )      \arrow[r, "{[\sigma_{i}]_{*}}" , swap    ]         &   M_{G}(K')   
\end{tikzcd}
\end{center}  
As     
$ \iota_{U_{i}, K_{i}',*} 
= [\gamma_{i}  ^ { - 1} ]_{L,K_{i}',*}   \circ  [   \gamma_{i} ]_{U_{i}, L , *} $ for each $ i $, we see that

\begin{align}   \label{parallelogram}   \notag  [\sigma]_{L,K',*} \circ   [\gamma_{i}]_{U_{i},L,*} & =      \big ( [     \sigma_{i}    ] _ {K_{i}',  K   ' , * }     \circ  [ \gamma _ { i } ^{-1} ] _{ L, K_{i}   '    , *}  \big  )   \circ   [\gamma_{i} ]_{U_{i}, L, *}  \\
&  =  [  \sigma_{i} ] _{K_{i} ' , K   '    }    \circ \iota_{U_{i}, K  _ {  i  }   '    ,*} 
\end{align}  
(see the diagram on the right above).   
Using $  [K \sigma K ' ]_{*}  =    [ \sigma ] _{   L , K  '    , * }   \circ  \pr _ { L , K   } ^  { *  }    $  in conjunction with eq.\ (\ref{parallelogram}) and   eq.\  (\ref{mixedmecke}), we see that  
\begin{align*}  [  K \sigma K  '   ]  _  { *   }      \circ  \iota_{U, K ,  * } 
&  =   [ \sigma   ]  _{ L  , K  '   ,   * } \circ \sum _{ i   }   \big    (  [ \gamma _{ i } ] _ { U_{i} , L     ,   * }   \circ  \pr _ { U _{ i }  , U } ^ { *  }   \big     )        \\
& = \sum _{ i }  \big (  [ \sigma   ]    _{L,K   '    ,  * } \circ   [ \gamma_{i} ] _ { U_{i} , L , * }   \big )  \circ \pr_{  U _{ i }  , U  }  ^ {    *    }   \\    
& =  \sum  _ { i }  \,   \big (  [  \sigma _ { i } ] _ {  K_{i}  '  , K  '  ,  * } \circ  \iota _{ U _ { i }  ,  K_ { i }    '   }   \big  )      \circ  \pr _   { U _{ i } , U }  ^ { * }   \\
&  =   \sum _ { i }  [ U \sigma   _ {i  } K  '] _ { * } \qedhere   
\end{align*}%which   proves the claim. 
\end{proof}

We end this subsection by showing that any two  (contravariant)   Hecke correspondences compose in the usual way.  For  $K _{1}, K_{2}, K_{3} \in \Upsilon_{G} $,  the \emph{convolution of double cosets} is the $ \ZZ $-linear homomorphism  
\begin{align*}  \circ    :  \mathcal{C}_{\ZZ}( K_{3} \backslash G / K_{2} )  \times   \mathcal{C}_{\ZZ}(K_{2} \backslash G / K_{1} )  &  \to  \mathcal{C}_{\ZZ}( K_{3} \backslash G / K_{1} )  
\end{align*}  
given by $ 
 \ch(K_{3} \sigma K_{2})   \circ  \ch(K_{2} \tau K_{1} )    = \sum \nolimits _{ \upsilon } c^{\upsilon}_{\sigma, \tau } \ch (K_{3}  \upsilon  K_{1}) $  
where $ c^{\upsilon}_{\sigma, \tau} = | (K_{3} \sigma K_{2} \cap \upsilon K_{1} \tau^{-1} K_{2} ) /  K_{2} | $. 
\begin{corollary}   \label{Heckecomposition}    Let $  M  =  M_{G}   $ be a  CoMack     functor  on $ G $, $ K_{1}, K_{2}, K_{3} \in   \Upsilon_{G}  $ and $ \sigma, \tau \in  G $. Then $  [K_{3}  \sigma  K_{2} ]  \circ  [ K _{2}   \tau   K_{1} ]  \in  \mathrm{Hom}_{R} ( M(K_{1}   )    , M(K_{3}) ) $ is a sum  of   Hecke correspondences obtained by the convolution of double cosets as above.    
\end{corollary} 
\begin{proof}  Let $  L  =  \tau  K_{1} \tau^{-1}  \cap  K_{2}   \in  \Upsilon_{G}     $ and suppose that $ K_{2}  \sigma  ^{-1} K_{3}  =  \bigsqcup_{i} L  \sigma _{i}  ^{-1} K_{3} $ for some $  \sigma  _{i} \in G  $.  Since  $ M $ is  Mackey,  we  see by   Lemma  \ref{MixedHeckecompose}   that    
\begin{align*} [K_{3}  \sigma  K_{2}]  \circ [ K_{2} \tau  K_{1} ]  &  =  [K_{3}   \sigma   K_{2} ] \circ  \pr_{L,K_{2},*}  \circ  [  \tau    ]  _ {  L  ,  K _{1} } ^{*} \\ 
&  =   \big (  [K_{2} \sigma^{-1}  K_{3} ]_{*} \circ  \pr_{L,K_{2},*}  \big )  \circ  [ \tau  ] _{L,K_{1} }  ^ { *  }   \\  
&  =    \sum   \nolimits    _{i}    \,     [L  \sigma _ { i } ^ {  -1 }  K_{3} ]_{*}  \circ  [ \tau  ]^{*}_{L,K_{1}}       
= \sum  \nolimits _ {i}  \,  [K_{3}  \sigma _ { i } L]  \circ    [  \tau   ] _ { L , K_{1}  }   ^ { * }  . 
\end{align*}
 For each $ i $, let  $ d_{i} : =  [ \sigma_{i} \tau K_{1} (\sigma_{i} \tau)^{-1}   \cap K_{3} :  \sigma_{i}  L  \sigma_{i}  ^ { - 1 } \cap   K_{3}    ]  $. Since   $ M $ is cohomological, we see that the diagram   
\begin{center}
    \begin{tikzcd}   &  M ( \sigma_{i} L \sigma_{i}^{-1} \cap K_{3}   )    \arrow[dd,   "{\pr_{*} } " ]  \arrow[dr, "{\pr_{*}}"]  &   \\ 
    M(K_{1})  \arrow[ur,   "{[\sigma_{i} \tau]^{*}}"]    \arrow[dr, "{d_{i} \cdot [\sigma_{i} \tau]^{*} }"']  &   &   M(K_{3}   ) \\
       &   M (   \sigma_{i} \tau K_{1} ( 
   \sigma_{i} \tau )   ^{-1} \cap K_{3} ) \arrow[ur, "{\pr^{*}}"']    &  
    \end{tikzcd} 
\end{center}
is commutative. So  $ [ K_{3} \sigma_{i}  L ] \circ  [  \tau   ]_{L, K_{1}}  ^{*}  =  d_{i}   \cdot     [K_{3} \sigma_{i} \tau K_{1} ] $ as maps $ M(K_{1}) \to M(K_{3})  $ (take the two routes in the diagram above) and therefore     $$  [  K_{3}    \sigma  K_{2} ]   \circ   [K_{2}  \tau K_{1}  ]  =    \sum 
  \nolimits    _{i}   d_{i}   [  K_{3}  \sigma_{i} \tau  K_{1}  ] .   $$     
To show that $  \sum_{i} d_{i}  \,    \ch(K_{3} \sigma _{i} \tau K_{1} )$ equals the convolution product,   
take $ M = M_{\pi} $ to be the functor associated with the smooth left $ G $-representation $ \pi $ where $ \pi =  \mathcal{H}_{\QQ}(G, \Upsilon) $  with  $ G $ acting via left translation $ \lambda $ (\S  \ref{heckeoperatorsubsection}).    
Let  $ \mu $ be a $ \QQ$-valued left Haar measure on $ G $. By  Corollary  \ref{HeckeAgree}, 
\begin{align*}  \mu(K_{2}) \mu(K_{1})  \cdot  j_{K_{3}}  \circ [K_{3} \sigma K_{2} ] \circ [K_{2} \tau K_{1}] \big  ( \ch(K_{1})  \big  )   =  \ch(K_{3} \sigma K_{2} ) * \ch(K_{2} \sigma K_{1}  ) 
\end{align*}    
as elements of    $ \widehat{M}_{\pi}   = \pi =   \mathcal{H}_{\QQ}(G, \Upsilon_{G}) $. 
   Similarly, $$  \mu (K_{2} )  \mu(K_{1} )    \cdot      j_{K_{3}} \circ   \sum   \nolimits    _{i} d_{i}  \,  [ K_{3} \sigma_{i} \tau K_{1} ] \cdot \big (  \ch(K_{1})  \big ) =   \mu(K_{2} )     \sum   \nolimits    _{i} d_{i} \,    \ch(K_{1} \sigma_{i} \tau K_{3})  .  $$
As the LHS of these two equalities are equal by the above   argument,    we must have   $$   \ch(K_{3} \sigma K_{2} ) * \ch(K_{2} \sigma K_{1} ) =    \mu(K_{2})  \sum\nolimits    _{i} d_{i} \, \ch(K_{1} \sigma_{i}  \tau K_{3} ) . $$   But the coefficient of $ \ch ( K _ {3 }  \upsilon K_{1}) $ in   $ \ch(K_{3} \sigma K_{2} ) * \ch(K_{2} \tau K_{1} ) $ equals $ \mu (  K_{3} \sigma K_{2} \cap  \upsilon  
K_{1} \tau ^{-1} K_{2} )   =  \mu(K_{2} )  c_{ \sigma  ,  \tau }  ^ { \upsilon }    $. Therefore $ \sum_{i} d_{i} \ch( K_{1} \sigma_{i} \tau K_{3}  ) $ must equal the convolution of double cosets.                 
\end{proof}

\subsection{Completed pushforwards}    
Let $ \iota : H \to G $ be    as  in  \S      \ref{mixedHeckecorrsec} and assume  moreover  that  $ H $, $ G $ are  unimodular.     Let $ \mu_{H} $, $ \mu_{G} $ Haar measures on  $ H $, $ G  $  respectively  with $ \mu_{ H } (  \Upsilon_{H} ) , \mu_{G} ( \Upsilon_{G})   \in R ^ {\times}  $.  
Let $ \mathcal{H} _ { R} ( G , \Upsilon  _ { G  } ) $  denote the Hecke algebra of $ G $  over   $ \Upsilon   _ { G  }   $.

\begin{definition}  \label{DefinitionOfIntertwiningMap}
Given  smooth  representations $ \tau $ of $ H $, $ \sigma $ of $ G $,  we consider  $ \tau \otimes \mathcal{H}_{R}(G ,  \Upsilon _ { G}  )   $ and $ \sigma $ smooth representations  of $ H \times G $ under the following  \emph{extended action}. 
\begin{itemize} [before = \vspace{\smallskipamount}, after =  \vspace{\smallskipamount}]    \setlength\itemsep{0.1em}   
\item $ (h, g) \in H \times G $ acts on $ x \otimes \xi \in \tau \otimes     \mathcal{H} _ { R } (G   ,   \Upsilon   _ { G  }    )   $ via $ x \otimes \xi \mapsto  h x  \otimes  \xi ( \iota ( h )  ^ { - 1  } ( - )  g )   $.   
\item $ (h,  g ) \in  H \times  G  $   acts  on $  y \in \sigma $ via $  y \mapsto g \cdot  y    $. 
\end{itemize}          
An    \emph{intertwining map} $ \Psi : \tau  \otimes \mathcal{H}_{R}(G,  \Upsilon    _ { G  }  )   \to \sigma  $ is  defined  to  be  a  morphism of $   H \times G  $  representations.  
\end{definition}

\begin{lemma}   \label{intertwining}     Let $  \Psi : \tau  \otimes \mathcal{H}_{R}(G,  \Upsilon    _ { G  }  )   \to \sigma  $ be an intertwining map. 
For any $   \xi _{1} , \xi_{2} \in \mathcal{H}_{R}( G , \Upsilon _ { G  } ) $ and  $ x  \in  \tau $,  $$  \xi_{1} \cdot \Psi (  x  \otimes  \xi_{2} )   =  \Psi  (   x \otimes \xi_{2} * \xi_{1}^{t} ) ,   $$ 
where $ \xi_{1} ^{t} $ is the transpose of $ \xi_{1}  $.   
\end{lemma} 
\begin{proof}    Since $  \Psi  $ is an intertwining map,  it is also a morphism of $ \mathcal{H}_{R}(G, \Upsilon_{G} )$  modules  under the induced  actions. Thus, $ \xi_{1} \cdot \Psi  ( x \otimes \xi_{2}) =  \Psi ( x \otimes \xi_{1} *_{\rho} \xi_{2} ) $.  But   Lemma  \ref{rhoconvolution} implies that $ \xi   _ { 1  }      *_{\rho} \xi_{2}  =  \xi_{2} * \xi_{1}^{t}  $. 
\end{proof}  

The proofs for the next two results are omitted and can be found in \cite{AESthesis} (cf.\  \cite{Anticyclo}).          

\begin{lemma} [Frobenius Reciprocity]  Let $ \sigma ^ { \vee}  $ denote the smooth dual of $ \sigma $ and $ \langle  \cdot  ,  \cdot   \rangle  : \sigma  ^ {  \vee }  \times  \sigma  \to  R $ denote the induced pairing. Consider $ \tau   \otimes   
\sigma ^ { \vee } $ as a smooth $ H $-representation   via $ h (x \otimes f   )  = hx \otimes \iota(h) f $.   Then for any intertwining  map  $ \Psi $ as above, there is a unique morphism $ \psi :  \tau   \otimes   
\sigma^{\vee} \to R $ of  smooth  $ H $-representations such that $$ \langle  f ,   \Psi ( x \otimes \xi )  \rangle  =  \psi (  x \otimes ( \xi \cdot f ) ) $$    
for all $ x \in  \tau $, $ f \in \sigma  ^ { \vee }    $ and $ \xi \in   \mathcal{H}_{R}(G,  \Upsilon_{G} )  $. The mapping $ \Psi \mapsto  \psi $  thus   defined    induces a bijection   between     $  \mathrm{Hom}_{H \times G } ( \tau \otimes \mathcal{H}_{R}(G, \Upsilon_{G} ) , \sigma ) $ and $ \mathrm{Hom}_{H} ( \tau \otimes \sigma^{\vee} ,  R ) $.
\label{frobeniusreciprocity}    
\end{lemma}

\begin{proposition}[Completed pushforward] \label{LemmEin} 
Suppose  $ M_ { H } $, $ M_{G} $ are RIC functors  with $ M_{H}   $ CoMack and  $ M_{G} $  Mackey.  Consider $ \widehat{M}_{H}   \otimes  \mathcal{H}_{R} (G  ,  \Upsilon  )    $ and $   \widehat{M}_{G}    $   as smooth $ H \times G $ representations via the extended action.  Then for any pushforward $  \iota_{*}  : M_{H}  \to  M_{G} $,   
there is a unique    intertwining map of $ H  \times  G  $ representations 
\begin{align*} 
\hat{\iota} _ { * } : \widehat{M}_{H} \otimes \mathcal{H} _ { R } (  G ,  \Upsilon  _   {  G   }     )      \to  \widehat{ M } _{G}
\end{align*} satisfying the following  compatibility    condition:  for all compatible pairs $ (U,K)  \in  \Upsilon _{H} \times \Upsilon_{G}  $, $   x \in M_{H}(U) $, we  have $  \hat { \iota } _{*} \big ( j_{U}(x ) \otimes \ch(K) \big  ) =  \mu_{H}  ( U   )  j_{K} \big (   \iota_{U,K,*} ( x )  \big  )       $.  Equivalently any $ \iota_{*} $ determines a unique   morphism $ \tilde{\iota}_{*}  :    \widehat{M}_{H}  \otimes   (\widehat{M}_{G})^{\vee} \to  R $ of $ H $-representations    such that  $ \tilde{\iota}_{*}  \big (  j_{U}(x ) \otimes f  \big  )    = \mu_{H}(U)    f \big ( j_{K} \circ \iota_{U,K,*}(x) \big )   $ for $ U$, $ K $, $ x   
  $  as above  and $ f \in   ( \widehat{M}_{G}   )     ^{\vee} $           
\end{proposition}

\begin{corollary}   \label{completedhecke}       Let $  \hat{\iota}_{*}    $ be as above. For any $ U \in \Upsilon_{H  }  $, $  K  \in  \Upsilon_{G} $, $ g \in G $ and $ x 
\in M_{H}(U) $,     $$  
\hat {\iota}_{* }( j_{U}(x_{U} ) \otimes  \ch ( g K  )  ) = \mu_{H} ( U \cap g K g^{-1} ) \cdot  j_{K}  \circ  [ U g K ] _ { * }  ( x  
) . $$
\end{corollary}

\begin{remark}    In the definition of $ \hat{\iota}_{*} $, we may replace $ \mathcal{H}_{R}(G) $ with $  \mathcal{C}_{R}( G )$ which is $ \mathcal{H}_{R}(G) $ considered as a $ R$-module with $ G $-action given by right translation, since the definition of $ \hat{\iota}_{*} $ does not require the convolution operation. In particular, $ \hat{\iota}_{*} $ is independent of $ \mu_{G}  $.    
\end{remark} 

\subsection{Shimura   varieties} 
In this subsection, we   briefly  outline how  the abstract formalism  here applies to the  cohomology of   general Shimura varieties.  We refer the  reader to \cite[Appendix B]{Anticyclo} for terminology which we will be used freely in what  follows.      

Let $ (\Gb, X) $ be a Shimura-Deligne (SD) datum and let $ \mathbf{Z} $ denote the center of $ \Gb $. For any neat compact open subgroup $ K \subset \Gb(\Ab_{f}) $, the double quotient $$  \Sh_{\Gb}(K)(\CC) := \Gb(\QQ) \backslash [ X \times \Gb(\Ab_{f})/K]  $$  is  the   set of $ \CC $-points  of a  smooth  quasi-projective variety over $ \CC  $.   If    $   (\Gb, X ) $ satisfies (SD3)  or  if  $  (\Gb, X) $ admits an embedding into a  SD datum which satisfies (SD3),   then $ \Sh_{\Gb}(K) $ admits a canonical model over its reflex field.  For two neat compact open subgroups $ K' , K \subset \Gb(\Ab_{f}) $ such that $ K ' \subset  K $,    it is not true in general  that $ \Sh_{\Gb}(K')(\CC) \to \Sh_{\Gb}(K) $ is a covering map of degree $ [K' : K ] $ unless (SD5) is also  satisfied (\cite[Lemma B.18]{Anticyclo}). However  one can establish the following (cf.\ \cite{kudla}).            

\begin{lemma}   \label{etalecovershimuralemma}       Let $  K,  K'  \subset \Gb ( \Ab_{f}   )    $ be  neat  compact open subgroups  such that $K ' \subset K $ and  $ K \cap \mathbf{Z}(\QQ) =   K'  \cap \mathbf{Z}(\QQ)    $.    Then the natural  map  $  \pr_{K', K} :  \mathrm{Sh}_{\Gb}(K')( \CC)     \twoheadrightarrow   \mathrm{Sh}_{\Gb}(K)   ( \CC)  $  of smooth $ \CC $-manifolds     is  an  unramified  covering map   of degree $ [K :  K' ] $. 
\end{lemma}    
\begin{proof}    
Suppose   that    there exists $ x \in X $, $ g \in \mathbf{G}(\Ab_{f} ) $ and $   k \in  K$ such that $ [x, g]_{K'}  = [x , g k   ]   _   { K'}      $  in  
$ \mathrm{Sh}_{\Gb}(K')(\CC) $. Let $ K_{\infty}  $  denote  the  stabilizer of $ x $ in $ \mathbf{G}( \RR)  $.     By  definition,    there exists a  $ \gamma \in \Gb(\QQ) \cap K_{\infty } $ such that 
\begin{equation}   \label{gkequality}              g k  =  \gamma g k '   
\end{equation}    for  some   $ k '  \in   K '  \subset  K  $. Then     $ \gamma   = g k (k')^{-1} g^{-1} $ is an element of $ \Gamma : =  \Gb(\QQ) \cap  g K g^{-1} $. Since $ \Gb(\QQ) $ is discrete in $ \Gb ( \Ab    ) $,  we  see  that    $ \Gamma $ is discrete in $ \Gb(\RR) $ and so   $ \Gamma \cap  K_{\infty} $ is discrete in  $ K_{\infty}  $.  In  particular,    the group $ C : = \langle \gamma \rangle  \subset \Gamma \cap K_{\infty}   $  generated by $ \gamma $ is  discrete in $ K_{\infty} $.        By  \cite[Lemma B.5]{Anticyclo}, the  quotient   $ K_{\infty}  /  (   \mathbf{Z} ( \RR )  \cap  K _ {\infty } )  $ is  a  compact    group. Since $ C / ( \mathbf{Z}(\RR) \cap C )  $  is a (necessarily closed)   discrete subgroup of this quotient,  it  must  be  finite.  There is  therefore  a  positive  integer $ n $   such that $$  \gamma ^{n}  \in  \mathbf{Z}(  \RR)    \cap   C     \subset     \mathbf{Z} (\QQ)  . $$   Since $ \Gamma  \subset  \Gb(\QQ) $ is neat, its image $ \bar{\Gamma}  \subset \mathbf{G}^{\mathrm{ad}}(\QQ) $ under the natural map $  \mathbf{G} ( \QQ   )   \to \mathbf{G}^{\mathrm{ad} } ( \QQ )  $ is also neat   \cite[Corollary 17.3]{Borel}.  Thus   $  \Gamma /  ( \mathbf{Z} ( \RR) \cap  \Gamma  ) = \Gamma / ( \mathbf{Z}(\QQ) \cap \Gamma) \subset \bar{\Gamma} $ is neat as well   and in  particular   torsion  free.    So  it must be the case that  $      \gamma \in \mathbf{Z}(\QQ) $. From  (\ref{gkequality}), we   infer that   $ k = \gamma k '  $ and this makes $ \gamma $ an element  $  K \cap  \mathbf{Z}(\QQ) $.  As $ K  \cap  \mathbf{Z}(\QQ) = K' \cap  \mathbf{Z}(\QQ) $, we see that   $ k = \gamma k ' \in K'  .  $     
The upshot   is  that   the fiber of $  \pr_{   K' ,K} $ above  $ [x,g]_{K} $ is of cardinality $  [K : K ' ]   $.   

Now  let $ L \subset K' $ be normal in $ K $. By replacing $ L $ with $ L    \cdot     (\mathbf{Z}(\QQ) \cap K) $, we may assume that $L  \cap \mathbf{Z}(\QQ)    = K \cap \mathbf{Z}(\QQ) $.   Applying the same argument to $ L $, we  see that $ \mathrm{Sh}_{\Gb}(K)(\CC) $ 
is a quotient of $ \mathrm{Sh}_{\Gb}(L)(\CC) $   by  the free action of   $ K / L $,    
hence the natural quotient map is an unramified covering of  degree $ [K : L] $.   Similarly for $ K' $. This implies the claim since $ [K : L ] = [K  : K   ' ] \cdot  [K' : L]  $.   
 \end{proof}  
\begin{corollary}   \label{shimuracartesian}   Let $ L    , L   ' ,  K  $ be neat  compact open subgroups of $ \Gb(\Ab_{f}) $ such that $ L  , L'  \subset K $ and all three have the same intersections     with $ \mathbf{Z}(\QQ) $.         For   $  \gamma \in K $, let $ L_{\gamma} = L \cap  \gamma L' \gamma ^{-1} $.  
Then  the diagram  below   
\begin{center}   
\begin{tikzcd}    \bigsqcup_{\gamma} \mathrm{Sh}(L_{\gamma} )  \arrow[r]  \arrow[d ,  "{\sqcup[\gamma]}",  swap ] &  \mathrm{Sh}(L)   \arrow[d,   ]     \\
\mathrm{Sh}(L')  \arrow[r]  &   \mathrm{Sh}(K) 
\end{tikzcd}
\end{center} 
where $ \gamma \in K $ runs over representatives of $ L  \backslash K / L '   $  is  Cartesian  in the category  of smooth $ \CC   $-manifolds.    
\end{corollary}    
If canonical models exist for $ \mathrm{Sh}_{\Gb}(K) $, then the following lemma allows us to descend the Cartesian property above to the level of varieties.   
\begin{lemma}   \label{pullbackdescent}         Let $ W , X,  Y , Z $ be geometrically reduced locally of finite type schemes over a field $ k  $ of characteristic zero  forming a commutative diagram 
\begin{center} \begin{tikzcd}  W   \arrow[r, "a"] \arrow[d, "g",  swap] & 
 X  \arrow[d, "f"]   \\
Z  \arrow[r, "b"]  &  Y 
\end{tikzcd}
\end{center} 
such that $ f ,    g $ are  \'{e}tale. Suppose for  each  closed point $ z \in Z $  the map $ a  : W \to X $ is injective on the pre-image $ g^{-1}(z) $ and surjects onto the pre-image of $ f^{-1}(b(z)) $. Then the diagram above is Cartesian in the category of $   k    $-schemes. 
\end{lemma}
\begin{proof}  Suppose that $   \mathcal{W}  = X \times _{ Y } Z  $ is a pullback. Let $ p_{X} : \mathcal{W} \to X $, $ p_{Y}  :   \mathcal{W}  \to Z  $ be the natural  projection  maps and $ \gamma : W  \to  \mathcal{W}     $ the map induced  by the universal property  of $  \mathcal{W} $. 
As  $  f  $  is  \'{e}tale, so is  $ p_{Z}    $  and since $ p_{Z}    \circ  \gamma  =   g  $   is  \'{e}tale, so is $ \gamma   $.   Let   $   \bar{k }  $ denote  the separable closure of $  k  $.  Since $ \mathcal{W}(\bar{k   }       ) = \left \{ (x,z) \in X(\bar{k}) \times  Z ( \bar  { k   }     ) \, | \,   f ( x ) =   b     (  z    )  \right \}  $,  the  condition on closed points (i.e.\  $ \bar{k}$-points)  implies that  $ \gamma :  W(\bar{k})  \to  \mathcal{W}(\bar{k}) $ is a bijection. The result follows since an \'{e}tale  morphism between such schemes that is bijective on $ \bar{k}  $ points is  necessarily    an  
isomorphism.
\end{proof}  

We now assume for the rest of this subsection that $ \Sh_{\Gb}(K)(\CC) $ admits a canonical model for each neat level $ K $. We let
$ \Upsilon $ be any collection of  neat compact open subgroups of $ \Gb(\Ab_{f}) $ such that the intersection of any  $ K \in  \Upsilon $ with $ \mathbf{Z}(\QQ) $  gives a subgroup of $ \mathbf{Z}(\QQ) $ that is independent of  $ K $. For instance,  we may take $ \Upsilon = \Upsilon(K_{0}) $ for any given neat level $ K_{0} $ where $ \Upsilon(K_{0}) $ is the set of all finite intersections of conjugates of $ K_{0} $. By Lemma \ref{UpsilonLemma}, such a collection satisfies (T1)-(T3) and clearly, the intersection of any group in $ \Upsilon(K_{0}) $ with $ \mathbf{Z}(\QQ) $ equals $ K_{0} \cap \mathbf{Z}(\QQ) $. Now   let $  \left \{ \mathscr{F}_{K} \right \}_{K \in \Upsilon}  $  be a collection of $ \ZZ_{p}$-sheaves $ \mathscr{F}_{K} $ on $  \Sh_{\Gb}(K) $  that are equivariant under the pullback action of $ \Gb(\Ab_{f} )$. 
More precisely, for any $ \sigma \in \Gb(\Ab_{f}) $ and $ L, K \in \Upsilon $ such that $ \sigma^{-1} L  \sigma  \subset K $, we assume that there are natural isomorphisms $ \varphi_{\sigma} :  [\sigma]_{L,K}^{*} \mathscr{F}_{K} \simeq \mathscr{F}  $ such that $ \varphi_{\tau \sigma} = [\tau]^{*}_{L',L}   \circ     \varphi_{\sigma}  $ for any $ L ' \in \Upsilon $ satisfying $ \tau^{-1} L ' \tau  \subset L $. 
For any integer $ i \geq 0 $ and $ K \in \Upsilon $, let $$ M(K) :=  \mathrm{H}^{i}_{\et}( \Sh_{\Gb}(K), \mathscr{F}_{K}) $$    denote Jannsen's continuous \'{e}tale cohomology.  Then   for  any  morphism $  ( L \xrightarrow{ \sigma } K ) \in \mathcal{P}(G, \Upsilon) $, there are induced $ \ZZ_{p}$-linear maps $ [\sigma]_{L,K}^{*} : M(L)  \to M(K) $ and $  [\sigma]_{L,K}^{*}  : M(K)  \to M(L) $  that make $ M $ a  RIC   functor  for  $ \mathcal{P}(G, \Upsilon) $ (see \cite[Appendix  A]{Anticyclo}).       
\begin{proposition} $ M $ is  a  cohomological Mackey functor.   
\label{ShimuraisCoMackproposition}    
\end{proposition}
\begin{proof} Lemma \ref{etalecovershimuralemma} and \cite[Tome 3, Expose IX, \S 5]{SGA4}) imply that $ M $ is cohomological.   Corollary     \ref{shimuracartesian} and \cite[Proposition A.5]{Anticyclo}  imply that $ M $ is Mackey.    
\end{proof}    

\begin{remark} 
Using similar arguments, one may establish that an injective morphism $ (\Hb, Y) \hookrightarrow (\Gb, X) $ of Shimura-Deligne data and a collection on sheaves of the two sets of varieties that are compatible under all possible pullbacks  induce  a Mackey pushforward on the corresponding cohomology of varieties over the reflex field of $ (\Hb, Y) $. See e.g.,  \cite[\S 4.4]{Anticyclo}.  Some care is  required in the case where  the centers of $ \Hb $ and $ \Gb $ differ and (SD5) is not satisfied for $ \Hb $. This is because one to needs to specify a collection of compact open subgroups for $ \Hb(\Ab_{f} )$ which contains the pullback of $ \Upsilon $ and which also satisfies the  conditions of intersection with the center of $ \Hb(\Ab_{f} )$.        
\end{remark}

%% file: AbstractZeta.tex
\section{Abstract zeta   elements}    
 \label{abstractsec}

In this section, we begin  by  giving  ourselves a  certain  setup that one    encounters in, but  which it is not necessarily  limited  to,  questions  involving  pushforwards  of  elements  in     the cohomology of Shimura  varieties   and  we formulate  a  general  problem in  the style of Euler  system  norm  relations within that setup. We then propose an abstract resolution for it by defining a notion we refer to as  \emph{zeta elements} and study its  various properties.   An example involving CM points on modular curves is provided in  \S \ref{toyexamples} and the reader is encouraged  to refer to it 
 while   reading   this section.  We note  for the convenience of the reader that in the said example, it is the group  denoted  `$\tilde{G} $' (resp., `$\tilde{K}$') that plays the role of the group denoted `$ G $' (resp., `$K$') below. 
\subsection{The setup}   
 \label{abstractzetasection}
Suppose for all of this  subsection that we are given 
\begin{enumerate} [before = \vspace{\smallskipamount}
]    \setlength\itemsep{0.1em}     
\item[$\bullet$]        $ \iota : H \hookrightarrow G $ a closed immersion of unimodular  locally profinite  groups,   
\item [$\bullet$] $\Upsilon_{H} $, $ \Upsilon_{G} $   non-empty    collections of compact open  subgroups satisfying (T1)-(T3) and    $ \iota^{-1}(\Upsilon_{G}) \subset \Upsilon_{H} $,       
\item[$\bullet$] $ \OO $  an  integral  domain  with field  of fractions a $ \QQ $-algebra,
\item[$\bullet$] $ M_{H,\OO} : \mathcal{P}(H,   \Upsilon_{H}   ) \to \mathcal{O}\text{-Mod}   $, $ M_{G,  \OO}    : \mathcal{P}(G,   \Upsilon_{G}    ) \to \mathcal{O}\text{-Mod} $ CoMack  functors,    
\item[$\bullet$]   $  \iota _ { * }  : M _ {H , \OO }     \to  M _{  G ,  \OO  }   $ a pushforward,   

\item[$\bullet$] $ U  \in  \Upsilon _{H} $, $ K \in  \Upsilon_{G} $ compact opens such that $ U = K  \cap  H   $ referred to as  \emph{bottom levels},   

\item[$\bullet$]   $ x _ {   U }  \in  M _ { H , \OO    }  (  U  )   $ which we call the \emph{source bottom class},         

\item[$\bullet$]  $ \mathfrak{H} \in   \mathcal  { C }   _  {  \OO   }    (   K  \backslash  G    /  K   )   $ a  non-zero   element which  we call the \emph{Hecke  polynomial},    

\item[$\bullet$] $ L \in  \Upsilon _ {  G }  $, $ L   \triangleleft     K  $ a  normal compact open subgroup referred to as a \emph{layer extension of degree} $  d =  [ K : L ]   $.
\end{enumerate}
As in   Definition    \ref{Heckecorrdefinition}, $ \mathfrak{H} $ induces a $ \OO $-linear map $ \mathfrak{H}_{*} =  \mathfrak{H}^{t}  :  M_{G , \OO } ( K ) \to M_{G, \OO} ( K ) $. Let $ y_{K} : =  \iota_{U, K ,  * } (  x _ { U }    )  \in  M_{G, \OO} (   K )  $ which we call the \emph{target bottom class}.      

\begin{question}   \label{NormRelationProblem}        
Does there exist  a  class $ y   _ {  L   }    \in  M_ { G ,  \OO }  ( L  )  $ such  that  $$     \mathfrak{H }_{*} (y_{K})  =    \pr  _ { L , K  ,  * } ( y _ { L  }   )    $$
 as  elements  of    $  M _ {  G  ,    \OO } ( K )  $?   
\end{question}

\begin{note}   \label{note}       Let us first  make a  few  general  remarks.  First note note is that if  $ d \in \OO^{\times} $, the class $ d^{-1} \cdot  \pr_{L,K}^{*}(y_{K}) \in M_{G,\OO}(L) $ solves  the   problem   above.   Thus  the non-trivial case occurs only when $ d $ is not invertible in $ \OO $, and in particular when $ \OO $ is not a field. In  Kolyvagin's bounding argument, the usefulness of such a norm relation  is indeed where $ d $ is taken to be non-invertible e.g.,  
$ \OO = \ZZ_{p} $ and $ d = \ell - 1 $ where  $ \ell \neq p $ is a prime such that a large power of $ p $ divides $   \ell  -   1    $.

Second, Problem \ref{NormRelationProblem} is meant to be  posed as a family of  such  problems where one varies $ L $  over a prescribed lattice of compact open subgroups of $K$  (which correspond to layers of certain abelian field extensions) together with the other parameters  above  and the goal is to construct  $ y_{L}  $  that satisfy such    relations  compatibly  in  a   tower.       This is typically achieved by breaking the norm relation problem into  `local'    components and varying the parameters componentwise. More precisely, $ H $ and $ G $ are in practice  the  groups of adelic points of certain  reductive  algebraic  groups over a number field and the class $ x_{U} $  has the features of a  restricted    tensor  product.    The problem above   is then   posed     for each   place in   a      subset of  all  finite  places of the number field.  Thus Problem \ref{NormRelationProblem} is  to be seen as one of a local  nature  that  is            extracted      from  a   global setting. See  \S \ref{gluingzetasection} for an  abstract formulation of  this   global     scenario.

Third, the  underling premise of \ref{NormRelationProblem} is that $ y_{K} $ is the image of a class $ x_{U} $ that one can vary  over the levels of  the   functor $ M_{H, \OO } $ and for which   one has a better description   as     compared to their counterparts  in   $  M _ {     G  ,    \OO   } $.  If  $ \iota_{*} $ is  also  Mackey, then Lemma   \ref{MixedHeckecompose} tells us  that $ \mathfrak{H}_{*}  (y_{K}) $ is the image of certain mixed Hecke correspondences. The    class $ y_{L} $  we are  seeking is  therefore    required   to be of a similar form.   As  experience  
suggests,  we assume that  $ y_{L} =    \sum _ { i = 1  }  ^ { r } [V_{i}g_{i} L]_{*}   ( x_{V_{i} } )   $   where
\begin{itemize}     
\item  $  g_{i} \in G $,
\item $ V_{i} \subset g_{i} L g_{i} ^ { -1 }  $, $ V_{i} \in \Upsilon_{H} $
\item  $  x_{V_{i}}  \in  M_{H, \OO } ( V_{i} )  $  
\end{itemize} 
are unknown quantities that we need to pick to obtain the said equality.    
If we only require equality up to $ \OO $-torsion (which suffices for applications, see  \ref{torsiondoesn'tmatter}), then one can use Proposition \ref{LemmEin}     to guide  these  choices. More precisely, let $ \mu_{H} $ be a  $  \QQ   $-valued   Haar  measure on $  H   $, $ \Phi $ a field containing $ \OO $, and $ M_{H, \Phi} $, $ M_{G, \Phi} $ denote the functors obtained by tensoring with $ \Phi $. Let  $  \widehat{M} _ {H , \Phi }  $,   $   \widehat {  M  }_{ G , \Phi } $ be the completions of $ M_{H, \Phi }   $,   $  M _ { G , \Phi     }  $    respectively. For $ V \in \Upsilon_{H} $, let  $ j_{ V  }    : M_{H,    \OO } (  V )  \to   \widehat { M } _{    H  ,   \Phi   }  $ denote   (abusing  notation)     the natural map and similarly for $ M_{G, \OO} $.  Let  $  \hat { \iota } _{*}     :    \widehat{M}_{H, \Phi}  \otimes   \mathcal{H}_{\Phi}(G , \Upsilon_{G} )         \to    \widehat {  M  
 } _  {G , \Phi }  $ the completed  pushforward of  Proposition  \ref{completedhecke}.     
As $ M_{G,\Phi} $ is cohomological, the kernel of $   j_{K}  : M_{G,\OO}(K) \to \widehat{M}_{H,\Phi} $ is contained in $ \OO $-torsion of $ M_{G, \OO}(K) $. An application of   Corollary  \ref{completedhecke}    then  implies that  $  \mathfrak{H}_{*}  ( y _ { K } )   -    \pr_{L,K, * } ( y_{L} ) $ is $ \OO $-torsion   if and only if \begin{equation}   \label{keyequality} 
\hat { \iota } _{ * }  \big (  j_{U} ( x _{U} ) \otimes  \mathfrak{H}   \big  )    =   \hat { \iota   } _{*}  \left ( \sum _ {i = 1}   ^ {  r   }     \mu_{H}(U)/\mu_{H} ( V_{i} ) (  j_{V_{i} }( x_{V_{i} } )    \otimes  \mathrm{ch}(g_{i} K  ) \big )   \right  )
\end{equation} 
as elements of the $ \Phi $-vector space $ \widehat{M}_{G, \Phi} $  (see the  proof of     Proposition  \ref{zetaworks} below). Thus we  are seeking a  specific ``test vector"   
  in $ \widehat{M}_{H, \Phi}  \otimes   \mathcal{H}_{\Phi}(G , \Upsilon_{G} )  ^{K} $  containing the data of certain  elements in  $ M_{H, \OO} $  
whose image under $ \hat{\iota}_{*} $  coincides with that of $ j_{U}(x_{U})\otimes \mathfrak{H}  $.  Any such test   vector can equivalently be seen as a right $ K $-invariant compactly supported function $ \zeta : G \to \widehat{M}_{H, \Phi} $.  The shape of the  element inside $ \hat{\iota}_{*} $ on the RHS of  (\ref{keyequality})    forces upon us  a notion of \emph{integrality} of such vectors. As $ \hat{\iota}_{*} $ is $ H$-equivariant,  a natural way of    enforcing    (\ref{keyequality}) is to require that the two functions  in the inputs of $ \hat{\iota}_{*} $ have \emph{equal $H$-coinvariants} with respect to the natural $ H$-action on the set of such functions. If a test vector  satisfying these two conditions exists,  Problem \ref{NormRelationProblem} is solved modulo $\OO$-torsion. In fact, such a vector solves the corresponding problem (modulo torsion) for \emph{any} pushforward emanating from $ M_{H,\OO} $ to a functor on $ \mathcal{P}(G,\Upsilon_{G})$, since the two  aforementioned properties  are    completely   independent of $ \iota_{*} $.
Under certain additional conditions, the  resulting norm    relation can be upgraded    to an equality.    
See \S    \ref{handling torsion}            
\end{note}
We now formalize the discussion   above. For $ \tau $ an arbitrary group, we let $  \mathcal{C} (G/K, \tau  ) $  denote the  set  of all  compactly supported functions $ \xi : G \to   \tau  $ that are invariant under right translation by $ K $ on the source. Here the support of $ \xi $ is the set of elements that do not map to identity element in $ \tau $.   If $ \tau $ is abelian  and  has the structure of  a $ \Phi $-vector space,   
$ \mathcal{C} (G/K, \tau ) $ is a $ \Phi $-vector space.  If $ \tau $ is in addition a $ \Phi $-linear  left $ H  $-representation,  so is $ \mathcal{C}(G/K, \tau) $ where  we let  $ h \in H $ act on $ \xi \in   \mathcal{C}   (G/K,   \tau  )  $ via   $ \xi \mapsto h \xi  : =  h   \xi ( h^{-1} ( - ) ) $.  In this case, we denote by $  \mathcal{C}(G/K, \tau ) _{H} $ the   space of $H$-coinvariants and  write $ \xi_{1}  \simeq     \xi_{2} $ if $ \xi_{1}, \xi_{2}   \in   \mathcal{C}(  G  / K ,   \tau )     $ fall in the same $ H $-coinvariant class.   Given a $ \phi  \in   \mathcal{C}   (G/K, \Phi) $ and $ x \in  \tau   $, we let $ x \otimes \phi \in   \mathcal{C}    (G/ K ,   \tau   ) $ denote the  
function given by  $ g \mapsto \phi(g) x $. Fix a   $  \QQ $-valued     Haar  measure $ \mu_{H} $ on $ H $. For $ V_{1}, V_{2} \in  \Upsilon_{H} $, we denote 
$ [V_{1} : V_{2} ] : = \mu_{H}(V_{1} ) /  \mu_{H}(V_{2}) $.  This is then independent of the choice of $ \mu_{H} $.         

\begin{definition} An  element    $ \xi \in     \mathcal{C}  (G/K , \widehat{M}_{H,\Phi})   $  is said to be  \emph{$ \OO $-integral at level $ L $}   if for  each $ g \in G $, there exists a finite collection $ \left \{ V_{i} \in \Upsilon_{H} | V_{i}  \subset gL g^{-1}  \right \}_{i \in I }   $ and  classes $ x_{V_{i}} \in  M_{H, \OO} ( V_{i})  $ for each $ i \in  I $  such that $$ \xi(g) =   \sum \nolimits    _{i \in  I    } [U:V_{i}]  \,   j_{V_{i}}( x_{V_{i}}    ) . $$     A  \emph{zeta element} for $ (x_{U}, \mathfrak{H}, L)$ with coefficients in $ \Phi $   is  an  element $ \zeta \in   \mathcal{C}  (   G    / K , \widehat{  M } _{H, \Phi} ) $ that  is     $ \OO $-integral at level $  L $ and lies in the $ H$-coinvariant class 
of          $  j_{U}(x_{U}) \otimes  \mathfrak{H}  $ 
\label{abstract zeta} 
\end{definition} 

\begin{remark} This notion of integrality  appears in \cite[Definition 3.2.1]{gu21}. Cf. \cite[Corollary 2.14]{Anticyclo}.       
\end{remark} 

Let $ \zeta $ be   a      zeta  element for $ (x_{U} ,  \mathfrak{H} , L  ) $. Then we may   write $ \zeta $ as a (possibly empty if $ \zeta = 0 $) finite   sum $ \sum_{  \alpha  } [U : V_{\alpha}]  \,  j_{V_{\alpha}} (x_{V_{\alpha}} )  \otimes \ch(g_{\alpha}K) $  where  for  each $ \alpha $,   $ V_{\alpha} \subset g_{\alpha} L g_{\alpha}^{-1} $ and $ x_{V_{\alpha}} \in  M_{H, \OO } ( V_{\alpha} ) $.    
Given such a presentation of $ \zeta $, we refer to 
$$ y_{L}   : =  \sum  \nolimits       _  { \alpha }   [V_  {  \alpha  } g_{ \alpha } L   ] _ {  *  } (x_{V_{\alpha}})  \in  M_{G, \OO} ( L ) $$    as an \emph{associated class} for $ \zeta $ under $ \iota_{*} $. 
It depends on the choice of the  presentation for $ \zeta  $.

\begin{proposition}   \label{zetaworks}     Suppose   there  exists a zeta 
element for $ (x_{U} , \mathfrak{H} , L )  $.   
Then  for    any  associated class $ y_{L} \in M_{G, \OO}(K) $,   the   difference     $ \mathfrak{H}_{*}(y_{K} )  - \pr _{ L , K , *}   ( y _ { L} )   $ lies in the $ \OO $-torsion of $ M_{G, \OO}( K )    $.    
\end{proposition}

\begin{proof} As above, let  $  \zeta =   \sum_{\alpha} [U : V_{\alpha} ]   \,     j_{V_{\alpha}}( x_{V_{\alpha}} )   \otimes \ch(g_{\alpha} K )  $ be  a choice (of a presentation)  of  a zeta element to which $ y_{L} $  is   associated.    
Let $  j_{K}   :  M_{G, \OO} (K)  \to  \widehat{M}_{G , \Phi}  $ denote  the natural map and  let     $  \mu_{G} $ be a Haar measure on $ G $ such that $ \mu_{G}(K) = 1 $.    
Corollary   \ref{HeckeAgree} and the properties of $ \hat{\iota}_{*} $  as 
 an   intertwining map (Lemma \ref{intertwining})  imply      that      $$  \mu_{H}(U)  \cdot  j_{K}(  \mathfrak{H}_{*}(y_{K}) )   =  \mathfrak{H}^{t}  \cdot   \hat{ \iota } _{ * } ( j_{U} ( x_{U} ) \otimes  \ch(K)    )    =   \hat { \iota  }    _{*} ( j_{U}( x_{U} ) \otimes \mathfrak{H}   )  .  $$ Since $ \hat{\iota}_{*} $ is $ H $-equivariant, its restriction to $ M_{H,\Phi} \otimes \mathcal{H}(G, \Upsilon_{G})^{K}  \simeq   \mathcal{C}  
(G/K,  
\widehat{M}_{H, \Phi} ) $ factors through the space  of  corresponding 
    $ H $-coinvariants. Since  $  \big ( j_{U}(x_{U}   )   \otimes \mathfrak{H} \big  
   )    \simeq     \zeta       $   
by  assumption,       $$ \hat{\iota}_{*}  \big ( j_{U}(x_{U} ) \otimes  \mathfrak{H}  \big )  =   \sum   \nolimits    _{\alpha} [U:V_{\alpha}] \,   \hat{\iota}_{*}     \big     ( j_{V_{\alpha}}(x_{V_{\alpha}    }) \otimes \ch(g_{\alpha} K ) \big ).   $$    Corollary \ref{completedhecke} 
allows us to rewrite each summand on the right hand side above  as   $   \mu_{H}(U)     \,      j_{K} \circ [ V  _  {  \alpha }  g_{ \alpha } K ]_{*}(x_{V_{\alpha} } ) $.  By  Lemma  \ref{mixedHeckecompose}, $  [ V_{\alpha} g_{\alpha} K ]  _{ * } = \pr_{L, K , * } \circ  [ V_{\alpha} g_{\alpha} L ]  _ {  *  } $.  Putting everything together, we get that         \begin{align*} 
\mu_{H}(U)   \cdot  j_{K} ( \mathfrak{H}_{*}( y_{K} )  )   & =   \mu_{H} ( U )   \cdot      j_{K} \circ \pr_{L, K , *} \left ( \sum   \nolimits    _{   \alpha }   [V_{ \alpha }  g_{ \alpha } L ]  _{ * } ( x_{V_{  \alpha } } )  \right  )   \\
& =   \mu_{H} ( U )  \cdot j_{K}  \left  (   \pr_{L,K,*} ( y_{L} )   \right     )   
\end{align*}
Thus $   j_{K} ( \mathfrak{H}_{*}(y_{K}) - \pr_{L,K,*}( y_{L} ) )  = 0 $.  This implies the claim since the kernel of $  j_{K}  :    M_{G, \OO}  ( K ) \to M_{G, \Phi} ( K ) \to  \widehat{M}_{G, \Phi }  $ is  $    M_{G, \OO}(K)_{\OO\text{-}\mathrm{tors}} $  by   Lemma \ref{cohoinj}.    
\end{proof}  

We next  study how  a given  presentation of a   zeta element  may be  modified. 
\begin{notation}  \label{notationf}  Given $ f \in   \mathcal{C}   (G/K, H) $ and $ \xi \in  \mathcal{C} (G/K, \tau ) $ for $ \tau $ any   left $ H $-representation over $ \Phi $, we define $  f  \xi \in   \mathcal{C}      (G/ K,  \tau ) $ by $ g \mapsto f(g) \xi( f(g)^{-1} g )  $.    
\end{notation}

\begin{lemma}   \label{zetatwisting}     If $ \zeta $ is a zeta element, so is $ f  \zeta $ for any $ f \in   \mathcal{C}    (     G/K, H) $.  Moreover the set of associated classes for the two elements under any    pushforward are equal.           
\end{lemma}

\begin{proof}   Clearly  
$ f\zeta $ lies in $   \mathcal{C}  (    G/K,   \widehat{M}_{H,  \Phi} )  $ and   $ f \zeta  \simeq 
    \zeta $.      Say  $ \sum_{\alpha}   [ U  :   V_{\alpha} ] \,  j_{V_{\alpha}} ( x_{V_{\alpha}} ) \otimes \ch( g_{\alpha}   K )  $ is a  presentation for $ \zeta $.    
Set $ h_{\alpha} : = f (
g_{\alpha} ) $, $  V_{\alpha}' : = h_{\alpha} V_{\alpha} h_{\alpha}^{-1} $ and $ x_{V_{\alpha}'} ' : = [h]_{V_{\alpha}', V_{\alpha}}^{*} ( x_{V_{\alpha}}) $. Then  $$   
f     \zeta =  \sum   \nolimits    _{\alpha}  [ U  :  V_{\alpha}   ]   \,      j_{V_{\alpha}'} ( x_{V_{\alpha}'} ' )  \otimes \ch( h g_{\alpha} K ) . $$
Since $ [U : V_{\alpha} ]  =  [ U :  V_{\alpha}' ] $  by   unimodularity of $ H $,  $  f \zeta $ is integral at   $ L $.   That the sets of   associated  classes for $ \zeta $ and $ \eta \zeta $ under a pushforward   are equal   follows  by  Lemma   \ref{mixedhecketwisting}.       
\end{proof} 

\begin{definition} Let $ \zeta $ be a zeta element.   A  presentation $   \zeta  =     \sum_{\alpha}  [U : V_{\alpha}]   \,      j_{V_{\alpha}}(x_{V_{\alpha}}) \otimes  \ch(g_{\alpha} K ) $ 
is said to be  \emph{optimal} if   $ V_{\alpha} = H \cap  g_{\alpha} L g_{\alpha}^{-1}$ for all $ \alpha $  and the    cosets     $ H g_{\alpha } K $ are pairwise   disjoint. We say that $ \zeta $ is  \emph{optimal} if it has an  optimal  presentation.   
\end{definition}    
\begin{lemma}   \label{optimal} If there exists a zeta element, there exists an optimal one and such that the set of associated classes of the latter element under any pushforward contains those of   the  former.    
\end{lemma}   

\begin{proof}    Let $   \zeta =  \sum _{ \alpha \in A} [U  : V_{\alpha} ]   \,    j_{V_{\alpha}}(x_{V_{\alpha}}   )     \otimes \ch( g_{\alpha }  K )     $  be  a   presentation of a zeta element. Say there is an index $ \beta \in A  $ such that $ V_{\beta} \neq   H \cap   g_{\beta}   L g  _{\beta}    ^{-1}   $.  Temporarily denote    $ V_{\beta} ' :  =  H \cap g_{\beta} L g _ { \beta }$ and $ x_{V_{\beta}'} = \pr_{V_{\beta},  V_{\beta}' ,  * }  ( x_{V_{\beta}} ) $.  Then $$ [U : V_{\alpha} ] \,     j_{V_{\beta}}  ( x_{V_{\beta}} )  \otimes \ch( g _{\beta} K ) ,  \quad  \quad    [U : V_{\beta} ' ] \,  j_{V_{\beta}'} (  x_{V_{\beta} '} )  \otimes  \ch( g_{\beta}   K ) $$ are equal in  $    \mathcal{C}( G/ K ,  \widehat{M}_{H, \Phi})  _  { H   }    $    by  Lemma  \ref{RICtracemakey}.   So the element $ \zeta ' $ obtained by replacing the summand indexed by $ \beta  $ in $ \zeta $  with $  [U : V_{\beta} ' ] \,  j_{V_{\beta}'} (  x_{V_{\beta} '} )  \otimes  \ch( g_{ \beta  }  K ) $  constitutes a  zeta   element.        Since $    [V_{\beta}  ' g_{\beta} L   ]  \circ  \pr_{V_{\beta} , V_{\beta} ' , *} = [V_{\beta}  g_{\beta} L ]_{*} $ by  Lemma  \ref{mixedHeckecompose}, the  associated classes    for  the  chosen    presentations of $ \zeta $ and $ \zeta' $ are equal.    So we  can   assume that there is no such index $ \beta $ in our chosen $ \zeta $. But then  $ f \zeta $ for any $ f \in   \mathcal{C}(G/K,  H ) $ is a zeta element (Lemma \ref{zetatwisting}) which also  possesses the same properties.    By choosing $ f $ suitably, we can  ensure that $ H g _{\alpha} K $  are  pairwise            disjoint. 
\end{proof}

\begin{remark}    
The   terminology
`zeta element'  is inspired by \cite{kkato} and motivated by the fact that  Hecke polynomials specialize to 
zeta functions of Shimura varieties \cite{Blasius-Rogawski},  \cite{Langlandszeta}. 
\end{remark} 
\subsection{Existence Criteria}   \label{zetacriteriasection}  
In this section, we  derive a necessary and sufficient criteria   for the     existence of zeta elements  that can be applied   in   practice.

Retain the  setup of \S \ref{abstractzetasection} and the notations introduced  therein. For $ X \subset H $ a group and $ g \in G $, we will often denote by $ X_{g} = X_{g,K} $ the intersection   $ X \cap g K g^{-1} $. For the result below,  we denote   by  $ \tau $ be an arbitrary  left   $ H $-representation over $ \Phi $.     
For $ \xi \in  \mathcal{C}( G/ K , \tau ) $, we  let    $  f_{\xi} \in  \mathcal{C}(G/K, H ) $ be an element  satisfying the following condition: $  \mathrm{supp}  (   f_{\xi} \xi  )  =  \sqcup_{i}  g_{i} K $ and $ H g_{i}     K $ are pairwise  disjoint (see Notation \ref{notationf}).  It is clear that a  $ f_{\xi} $ exist for each $ \xi  $.

\begin{lemma}   \label{mathcalMzero} The class $ [\xi]_{H} \in  \mathcal{C}(G/K, \tau)_{H} $   vanishes   if and only for each $ g \in G $, the class of $ (f_{\xi} \xi)(g)  \in \tau  $ in the space $ \tau_{H_{g} } $ of $ H_{g} $-coinvariants   vanishes.      
\end{lemma}    
\begin{proof} 
For each $ \alpha \in H \backslash G / K $,  fix a choice  $ g_{\alpha} K \in G / K $ such that $ H g_{\alpha} K =  \alpha $ and set  $$ \mathcal{M} : = \bigoplus_{\alpha \in H \backslash G / K }  \tau_{H_{g_{\alpha}}}  .  $$   We are going to define  a $\Phi$-linear map $ \varphi :  \mathcal{C}(G/K, \tau ) \to \mathcal{M} $. Since $ \mathcal{C}(G/K, \tau) \simeq  \bigoplus_{gK \in G/ K} \tau $, it suffices to specify  $ \varphi $ on simple tensors.  Given $ x \otimes \ch(gK) \in  \mathcal{C}(G/K, \tau) $, let $ \alpha : =  H g K $ and  pick $ h \in H $ such that $ h g K = g_{\alpha} K $. Then we  set $ \varphi  \big (  x \otimes \ch(gK )  \big   )  \in \mathcal{M} $ to be the element whose component at any index $ \beta \neq \alpha $ vanishes and at $ \alpha $ equals   the      $  H_{g_{\alpha}}   $-coinvariant class of $ hx $.   
It is straightforward to verify $ \varphi $ is well-defined  and factors through the quotient  $ \mathcal{C}(G/K, \tau)_{H} $.

We now prove the claim. If $ \xi = 0 $, the claim is obvious, so assume otherwise.   Since $[f_{\xi} \xi]_{H} = [\xi]_{H} $, we may replace $ \xi $ with $ f_{\xi} \xi $ and assume    wlog   that elements of $ \supp( \xi)/K \subset G/K $ represent distinct cosets in  $ H \backslash G  / K $. Say $ \xi = \sum_{i   }    x_{i} \otimes \ch(g_{i}K ) $. Denote $ \alpha_{i} : = H g_{i} K $ and let $ h_{i} \in H $ be such that $ h_{i} g_{i} = g_{\alpha_{i}} K $. If $ [\xi]_{H}  $  vanishes, 
so does $ \varphi(\xi) $  which in turn implies that the class of $ h_{i} x_{i} $ in $   H_{g_{\alpha_{i}} }$-coinvariants 
of $ \tau $ vanishes for each $ i $. By conjugation,   this is equivalent to the vanishing of $  H_{g_{i}}   $-coinvariant  class   of $ x_{i} $ for each $ i $. This proves the only if direction. The  if   direction is straightforward since the vanishing of $  H_{g_{i}}    $-coinvariant  class  of $ x_{i}  \in  \tau $ readily   implies    the same for the  $ H_{g_{i}} $-coinvariant 
 class  of $ x_{i} \otimes \ch(g_{i} K )  \in  \mathcal{C}(G/K , \tau )  $.  
\end{proof}

\begin{definition}  For $ g \in G $, the \emph{$ g $-twisted $ H $-restriction} or the \emph{$(H,g)$-restriction} of $ \mathfrak{H} $ is the function $$ \mathfrak{h}_{g} :  H \to \mathcal{O} $$
given by $ h \mapsto \mathfrak{H}(hg) $ for all $ h \in H   $. 
\end{definition} 

\begin{notation}  \label{notationalpha}    For each $ \alpha \in H \backslash  H \cdot \supp(\mathfrak{H} ) / K $, choose a representative $ g_{\alpha} \in G  $ for $ \alpha $. We denote (abusing notation)  $ H_{\alpha} = H \cap g_{\alpha} K g_{\alpha}^{-1} $, $ V_{\alpha} = H \cap g_{\alpha} L  g_{\alpha}^{-1} $, $ d_{\alpha} = [ H_{\alpha}  :  V_{\alpha} ]  $  and $ \mathfrak{h}_{\alpha} = \mathfrak{h}_{g_{\alpha}}  $ denote  the $ (H, g_{\alpha})$-restriction of $ \mathfrak{H} $.  
\end{notation}  

\begin{theorem}       \label{zetacriteria}  There exists a  zeta  element for $  ( x_{U} ,  \mathfrak{ H  }  ,  L  ) $  if  and  only if there exist classes $x_{V_{\alpha}} \in M_{H, \mathcal{O}}\left(V_{\alpha}\right)$ for all $ \alpha \in  H  \backslash ( H    \cdot  \supp   \mathfrak{H}  ) / K $  
such that 
$$  \mathfrak{h}_{\alpha}^{t}  
  \cdot  j_{U}(x_{U})   =  j_{H_{\alpha}}         \circ \pr_{V_{\alpha}, H_{\alpha},*}\left(x_{V_{\alpha}}\right)
$$
in $  \widehat{ M } _{H, \Phi} $.      
Moreover if $ y_{L} $ is an   associated  class for a given zeta element under $ \iota_{*} $ and $ M_{H, \OO} $ is $ \OO $-torsion  free,    the classes $ x_{V_{\alpha}}      $  satisfying the criteria   above can be picked  so  that    $ 
 \pr_{L, K, *}  ( y _ {  L }  )        =    \sum_{\alpha } [V_{ \alpha }  g_{\alpha} K ]_{*} (x_{V_{\alpha} 
   } ) $ where the sum runs over 
   $ \alpha \in H  \backslash ( H \cdot \supp \mathfrak{H}  ) / K $.         
\end{theorem}     
\begin{proof}      For notational convenience, we will  denote by $  A   : =  H \backslash  H  \cdot  \supp (   \mathfrak{H} )     / K  $ and $ x : = j_{U}(x_{U})  $  
in the    proof.   Since $ U  \backslash G / K 
 \to K \backslash G  / K  $ is surjective and  has finite fibers,  we have a natural injection $ \mathcal{C}_{\OO}( K \backslash G / K )  \hookrightarrow    \mathcal{C}_{\OO}  (  U \backslash G / K ) $ given by  $ \ch ( K  \sigma K )  \mapsto \sum  \ch ( U \tau   K ) $ where $ \tau $ runs over $ U \backslash K \sigma K / K $. Via this map, we  consider $ \mathfrak{H} $ as an element in $ \mathcal{C}_{\OO}( U \backslash G / K ) $. We  assume that  
\begin{equation} \mathfrak{H}   =  \sum   \nolimits  
   _{ j  \in J } c_{j }   \,   \ch ( U \sigma _{j} K ) 
\end{equation}    where $ J $ is a finite indexing set, $ U \sigma_{j} K $ are pairwise disjoint double cosets and   $ c_{j}  \neq 0   $ for all $ j \in  J $. For $ \alpha \in  A  $, let $ J_{\alpha} \subset J $ be the set of all $ j \in  J $     such that $ H \sigma _{j} K =  \alpha $. For each $ j \in J_{\alpha}  $, let $ h_{j} \in H $ 
be such that $ \sigma_{j} K =  h_{j} 
 g_{\alpha} K $.  Then   $ U h_{j}  H_{\alpha}   \in  U  \backslash H / H_{\alpha}  $ is independent of the choice of $ h_{j} $ and  \begin{equation}     \label{pseudo} \mathfrak{h}_{\alpha  } : = \sum  \nolimits    _{ j \in  J_ { \alpha } } c_{j}   \,    \ch (   U    h_{j}   H_{\alpha}    )     \in  \mathcal{C} _ {  \OO  }     ( U \backslash G /  H_{\alpha} 
 ) . 
\end{equation} 
%The degree of an element in a group algebra is defined to its image under augmentation map. \\ 

\noindent   $(\! \impliedby \!)$    Assume that there exist  $ x_{V_{\alpha}}  \in  M_{H, \OO} ( V_{\alpha}   ) $ satisfying the equality in the statement.   
We claim that  
 $\zeta=\sum_{\alpha  \in   A   } [ U : V_{\alpha} ]   \,  j_{V_{\alpha}}( x_{V_{\alpha}}  )  \otimes  \ch( 
 g _ { \alpha } K ) $ is  a   zeta element.  As  $ \zeta $ is clearly $ \OO $-integral, we need to  only  show that $   x 
 \otimes   \mathfrak{H}   \simeq    \zeta     $.  To this end,  note  that     
 \begin{align*}
  x    \otimes \mathfrak{H}  &  \simeq     \sum_{ j \in  J  }    c_{j}   \,     [ U : U_{\sigma_{j}    } ] \,   \big  (  
    x    \otimes  \ch (   \sigma_{j}  K )  \big )   \\
   &   \simeq   
   \sum_{j \in J}   c_{j}  \,     
 [ U :  H_{\sigma_{j}} ]   \,   \deg 
 [U    \sigma_{j} K ]_{*}  \big (  x  \otimes \ch ( \sigma_{j} K  )  \big  )  \\
&  \simeq   \sum_{ \alpha \in   A 
   }     [ U  : 
 H_{\alpha} ] \sum_{j \in J_{\alpha}} c_{j}  
 \deg     [U \sigma _{j} K]_{*}  \big  (   h_{j}^{-1}   x   \otimes  \ch  (  g_{  \alpha }  K  )  \big  )    
\end{align*}
where    the third relation uses that $ \mu_{H} ( H_{\sigma_{j}  })  = \mu_{H}( H_{\alpha} )   $ for 
$ j \in J_{\alpha} $.  For each $ \alpha   \in   A     $, let $ \theta _{\alpha} \in \Phi [H_{g_{\alpha}}] $ denote the sum  over  
 a set of representatives of $ H_{\alpha  } / V_{\alpha} $ and for each $ j \in  J_{\alpha} $,  let $ \varsigma_{j}  
\in 
\Phi[H_{\alpha}] $ denote the sum over  a set of  representatives for $   H_{\alpha} /  ( h_{j}^{-1} U _{\sigma_{j}}  h_{j}  )  $.  By  Lemma \ref{HeckeAgree} and Lemma  \ref{RICtracemakey}, we have $$   
\mathfrak{h}_{\alpha}^{t} \cdot x =   \sum_{ j \in J_{\alpha}}  c_{j} \varsigma_{j} h_{j}^{-1}  x,   \quad \quad   \quad         
j_{H_{g_{\alpha}} }  \circ \pr_{V_{\alpha}, H_{g_{\alpha}},*}(x_{V_{\alpha}})   =     \theta_{\alpha}   \,        j_{V_{\alpha}} ( x_{V_{\alpha}} )      $$
Now note that $  \deg  \varsigma_{j} = [H_{\alpha} :  h_{j}^{-1}  U_{\sigma_{j}} h_{j}   ] =   \deg  [U \sigma_{j}  K ] _{*} $ where degree of an element in a group algebra denotes its image under the augmnetation map.  Thus for each $ \alpha \in A   
$,                  
\begin{align*}    \sum    \nolimits  
_  {j \in J_{\alpha} } c_{j} \operatorname{deg} \, [U \sigma _{j} K]_{*}   \big  (  h_{j}  ^ { -  1  }   
 x \otimes  \ch (  g_{\alpha} K   )   \big   )  &  \simeq   \sum  \nolimits  _{ j \in  J_{ \alpha }   }      
 c_{j}  \big (   \varsigma_{j} h_{j} ^{-1}  x   \otimes  \ch (  g_{\alpha }   K   )    \big  )     \\
  & =   ( \mathfrak{h}_{\alpha}^{t}  \cdot   
  x ) \otimes \ch ( g_{\alpha} K )  \\
  & =   ( j_{H_{\alpha} }  \circ \pr_{V_{\alpha}, H_{\alpha}, *} (x_{V_{\alpha}} )  )   \otimes  \ch( g_{\alpha}  K ) \\
  & =  \big ( \theta_{\alpha}  
  \,    j_{V_{\alpha}} ( x_{V_{\alpha}} )  \big )   \otimes  \ch( g_{\alpha} K  )   \\   &       \simeq    [H_{\alpha} : U_{\alpha} ]   \cdot     j_{V_{\alpha}} (  x_{V_{\alpha}}   )     \otimes  \ch( g_{\alpha} K )  
\end{align*}
Putting everything together, 
we deduce that   \begin{align*} x \otimes  \mathfrak{H}  & 
   \simeq   \sum _{\alpha \in  A   
 }  [U : H_{\alpha}] [ H_{\alpha} : V_{\alpha} ]  \cdot    j_{V_{\alpha}}  ( x_{V_{\alpha}  }  ) \otimes \ch(g_{\alpha} K )  \\
 & =  \sum_{\alpha \in A } [U : V_{\alpha} ] j_{V_{\alpha}}(x_{V_{\alpha}  } ) \otimes \ch( g_{\alpha} K )  = \zeta .   
\end{align*}  
This completes the proof of the if direction.  \\

\noindent  $ (\!  \implies  \!) $   Suppose that $ \zeta  $ is a zeta element for $  (x_{U}, \mathfrak{H}, L ) $. Invoking  Lemma 
\ref{optimal}, we may assume that $ \zeta $ is optimal. Say $ \zeta  = \sum \nolimits  _ {\beta \in B} [U : V_{\beta} ] j_{V_{\beta}}(x_{V_{\beta}}) \otimes \ch( g_{\beta} K) $ is an optimal presentation where $ B  $ is some finite indexing set. We identify    $ B   $   with   a subset of $ H \backslash G / K $ by identifying $ \beta   \in   B   $ with $ H g_{\beta} K  \in  H \backslash G / K   $.   Extending $  B  $ by adding zero summands to $ \zeta $ if necessary, we may assume that $ A  \subset  B    $. 
Lemma \ref{zetatwisting} allows us to further assume that $ \left \{ g_{\alpha}   K  \, | \, \alpha \in A    \right \}  \subseteq \left \{ g_{\beta}   K  \, | \, \beta \in  B        \right \}   $.   

We claim that $ x_{V_{\beta}} $ for $ \beta \in  A   $    are the  desired elements.   Set $ \mathfrak{h}_{\beta} : = 0 $  and $ d_{\beta} : = [H_{g_{\beta}   }     : V_{\beta} ]    $ for $  \beta \in  B  
 \setminus   A    $.  Our calculation in the first part and the fact that  $ \zeta \simeq  x \otimes  \mathfrak{H}  $  imply   that      
$$ \sum_{\beta \in  B } [U : H_{g_{\beta}} ]  \big (  
 d_{\beta}  j_{V_{\beta}} (  x_{V_{\beta} }   )    - \mathfrak{h}_{\beta} ^{t }\cdot  x  \big  ) \otimes \ch(g_{\beta} K )  \,  \simeq  \, 0 . $$
Lemma \ref{mathcalMzero} now implies that  $ H_{g_{\beta}} $-coinvariant  class  of $ d_{\beta} j_{V_{\beta}} (x_{V_{\beta}} ) -    \mathfrak{h}_{\beta} x $ vanishes   for each $  \beta \in 
 B $. Fix a $ \beta \in   B      $ for the remainder of this   paragraph     and write  
\begin{equation}  \label{Hcoinvariants} d_{\beta}   j_{V_{\beta}} ( x_{V_{\beta}}  )   -   \mathfrak{h}_{\beta}^{t}  \cdot x =  \sum_{i=1}^{k} (\gamma _{i} -1) x_{i}   
\end{equation}   
where $  \gamma _{i} \in H_{g_{\beta}} $ and $ x_{i} \in  \widehat{  M} _{H,\Phi}  $. Let $ W \subset  H_{g_{\beta}} $ be a normal compact open subgroup contained in $ V_{\beta} $  
such that $ W $ fixes $ x_{i} $ for each $ i = 1, \ldots, k $. If $ \beta \in  A  $,  
we require in addition  that   $ W \subset h_{j}  ^{-1}  U h_{j}  \cap g_{\beta} K g_{\beta}^{-1} $ for each  $ j \in J_{\beta} $.  
Let $ Q  = H_{g_{\beta}}/W $ and 
$ e_{Q}   \in \Phi[H_{g_{\beta}}] $ denote    $|Q|^{-1} $ times the sum over a set of representatives in $ H_{g_{\beta}} $ for $ Q $.  Then left multiplication of elements in $ \widehat{M}_{H,\Phi}$ by $ e_{Q}  $ annihilates $ (\gamma_{i} - 1) x_{i} $ for all $ i = 1, \ldots, k $,  stabilizes    $ \mathfrak{h}_{\beta}^{t}  \cdot  x \in (\widehat{M}_{H,\Phi})^{H_{g_{\beta}}}   $         and sends $  d_{\beta} j_{V_{\beta}} (x_{V_{\beta}} ) $ to  $  j_{H_{g_{\beta}}} \circ \pr_{V_{\beta}, H_{g_{\beta}, *}} ( x_{V_{\beta}} ) $. Thus multiplying  (\ref{Hcoinvariants}) by $ e  _ { Q  }$ on both sides yields an equality involving  $\mathfrak{h}_{g_{\beta}} $ and $ x_{V_{\beta}} $. The equalities for $ \beta \in  A  $ are the ones  sought  after. This completes the proof of the only if   direction.   \\

It remains to prove the seoncd claim. So assume that  $ M_{H,\OO} $ is $ \OO$-torsion free and let $ y_{L} $  be   an associated class for an arbitrary  zeta element.   By the second part of  Lemma \ref{optimal}, we may pick the optimal presentation for $ \zeta $ (as we did at the start of the if direction above) to  further     ensure that  $$ y_{L} =   \sum_{\beta \in  B   }     [V_{\beta} g_{\beta} L]_{*} ( x_{V_{\beta}}     )  .   $$     Since  
$  \pr_{L, K, *}  \circ [ V_{\beta} g_{\beta}  L ] =  [V_{\beta} g_{\beta} K ] _{*} $   for  each  $ \beta \in  B $ (see Lemma \ref{mixedHeckecompose}),   it suffices to show that $  [V_{\beta} g_{\beta} K]_{*} ( x_{V_{\beta}}   ) = 0 $ for each $ \beta \in  B  \setminus   A     $ to establish the second claim.  Torsion freeness   of  $ M _{ H , \OO  }   $  implies that  
$$ j_{ H_{g_{\beta}} }        : M_{H, \OO}   ( H_{g_{\beta}}  )   \to \widehat{M}_{H, \Phi}   $$  is injective for all $ \beta$.  
Thus the conclusion of the  previous  paragraph gives  $ \pr_{V_{\beta} , H_{g_{\beta}}, * } ( x_{V_{\beta}}  ) = 0 $ for all $ \beta \in  B  \setminus   A  $      (recall that $ \mathfrak{h}_{\beta} = 0 $ for such $ \beta $).   Invoking  Lemma  \ref{mixedHeckecompose}   again, we get that     $ [V_{\beta} g_{\beta} K ]_{*}(   x_{V_{\beta}} ) =  [H_{g_{\beta}} g_{\beta} K ] \circ \pr_{ V_{\beta} , H_{g_{\beta} }  *  }( x_{V_{\beta}}  )  = 0  $ 
which finishes the proof.    
\end{proof}

\begin{corollary} 
A  zeta element   for $ ( x_{U} , \mathfrak{H} ,   L )  $   exists      over $ \Phi $ if and only if one   exists   over  $  \mathrm{Frac}  ( \OO   )     $. \label{easyzeta-1}  
\end{corollary}

\begin{proof}   The if direction is trivial and the only if direction follows by  Theorem  \ref{zetacriteria} and injectivity of $ M_{H, \mathrm{Frac} ( \OO   ) }  ( V )  \to  M_{H, \Phi}(V) $ for any $ V   \in \Upsilon _   {  H   }    $.    
\end{proof}    

\begin{corollary} Suppose in the notation introduced at the start of the proof of Theorem \ref{zetacriteria} that for each $ \alpha \in H \backslash  H \cdot \supp(\mathfrak{H}) / K $,   
\begin{itemize}    [before = \vspace{\smallskipamount}, after =  \vspace{\smallskipamount}]    \setlength\itemsep{0.1em}   
\item $ j_{U}(x_{U}) \in  \widehat{M}_{H, \Phi}  $ lifts    to  a  class  in   $   
 M_{H,\OO}  ( H _  {\alpha } ) $\footnote{this condition is automatic if $ H_{\alpha }  \subseteq  U $}, 
 \item for each $ j \in J  _  {  \alpha    }    $,   we  have     $ h_{j}^{-1}  \cdot j_{U}(x_{U}) =  a _  {j} \,  j_{U}(x_{U} )   $    for  some  $ a _ { j }  \in  \Phi   $.  
 \end{itemize}  
 Then there exists a zeta element  for  $ ( x_{U}  ,  \mathfrak{H}   ,  L  )   $     if $ \sum   \nolimits   _{j \in J_  {  \alpha } 
  }  c_{j}  a_{j}    \deg [ U \sigma_{j}  K ] _{*}    \in d_{\alpha } \mathcal{O}   $ 
 for all $  \alpha $.      \label{easyzeta0}  
\end{corollary}
\begin{proof} For  each  $ \alpha 
 $, let    $ x_{H_{\alpha}  }  \in M_{H, \OO} ( H_{\alpha  }    ) $ denote an element satisfying 
  $ j_{H_{\alpha} } ( x_{H_{\alpha}} ) =  j_{U}( x_{U} ) $. Then   
$ \ch  ( H_{\alpha} h_  { j }^{-1} U ) \cdot j_{U}(x_{U})  $ equals $  
a_{j} 
 \deg  [U \sigma_{j} K ]_{*} \,  j_{U}(x_{U}) $ for each $ j \in J_{\alpha} $.   
So by  Theorem \ref{zetacriteria}, a zeta element exists in this case if (and only  if) there exist $ x_{V_{\alpha}  } \in  M_{H, \OO} ( V_{\alpha}    ) $  for  each $  \alpha  $     such that   
\begin{equation}   \label{trivialzeta}  \sum   \nolimits    _{j \in J_{  \alpha  } }  \big (  c_{j}  a_{j}    \mathrm{deg}\,[U \sigma_{j} K ]_{*}    
\big   )  
\, j_{U} ( x_{U}    )     =    j_{H_{ \alpha }   }             \circ \operatorname{pr}_{V_{\alpha }, H_{\alpha   }    , *}\left(x_{V_{\alpha} }\right)    . 
\end{equation}    
But if $ \sum_{j \in  J_  {  \alpha   }  
  }  c_{j}  a_{j} \deg  [ U \sigma _ { j } K  ] _{*}   
= d_{\alpha} f _{\alpha} $ for some $ f_{\alpha} \in \OO    $,   
 (\ref{trivialzeta}) holds with   
 $ x_{V_{\alpha} } : =  
 \pr_{V_{\alpha} ,  H_{\alpha   }  
 } ^{ * } (  f_{\alpha}     x_{H_{g_{\alpha}}   }    ) $.      
\end{proof}

\begin{corollary}   \label{easyzeta1}      Suppose that $M_{H, \mathcal{O}} $ is the trivial functor on $ H $ and $ x_{U}  \in  M_{H,  \OO}(U)  =    \OO $ is an invertible element.     Then a   zeta  element   exists  for   $  ( x_{U}  ,   \mathfrak{ H  }  ,  L    )   $     if and only if  $$ \deg    (  \mathfrak{h}_{ \alpha }^{t}   ) 
 \in   d_{\alpha}  \OO   $$  for  all $ \alpha  \in     H \backslash   H \cdot  \supp  (\mathfrak{H}   )    
 / K  $.   
\end{corollary}

\begin{proof}    The if part is clear since the conditions of Corollary \ref{easyzeta0}  are satisfied  with $   a_{j}  =   1   
 $ and the sum $ \sum_{ j \in J_{\alpha}} c_{j} \deg [U \sigma_{j} K]_{*} $ equals 
 $ \deg  \mathfrak{h}_{\alpha} $.       
For the only if part, note  that (\ref{trivialzeta}) in the previous  proof  is also 
necessary.  
So we may  assume that there exist   $ x_{V_{\alpha}} \in M_{H, \OO}(V_{\alpha})  =    \OO  $   such that (\ref{trivialzeta}) holds. Now $ x_{U} = x_{H_{\alpha}} \in \OO $,   
$ x_{V_{\alpha}} = f_{\alpha} x_{ U } $ where $ f_{\alpha} = x_{V_{\alpha}}  \cdot 
 x_{U}^{-1}  \in \OO  $ and   $ \pr_{V_{\alpha} , H_{g_{\alpha}} , * }  (x_{V_{\alpha} }  ) = d_{\alpha}f_{\alpha}  x_{U} $.  Thus  
(\ref{trivialzeta})  is equivalent to   $$   \sum   \nolimits   
 _{ j \in J_{\alpha}}  
 \big  (   c_{j} \deg [U \sigma_{j} K ]_{*} 
 \big    )  
 x_{ U }      = d_{\alpha}  f_{\alpha} x_{U }  $$   
and    the claim follows by multiplying by $ x_{U}^{-1}$ on both sides.        
\end{proof}   
The next result is included for completeness and shows that if the norm relation problem \ref{NormRelationProblem} is trivial, so is the existence of a zeta element. 
\begin{corollary} If   $  \mathfrak{h}_ { \alpha}^{t}   \cdot j_{U}(x_{U}  ) \in \widehat{M}_{H, \Phi} $ lifts to class in $  d_{ \alpha }     \cdot     M_{H, \OO} ( H_{\alpha}  ) $ for all $ \alpha  \in  H      \backslash     H   \cdot  \supp  (  \mathfrak{H}  )      / K   $,     a zeta element  exists  for $ (x_{U} ,  \mathfrak{H},  L )  $.  The   lifting  condition holds  automatically for  an  $ \alpha $  if $ d  _  {  \alpha   }    $  is  invertible  in   $ \OO   $.  In particular, a zeta element exists   unconditionally     if $ d  $ is invertible in $ \OO $.       
\label{easyzeta2} 
\end{corollary}
\begin{proof}  Let  $ x_{H_{\alpha  }    }      \in M_{H, \OO} ( H_{ \alpha } ) $  be  such that $ j_{H_{\alpha }      }      ( d_{\alpha}  x_{H_{\alpha }   }     )  =    
 \mathfrak{h}_{\alpha} ^{t}  \cdot j_{U} ( x_{U}  ) $. The   criteria  of   Theorem     \ref{zetacriteria} is 
 then satisfied by taking $   x_{V_{\alpha}} : = \pr_{V_{\alpha} , H_{g_{\alpha}} } ^{*}   ( x_{H_{\alpha}  } )    \in M_{H,\OO}(V_{\alpha})  $.  By Lemma  \ref{HeckeAgree}, 
 we have $$  j_{H_{\alpha}}   (  [ U h_{j} H_{\alpha}  ]_{*} ( x_{U} )  )  =   \ch( H_{\alpha} h_{j}^{-1} U )    \cdot  j_{U} ( x_{U} ) .  $$
Since $ [H_{\alpha} h_{j}^{-1} U ] ( x_{U} )   \in  M_{H, \OO} (H_{\alpha}   )  $, the class 
  $  \mathfrak{h}_{\alpha}^{t}   
  \cdot j_{U}(x_{U})   $   always lifts to an element in $ M_{H, \OO}  (  H_{g_{\alpha}}  ) $. In particular when  $  d_{ \alpha  
 }    \in  \OO  ^ { \times } $, this class   lifts to an element in $ d_{ \alpha }   \cdot  M_{H , \OO }  (H_{g_{ \alpha }  }  )    = M_{H, \OO} ( H _ { g_{ \alpha } 
  }      )    $.  
\end{proof}

As noted in the proof above,  each  $ \mathfrak{h}_{\alpha} $ gives rise to an $ \OO$-linear map $$ \mathfrak{h}_{\alpha,*} :    M_{H, \OO}(U)  \to  M_{H, \OO}(H_{g_ { \alpha } }) $$
given by Hecke correspondences in the covariant convention. 
Theorem \ref{zetacriteria} then says that  constructing a  zeta  element  amounts to finding  $ x_{V_{\alpha}} \in M_{H, \OO} (V_{\alpha}) $ such that   
 \begin{equation}   \label{pseudoequal}       j_{H_{g_{\alpha} } }    \circ \mathfrak{h}_{  \alpha  ,  *   }  
 (  x_{U}   ) = j_{H_{ \alpha }  }  \circ \pr_{V_{\alpha} , H_{ \alpha  } , *  }    (    x_{V_{\alpha}} ) . 
 \end{equation}
 Using this, we can record the following version of Theorem \ref{zetacriteria}.   
\begin{corollary}   \label{torsionfreezeta}      Suppose $ M_{H, \OO} $ is $ \OO$-torsion free. Then a zeta element exists for $ (x_{U} , \mathfrak{H}, L) $ if and only if there exist $ x_{V_{ \alpha  }}    \in M_{H, \OO}(V_{\alpha} ) $ such that $$ \mathfrak{h}_{\alpha, *} (x_{U}) = \pr_{V_{\alpha} , H_{ \alpha    }     , * } ( x_{V_{ \alpha }   }  ) $$ 
for all $ \alpha \in H \backslash 
  H  \cdot   \supp (   \mathfrak{H}    )      / K $. 
   Moreover if $ y_{L} $ is an associated class for a zeta element under $ \iota _{*} $, the classes $ x_{V_{\alpha}} $ can be picked to ensure that $ \pr_{L,K,*}(y_{L}) = \sum_{\alpha} [V_{\alpha} g_{\alpha } K]_{*}  ( x_{V_{\alpha}}  )  $.           
\end{corollary}

\begin{corollary}    \label{looksowise} Let  $ \mathfrak{H}  '  \in  \mathcal{C}_{\OO} ( K \backslash G / K ) $  such   that $  \mathfrak{H} - \mathfrak{H} ' \in  d  \cdot   \mathcal{C}_{\OO} ( K \backslash G / K ) $. Then  there   exists a zeta element for $   (x_{U}   ,   \mathfrak{H }   , L ) $ if and only if there exists one for  $ ( x_{U  } ,  \mathfrak{H } ' , L ) $   \end{corollary}
\begin{proof} For $ g \in G $, let $ \mathfrak{h}_{g}' \in \mathcal{C}_{\mathcal{O}} ( U \backslash H / H_{g} ) $ denote the $ (H,g) $-restriction of $ \mathfrak{H}' $. Then $ \mathfrak{h}_{g} - \mathfrak{h}_{g}' \in  d \cdot \mathcal{C}_{\mathcal{O}}(U\backslash H / H_{g} ) $. 
Since  $ d \cdot  M_{H, \OO} ( H \cap g K g^{-1}   ) $ is in the image of the trace map from $ M_{H, \mathcal{O}}(H \cap g L g^{-1}) $, the claim follows.   
\end{proof}

\begin{remark}   \label{cyclesarecongruences}   The motivation  behind these  criteria 
is that in practice   the source  functor $ M_{H,\OO} $ is    much better understood (e.g.,  ones    arising from   cycles or   Eisenstein  classes)     than the  target  $ M_{G,\OO} $  and  the results above   provide   a means for parlaying this  additional   knowledge (and that of the Hecke polynomial)   for  Euler  system  style   relations. In fact in all the cases that  we will consider, $M_{H,\OO}    $  will      be a space of functions on a suitable topological space that would parametrize classes in the cohomology of Shimura varieties. In \S \ref{Schwartz} we study the trace  map      for     such spaces  in detail.    
\end{remark}  
\begin{remark} 
For the case of cycles  coming from a sub-Shimura datum, the collection of fundamental classes of the  sub-Shimura  variety   constitutes the trivial functor on $ H $ (see  \cite{Anticyclo} for a concrete instance).  In this case,  Corollary  \ref{easyzeta1}  applies   and  proving norm relations amounts to verifying certain congruence conditions. One may of course use the finer structure of the connected components of a Shimura variety as   prescribed by the reciprocity laws  of  \cite{DeligneTS}.   See   \cite{explicitdescent}  where   a general      formula   for  the    action of Hecke operators  is provided. 
We  however  point out   that  for the case considered in \cite{Anticyclo},  working with the trivial functor turns out to be necessary as the failure of axiom (SD3)   precludes the possibility of describing the geometric connected components of the source Shimura variety.   
\end{remark}

\begin{remark} While the if direction is the ``useful" part of Theorem \ref{zetacriteria}, the 
only if direction provides strong evidence that one does not need to look beyond the test vector specified by twisted restrictions, say, in local zeta integral computations.    
\end{remark}

\subsection{Handling Torsion} \label{handling torsion} We now address the equality of norm relation asked for in  Problem   \ref{NormRelationProblem}    without forgoing torsion. We retain the setup at the  start of \S \ref{abstractzetasection} and \S \ref{zetacriteriasection}, in particular  Notation \ref{notationalpha}.

\begin{theorem}  Suppose  that        $ \iota_{*}$ is Mackey and $  M_{H, \OO}  $  is  $\OO$-torsion  free. 
If a zeta element exists, any associated class $y_{L} \in M_{G,\OO}(L) $   
satisfies   
\label{eventorsion}     
$ \mathfrak{H}_{*}\left(y_{K}\right)=\operatorname{pr}_{L, K, *}\left(y_{L}\right)$.  
\end{theorem}
\begin{proof} 
By  Corollary   \ref{torsionfreezeta},  we can find $ x_{V_{\alpha}}  \in   M_{H , \OO } (  V_{\alpha } ) $   satisfying  
$   \pr_{L, K , * }  ( y_{L}    )  =   \sum _ {\alpha  \in   A    }    [ V_{\alpha}  g_{\alpha} K   ]_{*} ( x_{V_{\alpha}  }  )   $ and   
\begin{equation}    
\label{psuheckeequal} \mathfrak{h}_{\alpha, * }( x_{U} )  = \pr_{V_{ \alpha } ,  H_{\alpha} ,   * } ( x_{V_{\alpha} } ).    
\end{equation} In the notation introduced at the start of the proof of Theorem \ref{zetacriteria},  we see using 
Lemma  
 \ref{mixedHeckecompose} 
 that $$\mathfrak{H}_{*}\left(y_{K}\right) 
 = \sum_{j \in J}\left[U \sigma_{j} K\right]_{*}\left(x_{U}    \right)  = 
  \sum_{ \alpha \in  A    } \sum_{j \in J_{\alpha}} c_{j}\left[U \sigma_{j} K\right]_{*}\left(x_{U}\right).  $$   
For $ j \in J  $, let $W_{j}  :    =h_{j}  ^{-1 }U_{\sigma_{j}} h_{j}  $. Then for $ j \in J_{\alpha} $,  $ W_{j}  = h_{j}^{-1}  U h_{j}  \cap g_{\alpha}  K g_{\alpha}^{-1} \subset H \cap g_{\alpha} K g_{\alpha}^{-1} = H_{\alpha } $. 
By Lemma  \ref{mixedhecketwisting}  and   \ref{mixedHeckecompose},   we  see that       
\begin{align*}
\mathfrak{H}_{*}( y_{K}  )  &=\sum_{ 
 \alpha \in  A    } 
 \sum_{j \in J_{\alpha}}    c_{j}\left[U_{\sigma_{j}} \sigma_{j} K\right]_{*} \circ \operatorname{pr}_{U_{\sigma_{j}}, U}^{*}\left(x_{U}\right) \\
&=\sum_{  \alpha  \in  A    }  \sum_{j \in J_{\alpha}}   c_{j}\left[W_{j} g_{ \alpha }    K\right]_{*} \circ\left[h_{j}^{-1}\right]_{ W_{j}, U }^{*} \left(x_{U}\right) \\
&=\sum_{  \alpha  \in   A    } \sum_{j \in J_{\alpha}   } c_{j}\left[H_{\alpha} g_{\alpha} K\right]_{*} \circ \operatorname{pr}_{W_{j}, H_{g_{\alpha}} 
 , *} \circ\left[h_{j}^{-1}\right]_{W_{j}, U}^{*}\left(x_{U}\right)
\end{align*}
Now note that $  U \cap \,  h_{j} H_{g_{\alpha}}    h_{j}^{-1} =     U _ {\sigma_{j} }    $ and   $ h_{j} ^ { - 1 }  U   
h_{j}  \cap H_{  \alpha     }  =   W_{j}  $. Thus   $$ \mathrm{pr}_{W_{j}, H_{\alpha}     , *} \circ\left[h_{j}^{-1}\right]_{W_{j}, U}^{*}\left(x_{U}\right)=\left[H_{  \alpha   } h_{j }^{-1}   U\right]\left(x_{U}\right) = [U h_{j} H_{\alpha}]_{*}(x_{U)}  $$ (see   the diagram below) and so  $ \sum_{j \in J_{\alpha}   }    c_{j}  \cdot     \pr_{W_{j}, H_{g_{\alpha} },  * } \circ [ h_{j}^{-1}  ]^{*}  _ {   W_{j}  , U } ( x_{U} )   =   \mathfrak{h}_{\alpha, *}   ( x_{U} )  $.             
\begin{center} 
\begin{tikzcd}
  & M_{H, \mathcal{O} } \left(U_{\sigma_{j}}\right)   \arrow[r, "{[h_{j}^{-1}]^{*}}"] & M_{H, \mathcal{O}}\left(W_{j}\right)   \arrow[dr, "{\pr_{*}}"] &  M_{H, \mathcal{O}}\left(V_{\alpha   }  
  \right)  \arrow[d, "{\pr_{*}}"]  \\
M_{H,\OO}(U) \arrow[ur,"{\pr^{*}}"]   \arrow[rrr, "{   [  U  h_{j   }  H_{\alpha}  ]_{*}    } "   ,  swap     ]         & & &  M_{H, \OO}  ( H_{\alpha}    ) 
\end{tikzcd}
\end{center}
Therefore by eq.\   (\ref{psuheckeequal}) 
and Lemma  \ref{mixedHeckecompose},  we   have     
\begin{align*}
\mathfrak{H}_{*}\left(y_{K}\right) 
& =   \sum_{\alpha \in  A    } [H_{\alpha } g_{ \alpha } K ]_{*} \big ( \mathfrak{h}_{\alpha,  * }  ( x_{U} )  \big )  \\
&=\sum_{\alpha \in   A    } \left[H_{\alpha} g_{\alpha} K\right]_{*} 
 \circ \operatorname{pr}_{V_{\alpha}, H_{\alpha}, *}\left(x_{V_{\alpha} 
 }    \right)      \\
&=\sum_{\alpha \in  A } 
\left[V_{\alpha} g_{\alpha} K\right]_{*} (x_{V_{\alpha}} ) =\operatorname{pr}_{L, K, *}\left(y_{L}\right)
\end{align*}   
which   finishes  the   proof.    
\end{proof}

\begin{remark}  \label{nozeta}      In the proof above, the relation $ \mathfrak{H}_{*}\left(y_{K}\right)=\operatorname{pr}_{L, K, *}\left(y_{L}\right)$ in Theorem \ref{eventorsion} can also be derived  under the weaker assumption that  the relations eq.\ (\ref{psuheckeequal}) is satisfied  modulo the kernel of $   \iota_{*}  :   M_{H,\OO}(H_{\alpha})  \to  M_ {  G , \OO} (g_{\alpha} K  g_{\alpha}^{-1} )  $ for each $ \alpha $ and   even when  $ M_{H,\OO}$ is not $\OO$-torsion free.   In particular, this result (which is all we really need for Euler systems) can be  stated   
without ever referencing  zeta elements.    However as noted in the introduction, the notion of zeta elements to connects the approach of \cite{LSZ}, \cite{Anticyclo} etc., with ours, and also ``explains" the nature of integral test vectors chosen in these works.     
\end{remark}

\begin{remark}   \label{torsiondoesn'tmatter}                 In   applications to Shimura varieties, one eventually projects the norm relations to a $\pi_{f}$-isotypical component of the cohomology of the target Shimura variety,  where $\pi_{f}$ is (the finite part
of) an irreducible cohomological automorphic representation of the target reductive group,  in order to land in the first Galois cohomology $\mathrm{H}^{1}$ of a Galois representation $\rho_{\pi}$ in the multiplicity space of $\pi$. The projection step, to our knowledge,   requires the coefficients to be in a field. Thus the information about torsion is lost anyway i.e.,   one apriori obtains norm relations in the \emph{image}  of $\mathrm{H}^{1}$ of a Galois stable lattice $ T_{\pi}  \subset \rho_{\pi}  $ inside $\mathrm{H}^{1}$ of the Galois representation $ \rho_{\pi}  $.       One way to retrieve the torsion in the norm relation after projecting to Galois representation   is to to use Iwasawa theoretic arguments e.g.,  see      \cite{Anticyclo}  or 
   \cite{LSZ}.        We will however not address this question here.
\end{remark}

\subsection{Gluing Norm Relations} \label{gluingzetasection}   We  now  consider  the  `global'          version   of  Problem  \ref{NormRelationProblem}. Let 
$ I $ be an indexing set and 
$ \iota_{v} : H_{v}  \to   G_{ v  }  $ be a collection of  embeddings of  unimodular locally profinite  groups indexed by $ v \in I $.   We will consider $ H_{ v  }$ as  a subgroup of $  G_{ v }  $ via $ \iota_{v} $.   For each  $ v \in I $,  let     $ K_{v } \subset  G_{ v } $ be a  compact open subgroup and set $ U_{v} := H \cap K_{v} $. Let $ G $, $ H $ denote respectively the restricted direct product of $ G_{v} $, $H _{v} $ with respect to $ K_{v} $, $ U_{v} $ over all $ v $.  Let $ K $, $ U$ denote respectively the products of $ K_{v} $, $ U_{v} $ over all $ v $.   For any finite subset  $ \nu \subset I $,   we  define  $ G_{\nu} = \prod_{v \in \nu} G_{v}  $ and $ G^{\nu} = G/G_{\nu} $.   If $ \nu = \left \{ v \right \} $, we denote $ G^{\nu} $  simply  as     $ G^{v} $.    We  similarly define the $  H $, $ K $ and $ U $ versions.

For all but finitely many $  v  \in   I $, say we are given   $ L_{v}   $    a normal compact open subgroup of $ K_{v} $. Let $ I' \subset 
I $ denote the set of  all   such $ v $  and let $ \mathcal{N} $ denote the set of 
 all     finite subsets of $ I ' $.   Let $ \Upsilon_{G} $ be   a     collection of compact  open   subgroups     satisfying (T1)-(T3)  that  contains $ K $ and the groups $ L_{v}K^{v} $  
 
 for all $ v \in I' $. 
Let $ \Upsilon_{H} $ be a collection of  compact open subgroups of $ H $ satisfying (T1)-(T3) and  which contains $ \iota^{-1}  ( \Upsilon_{G}     )  $.  Fix  $ \OO $ an integral domain whose field of fractions $ \Phi $ is a $ \QQ$-algebra. Let $$ M_{H , \OO } : \mathcal{P}(H, \Upsilon_{H})  \to  \OO  \text{-Mod},  \quad \quad   \quad       \quad       M_{G ,\OO}  : \mathcal{P}(G,  \Upsilon_{G} ) \to  \OO\text{-Mod} $$ be CoMack functors and $ \iota  _{  * } : M_{H, \OO }  \to  M_{G, \OO} $  be a Mackey pushforward. Let $x_{U} \in M_{H, \OO} ( U ) $  be  a class and 
  denote      $ y_{K} :  = \iota_{U,K,*}(x_{U}) $ its image  in  $ M_{G, \OO}(K) $.  Suppose  we are also given for each $ v \in I ' $  an element   $ \mathfrak{H}_{v}  \in  \mathcal{H}_{\OO}( K_{v}  \backslash G_{v}  / K_{v}  )  $.    Given any $ \nu \in \mathcal{N} $, any $  K' \in \Upsilon_{G} $ of the form  $ K_{\nu} K'' $ with $ K'' \subset   G^{\nu} $,   we obtain  by  Lemma  \ref{Hecketensor} a well-defined $ \OO $-linear endomorphism $$    \mathfrak{H}_{\nu,  * }   : M_{G, \OO}( K '  ) \to  M_{G, \OO} ( K ' ) $$  induced by the tensor product  $  \mathfrak{H}_{ \nu   }    :=  \ch( K'' ) \otimes   \bigotimes_{v \in    \nu }  \mathfrak{H}_{v }   $.    
For $   \nu  \in  \mathcal{N} $,  denote 
$ K[\nu] :  =  K^{\nu} \times      \prod_{v \in \nu} L_{v}  $. If $ \nu = \left \{ v \right \} $, we denote this group   simply    by $ K[v] $. Note that $ K[\nu] = K $ if $ \nu = \varnothing $. For $ \nu ,\mu \in \mathcal{N} $  that   satisfy    $ \nu  \subset  \mu $, denote the pushforward  $ M_{G, \OO} (K[\mu]) \to M_{G,\OO} (K[\nu]) $ by $   \pr_{\mu, \nu, *}   $.    
\begin{problem}   \label{globalnormproblem}        Construct  classes $ y_{\nu} \in   M_{G, \OO}(K[\nu] ) $ for $  \nu  \in \mathcal{N} $ such that $ y_{\varnothing} = y_{K} $ and for all $ \mu, \nu \in  \mathcal{N} $   satisfying     $ \nu  \subset  \mu  $, we have $ \mathfrak{H}_{  \mu     \setminus  \nu  ,  *    } ( y_{\nu}  )  =  \pr_{  \mu,  \nu, * } (  y_{ \mu   }    ) $.

 \end{problem}

Our ``resolution"  to this problem is by assuming the existence of abstract   zeta elements at each $ v $ in  suitable RIC functors  whose restricted tensor product parameterizes classes in $ M_{H, \OO} $.       Let $ \Upsilon_{G_{v}} $ denote the  collection  of all compact open subgroups that are obtained as finite intersections of conjugates of $ K_{v} $ and $ L_{v} $.   By Lemma \ref{UpsilonLemma}, $ \Upsilon_{G_{v}} $ satisfies (T1)-(T3). Then any compact open subgroup in $  
\prod_{v \in I} \Upsilon_{G_{v}} $ whose component at $ v $ equals $ K_{v} $ for all but  finitely many $ v $ belongs to $  \in  \Upsilon_{G} $.    Let $ \Upsilon_{H_{v}} = \iota^{-1}_{v}  ( \Upsilon_{G_{v}}     )      $ and  let   $ \Upsilon_{H, I} \subset  \prod_{v \in I}  \Upsilon_{H_{v}} $ denote the collection of all subgroups whose component group at $ v $ is $ U_{v} $ for  all but finitely many $ v $. Then $ \Upsilon_{H , I } $ satisfies (T1)-(T3) and $ \Upsilon_{H, I}   \subset \iota^{-1} (  \Upsilon_{G}   ) \subset   \Upsilon_{H}  $.

\begin{theorem}   \label{gluingzeta}        Suppose  that there exists a 
 morphism $ \varphi : N \to M_{H, \OO} $  where $ N : \mathcal{P}(H, \Upsilon_{H,I}) \to  \OO$-Mod  is a restricted tensor product $  \otimes'_{v \in I} N_{v} $ of  $ \OO$-torsion free functors $ N_{v}  :    \mathcal{P}(H_{v},   \Upsilon_{H_{v}})   \to   \mathcal{O} $-Mod      taken  with respect to a collection $ \left \{ \phi_{U_{v}} \in  N_{v}( U_{v} ) \right \}_{v \in I} $ that satisfies $  \varphi ( \otimes_{v \in I} \phi_{U_{v}} )  = x_{U} $. If  a  zeta element $ \zeta_{v } \in   \mathcal{C}( G_{v}/  K_{v}   ,     \widehat{N}_{v,  \Phi}   )      
 $ exists for $ (\phi_{U_{v}} ,  \mathfrak{H}_{v} ,  L_{v}) $ for  every $ v \in I' $,    then    there exist classes $ y_{\nu} \in  M_{G, \OO}(K[\nu]) $  for each $ \nu \in \mathcal{N} $ such that $ y_{\varnothing} = y_{L} $ and     $$\mathfrak{H}_{\mu   \setminus    \nu  ,   *} ( y_{\nu} )   =  \pr_{ \mu ,  \nu,  *} ( y_{\mu} ) $$ for all $ \nu , \mu \in \mathcal{N}   $   satisfying   
     $ \nu \subset \mu  $.        
\end{theorem}

\begin{proof}  

For $ v \in I ' $, denote  $ A_{v} : =  H _{v} \backslash  H_{v} \cdot   \supp  (  \mathfrak{H}_{v}    )      /  K_{v}  $ and for each 
$    
\alpha_{v}  \in   A_{v}$,  
let $ g_{\alpha_{v}} \in G_{v} $ be a representative for the class $ \alpha_{v} $. Denote  $ H_{\alpha_{v}} : =  H_{v} \cap   g_{\alpha_{v}}  K_{v}  g_{ \alpha  _{v}  } ^ { - 1 } $, $  
V_{ \alpha_{v} }  :  =  H_{  v    }  \cap   g_{  \alpha _{v} } L _ {v }  g_{   \alpha  _{v}     }  ^{ - 1}   $ and  $\mathfrak{h}_{\alpha_{v} }  \in  \mathcal{C}_{ \OO  } ( H_{ \alpha_{v}  }     \backslash  H_{ v }  /  U_{v}   )$  be the  $(H_{v}, g_{\alpha_{v}})$-restriction 
$ \mathfrak{H}_{v} $ with  respect to    
 $ g_{\alpha_{v}    }$.   
By  Corollary  \ref{torsionfreezeta},   the existence of $ \zeta_{v} $ 
is equivalent to the existence of $ \phi_{\alpha_{v} } \in N_{v}(V_{\alpha_{v}  })  $ for all $ \alpha_{v} \in A _{v} $    
such that   $$  \mathfrak{h}_{\alpha_{v}, * } ( \phi_{U_{v}} ) =   \pr_{V_{\alpha_{v}}, H_{\alpha_{v} }  , * }  (  \phi_{\alpha_{v} }  )    . $$  
Denote by $ \imath_{*}    : N \to M $ the    pushforward given by the composition $  \iota_{*} \circ  \varphi $.  Then $ \imath_{*}    $ is Mackey since $ \iota_{*} $ is.  Recall that $ \mathcal{N} $ denotes the set of finite subsets of $ I ' $.   For  $ \nu  \in  \mathcal{N} $, we      denote  $  A _{ \nu }  :=  \prod_{   v    \in \nu }     A _{v} $. Given a $ \nu \in \mathcal{N} $ and $ \alpha    =  \alpha_{\nu}    \in  A _{\nu} $, we let  
    $ \alpha_{v} $  denote  the $ v $-th component of $  \alpha   
    $  for  $ v \in \nu $ and set  $$ H_{\alpha} :=  \prod_{v \in \nu} H_{\alpha_{v}} , \quad V_{\alpha} = \prod_{v \in \nu} V_{\alpha_{v}} ,   \quad     g_{\alpha}  :  = \prod_{v \in \nu} g_{\alpha_{v}} \quad 
 \phi_{\alpha
 } = \bigotimes_{v \in \nu} \phi_{\alpha_{v}} \in \bigotimes _  {v \in \nu } N(V_{\alpha_{v}} ) .  $$ For  $ \nu \in \mathcal{N} $, we  let $ \phi_{U^{\nu}} $   denote 
   the restricted tensor product $ \otimes'_{v \notin \nu } \phi_{U_{v}} $ 
and  define  $$ y_{\nu} : = \sum   \nolimits    _{ \alpha \in  A _ {  \nu  }   }     [U^{\nu}V_{\alpha}g_{\alpha} L_{\nu} K^{\nu} ]_{*} ( \phi _ { U ^ {\nu}   }  \otimes   \phi_{\alpha}    ) \in M_{G, \OO}(K[\nu])  $$
i.e.,  $ y_{\nu} $ is the  sum of classes  obtained by applying  mixed Hecke  correspondences $ [U^{\nu} V_{\alpha} g_{\alpha} L_{\nu} K^{\nu} ]_{*} : N(U^{\nu} V_{\alpha}) \to M_{G, \OO} (L_{\nu}K^{\nu} ) = M_{G, \OO} (K[\nu]) $  to    $ \phi_{U^{\nu}} \otimes \phi_{\alpha}  \in  N( U^{\nu} V_{\alpha} )    $ over all $ \alpha \in A_{  \nu }  $. 

We claim that $ y_{\nu} $ for $ \nu \in 
 \mathcal{N}  $    are the desired classes.  It is clear that $ y_{\varnothing} = y_{K} $ as $ \varphi( \phi_{U} ) = x_{U} $.  By Lemma   \ref{Hecketensor},  it suffices to prove the norm relation $ \mathfrak{H}_{\mu \setminus  \nu, *} (y_{\nu} )   = \pr_{\mu, \nu, *} ( y_{\mu} ) $ for $ \nu   \subset    \mu $ such that $ \mu \setminus \nu = \left \{ v \right \} $.   To this end, fix  an  $ \alpha \in  A _{\nu} $ for the remainder of this proof and 
 consider the inclusion  $ \Upsilon_{H_{v}}  \hookrightarrow  \Upsilon_{H} $ (of sets) given by  $ W_{v} \mapsto  W_{v  } V_{\alpha} U^{\mu}  $ and the inclusion $ \Upsilon_{G_{v}} \hookrightarrow  \Upsilon_{G} $ given by $ K_{v} ' \mapsto K_{v'} L_{\nu} K^{\mu}   $. 
 Let 
 $ N_{H_{v}, \alpha} : \mathcal{P}(H_{v}, \Upsilon_{H_{v}}) \to \OO $-Mod, $  M_{G_{v}, \nu } :  \mathcal{P}(G_{v}, \Upsilon_{H_{v}}) \to \OO $-Mod be  respectively      the functors obtained by fixing levels away from $ \nu $  as  specified  by  these  inclusions.      We then have a   Mackey     pushforward  $$ \imath_{ v , \alpha , *}  : N_{H_{v}   ,   \alpha     }    \to   M_{G_{v}  ,    \nu      }   $$   where for  a compatible pair $ (W_{v}, K_{v}')  \in \Upsilon_{H_{v}}  \times  \Upsilon_{G_{v}}  $,  
 the map  $ N_{H_{v},  \alpha  }(W_{v} ) \to M_{G_{v},  \nu }( K'_{v} ) $ is equal to  map     $ [U^{\mu}W^{v}V_{\alpha} g_{\alpha} L_{\nu} K'_{v} K^{\mu}]_{*} :   N(U^{\mu}W_{v}V_{\alpha})   \to      M_{G, \OO}(L_{\nu} K_{v}' K^{\nu}) $.  
 Given $ \phi_{ W_{v}} \in  N_{v}(W_{v}) $, we denote by $ \phi_{W_{v}, \alpha}  \in   N_{H_{v}   ,  \alpha} ( W_{v})  $ the element 
$ \phi_{U^{\mu} }  \otimes \phi_{W_{v}} \otimes \phi_{\alpha} \in  N_{H_{v}, \alpha} ( W_{v} ) = N(U^{\mu} W_{v} V_{\alpha}) $. Similarly for any $ \beta_{v}  \in A_{v} $,  we  let     $ \phi_{\beta_{v}, \alpha}   \in  N_{H_{v},\alpha}(V_{v})  $ denote  the  element    $ \phi_{U^{\mu}} \otimes \phi_{\beta_{v}} \otimes \phi_{\alpha} $. Then for any $ \beta_{v} \in A_{v} $, we have 
\begin{align*}    \mathfrak{h}_{\beta_{v},* }( \phi_{U_{v}, \alpha } ) & = \phi_{U^{\mu}} \,  \otimes     \,  \mathfrak{h}_{\beta_{v} ,* } (   \phi_{U_{v}} ) \,  \otimes   \,   \phi_{\alpha}    \\ & = \phi_{U^{\mu}} \, \otimes \,  \pr_{V_{\beta_{v} }, H_{\beta_{v}      } ,  *  }  ( \phi_{\beta_{v} } ) \, \otimes  \phi_{\alpha}  \\  &   =  \pr_{V_{\beta_{v}  }, H_{\beta_{v}    }, *} ( \phi_{\beta_{v}, \alpha} )      \in  N_{H_{v} , \alpha} (H_{\beta_{v}} )       
\end{align*}    
A zeta element for the triple $ ( \phi_{U_{v}, \alpha} ,  \mathfrak{H}_{v} , L_{v} ) $ therefore exists in $ 
  \mathcal{C}( G_{v}/K_{v} , \widehat{N}_{H_{v}, \alpha, \Phi} ) 
  $ since the $ \phi_{\beta_{v}, \alpha} \in N_{H_{v}, \alpha } (V_{\beta}) $ (for $ \beta_{v} \in A_{v}) $ satisfy  the criteria Theorem \ref{zetacriteria}.     By Theorem 
 \ref{eventorsion}, we see that  $$   \mathfrak{H}_{v, * }     \circ  \imath_{v, \alpha, * } ( \phi_{U_{v}, \alpha}  )  =   \sum_{ \beta_{v} \in  A_{v}}   [V_{\beta_{v} }  g_{\beta_{v} } K_{v} ]_{*} ( \phi_{\beta_{v} , \alpha} )   
$$
Therefore   
\begin{align*}  \pr_{\nu, \mu, * } (y_{ \mu } )   &     = \sum_{ \alpha \in A_{\nu} } \sum _ { \beta_{v}    \in A_{v} }   [U^{\mu} V_{\beta_{v}} V_{ \alpha}  \,  g_{\beta_{v}}  \,  g_{\alpha} \,    L_{\mu} K^{\nu}]_{*} ( \phi_{U^{\mu}  }  \otimes  \phi_{\beta_{v}}     \otimes  \phi_{\alpha}  )     \\
&  =   \sum _ { \alpha \in   A_{\nu}  }   \sum_{\beta_{v}   \in A_{v}}  [V_{\beta_{v}} g_{\beta_{v}} K_{v}]_{*} (  \phi_{\beta_{v} , \alpha }  )    \\ 
&  =  \sum _ { \alpha \in  A_{\nu} }  \mathfrak{H}_{v} \circ   i    _{v, \alpha, *} ( \phi_{U_{v}, \alpha    }      ) \\
& =  \sum_{ \alpha \in A_{\nu} }    \mathfrak{H}_{v, * }   \circ   [U^{\nu} V_{\alpha} g  _  { \alpha }  L_{\nu}  K^{\nu} ]_{*} ( \phi_{U^{\nu}}  \otimes  \phi_{\alpha} )    =   \mathfrak{H}_{v, * } (  y_{\nu} )    
\end{align*}  
which  completes  the   proof.   
\end{proof}

\begin{remark}  The  intended  application 
 to    Shimura varieties we have in mind is where we take $ I $ to be the set of all places where all groups at hand are unramified and reserve one element $ v_{\mathrm{bad}} \in I $ for all the bad places lumped together i.e.,  if $ S $ is the set of all bad places, $ G_{v_{\mathrm{bad}}} = \prod_{v \in  S   }    G_{v} $, $ \phi_{v_{\mathrm{bad}}} = \phi_{U_{S}} $ etc.,  and we take   $ I' = I \setminus  \left \{  v_{\mathrm{bad}} \right \}  $.    
\end{remark}

\subsection{Traces in Schwartz    spaces}  Since the machinery   developed  so far 
only allows  us to recast norm relation problem from a  larger group to the smaller one, it is  useful   to have some class of functors  where identifying the image of the trace map is  a more  straightforward check. For instance when $ M_{H,\OO} $ is the trivial functor, the trace map is multiplication by degree and  Corollary \ref{easyzeta1} uses this to give us a congruence  criteria involving certain mixed degrees.     This applies to  pushforwards of fundamental cycles. For Eisenstein classes and cycles constructed from connected components of Shimura varieties, the parameter spaces are certain adelic Schwartz spaces. 
In this subsection, we study  the  image  of  the trace  map for  such  spaces and derive  an  analogous     congruence criteria.    
\label{Schwartz}

Let $ H $ be locally profinite group with identity element $ e $  and  $ X $ a locally compact Hausdorff  totally disconnected   space  endowed  with a  continuous  right  $ H $-action $ X \times   H  \to X $. By definition,  $  X  $ carries a  basis of compact   open   neighbourhoods.  
For a ring $ R $,  we denote by $ \mathcal{S}_{R}(X) $  the  $ R $-module of locally constant  compactly  supported  functions on $ X  $    valued in $ R $.   Under the right  translation action on functions, $  \mathcal{S}_{R}(X) $  becomes   a smooth left representation of $ H $.    In  what  follows, we will frequently use the following fact: the set of all compact open subsets of $ X $ is closed under finite unions, finite intersection and relative complements.  Moreover if $ U \subset H $ is a compact open  subgroup,    then the set of compact open subsets  of  $   X    $     that are invariant under $ U $ is  such a  collection  as well.

\begin{definition}   \label{smoothdefi}       Let $  W  ,  V   \subset H $ be  compact open subgroups with $ V  \subset   W $.  We say that $ x \in X $ is \emph{$ (W,V) $-smooth} if there exist a $ V $-invariant  compact  open   neighbourhood  $ Z $ of $ x $  such that $ Z  \gamma  $ for $  \gamma \in V \backslash W $  are pairwise  disjoint. 
A $ W $-invariant compact  open   neighbourhood     $ Y \subset X $ is said to be \emph{$ (W,V) $-smooth} if $ Y = \bigsqcup _ { \gamma  \in  V  \backslash  W } Z_{\gamma} $ such that $ Z_{ e }    $ is  a    $ V $-invariant compact  open  neighbourhood     of $ X $ and    $ Z_{\gamma} = Z_{  e  }   \gamma     $ for all $ \gamma \in  V   \backslash    W    $.   
\end{definition}
If  $  Y =  
\bigsqcup_{\gamma \in V \backslash W }  Z_{\gamma }$ is $ (W,V) $-smooth, the points of $ Z_{\gamma}  $ are $ (W, \gamma V \gamma^{-1} ) $-smooth but not  necessarily  $ (W,V) $-smooth unless $ V \trianglelefteq W $. It is clear that any $ (W,V) $-smooth neighbourhood is also  $ ( W, \gamma V  \gamma^{-1} ) $-smooth for all $  \gamma \in  W $. Smooth neighbourhoods behave well with respect to     finite unions, finite intersections  and   relative   complements.

\begin{lemma}   \label{cutpaste}        Suppose that $ Y, S \subset X $ are compact opens such that  $ Y $ is  $ (W , V)$-smooth  and $ S $ is $ W $-invariant.    Then $ Y    -  S $ and $ Y \cap S $ are    $ (W,V)  $-smooth. If $ S $ is also  $ (W,V) $-smooth, then so is $ Y \cup S $.     
\end{lemma}    
\begin{proof}   Let $ Y = \bigsqcup _ {  \gamma \in   V  \backslash W } Z_{\gamma} $ where $ Z_{e} $ is a $ V $-invariant compact open  and $ Z_{\gamma} = Z_{e} \gamma $.   Then $ Y \cap S $ is a $ W $-invariant compact open  neighbourhood, $ Z_{e} \cap S $ is  a  $ V $-invariant compact open  neighbourhood  contained in $ Y \cap S $ and $ Z_{\gamma} \cap S  $ $ =    (Z _{e}  \cap S ) \gamma $. Thus $ Y \cap S =  \bigsqcup _{  \gamma \in  V    \backslash W } (Z_{e} \cap S ) \gamma  $  which implies $ (W,V)$-smoothness of $ Y \cap  S  $.  Similarly    $ Y  -  S  =  \bigsqcup _ { \gamma \in    V  \backslash W  } (Z_{e}   -    S ) $.  If $ S $ is also $ ( W, V ) $-smooth, then  since   $$ Y \cup S =  (  Y - 
 S ) \sqcup ( S \cap Y )  \sqcup  ( S - Y ) $$
is a disjoint union of $ (W,V)$-smooth neighbourhoods,  $ Y   \cup    S $ is $ (W,V)$-smooth as well.      
\end{proof}

\begin{corollary}   \label{localtoglobalsmooth}               Suppose that $ S \subset X $ is a $ W $-invariant   compact  open  subset  that  admits a covering by $ (W,V) $-smooth neighbourhoods    of $ X  $.       Then $ S $ is    $ (W,V)  $-smooth.       
\end{corollary}

\begin{proof} 
For all $ x \in S $, let $ Y_{x}  \subset X $ denote  a  $ (W,V $)-smooth  neighbourhood  around  $ x $. By Lemma  \ref{cutpaste}, $ Y_{x} \cap S $  is $ (W,V ) $-smooth and we may therefore  assume  that    $ Y_{x}  \subset  S $ for all $ x \in  S$.   Since $ S $ is compact  and     $ S = \bigcup _{ x \in S }  Y_{x} $,  we  have     $ S  =  \bigcup _ { i = 1 } ^ {  n    }       Y_{i} $ where $ Y_{1} , \ldots, Y_{n}   $   form  a  finite  subcollection of $ Y_{x} $. Thus $ S $ is a finite union of $ (W, V)$-smooth neighbourhoods and is therefore  itself  $ (W,V) $-smooth by     Lemma \ref{cutpaste}.   
\end{proof}

Next  we   have the following criteria for  checking $ (W,V) $-smoothness of point. For $ x \in X $, let $ \mathrm{Stab}_{W}(x) $ denote the stabilizer of $ x $ in $ W   $.

\begin{lemma}  \label{smoothcriteria}   
A point  $ x $ is $ (W,V)$-smooth if and only if $ \mathrm{Stab}_{W} (     x   )    \subset V $. 
\end{lemma}

\begin{proof} The only if direction is clear, so assume that $ \mathrm{Stab}_{W}(x) \subset V $. Let $ U \subset V $ be a compact open subgroup that is normal in $ W $.  For $ \sigma  \in W  $, let $ C_{\sigma } :  =  x  \sigma   U    \subset X $ denote the $ U $-orbit of $ x  \sigma   $.    Thus  two such subsets   are      disjoint if they  are distinct.   By continuity of the action of $ H $,  $ C_{\sigma} $ are compact and therefore closed in $ X $. Since $ U \trianglelefteq    W $, we have $ C_{\sigma} =  x U  \sigma   $ and $ C_{\sigma \tau } =  x \sigma U  \cdot   \tau U $. Thus  $ U  \backslash  W  $  acts transitively on the orbit space  $  ( xW )/ U =  \left \{ C_{\sigma} \, | \, \sigma \in W \right \} $ via  the right  action     $ (C_{\sigma}  ,   U   \tau     )  \mapsto  C_{ \sigma   \tau    } $. Let $ U ^{\circ}  $  denote   the inverse image in $ W $ under $ W  \twoheadrightarrow   U   \backslash    W   $    of the stabilizer of $ C_{ e } $ under this action.   Clearly $ \mathrm{Stab}_{W}(x) \subset   U ^ { \circ  }   $.     If $ \gamma \in  U^{\circ} $, then $  x  \gamma   = x u $ for some $ u \in U $ by definition.  This implies that $ u \gamma^{-1} \in \mathrm{Stab}_{W}(x)  \subset V $ and since $ U \subset V $, we have $ \gamma \in V $. So $ U^{\circ} $ is a   compact open subgroup of $ W $  such  that   $ \mathrm{Stab}_{W}(x)  \subset  U^{\circ}  \subset   V  .     $
It therefore suffices to show that $ x $ is $ (W,U^{\circ} )$-smooth.

Let $ \gamma_{1} , \ldots, \gamma_{n} \in W $ be a set of representatives for $   U ^ {\circ }   \backslash   W    $,  $ \delta_{1} ,   \ldots,  \delta_{m} \in U  ^ {   \circ   }    $ be a set of representatives for $   U   \backslash     U^{\circ}    $ and  denote  $ C_{i}  :     = C_{\gamma_{i} } $.  Then   $  C_{i} $ for $ i = 1 ,\ldots, n $  are pairwise disjoint  and   each     $ C_{i} $ is stabilized (as a set) by $  \delta_{j,i} : = \gamma_{i}^{-1}\delta_{j}\gamma_{i} $ for all $ j   = 1 ,  \ldots ,  m     $. For any  compact  open neighbourhood $ T $ of $ x $, $  X' :=    T   W     $ is a  compact open      neighbourhood of   $ X $ that contains $ C_{i} $  for   all $ i $.  Since $ X'  $ is compact Hausdorff, it is normal and we may therefore choose  compact open   neighbourhoods  $ S_{i} $  contained in  $   X'  $ 
such that $ S_{i}  $  contains $ C_{i} $  and $ S_{1} , \ldots, S_{n} $  are pairwise disjoint.    For each fixed $ k = 1 , \ldots n $, $ \ell = 1 , \ldots  ,  m   $,     let   
$$ Z _{ k  ,   \ell } : = S_{k}      U  \delta_{\ell,k}  -   \bigcup  _ {     i \neq  k  } \bigcup_{j = 1 } ^{m} S _{ i }  U   \delta_{j,i}       $$
where $ i $ runs over all integers $ 1 $ to $ n $ except for  $ k $. Since $  S_{i} U  \delta_{j, i} =    S_{i}  \delta_{j, i}  U $  by normality of $ U $ in $ W $  and $ \left \{   S_{i}    \delta_{j ,  i  }      U   \, | \, j = 1 , \ldots, m , i =  1 , \ldots , n  \right \} $ is a collection of $ U $-invariant compact open  neighbourhoods,  $ Z_{k , \ell} $ are $ U $-invariant compact open    neighbourhoods  as  well.  By  construction, $ Z_{k,\ell} $ intersects $ Z_{k' , \ell'} $ if and only if $ k = k  ' $.  We    claim  that       $  x   \delta_{\ell }   \gamma_{k}    =  x \gamma_{k}  \delta_{\ell, k }        \in    C_{\gamma_{k}}  \subseteq  S_{k} U \delta_{\ell, k} $ is a member of  $ Z_{k, \ell} $. Suppose for the sake of deriving a 
 contradiction that $  x  \delta_{\ell }    \gamma_{k}      \in S_{i}  U  \delta_{j  , i }     $ for  some  $ i , j $ with  $ i \neq k $, so that  $$ x   \delta_{\ell} \gamma_{k}     \delta_{j,  i } ^{-1}   U  =  C_{e} \delta_{ \ell    }  \gamma_{k}  \delta_{j, i } ^{-1}  = C_{e} \gamma_{k}  \delta_{ j , i }  ^ { - 1 }       =   C_{ \gamma_{k}  \delta_{j, i }^{-1}}     $$     intersects $ S_{i} $. As $ C_{ \gamma_{i} }    $ is the only element in $ \left \{ C_{\sigma}   \, |  \,  \sigma \in W  \right \} $  contained in $ S_{i} $,  this   can only happen if $ C_{\gamma_{k} \delta_{j,i} ^ { - 1   }     } = C_{\gamma_{i}} $ or equivalently if $ C_{\gamma_{k}} = C_{\gamma_{i}} \delta_{j,i} $. But since $ \delta_{j,i} $ stabilizes $ C_{\gamma_{i}} $, this means that $ C_{\gamma_{k}} = C_{\gamma_{i}} $   which in turn  implies  
 $ i = k $,  a contradiction.  Thus $ x \delta_{\ell}  \gamma_{k}  \in  Z_{k, \ell} $ or equivalently, $ x \in Z_{k , \ell } \gamma_{k} ^{-1}   \delta_{\ell}  ^  { - 1  } $.     Now let  $$ Z   :  =   \bigcap_{ k = 1 } ^{n}  \bigcap_{\ell = 1 } ^ { n    }      Z    _ { k , \ell } 
 \gamma_{k}^{-1}  \delta_{\ell}^{-1}   . $$ Then $ Z  $ is a $ U $-invariant compact open  neighbourhood of $ x $ as each $ Z_{k, \ell} $ is and $ U \trianglelefteq      W    $.    Since  $ Z \delta_{\ell} \gamma_{k} \subseteq  Z_{k, \ell } $,  $ Z   \delta_{\ell}   \gamma_{k}     $ and $    Z   \delta_{\ell'}   \gamma_{k'}   $ are disjoint for any $ 1 \leq \ell  '   ,  \ell  ' 
    \leq   m  $, $ 1 \leq  k , k ' \leq n $ with $ k \neq k'  $.   If we  now    let  $ Z^{\circ}  : =  \bigcup_{ \ell = 1 } ^ { m }  Z   \delta_{\ell}   $, then  $ Z^{\circ} $ is $ U^{\circ}$-invariant  compact  open  subset of $ X  $  such that $ Z^{\circ}  \gamma_{1}  , \ldots ,  Z^{\circ}  \gamma_{n}  $ are pairwise disjoint. Thus   $ x $ is $ (W,U^{\circ} ) $-smooth.     
\end{proof}

For each $ x \in X $, we let $ V_{x} $  denote the subgroup of $ W $ generated by $ V $ and  $ \mathrm{Stab}_{W}(x) $.  By Lemma \ref{smoothcriteria} $ V_{x} $ is the unique smallest    subgroup of $ W $ containing $ V $ such that $ x $ is $ (W,V_{x} ) $-smooth.  Let $ \mathcal{U} $ be the lattice of subgroups of $ W $ that contain $ V $. For $ \mathcal{T} \subset \mathcal{U}  $ a sub-collection, we denote by $ \max     \,  \mathcal{T}    $ the set  of maximal elements of  $ \mathcal{T} $ i.e.,   $ U \in   \max 
 \,    \mathcal{T}  $   if no $ U' \in \mathcal{T} $ properly contains $ U $. We have a filtration   $$   \mathcal{U} = \mathcal{U}_{0} \supsetneq  \mathcal{U}_{1}  \supsetneq  \ldots  \supsetneq \mathcal{U}_{N}  =  \left \{ V  \right \}   $$    
defined inductively  as    $   \mathcal{U  }     _{k+1}  :  =   \mathcal{U } _{k} - \max  \,  \mathcal{U  }    _{k}  $ for $ k  = 0, \ldots , N - 1 $.  We let $ \mathrm{dep} : \mathcal{U} \to \left \{ 0 , \ldots, N \right \} $ be the function $ U  \mapsto  k  $  where $ k $ is the largest integer such that $ U \in \mathcal{U}_{k} $ i.e.,      $ k $ is the unique integer such  that      $ U \in \max   \,     \mathcal{U}_{k} $.   It  is  clear  that $ \mathrm{dep} $ is  constant  on  conjugacy   classes   of  subgroups.      We let   \begin{align*} \mathrm{dep} = \mathrm{dep}_{W,V} : X & \to  \left \{  0 ,   1,  2, \ldots, N  \right \} \\ 
x   &  \mapsto   \mathrm{dep}(V_{x} )    
\end{align*} 
and refer to   $ \mathrm{dep}(x) $ as the \emph{depth of $ x $}.  We   say that  \emph{$ S \subset X $ has depth $ k $} if $ \mathrm{inf} \left \{ \mathrm{dep}(x) \, | \,  x  \in  S  \right    \} = k $.   

\begin{lemma}   \label{depth}        If   $ S  \subset  X  
$  has  depth  $ k $,
the set of depth $ k $ points in $ S  $ is closed in $ S $.               
\end{lemma}

\begin{proof}  Let $ T \subset  S   $ be the  set of depth $ k $ points. By assumption, the  depth of any point  in    $ S - T $ is at least $ k + 1 $.     For $ x  \in S   -      T $, choose $ Y_{x}   $ a $ (W, V_{x}) $-smooth neighbourhood of $ x $ in $ X $. Then each     $ y \in Y_{x} $  is   $ (W, \gamma  V   _ { x  }   \gamma^{-1} )$ smooth for some $ \gamma \in W $. Thus $ V_{y}  \subseteq \gamma V_{x}  \gamma^{-1}  $ by  Lemma   \ref{smoothcriteria}  and so    $$ \mathrm{dep}(y) = \mathrm{dep}   (  V_{y}   ) \geq   \mathrm{dep}   (  \gamma V_{x} \gamma^{-1}   )  =  \mathrm{dep}  (  V_{x} 
  )     =  \mathrm{dep}(x)   >  k  $$    for all $ y \in Y_{x} $.    Therefore  $ Y_{x} \cap T  =  \varnothing  $    which makes $ Y_{x} \cap S $ an open (relative to $ S$) neighbourhood of $ x $ contained in $ S  -    T  $.  As  $ x $ was arbitrary, $  S   -    T $ is open in $ S $ which makes $ T $ closed in $ S   $.        
\end{proof}

If $ V \trianglelefteq  W $, then $ V_{x} = \mathrm{Stab}_{W}(x)   \cdot         V $ and $ [V_{x} : V ] =  [ \mathrm{Stab}_{W}(x) : \mathrm{Stab}_{W}(x) \cap V ]  $. The next result provides a necessary and sufficient criteria for a given function in $ \mathcal{S}_{R}(X) $ to be the trace of a $ V $-invariant  function in terms of these  indices.      
\begin{theorem} 
  \label{traceriteria}                                    Suppose  that $ V  \trianglelefteq     W $,  $ R $  is  an integral  domain and $ \phi \in  \mathcal{S}_{R}(X) ^ { W } $. 
Then there exists    $ \psi \in \mathcal{S} _{   R  } (  X ) ^ {  V} $ such that $ \phi = \sum _{ \gamma \in W  /  V   }  \gamma \cdot   \psi    $ if and if  only for all $ x  \in   \supp ( \phi ) $,    $ \phi(x) \in [V  _ { x }   :      V   ]  R     $.
\end{theorem}

\begin{proof} ($\!\impliedby \!$) Let $ \psi \in \mathcal{S}_{R}(X)^{V} $ be an  element satisfying the trace  condition. For $ x \in X $,   let   $   V_{x}   $   be as above, $ \gamma_{1} , \ldots,   \gamma_{n} \in  W $ be a set of representatives  for  $ W  / V_{x}  $    and      $ \delta_{1}, \ldots, \delta _{m} \in V_{x} $ be a set of representatives of $  V_{x}  / V     $,  so that $ \gamma_{i} \delta_{j} $ run over a  set    of    representatives for $ W / V  $. As $ V_{x} = \mathrm{Stab}_{W}(x) V $, we may assume that $ \delta_{i} $ (and therefore  $ \delta_{i}^{-1} $) belong to $ \mathrm{Stab}_{W}(x)     $.  Since  $ W / V  $  is a group, $ \delta_{j}^{-1}  \gamma_{i} ^{-1}  $ also   run over    a set of representatives for $ W / V $. Therefore 
  \begin{align*} \phi(x)  =  \sum   \nolimits       _{  \gamma \in W / V }  \gamma \cdot \psi (x)          &   =     \sum _ { i = 1 } ^ { n}  \sum_{j =1}^{m}    \psi (  x  \delta_{j}  ^  { - 1}       \gamma_{i} ^{-1}  )  \\  &  =  \sum _ { i     = 1 }   ^ { n  }      m   \cdot     \psi ( x    \gamma_{i}    ^ { -  1   }      )   \in [V_{x} : V ] R . 
  \end{align*}  
\noindent    ($\! \implies   \! $) Set $ S    :      = \supp \phi $ and $ N : =  \mathrm{dep} \, V $.  By  definition of $ \mathcal{S}_{R}(X) $,    $ S $ is a $ W $-invariant compact  open subset $ X $.  We  inductively define a  sequence  $ S_{0} , \ldots , S_{N} $   of $  W $-invariant  compact open  subsets of $ S $ such that   
\begin{itemize} [before = \vspace{\smallskipamount}, after =  \vspace{\smallskipamount}]    \setlength\itemsep{0.1em}   
\item $ S = S_{0} \sqcup  S_{1} \sqcup \ldots \sqcup S_{N} $,
\item all depth $ k $ points of $ S $ are contained in $ S_{0} \sqcup \ldots \sqcup S_{k} $ for each $ 0 \leq k \leq N $,    
\item each  $ S_{k} $ admits a sub-partition $ \bigsqcup_{ U \in \max \,  \mathcal{U}_{k}}   S_{U}  $ where $ S_{U} $ is a   $ (W,U) $-smooth neighbourhood    on which $ \phi $ is  constant and valued in  $ [U :V ] R $. 
\end{itemize}    
We  provide the inductive step for going from $ k - 1 $ to $ k $ which covers base case as well by taking $ k = 0 $, $ S_{-1} = \varnothing $.      So assume that for $ k \in \left \{0 , \ldots, N-1 \right \} $, the  subsets $S_{0},  \ldots, S_{k-1} $  have been constructed. 
   Let $ T_ { k } $ be the (possibly empty)  set of   all     depth $ k $ points in $$ R_{k} : =  S - \bigsqcup_{i = 0 } ^ { k - 1} S_{ i } $$  where $    R_{k} = S $ if $ k = 0  $.  By  construction,  $ R_{k} $ is   a $ W  $-invariant compact  open subset of $ S $  and depth of $ R_{k} $ is at least $ k $.    By Lemma  \ref{depth}, $ T_{k} \subset R_{k} $ is closed and therefore  compact.   For each $ x \in T_{k} $, let  $ Y_{x}  $ be a $ (W,V_{x}) $-smooth neighbourhood of $ x $. By Lemma \ref{cutpaste}, we may assume $ Y_{x}  \subset   R_{k}  $.   
   As $ \phi  $  is   $   W $-invariant and locally constant, $ x $ is contained in a $ W $-invariant  compact open    neighbourhood on which  $ \phi $ is  constant.  By intersecting such a neighbourhood with $ Y_{x }$  if necessary, may also 
 assume that $ \phi $ is constant on $ Y_{x} $ for each $ x  $. Since $ T_{k} $ is  compact  and     covered by $ Y_{x} $, there exist $ x_{1}, \ldots, x_{n} \in T_{k} $  such that $  T_{k}  \subseteq Y_{x_{1}} \cup \ldots \cup Y_{x_{n}} $. Let $$ S_{k} : = Y_{x_{1}} \cup \ldots  \cup Y_{x_{n}}  . $$ 
   Clearly $ S_{k} $ is   a  $ W $-invariant compact open subset of $ R_{k} $   since $ Y_{x} $ are and $ S_{k} $ is disjoint  from   $  S_{1} , \ldots, S_{k-1} $.   By construction, all the depth $ k $ points of $ R_{k} $ are in $ S_{k} $ and thus all the depth $ k $ points of $ S $ are in $ S_{1} \sqcup \ldots \sqcup S_{k} $. Let $ Y_{i}  :   = Y_{x_{i}}  -  (Y_{x_{i+1}}  \cup   \ldots  \cup     Y_{x_{ n } }   ) $. Then $ Y_{i} $ are $ (W, V_{x_{i}} )$-smooth by Lemma   \ref{cutpaste}  and    $ S = Y_{1} \sqcup \ldots   \sqcup       Y_{n} $.      As $ x_{i} \in T_{k} $, we have  $ V_{x_{i}} \in   \max   \,    \mathcal{U}_{k}  $ and by construction, $ \phi $ takes the constant value $  \phi(x_{i}) \in [V_{x_{i}} : V]  R $  on $  Y_{i}$.  For each $ U \in  \max \,  \mathcal{U}_{k} $, we let $ S_{U}      :     = \bigsqcup _ { V_{x_{i}} = U }  Y_{i} $.  Then $ S_{k} = \bigsqcup _ { U \in \max  \mathcal{U}_{k}}  S_{U}      $  is the desired   sub-partition  
   and the inductive step is complete.    

Now for each $ U \in \mathcal{U} $, let $ Z_{U}  \subset S_{U} $ be a $ U $-invariant  neighbourhood whose $ U  \backslash W $  translates  partition $ S_{U} $.   We define $ \psi : X \to R $ by $$\psi   ( x  )  =   \begin{cases}    [U : V ] ^{-1} \phi(x)   & \text{ if } x \in Z_{U}  \\
0  & \text{ otherwise } 
\end{cases}     $$ 
Then $ \psi $ is well-defined since for all $  x \in S_{U}  $, $ \phi(x)  =    [U:V] \cdot r $ for a unique $ r \in  R -  \left \{ 0 \right \}     $.     As $ \psi $ takes  a     non-zero constant value on $ Z_{U} $ and is zero elsewhere, $ \supp \psi = \bigsqcup_{ U \in \mathcal{U}}  Z_{U}  $. As each  $ Z_{U} $ is $ V $-invariant, $ \psi $ is $ V $-invariant. Thus $ \psi \in  \mathcal{S}_{R}(X)^{V} $. Let $ \phi' = \sum _ { \gamma \in W / V } \psi $. As $ S $ is $ W $-invariant and $ \supp \psi \subseteq S $,   $ \supp \phi ' \subset S $ as well.  Thus $ \phi $ and $ \phi'  $ agree on $ X - S 
 $ and we show that they  agree on $ S $ as well.   For each $ x \in S $, there exists  a unique $ U \in \mathcal{U} $  and a unique $ \gamma \in  U \backslash  W $ (both of which depend on $ x $) such that $  x \gamma \in Z_{U} $.  Let $ \gamma_{1} , \ldots, \gamma _ { n } \in W $ be a set of representatives of $ U \backslash  W $ 
 and $ \delta_{1}, \ldots, \delta_{m} \in  U $ a set of representatives of $ V   \backslash U $. Then $  \gamma  \delta_{j}   \gamma _{i}    $ run over set of representatives for $V \backslash W =   W / V $. Since $ x \gamma \in Z_{U} $ and $ Z_{U} $ is $ U $-invariant,  $x \gamma \delta_{j} \gamma_{i} \in  Z_{U} \gamma_{i} $ for all $ i $, $ j $.   
 Thus  $ x \gamma \delta_{j}  \gamma_{i} \in Z_{U} $  if and only if $ \gamma_{i} $ represents the identity class in $ U \backslash W $.  One then easily sees that 
 \begin{align*} \phi '  ( x )  = \sum_{ i , j } \psi ( x \gamma \delta_{j}  \gamma_{i}  )   &     =   \sum_{ j}    
   \psi ( x \gamma \delta_{j} )  = [U : V]  \psi (x) =  \phi( x)   .
 \end{align*}
Hence $ \phi = \phi' $  and  so   $ \psi   $  is  the    desired  element.      
\end{proof}

\begin{corollary}  Let $ \phi \in  \mathcal{S}    _{R}(X)^{W} $ and  let $ x_{\alpha} \in X $ for $ \alpha \in I $  be a set of representatives for $  (    \supp \phi   ) /  W     $. 
 Then $ \phi $ is the trace of an element $  \mathcal{S}_{R}(X) ^{V} $ for $ V \trianglelefteq   W $ if and only if $ \phi(x_{\alpha}) \in  [ V_{x_{\alpha}}  : V ] R $ for all $ \alpha \in I $.      
\end{corollary}
\begin{proof}  The only if direction is clear by Proposition  \ref{traceriteria}. The if direction also follows from it since any $ x \in x_{\alpha} W $  is $ (W, \gamma V_{x_{\alpha}}     \gamma^{-1} ) $-smooth for some $ \gamma \in W $,  so  that    \[  \phi(x) =  \phi(x_{\alpha})  \in     [V_{x_{\alpha}} : V] R = [V_{x} : V] R   . \qedhere  \]  
\end{proof} 

We   resume  the setup of \S \ref{abstractzetasection} and retain Notation \ref{notationalpha}.  Assume   moreover that $ M_{H, \OO} $ is the functor associated with the smooth $ H $-representation $  \mathcal{S}_{\OO}(X) $. In particular, $ x_{U} $ is a $ U$-invariant Schwartz function $ \phi_{U} : X \to \OO  $.   
\begin{corollary}   \label{checkfixfirst}      Suppose  that $  p  \in  \supp  \phi_{U}  $ is an $ H$-fixed point.  Then a zeta element exists for $ (\phi_{U} , \mathfrak{H}, L) $ only if $  \phi_{U}(p) \cdot  \deg(\mathfrak{h}_{\alpha,*}) \in [H_{\alpha} : V_{\alpha}]  \OO  $  for all $ \alpha \in H \backslash  H  \cdot  \supp(\mathfrak{H}) / K  $.  
\end{corollary}
\begin{proof} By Theorem \ref{zetacriteria}, a zeta element exists if and only if $  \mathfrak{h}_{\alpha}^{t}   \cdot  \phi_{U} \in  M_{H, \OO}( H_{\alpha} )    $ is the trace of an element in $  M_{H, \OO} (V_{\alpha} ) =  \mathcal{S}_{\OO}(X)^{V_{\alpha}} $.  By  Theorem   \ref{traceriteria}, this can  happen   only if $$ \mathfrak{h}_{\alpha}^{t}   \cdot   \phi_{U}(p) \in   [H_{\alpha} : V_{\alpha}]   \OO . $$ Since $ p $ is $ H $-fixed, $ \mathfrak{h}_{\alpha} ^{t}    \cdot     \phi _{U}(p) = \phi _{U}(p) \cdot  \deg(\mathfrak{h}_{\alpha,*}) $.  
\end{proof}    

\begin{remark} For Eisenstein classes, the local Schwartz functions are  characteristic functions on lattices in  certain  vector spaces and the group $ H $ at hand acts  via 
linear transformations. The origin is therefore a fixed point for its action  and  Corollary \ref{checkfixfirst}  provides a quicker initial check\footnote{this proved  particularly helpful in \cite{Siegel1} where the `convolution step'  was quite involved} for applying the  criteria of Theorem  \ref{traceriteria}. %This proves particularly helpful to us in \cite{Siegel1}, where the `convolution step' was quite involved. 
Incidentally, this is the  same check as in    Corollary  \ref{easyzeta1} which applies  to   fundamental cycles.    
\end{remark}    

\subsection{Miscellaneous results}    
We will study zeta elements for groups $ G $ that are product of two groups, one of which is abelian and it would be useful to record some auxiliary results that would be helpful in applying the criteria to such groups.

Suppose for the  this subsection   only    that  $ G = G_{1} \times T $ where $ T $ is abelian with a unique maximal compact subgroup $ C $. Suppose also     that $ K = K_{1} \times C $, $ L = K_{1} \times D $ where $ K_{1} \subset G_{1} $, $ D \subset C  $ (so that $ d = [C:D]$)     and that    $$ \mathfrak{H} 
=  \sum_ {  k \in I }   e_{ k }    \,   \ch ( K \gamma_{  k } \phi_{ k }   K    )    \in  \mathcal{C}_{\OO} ( K \backslash G / K )   $$ 
where $  e  _  {  k }     \in \OO $,  $ \gamma   _ { k   } \in G_{1} $ and  $    \phi_{ k }  \in T $.
Let $ \iota_{1} : H \to G_{1} $, $ \nu : H \to  T $  denote   the   compositions   $   H  
\xrightarrow{ \iota } G \to G_{1}     $, $ H \to G \to T $ respectively. We suppose that $ \iota_{1} $ is injective, so we may consider $ H $, $ U $ as a subgroup of $ G_{1} $ as well as $ G $.  When we consider $ H $, $ U $ as subgroups of $ G_{1} $, we denote them by  $ H_{1} $, $ U_{1} $ respectively.   

\begin{lemma}   \label{Heckefrobzetalemma}                     Suppose that $    K _{1} \gamma_{ k } K_{1} = \bigsqcup_{ j \in J_{ k } }   U_{1}   \sigma _ { j} K _{1}     $ where $ J_{k } $ is an indexing set and $ \sigma  _ { j }  \in  G_{1} $. Denote  $ \sigma_{j,k} = \sigma_{j} \phi_{ k } $ and $ H_{1, \sigma_{j}} = H_{1} \cap \sigma_{j}  K _{ 1 }    \sigma_{j}^{-1} $  
Then 
\begin{enumerate}[label = \normalfont   (\alph*)]      
\item $   \mathfrak{H}   =  \sum_{ k \in I }  \sum _{ j \in J_{k }  }  e_{k }   \,   \ch (  U   \sigma_{j, k }  K    )   $ \vspace{0.2em} 
\item $ 
\deg  \,  
[U \sigma _{j ,  k }   K ]_{*}   =    \deg  \, [ U_{1}  \sigma_{j} K _{1}  ]_{*}   
$,  \vspace{0.2em}    
\item   $   \displaystyle {   [H \cap \sigma_{j  ,  k}   K   \sigma_{j , k }      : H  \cap  \sigma_{j , k }  L \sigma_{j, k } ^{-1}    ] = [ H_{1, \sigma_{j} }  : H_{1 , \sigma_{j} }     \cap \nu^{-1}(D)  ]  }     $. 
\end{enumerate}    
\end{lemma} 
\begin{proof} Since $ \nu $ is continuous and $ C $ is the unique   maximal compact subgroup of $ T $, the image  under $ \nu $  of any compact subgroup of $ H $ is contained in $ C $.  For (a),  it suffices to note that $$ K_{1} \gamma _{k} K_{1} = \bigsqcup _{ j \in J_{k} } U _{1}  \sigma_{k} K_{1}  \implies  K \gamma_{k } \phi_{k}  K =  \bigsqcup _ { j \in  J _{k}  }  U \sigma_{j} \phi_{k}  K $$
since $ K  \gamma  _ { k } K   =     K_{1} \gamma_{k} K_{1} \times \phi_{k} C  $ and $ \nu(U) \subset C $.  For (b), note that   $  H \cap \sigma_{j ,  k } K \sigma_{j, k } ^{-1}  = H \cap \sigma_{j} K \sigma_{j}^{-1} $ as $ T $ is abelian. Since $ H  _ { 1 }  \cap \sigma_{j} K _ { 1 }  \sigma_{j}^{-1} $ is compact, $ \nu ( H _{1}  \cap \sigma_{j } K _{1}  \sigma_{j}^{-1} ) \subset C $  and  therefore   $$   H \cap \sigma_{j, k } K \sigma_{j, k }^{-1} =   \iota_{1}^{-1} \big (  \sigma _{j} K_{1} \sigma_{j}^{-1} \big )  \cap \nu^{-1}(C) =  H_{1}   \cap \sigma _{j} K_{1} \sigma_{j}^{-1}      .   $$     Similarly  $ U \cap \sigma_{j} \phi_{k} K  ( \sigma_{j} \phi_{k} ) ^{-1}  =U _{1}  \cap  \sigma_{j} K_{1} \sigma_{j}^{-1} $ and (b) follows. The  argument for  (c)  is  similar.    
\end{proof}

\subsection{A prototypical example}   \label{toyexamples} 
In this subsection, we show how the machinery above  may be applied to the case of CM points on modular curves to derive the Hecke-Frobenius valued norm relations at a split prime, which is essentially the $ n = 2 $ case of the example studied in \S \ref{GLnLfactorsection}.  See also \cite{Norm} for a more thorough treatment.

Let $ E $ be an imaginary quadratic field. Set $   \Hb   = \mathrm{Res}_{E/\QQ} \GG_{m} $ and 
 $ \Gb = \GL_{2} $. Fix a $ \QQ $-basis of $ E $ and let $ \iota : \mathbf{H} \to  \Gb  $ be the resulting embedding.  Let $  \mathbf{T} $ be the torus of norm one elements $ E $ and set $ \tilde{\Gb} = \Gb \times \Tb  $.  Let  $ \nu : \mathbf{H}    \to   \mathbf{T} $ denote $ h \mapsto h \gamma(h)^{-1} $ where $ \gamma \in  \Gal  ( E /   \QQ   ) $ is the non-trivial  element and let $$ \iota_{\nu} : \Hb \to  \tilde{\Gb }  , \quad \quad  h \mapsto ( \iota (h ),  \nu ( h ) )  . $$
Then both $ \iota $ and $ \iota_{\nu} $ is a morphisms of   Shimura datum.    The embedding $ \iota $ signifies the construction of CM points on the modular curve.         Under Shimura-Deligne reciprocity law for tori,  the extensions corresponding to $  \mathbf{T} $ by class field theory  are  \emph{anticyclotomic} 
  over $ \QQ $.

Let $ G_{f} $, $ \tilde{G}_{f} $, $ H_{f} $,  $ T_{f} $ denote the $  \Ab_{f} $ points of $ \Gb $, $ \tilde{\Gb} $,  $ \mathbf{H} $,  $   \mathbf{T} $    respectively.  
Let $ \Upsilon_{\tilde{{G}}_{f}} $ denote the collection of all neat compact open subgroups of $  \tilde{G} _{f} $ of the form $ K \times C $ where $ K \subset  G_{f}  $, $ C \subset T_{f} $ and let  $ \Upsilon_{H_{f}} $ denote the collection of all neat compact  open subgroups of $ H_{f} $. These collections satisfy (T1)-(T3) and $ \iota_{\nu}^{-1}   ( \Upsilon_{\tilde{G}_{f}}          )         \subset  \Upsilon_{H_{f}   }    $.  For any rational prime $ p $,        the mappings   \begin{alignat*}{4}   N _ { \ZZ_{p}  }  : \Upsilon_{H_{f}}      &      \to  \ZZ_{p} \text{-Mod} &  \hspace{0.3in}    M   _ { \ZZ_{p}  }   :   \Upsilon_{\tilde{G}_{f} }  & \to  \ZZ_{p} \text{-Mod} \\
U     &     \mapsto  \mathrm{H}^{0}_{\et} \big ( \mathrm{Sh}_{\Hb}(U) , \ZZ_{p}  \big   )   &  \hspace{0.3in} \tilde{K}    &  \mapsto \mathrm{H}^{2}_{\et}  \big  ( \mathrm{Sh}_{\Gb}  (  \tilde{K} )    ,  \ZZ_{p}   ( 1  )       \big  )  
\end{alignat*} 
that send each compact open subgroup to the corresponding arithmetic \'{e}tale  cohomology  of the corresponding   Shimura    varieties over $ E $  constitute CoMack functors. We note that  if $ \tilde{K}  : = K \times C $, the Shimura variety $ \mathrm{Sh}_{\Gb}  (\tilde{K} ) $  is the base change of the modular curve over $ \QQ $ of level $ K $ to the extension of $ E $ determined by the compact open subgroup $ C $.    
The embedding $ \iota :  \mathbf{H} \to \mathbf{G} $ induces a  Mackey    pushforward $ \iota : N_{\ZZ_{p}}  \to M_{\ZZ_{p}} $ of RIC  functors. For each $ U $, $ N_{\ZZ_{p}}(U) $ is the free $ \ZZ_{p} $-module on the class of $  1_{\mathrm{Sh}_{H}(U)} $ and $ N_{\ZZ_{p}} $ is the trivial functor on $ \Upsilon_{H_{f} } $.   
Let $ \ell  \neq   p  $ be a rational prime that is split in $ E $. Then   $$ \mathbf{H}_{\QQ_{\ell} } \simeq \GG_{m} \times \GG_{m} , \quad   \quad  \mathbf{T}_{\QQ_{\ell}} \simeq  \GG_{m}  $$ where  the  isomorphisms are chosen so that the map $ \nu $ is identified with the map that sends  $ (h_{1}, h_{2})  \in \Hb_{\QQ_{\ell}} $ map to $ h_{2} / h_{1}  \in \mathbf{T}_{\QQ_{\ell}} $. The particular choice is so that the action of uniformizer  $  \ell    \in  \QQ_{\ell}^{\times} \simeq \Tb(\QQ_{\ell})    $ (in the  contravariant convention)  is identified with the action of  geometric Frobenius $ \mathrm{Frob}_{\lambda}^{-1} $ on cohomology  where $ \lambda $  corresponds to the first component in the identification $ H_{\QQ_{\ell}}  \simeq \GG_{m} \times  \GG_{m}   $.

Fix for the rest of this  discussion a  split prime $ \ell $ as above and a compact open subgroup  $ \tilde{K}  =   K  \times  C  \in \Upsilon_{\tilde{G}_{f}} $   such that $ K = K^{\ell}K_{\ell}$, $ C = C^{\ell} C_{\ell} $ where $ K_{\ell} = \GL_{2}(\ZZ_{\ell}) $, $ C_{\ell} = \ZZ_{\ell}^{\times}  $ and $ K^{\ell} $, $ C^{\ell} $ are groups away from $ \ell  $.    Let $ U : = \iota^{-1}(K) $ and  similarly 
 write $ U = U^{\ell} U_{\ell} $ where $ U_{\ell} = \ZZ_{\ell}^{\times} \times  \ZZ_{\ell}^{\times}$.  Let 
\begin{equation}    
\label{HeckepolyGL2}     \mathfrak{H} _{\ell} (X)  :  =    \ell \cdot   \mathrm{ch} (  K   ) -         \mathrm{ch} ( K   \sigma _ { \ell }  K   )  X   +   \mathrm{ch} ( K \gamma  _  {  \ell  }     K)   X ^ { 2 }    \in  \mathcal{H}_{\ZZ}( K  \backslash   \Gb    ( \Ab_{f} ) / K )  [  X ]         
\end{equation}
where $ \sigma_{\ell}  : =  \mathrm{diag}(\ell , 1 ) $  and   $  \gamma _{\ell} :   =  \mathrm{diag}   ( \ell ,  \ell )   $. 
Then  $$ \tilde{\mathfrak{H}}_{\ell}    :    =     \mathfrak{H}_{\ell}  ( \mathrm{Frob}_{\lambda} )   =   \mathfrak{H}_{\ell} ( \ell^{-1} C ) \in \mathcal{C}_{\ZZ}(\tilde{K} \backslash \tilde{G} /  \tilde{K}_{\ell} )   $$    induces a $ \ZZ _{p} $-linear map  $  \tilde{\mathfrak{H}}_{\ell, * }   :    M_{\ZZ_{p}}  ( \tilde{K}  ) \to M_{\ZZ_{p}}(\tilde{K})   $. Let $ D = C^{\ell} D_{\ell }$ where $ D_{\ell} = 1 + \ell \ZZ_{\ell } $ and let $ x_{U} = 1_{\mathrm{Sh}(U) } $. Set $ \tilde{L} = K \times D $.      We ask if there is a zeta element $ ( x_{U} ,  \tilde{\mathfrak{H}}_{\ell} ,  \tilde{L} ) $.  Recall that such an element  would solve the corresponding question posed in   \ref{NormRelationProblem}. It   is  also clear that this checking can be done locally at the prime $ \ell $ and that via Theorem 
  \ref{gluingzeta}, one can produce a compatible system of such relations for such $ \ell $.

The local embedding of $ H_{\ell} \hookrightarrow \tilde{G} _{\ell} $ is not the  diagonal one on the $ \GL_{2} (\QQ_{\ell} ) $ copy. We may however conjugate this embedding by an appropriate element of $ K_{\ell} $ to study the zeta element problem and conjugate everything back at the end by the inverse of the said element\footnote{See \cite[\S 2.2]{Norm} for details.}. So say that $ H_{\ell}  \hookrightarrow \tilde{G}_{\ell} $ is the diagonal embedding where the first component of $ H_{\ell}$ corresponds to the top the left matrix entry in $ \GL_{2}(\QQ_{\ell}) $.    Define the following elements of $ G_{\ell} $: $$   \sigma_{1} = \begin{pmatrix} \ell & \\ &  1 \end{pmatrix},  \quad     \sigma_{2} =  \begin{pmatrix}  \ell & 1  \\ & 1  \end{pmatrix},   \quad    \sigma_{3} =  \begin{pmatrix}  1 & \\  & \ell  \end{pmatrix}    ,   \quad   \sigma_{4    } =     \begin{pmatrix} \ell & \\ &  \ell   \end{pmatrix}   ,      \quad 
  \tau =  \begin{pmatrix}  1  & \ell^{-1}  \\ & 1   \end{pmatrix}    $$
  and set $ \tilde{\sigma_{i}} =  ( \sigma_{j},  \det(  \sigma_{j}^{-1}) ) $.  Then 
  $$ \tilde{\mathfrak{H}}_{\ell} =   \ell \cdot  \ch(U_{\ell} \tilde{K}) -  \left (  \ch(U  \tilde{\sigma}_{1} \tilde{K}) +   \ch(U_{\ell}  \tilde{\sigma}_{2}   \tilde{K}_{\ell})  + \ch(U_{\ell} \tilde{\sigma}_{3} \tilde{K}_{\ell}) \right )  +  \ch (  U_{\ell}  \tilde{\sigma}_{4} 
\tilde{K}_{\ell})   . $$
It is then clear that 
 $$ g_{0} : =  (1,1)  ,  \quad \quad  g_{1}  :  =  ( \tau , 1 )  ,  \quad   \quad     g_{2} : = ( 1,   \ell  ^ {-2} )  
 $$
 form a complete system of  representatives for $ H_{\ell} \backslash  H_{\ell}  \cdot \supp ( \tilde{\mathfrak{H}} )  /  \tilde{K}_{\ell} $. 
 For $ i = 0, 1, 2 $, let $ H_{\ell, i } = H_{\ell} \cap g_{i} \tilde{K}_{\ell}  g_{i}^{-1} $ and $ \mathfrak{h}_{\ell, i} \in \mathcal{C}_{\mathcal{\ZZ}}(U_{\ell} \backslash  H_{\ell} / H_{\ell, i} )  $   
 denote the $ (H, g_{i}) $-restriction of $  \tilde{H}_{\ell} $. Then  
 \begin{align*}    \mathfrak{h}_{\ell,0} &  = \ell  \cdot \ch \big ( U_{\ell}  \big  )  -  \ch \big (U_{\ell}  (\ell, 1)  U_{\ell} \big )  \\ 
 \mathfrak{h}_{\ell,1} & =   \ch \big  ( U_{\ell}  (\ell, 1) H_{\ell, 1}  \big   )   \\
 \mathfrak{h}_{\ell, 2 } &  =  \ch(U_{\ell} (1,\ell) U_{\ell} ) -  \ch( U_{\ell} (\ell , \ell )  U_{\ell} )  
 \end{align*} 
from which it is easily seen that $$ \deg(\mathfrak{h}_{\ell,0,*}) =  \ell - 1 , \quad \quad  \deg(\mathfrak{h}_{\ell,1,*}) = 1, \quad \quad   \deg(\mathfrak{h}_{\ell,2,*}) = 0 . $$ 
Finally, let  $ d_{i} : =   [ H_{\ell, i}  : H_{\ell}  \cap g_{i} \tilde{L} _{\ell}  g_{i}^{-1} ] $. Then  $ d_{0} = d_{2} =  \ell  - 1 $, $ d_{1} = 1 $. 
Since $$ \deg ( \mathfrak{h}_{\ell, i,*}) \in d_{i} \ZZ_{p} $$ for $ i = 0 , 1, 2 $, Corollary \ref{easyzeta1}  implies that a zeta element exists for $ (1_{\mathrm{Sh}_{U}} , \tilde{\mathfrak{H}}_{\ell}, \tilde{L}) $.    
\begin{remark} Note that our zeta element is   supported on  $ g_{0}K \cup  g_{1}K  $, even  though $ H_{\ell} \backslash  H_{\ell}  \cdot \supp ( \tilde{\mathfrak{H}} )  /  \tilde{K}_{\ell} $ has three elements.  See also Remark  \ref{sametestvector} for a similar observation.      
\end{remark}

%% file: HeckePolynomials.tex
\section{Hecke   polynomials}

\label{LfactorHeckesection} 
In this section, we describe the Hecke algebra valued polynomials associated   with     representations     of the Langlands dual of a reductive group and  record  some  techniques  that can be used to compute them. On the way, we fix notations and terminology that will be  used in  carrying out the computations in Part II of this article.      
\begin{notation}   
\label{LfactorHeckesectionnota}                   Throughout this  section,  we  let  $F$  denote  a local field of characteristic zero, $\mathscr{O}_{F}$ its ring of integers, $\varpi$ a uniformizer, $ \kay =\mathscr{O}_{F} / \varpi$  its  residue field,  $q=|\kay|  $ the cardinality of $ \kay   $ and $ \mathrm{ord} :  F \to \ZZ \cup \left \{ \infty \right\} $ the additive valuation assigning $ 1 $ to $ \varpi $. We pick once and for all    $ [\kay ]   \subset   \Oscr_{F}     $  a fixed  choice  of    representatives for $ \kay $.  We let  $ \bar{F} $ denote  an  algebraic  closure  of   $ F $  and let $ F^{\mathrm{unr} } \subset  \bar{F}   $ denote the maximal unramified subextension.   For $ M $ a free abelian group of finite rank, we  will often denote by $ M_{\QQ} $   the   $ \QQ $-vector space $ M \otimes_{\ZZ} \QQ $.       
\end{notation}

\subsection{Root data}       

\label{rootdatasection}
Let $\mathbf{G}$ be an unramified 
reductive group over $F$. This  means that $ F $ is quasi-split over $ F $ and split over a finite unramified extension of $ F $.  Let  $  \mathbf{A}  $  be  a maximal $F$-split torus in $ \Gb $, $\mathbf{P} \supset  \mathbf{A} $  a $F$-Borel subgroup and  $\mathbf{N}$ the unipotent radical of $\mathbf{P}$. Let  $\mathbf{M}:=\mathbf{Z}(\mathbf{A})$   be      the centralizer of $\mathbf{A}$ which is a maximal $F$-torus in $\mathbf{G} $.   We will denote by  $G $,  $ A  $,  $ P $, $  M $,  $  N$ the corresponding groups of  $F$-points of $\mathbf{G}, \mathbf{A}, \mathbf{P}, \mathbf{M}, \mathbf{N}$ respectively.    Let    $ X^{*}(\mathbf{M}) $ (resp., $ X_{*}(\mathbf{M}) $) denote the  group of  characters (resp.,  cocharacters) of $ \mathbf{M} $ and let  
\begin{equation}     \langle-,-\rangle: X_{*}(\mathbf{M}) \times X^{*}(\mathbf{M}) \rightarrow \mathbb{Z}   \label{nonsplittoruspairing}   
\end{equation} 
denote the natural  integral   pairing. The natural extension of (\ref{nonsplittoruspairing}) to $ X_{*}(\mathbf{M})_{\QQ} \times X^{*}(\mathbf{M})_{\QQ} \to \QQ $ is also denoted  as  $ \langle - , - \rangle $.

Let 
 $  {\Phi}_{\bar{F}} \subset X^{*}(\mathbf{M}) $  denote  the set of absolute    roots        of  $ \mathbf{G} $ with respect to $ \mathbf{M} $,   $  \Phi^{+}_{\bar{F}}   \subset  \Phi _ { \bar{F}}    $ the set of positive roots associated with $ \mathbf{P} $ and $ \Delta_{\bar{F}} $ a base for $ \Phi_{\bar{F}} $.      For $ \alpha \in  \Phi _ { \bar{F}} $, we denote by $  \alpha^{\vee} \in X_{*}   (  \mathbf{M} ) $ the corresponding coroot and denote the set of coroots by $ \Phi_{\bar{F}} ^ { \vee }   $.    Since  $ \Phi_{\bar{F}} $ is reduced, $ \Delta_{\bar{F}}^{\vee} = \left \{ \alpha^{\vee} \, | \,   \alpha   \in  \Delta_{\bar{F}}     \right \} $ is a  base for  the positive  coroots in  $ \Phi_{\bar{F}}^{\vee} $.  We  let    $ \mathbf{W}_{\mathbf{M}} =  \mathbf{N} _{\Gb}(\mathbf{M}) / \mathbf{M} $ denote the absolute Weyl group  scheme of $ \mathbf{G} $ and set $ W_{M}  :    = \mathbf{W}_{\mathbf{M}}(\bar{F} ) $.
Then left action $ W_{M} $ on $ X^{*}(\mathbf{M}) $, $ X_{*}(\mathbf{M}) $ induced by conjugation action  on $ \mathbf{M}_{\bar{F}} $ identifies it with the Weyl group of the   (absolute)          root datum  $ ( X^{*}(\mathbf{M}) ,   \Phi_{\bar{F}} ,  X_{*}( \mathbf{M}) ,   \Phi_{\bar{F}}^{\vee} ) $.       Thus for $ \alpha \in \Phi_{\bar{F}} $, there is  reflection element $ s_{\alpha} =  s_{\alpha^{\vee}}  \in W _ {  M   }      $ that acts on  $ \lambda  \in X_{*}( \mathbf{M} )$ and $ \chi \in X^{*}(\mathbf{M}) $ via  
\begin{equation}   
\label{reflectionmaproots} \lambda \mapsto  \lambda    - \langle  \lambda  ,   \alpha 
\rangle \alpha  ^ { \vee  }       \quad \quad   \quad      \chi \mapsto  \chi - \langle  \alpha^{\vee}, \chi  \rangle  \alpha     
\end{equation} 
The pair    $  \big  ( W_{M} , \left \{s_{\alpha} \right\} _{\alpha \in \Delta_{\bar{F}} }  \big  )    $ is a Coxeter system. We let  $ \ell _ { \bar{F}}    : W   _ { M } \to \ZZ $  the   corresponding length function.

We will also need to work with the relative root datum of $ \Gb $.   Let $ X^{*}(\mathbf{A})$, $X_{*}(\mathbf{A}) $ denote  respectively    the set of characters and 
 cocharacters  of $ \mathbf{A} $.   As $ \mathbf{A} $ is split, all characters and cocharacters are defined over $ F $.   
Let $$ \mathrm{res} : X^{*} ( \mathbf{M } ) \twoheadrightarrow  X ^ { * } (  \mathbf{A} )      ,\quad \quad     \quad      \mathrm{cores} : X_{* } (  \mathbf{A} )   \hookrightarrow X_{*} ( \mathbf{M} ) $$   denote respectively the natural injection and  surjection induced by  $  \mathbf{A}  \hookrightarrow   \mathbf{M}   $. Let $ \Gamma  = \Gal(F^{\mathrm{unr}} / F )  \simeq  \widehat{\ZZ} $ denote the unramified Galois group of $ F $. Then $ \Gamma $ acts on $ X^{*}(\mathbf{M}) $  via   $ (\gamma,\chi)  \mapsto    \gamma  \chi(   \gamma^{-1}   x  )$  where  $ \gamma \in \Gamma $, $ \chi \in X^{*}(\mathbf{M}) $ and $ x \in \mathbf{M}( F^{\mathrm{unr} } )   $.   Similarly      $ \Gamma $ acts on $ X_{*}(\mathbf{M}) $ 
and the pairing  (\ref{nonsplittoruspairing})     is $ \Gamma $-invariant under  these   actions.         
Since $  \mathbf{M} $ is defined over $ F $,  the action of $ \Gamma $ preserves $ \Phi_{\bar{F}} $, $\Phi_{\bar{F}}^{\vee} $ (as sets).  
  Since $ \mathbf{P} $ is  defined over $ F $,   $ \Gamma $   also  preserves $ \Delta_{\bar{F}} $, $ \Delta_{\bar{F}} 
  ^ { \vee  }     $   and the  action of $ \Gamma $  on  these  bases is via  diagram   automorphisms.     We have  $$ X^{*}(\mathbf{M}) _{ \Gamma, \mathrm{free}}  \overset{\mathrm{res}}{\simeq}  X^{*} ( \mathbf{A} ) , \quad  \quad    \quad      X_{*}(\mathbf{A})   \overset{\mathrm{cores}}{\simeq}      X_{*}(\mathbf{M})^{\Gamma}  $$   where $ X_{*}(\mathbf{M})_{\Gamma, \mathrm{free}} $ denotes the quotient of the group of coinvariants by torsion.  
The  pairing     \begin{equation}   \label{splittoruspairing} \langle-,-\rangle: X_{*}(\mathbf{A}) \times X^{*}(\mathbf{A}) \rightarrow \mathbb{Z} 
\end{equation}  
is  compatible with  (\ref{nonsplittoruspairing}) i.e.,  if $ \lambda : \GG_{m} \to \mathbf{A} $, $ \chi : \mathbf{M} \to \GG_{m} $ are  homormophisms defined  over $ \bar{F} $, then $ \langle \mathrm{cores} \, \lambda , \chi \rangle = \langle \lambda, \mathrm{res} \, \chi  \rangle $.   
Let   $  \mathbf{W} _{\mathbf{A}} : = \mathbf{N} _{\Gb}(\mathbf{A})/\mathbf{M}  $      denote the Weyl group   scheme  of $ \Gb $ with respect to $ \mathbf{A} $. 
Then $  \mathbf{W} _{\mathbf{A}} $ is a constant group scheme over $ F $ and $$  \mathbf{W} _{\mathbf{A}}(F) = \mathbf{N}  _{\mathbf{\Gb}}(\mathbf{A})(F)/ \mathbf{M}(F)  =  N_{G}(A)/M $$ by \cite[Proposition C.2.10]{CGP}.    Using quasi-splitness of $ \Gb $,  it can also be shown that  $  \mathbf{W} _{\mathbf{A}}(F) =   \mathbf{W}    _{\mathbf{M}}(F )  $   (\cite[\S 6.1]{Borelautomorphic})  and that $  \mathbf{W} _{\mathbf{M}}(F) =   \mathbf{N} _{\Gb}(\mathbf{M})(F)/ \mathbf{M} (F) $ (\cite[Lemma  2.6.32]{BTbook}).   In   particular      $ \mathbf{W}_{\mathbf{A}}(F) $ is the subgroup of $ \Gamma $-invariant elements in $ W_{M} $.       We call $ W := N_{G}(A)/M $ the \emph{relative Weyl group}   of $ \Gb $. It is clear that $ \mathrm{res} $ and $ \mathrm{cores} $ are  equivariant under the action of $ W $.

Let $ \Phi_{F}  \subset  X_{*}(\mathbf{A})  $ denote the set of restrictions 
of elements of 
$ \Phi_{\bar{F}} $ to $ \mathbf{A} $.       The elements of $ \Phi_{F} (\mathbf{A})   $ are called the \emph{relative roots} 
of $ \mathbf{G} $ with respect to $ \mathbf{A} $.  We denote by $ Q(\Phi_{F}) $ the $ \ZZ $-span of $ \Phi_{F} $ in $ X_{*}(\mathbf{A}) $.  Then $ \Phi_{F} $ forms a (possibly non-reduced) root system  in    $ Q(\Phi_{F})_{\QQ}    $.    Since $ \Gb $ is quasi-split, $ \Phi_{\bar{F}} $ does not intersect the kernel of  restriction map.  The set of elements of $ \Phi_{\bar{F}} $ that restrict to the same element in $ \Phi_{F} $ form a single $ \Gamma $-orbit.    The restrictions obtained  from the $ \Gamma $-orbits of   $ \Delta_{\bar{F}} $   constitute a base $ \Delta_{F} $ for $ \Phi_{F}$ (\cite[Proposition 6.8]{Borel-Tits})  and we denote by $ \Phi_{F}^{+} $ the corresponding positive root system.   The natural action of $ W $ on $ X^{*}(\mathbf{A}) $  identifies it with the Weyl group of the root system of relative roots.  To each root $ \alpha \in \Phi_{F} $, there is by definition an element $ \alpha^{\vee} $ in the  vector space dual  of $ Q(\Phi_{F})_{\QQ}    $. The totality $ \Phi_{F}^{\vee} $ of these elements $ \alpha^{\vee} $ naturally forms a root system (\cite[Ch.\ VI \S1 n$^{\circ}1$ Proposition 2]{Bourbaki}).  We refer to $ \Phi_{F}^{\vee} $ as the set of  \emph{relative coroots} of $ \Gb $.   The set $\left \{ \alpha ^ { \vee } \, |  \alpha  \in  \Phi_{F}^{+} \right \} $ is then a system of positive (co)roots for $ \Phi_{F}^{\vee} $. The subset $ \Delta_{F}^{\vee} = \left \{ \varphi(\alpha) \, | \,  \alpha \in \Delta_{F} \right \} $ where $ \varphi(\alpha) = \alpha^{\vee} $ if $ 2 \alpha   \notin  \Phi_{F} $ and $ \frac{1}{2} \alpha^{\vee} $ if $ 2\alpha \in \Phi_{F} $  is  a  base for the positive relative coroots  (\cite[Ch.\ VI \S1 n$^{\circ}5$ Remark 5]{Bourbaki}). By  \cite[Lemma 2.6.5]{BTbook}, $ \Phi_{\bar{F}}^{\vee} $   embeds naturally into $ X_{*}(\mathbf{A}) $. The quadruplet $ (X^{*}(\mathbf{A}),   \Phi_{F},   X_{*}(\mathbf{A} )  , \Phi_{F}^{\vee} ) $    thus      constitutes  a root datum  and   will be  referred to as the \emph{relative root datum} of $ \Gb $.     See    also   \cite[Theorem C.2.15]{CGP}.   

\subsection{Orderings}         In this subsection, we  work with    an abstract  root datum first  and then specialize the  notations to the situation of the previous subsection.  This is done to  address  the absolute and relative cases simultaneously.  The notations for abstract datum will also be used in \S \ref{minusculesec}.    \label{orderings}     

Let $ \Psi = (X, \Phi , X ^{\vee} , \Phi^{\vee} ) $ be a root datum. The perfect pairing $ X^{\vee} \times X  \to \ZZ   $ given as part of this datum  will be denoted by $ \langle - , - \rangle $. Given $ \alpha \in \Phi $, $ \beta  \in \Phi^{\vee} $, we denote by $ \alpha^{\vee} \in \Phi ^{\vee} $, $ \beta   ^ {   \vee    }       \in \Phi $ the associated elements  under the bijection $ \Phi  \xrightarrow{\sim}   \Phi^{\vee} $ given as part of    $ \Psi $. We let $ W_{\Psi} $ denote the Weyl group of $ \Psi  $. If $ \alpha $ is in $ \Phi $ or $ \Phi^{\vee} $, we denote by $ s_{\alpha} \in W_{\Psi} $ the corresponding reflection.

Let $ Q $ be the  span of $ \Phi $ in $ X $, $ Q^{\vee} $ the span of $ \Phi^{\vee} $ in $ X^{\vee }$,  $ X_{0} $ the subgroup of $ X $ orthogonal to $ \Phi^{\vee} $ and $ P \subset  Q_{\QQ} =  Q \otimes_{\ZZ} \QQ $ the $ \ZZ $-dual of $ Q^{\vee}  $.  Then $ Q \subset P $ are lattices in $ Q_{\QQ}   $.   We define $X_{0}^{\vee} $, $  P^{\vee} $ in an analogous fashion. We refer to $ Q $  (resp.\ $ P $, $ Q^{\vee} $, $ P^{\vee} $) as the root (resp.\ weight, coroot, coweight) lattice.  The groups $ P / Q $, $ P^{\vee}/ Q^{\vee} $ are in duality and finite. It is clear that the  action of $ W_{\Psi} $ preserves $ Q $, $ P$, $ Q ^{\vee} $, $ P^{\vee}  $. If $ \chi \in X_{0} $,  $ \chi -  s_{\alpha} \chi  = \langle \alpha^{\vee} , \chi \rangle \alpha  =    0 $ for all $ \alpha \in \Phi $ and thus $ W_{\Psi} $ acts trivially on $ X_{0} $. Similarly it acts trivially on $  X_{0}^{\vee}  $.    
By \cite[Lemma 1.2]{SpringerBorel},  the subgroup $  Q + X_{0} $ of $ X $ has finite index in $ X $ and $ X_{0} \cap Q $ is trivial. Thus each $  \chi \in X $ can be written uniquely as $  \chi_{0} +  \chi_{1} $ for $ \chi_{0} \in X_{0, \QQ} $, $ \chi_{1} \in Q_{\QQ} $. We refer to $ \chi_{0} $ as the \emph{central component} of $ \chi $.     As $  \langle    \lambda   , \chi_{1}  \rangle = \langle   \lambda  ,  \chi \rangle $ for all $ \lambda \in Q^{\vee} $ and $  \langle \lambda , \chi  \rangle \in \ZZ $ as $ \chi \in X $,     we see that $   \chi_{1} \in P   $ for every $ \chi \in X $.    There  is  thus  a  well-defined   $  X \to P $ and its kernel is easily seen to be $ X_{0} $.   We call the map $ X \to P $  the \emph{reduction modulo $ X_{0}  $} and $ \chi_{1} $ the \emph{reduction of $ \chi $ modulo $ X_{0} $}.  We similarly define these notions for $ X^{\vee} $.   
\begin{remark} It is however  not true in general that  $ X \subset X_{0} + P $ e.g.,  consider the root datum of    $ \GL_{2} $.   
\end{remark}   
Let $ \Delta \subset \Phi $  be a  base for $ \Phi $ giving a positive system $ \Phi^{+}$ for $ \Phi $,  $ \Delta^{\vee} $ a base for the corresponding positive system for $ \Phi^{\vee} $, $ S $ the set of reflections associated to $ \Delta^{\vee} $ and $ \ell : W   _ {  \Psi   }     \to \ZZ $ the resulting length function.  We say that $ \lambda \in X^{\vee} $ is \emph{dominant} (resp.,  \emph{antidominant}) if for all  $ \alpha \in \Delta $, we have  $  \langle  \lambda ,  \alpha  \rangle  \geq 0  $ (resp.,    $ \langle  \lambda   ,    \alpha   \rangle    \leq 0  $) and we  denote the set of such $ \lambda $ by $ (X^{\vee})^{+} $ (resp.,  $ (X^{\vee})^{-} $). It is clear that $ \lambda \in (X^{\vee})^{+} $ if and only if $ \langle \lambda , \beta^{\vee} \rangle \geq 0 $ for  all   $ \beta \in \Delta^{\vee} $ (since any element of $ \Delta $ can be written as $ \beta^{\vee} $ or $   \beta^{\vee} /  2  $   for some   $ \beta \in  \Delta^{\vee} $).   We similarly define  dominant elements in $ P^{\vee} $ and denote their collection by $ (P^{\vee})^{+} $. Then $ \lambda \in X^{\vee}$ is dominant if and only if its image $ \bar{\lambda} \in  P^{\vee} $ under reduction modulo $ X_{0}^{\vee}$   is  dominant.      

There exists a partial ordering  $ \succeq  $   on $  X^{\vee} $ which also  depends on the choice of basis $ \Delta^{\vee} $.  It is defined   by declaring $ \lambda \succeq \mu $  for    $ \lambda , \mu \in   X^{\vee}  $  if   $$ \lambda - \mu = \sum   \nolimits     _ { \beta  \in  \Delta  ^ {  \vee  }    } n_{ \beta  }  \beta   $$ for some  non-negative  integers $ n_{ \beta }   \in \ZZ   $. In particular, $ \lambda $ and $ \mu $ are required to have  the  same central   component.    We say that $ \lambda $ is \emph{positive} with respect to $ \succeq $ if $ \lambda  \succeq 0 $ and \emph{negative} if $ \lambda  \preceq  0 $.   We similarly define the   ordering  $ \succeq $ 
for $ P^{\vee} $. It is easily seen that $ \lambda \succeq \mu $ 
for $ \lambda , \mu  \in  X^{\vee  }    $     iff $ \lambda $, $ \mu $ have the same central component and $  \bar{\lambda}  \succeq \bar{\mu} $ 
where $ \bar{\lambda}, \bar{\mu} \in P^{\vee}  $ denote  respectively  the reductions of $ \lambda , \mu $.   
\begin{lemma}   \label{wbetareflection}      Let $ w \in W _  {  \Psi   }    $, $ \beta \in \Delta^{\vee} $ be such   that    $ \ell(w) = \ell( w s_{\beta}) + 1 $. Then $  w \beta $ is negative.         
\end{lemma} 
\begin{proof}  Let $ V = Q^{\vee} \otimes \QQ $. Then $ \Phi ^ { \vee }  $ embeds in $ V $ and $ (V, \Phi^{\vee}) $ is a root system. Let $ \Phi '  \subset \Phi $ the set of all indivisible roots. Then $ (V, \Phi') $ is a reduced root system with the same Weyl group $ W_{\Psi} $  and $ \Delta^{\vee} \subset \Phi' $ is a base for $ \Phi' $. The result then  follows  by   \cite[Ch.\ VI \S1 n$^\circ 6$  Prop.\ 17(ii)]{Bourbaki}. 
\end{proof}    
 In general, a dominant $ \lambda \in   X^{\vee} $  need not be   positive (consider $ \lambda \in X_{0}   ^ {  \vee   }                 $) and a positive $ \lambda $  need not be dominant (cf.\ the `dangerous bend' in \cite[Ch.\ VI   \S 1 n$^\circ 6$]{Bourbaki}).  We however have the following result.

\begin{lemma}  \label{dominateWeylorbit}  $  \lambda  $ in $  X^{\vee} $ or $ P^{\vee} $  is dominant if and only if for all $ w \in W_{\Psi }  $, $ \lambda  \succeq  
w \lambda $. 
\end{lemma} 

\begin{proof}
This is essentially  \cite[Ch.\ VI \S1 n$^\circ 6$  Prop.\ 18]{Bourbaki} where it is proved in the setting of root systems and where the ordering $ \succ $ is defined by taking positive real coefficients. 
We   provide the necessary modifications.    Since both the dominance relation and $ \succeq
$ on $ X ^{\vee} $ are compatible modulo $ X_{0} ^  { \vee } $
    and since the action of $ W_{\Psi} $ on $ X^{\vee}$ preserves    central  components, 
  the claim for $ X^{\vee} $ follows from the  corresponding   claim for $ P^{\vee} $. So let $ \lambda \in P^{\vee} $.  Since  $ \lambda -  s_{\beta}   \lambda =  \langle \lambda , \beta^{\vee}  \rangle \beta $ for any $ \beta \in \Delta^{\vee} $ (see eq.\ (\ref{reflectionmaproots})),   we see that $   \lambda $ is dominant if and only if $ \lambda   \succeq 
  s_{\beta} \lambda $ for all $ \beta \in \Delta^{\vee} $. So it suffices to show that $ \lambda \succeq    w \lambda $ for all $ w \in S $ implies  the same for all $ w \in W $. This is   easily  proved by induction on $ \ell(w) $. Write $ w = w' s_{\beta} $ where $ \beta \in \Delta ^{\vee} $ and $ \ell(w) = \ell(w') + 1 $.  Then \begin{equation}   \label{positivedominant}  \lambda 
 - w\lambda = \lambda - w'   \lambda  + w'( \lambda  - s_{\beta}\lambda) .    \end{equation}       Now $ \lambda - w' \lambda   $ is positive by induction hypothesis. On the other hand, $ w'(\lambda - s_{\beta}\lambda   )    = w (s_{\beta} \lambda - \lambda) = - \langle \lambda,  \beta ^{\vee}  \rangle w\beta$. Since $ -w\beta \in Q^{\vee} $ is positive  by 
    Lemma  \ref{wbetareflection} and $ \langle \lambda , \beta ^{\vee} 
 \rangle \in \ZZ_{\geq 0} $ since $ \lambda \succeq s_{\beta} \lambda $,  we see from (\ref{positivedominant}) that $ \lambda \succeq   
 w \lambda $. This completes the induction  step.     
\end{proof}

We now specialize  back   to the notation of \S  \ref{rootdatasection}. If $ \lambda, \mu \in X_{*}( \mathbf{A}) $, we write $ \lambda \succeq \mu   $ to denote the ordering with respect to the relative root datum.  If $ \lambda, \mu \in X_{*}(\mathbf{M}) $, we write $ \lambda \succeq_{M} \mu $ to emphasize that the ordering is with respect to the absolute root datum.      
The set of dominant relative (resp., absolute) cocharacters  is denoted $ X_{*}(\mathbf{A}) ^{+} $ (resp.,  $ X_{*}(\mathbf{M})^{+} $).    Since $ \mathrm{res}(\Delta_{\bar{F}} ) = \Delta_{F} $, $  \mathrm{cores} $ induces an inclusion   $  X_{*}(\mathbf{A})^{+}   \hookrightarrow        X_{*}(\mathbf{M})^{+} $. We denote by $ X_{*}(\mathbf{A})_{0} $, $ X_{*}(\mathbf{M})_{0} $ the groups orthogonal to $ \Delta_{F} $, $ \Delta_{\bar{F}} $ respectively. Then $ X_{*}(\mathbf{A})_{0} = X_{*}(\mathbf{M})_{0} ^  {\Gamma} $.     

Recall that we denote by $ W $ the relative Weyl group for $ \Gb $. Let $ S := \left \{ s_{\alpha}  \, | \,      \alpha  \in  \Delta_{F}  \right \}     $ be the set of simple reflections and $ \ell   =  \ell_{F} : W \to \ZZ  $ the  resulting  length function.  The \emph{longest Weyl element} $ w_{\circ} \in W $ is defined to be the unique element which attains the maximum length in $ W $. Then $ w_{\circ} $ is also maximal under Bruhat ordering and is the  unique element of $ W $ satisfying  $ w_{\circ}   \cdot    \Delta_{F} = - \Delta_{F} $ (as a set).  We have $ w_{\circ}^{2} = \mathrm{id}_{W} $.     For each $ \lambda   \in    X_{*}(   \mathbf{A} ) $, we define  $   \lambda ^ {\mathrm{opp}} := w_{\circ}   \lambda  $. Then for $ \lambda \in  X_{*}(\mathbf{A}) ^ {+ }$, $ \lambda^{\mathrm{opp}}  $  is the unique element in the Weyl  orbit of $ \lambda $  that  lies  in    $  X_{*}(\mathbf{A})^{-} $. Moreover 
\begin{align} \label{oppositesucc}  \lambda \succeq \mu  \Longleftrightarrow - \lambda ^ {\mathrm{opp} } \succeq - \mu ^ { \mathrm{opp } }
\end{align}  
for  any   $ \lambda , \mu \in X_{*}(\mathbf{A} )   $        since $ - w_{\circ} ( \lambda - \mu )  \succeq 0  $.
We will say that $ w_{\circ} =  - 1 $ as an element of $ W $ if $ w_{\circ} (\alpha) = - \alpha $ for all $ \alpha \in \Delta_{F} $.    We can similarly define $ \lambda^{\mathrm{opp}} $ for  any  $ \lambda \in X_{*}(\mathbf{M}) $.  This is compatible with $ \mathrm{cores}  $ by the following.     
\begin{lemma}   \label{longestlemma}           $ w_{\circ}   $  is  also the longest element in $ W_{M}   $.  
\end{lemma}
\begin{proof} Since $ w_{\circ}  \in  W   =    (W_{M})^{\Gamma} $, the action of $ w_{\circ}  $  on $ \Phi_{\bar{F}} $ is $ \Gamma $-equivariant. In particular, $ w_{\circ} $ preserves $ \Gamma $-orbits. Since restriction $ \mathrm{res}: \Phi_{\bar{F}} \to \Phi_{F} $  is $ W $-equivariant and sends positive (resp., negative) absolute roots to positive (resp.,  negative) relative roots, we see that $  w_{\circ} \cdot \Delta_{\bar{F}} =  - \Delta_{\bar{F}} $.           
\end{proof}

\begin{lemma}   \label{wcircnegativeresult}              If $ w_{\circ} = - 1 $ as an element of $ W $, then $ \lambda + \lambda^{\mathrm{opp}} \in X_{*}(\mathbf{A})_{0} $ for any $ \lambda \in X_{*}(\mathbf{A}) $.  Moreover if $ \lambda  \succeq \mu $ for some $ \mu \in X_{*}(\mathbf{A}) $, then $ \lambda + \lambda^{\mathrm{opp}} = \mu + \mu^{\mathrm{opp}} $.      
\end{lemma} 
\begin{proof} The  first claim  follows since $ \langle  \lambda + \lambda^{\mathrm{opp}} , \alpha \rangle =  \langle \lambda , \alpha + w_{\circ} \alpha \rangle $ for any $  \alpha \in \Delta_{F} $.   The second claim follows since $ \lambda - \mu $ is a positive integral sum of positive coroots and applying $- w_{\circ} $ acts as  identity on this sum.    
\end{proof}

\subsection{Iwahori Weyl group}     \label{Weylgroupsection}      
From now on, we denote $ X_{*}(\mathbf{A}) $ by $ \Lambda $.      We fix  throughout a smooth reductive group scheme $\mathscr{G}$ over   $\mathscr{O}_{F}$ such that $\mathbf{G} $ equals the  generic fiber $ \mathscr{G}_{F} $ of $ \mathscr{G} $.       
Then $ K :=   \Gscr(\Oscr_{F})  $ is a \emph{hyperspecial} maximal compact subgroup of $ G = \mathscr{G}(F) $. Let $ A^{\circ} :=  A \cap K $, $ M^{\circ} :  = M \cap K $. As $ \Gb $ is unramified, $ A^{\circ} $, $ M^{\circ} $ are the unique maximal compact open subgroups of $ A $, $ M $ respectively. 
In  particular, these do not depend on $ \mathscr{G} $. Moreover $ W $ is identified with $ (K \cap N_{G}(A)  )   /M^{\circ} $.  We have isomorphisms $   \Lambda  \stackrel{\sim}{\rightarrow} A / A^{\circ}  \stackrel{\sim}{\rightarrow} M / M^{\circ} 
 $    induced respectively   by $\lambda \mapsto \varpi^{\lambda} A^{\circ}, A \hookrightarrow M$ (see \cite[\S 9.5]{Borelautomorphic}). We denote by 
\begin{equation}   \label{vmap}     v : A / A^{\circ}  \to   \Lambda   
\end{equation}  the inverse of the  \emph{negative} 
   isomorphism $   \Lambda   
   \to A/ A^{\circ} $, $ \lambda \mapsto \varpi^{-\lambda} A^{\circ} $.   
The quotient $  W_{I} : =   N_{G}(A) /  M ^{\circ} $ is called the  \emph{Iwahori Weyl group} of $ G  $.  It  naturally   isomorphic  to the semi-direct products $ M/M^{\circ} \rtimes  W \simeq A/A^{\circ} \rtimes W $   (\cite[\S 3.5]{CartierSatake})   and we identify $ W_{I} $ with these groups.   The   mapping   (\ref{vmap}) induces a further 
isomorphism $ v :  W_{I}  =    A/A^{\circ} \rtimes W  \xrightarrow{\sim}  \Lambda  \rtimes W $ where  $ \varpi^{\lambda} A^{\circ} \in W_{I} $ for $ \lambda \in \Lambda $  is     identified with  $ (-\lambda, 1) $.

Let $   Q ^ { \vee } _  { F }  =  Q(\Phi_{F}^{\vee} ) $ denote the relative coroot  lattice.   The subgroup $  W_{\mathrm{aff}}   : =     Q_  {F}  ^  {  \vee }    \rtimes W $ of $ \Lambda \rtimes  W $ is   called the     (relative)  \emph{affine Weyl group}.   
The group $ W_{\mathrm{aff}} $ acts on the vector 
space $ Q^{\vee}_{F}  \otimes \RR $ by translations and it is customary to denote the element $ (\lambda,1) \in W_{\mathrm{aff}} $ by $ t_{\lambda} $ or $ t(\lambda) $. Similarly when the coroot lattice $ Q^{\vee}_{F}  $ is viewed as a subgroup of $ W_{\mathrm{aff}}    $, it is written as  $ t ( Q ^{\vee} _{F}  )  $. More generally, we denote the element $ ( \lambda , 1 ) \in \Lambda \rtimes W $  by  $ t(\lambda)  $ and consider it as a translation of $ \Lambda  \otimes \RR  $.   If $ \Phi _ { F } $ is irreducible, $ \alpha _{ 0 } \in  \Phi_{F} $  is 
 the highest root and $ s_{\alpha_{0}} \in W $ denotes   the reflection associated with $ \alpha_{0} $, the group   $ W_{\mathrm{aff}} $ is   a  Coxeter group with  generators  $      S_{\mathrm{aff}} : = S \sqcup  \left \{ t_{ \alpha_{0} ^ {\vee}  }   s_{\alpha_{0}}   \right \} $.    In general,   $ W_{ \mathrm{aff} }  $ is    a Coxeter group whose set of generators $ S_{\mathrm{aff} } $ is  obtained by extending the set $ S $ by the reflections associated to the simple affine  root of  each irreducible component of $ \Phi_{F} $.   In particular, its rank (as a Coxeter group) is the number of irreducible components of $ \Phi_{F} $ added to the  rank of $ W   $.  
We denote by $ \ell : W_{\mathrm{aff}}  \to  \ZZ  $ the extension of $ \ell : W \to \ZZ $ and by $  \geq $  the strong Bruhat order   on   $  W _ { \mathrm{aff}  }  $     induced by the set $ S_{\mathrm{aff}} $.    Via the isomorphism $ W _{I}  \xrightarrow {v } 
 \Lambda  \rtimes  W $, we  identify $ W_{\mathrm{aff} }$ as a subgroup of $ W_{I} $.    The  quotient $  \Omega :=  W_{I} / W _{\mathrm{aff} }  $ acts  on    $ W_{\mathrm{aff}} $ by automorphisms (of Coxeter groups) and one has an  isomorphism    $ W_{I}   \simeq  W_{\mathrm{aff}}   \rtimes \Omega  $. One extends the  length function to a function  $$ \ell :  W_{I}  \to \ZZ $$  by declaring the length of elements of $ \Omega $ to be $ 0 $.  Similarly,  the    strong      Bruhat ordering  on $ W_{\mathrm{aff}}  $  is  extended    to $ W_{I} $ by declaring $ w \rho \geq w' \rho' $ for $ w , w ' \in  W_{\mathrm{aff}  } $,  $  \rho  ,  \rho '  \in  \Omega $ if $ w \geq w' $ and $ \rho = \rho' $. Each $ W \backslash W_{I} / W $    
 has a unique minimal length representative in $ W_{I}$ via which we can define a partial ordering on the double cosets.     Under the identification $ \Lambda^{+} \simeq W \backslash W_{I} / W $, the ordering $ \succeq $  restricted  to  $ \Lambda^{+} $ is identified with the ordering on  representatives in $ W \backslash W _{I} /  W  $. See \cite[\S Corollary 4.7]{Stembridge}.

\begin{remark}   See  \cite[\S 3.5]{CartierSatake} and \cite[Ch.\ 1]{Tits} for the role of buildings in  defining  these  groups.
Buildings will be briefly   used in \S  \ref{reductivetitssection}.      
\end{remark}     
\subsection{The Satake transform}   \label{Sataketransormsection}    

Fix a Haar measure $\mu_{G}$ on $G$ such that $\mu_{G}(K)=1$. For a ring $ R $, let   $\mathcal{H}_{R}(K \backslash G / K)$ be the Hecke algebra of level $ K $ with coefficients in $ R $ (Definition \ref{Heckealgebradefinition}) and $   R\langle G / K\rangle$ be the  set of finite $R$-linear combinations on cosets in $G / K  $.  For $ \sigma  \in G $, we denote by $ \ch  (K \sigma  K ) \in   \mathcal{H}_{R}( K  \backslash  G  / K )  $ the characteristic function of $ K \sigma K $ which we will occasionally also write simply as $ (K \sigma K) $.        For $\lambda \in \Lambda$, denote by  $e^{\lambda}$ the element corresponding to $\lambda$ in the group algebra $\mathbb{Z}[\Lambda]$ and $ e^{W\lambda} $ the  (formal)  sum $ \sum_{\mu \in W \lambda}  e^{\mu} $.  This allows one to convert from additive to  multiplicative notation for  cocharacters.    The half sum of positive roots   $ \delta : =  \frac{1}{2} \sum   \nolimits    _{\alpha \in \Phi_{\bar{F}}^{+} } \alpha $ is an element of $ P (\Phi_{\bar{F}}) $    by  \cite[\S 13.3 Lemma A]{Humphreys}.                      For $ \lambda \in \Lambda  = X_{*}(\mathbf{A}) $, let  $ \langle \lambda, \delta \rangle $ denote the  quantity $   \langle \mathrm{cores}(\lambda) ,  \delta \rangle =  \langle \lambda, \mathrm{res}(\delta ) \rangle $.

Let  $  \mathcal{R}=\mathcal{R}_{q} $ denote the ring $ \mathbb{Z}[q^{\pm \frac{1}{2}}] \subset \CC $ where $ q^{\frac{1}{2}} \in \CC  $ denotes a root of $  x^{2}-q $ and $ q^{-\frac{1}{2}} $ denotes its inverse.    Denote by $ p: G / K \rightarrow K \backslash G / K$ the natural map and  $p^{*}: \mathcal{H}_{\mathcal{R}}(K \backslash G / K) \rightarrow \mathcal{R}\langle G / K\rangle$ the induced map that sends the characteristic function of $ K \sigma K $ to the formal sum 
of left cosets $ \gamma K $ contained in $ K \sigma K $.     Let $  \mathscr{I}: \mathcal{R}\langle G / K\rangle  \rightarrow \mathcal{R}[\Lambda]  $  denote the $ \mathcal{R} $-linear  map defined by $    \ch (   \varpi ^ { \lambda } n  K  )     \mapsto    q^{-\langle\lambda,  \delta  \rangle} e^{\lambda} $ for $\lambda \in \Lambda, n \in N $. 
This is well defined  by \cite[Lemma 5.3.5]{BTbook} (since   $M K /  K  \simeq M / M^{\circ}  \simeq   \Lambda$). The  composition \begin{align} \mathscr{S} :    \mathcal{H}_{\mathcal{R}}(K \backslash G / K)  \xrightarrow{  p ^{*}    }  \mathcal{R}  \langle   G / K  \rangle     \xrightarrow{  \mathscr{I} }  \mathcal{R} [\Lambda] 
\end{align}  
is then  a homomorphism of $ \mathcal{R} $-algebras known as the  \emph{Satake transform}. Its image lies in the Weyl invariants $ \mathcal{R}[\Lambda]^{W} $.  By 
\cite[Theorem 4.1]{CartierSatake} or \cite[Theorem 3]{Satakeoriginal}), the induced map $ \mathscr{S}_{\CC} $ over $ \CC $ is an isomorphism onto $ \CC [ \Lambda ] ^{W} $.        We note that $ \left \{   (  K \varpi ^{\lambda} K ) \, | \,   \lambda  \in  \Lambda^{+}  \right \} $ is a basis for $ \mathcal{H}_{\mathcal{R}}(K \backslash G / K )  $ by Cartan decomposition. We are therefore interested in the Satake transform of such   functions.       For    $ \lambda \in \Lambda^{+} $,  write   
\begin{align}  \label{Satakeform} \mathscr{S}(K \varpi^{\lambda} K ) 
& = \sum_{ \mu \in \Lambda} q^{- \langle  \mu, \delta   \rangle    }  a_{\lambda}(\mu)  e^{\mu}
\end{align} 
where $ a_{\lambda}(\mu) \in \ZZ_{\geq 0}     $.  By  definition,   $  a_{\lambda} ( \mu ) $ is equal to the number of distinct left cosets $ \varpi^{\mu} n K $ for $ n \in N $ such that $ \varpi ^ { \mu } n K  \subset  K \varpi ^ { \lambda } K  $.  The $ W $-invariance of $ \mathscr{S} $ implies that $ q ^ { - \langle \mu_{1}  , \delta \rangle }  a_{\lambda}(\mu)  =  q ^ { -  \langle \mu_{2} , \delta  \rangle    }    a_{\lambda}(\mu_{2}  ) $ 
for all  $ \mu_{1}, \mu_{2} \in \Lambda $ such that $ W \mu_{1} = W \mu _  {  2 }   $.  Let $ \succeq $ denote the same partial ordering in \S  \ref{orderings}.

\begin{proposition}   \label{Satakeupperprop}             For $ \lambda , \mu \in \Lambda^{+ } $, $ a_{\lambda}(\mu ) \neq 0 $ only if $ \lambda \succeq \mu $. Moreover, $ a_{\lambda}(\lambda^{\mathrm{opp}}) = 1 $. 
\end{proposition}     
\begin{proof} Set  $ \kappa = \lambda ^ { \mathrm{ opp }} $ and $  \nu := \mu  ^  { \mathrm{opp } } $.  Then  $ - \kappa , - \nu \in \Lambda^{+ } $.   Since the image of $ \mathscr{S} $ is $ W $-invariant, $ a_{\lambda} ( \mu ) \neq 0  $ if and only if $ a_{\lambda}(\nu) \neq 0 $. By definition, this is  equivalent to  $  \varpi^{\nu } N K  \cap K \varpi^{\lambda} K \neq \varnothing $. Now $  \varpi^{\nu} N =   N \varpi ^  { \nu }   $ as $ A $ normalizes $ N $ and $ K\varpi^{\lambda} K = K \varpi^{ \kappa }  K $ as $ K \cap N_{G}(A)  $ surjects onto $ W $.  Thus $$  a_{\lambda}(\mu) \neq 0 \,\,\,  \Longleftrightarrow  \,\,\,   K \varpi^{ \kappa } K    \cap        N \varpi^{\nu} K    \neq     \varnothing . $$ By \cite[Lemma 10.2.1]{HainesRostami} and the identification of $ \succeq $ on $ \Lambda^{+}$ with the  Bruhat  ordering on $ W \backslash  W_{I}  / W $,  
we  get that     $ K \varpi^{ \kappa } K  \cap   N \varpi ^{ \nu} K  \neq \varnothing  \implies  - \kappa \succeq - \nu $\footnote{the negative sign arising from  the normalization (\ref{vmap})}.      
But the last condition is  the same as $ \lambda \succeq \mu $  by    (\ref{oppositesucc}). This   establishes the first part.  By \cite[Proposition 4.4.4(ii)]{TitsBruhat}, $K \varpi^{ \kappa } K   \cap    N \varpi ^ {   \kappa  } K    =  \varpi^{ \kappa } K    $ i.e.,    the only coset of the form $ \varpi^{  \kappa }     n K  $ where $  n \in N $ such that $ \varpi^{ \kappa }  n  K \subset K \varpi^{\lambda} K  $ is $ \varpi^{\kappa} K $.  The second claim  follows.     
\end{proof}

\begin{remark}   
 A weaker version  of above 
 appears in \cite[p.148]{CartierSatake}. 
   See also \cite[Th\'eor\`eme 5.3.17]{Matsumoto}.   
\end{remark}

\begin{corollary}   \label{Satakeuppercoro}   For $ \lambda \in \Lambda^{+} $, $ \mathscr{S} ( K \varpi ^ { \lambda} K ) - q ^ { \langle  \lambda,  \delta  \rangle  } e^{W \lambda } $ lies   in the  $ \mathcal{R} $-span of $ \left \{ e ^ { W \mu } \, | \, \mu \in \Lambda^{+} ,  \mu  \precneq \lambda \right \} $.
\end{corollary}

\begin{proof} 
Since $ w_{\circ} \delta = -  \delta   $ by Lemma   \ref{longestlemma},  we see that 
   $   \langle  \lambda^{\mathrm{opp} } ,  \delta \rangle = \langle \lambda ,  w_{\circ} \delta    \rangle = - \langle    \lambda  ,   \delta \rangle $. 
The second part of 
Proposition  \ref{Satakeupperprop} therefore implies that $$ q^{- \langle \lambda^{\mathrm{opp}} ,   \,       \delta   \rangle    }  a_{\lambda}(\lambda^{\mathrm{opp}}) =  q ^ { \langle  \lambda,  \delta  \rangle  } . $$ Thus the coefficient of $ e^{W\lambda} $ in $ \mathscr{S}(K \varpi^{\lambda} K ) $ is $ q^{   \langle  \lambda  , \delta \rangle   } $. The claim now follows by the first part of  \ref{Satakeuppercoro}.     
\end{proof}

\begin{corollary} The Satake transform induces an isomorphism $ \mathcal{H}_{  \mathcal{R}  } (K \backslash G / K ) \simeq  \mathcal{R}  [\Lambda]^{W} $ of   $ \mathcal{R} $-algebras.  
\end{corollary}   

\begin{proof} Fix $ \lambda \in \Lambda^{+} $. We wish to show that $ e^{W \lambda} $ lies in the image of $ \mathscr{S} $. Let  $ \mathcal{U}_{0}   = \left \{ \mu \in \Lambda^{+} \, | \, \mu  \preceq \Lambda   \right \}  $ and inductively  define $ \mathcal{U}_{k} $ as the set $  \mathcal{U}_{k} \setminus 
  \max      \,     \mathcal{U}_{k} $ for $ k \geq 1 $. It is clear that $ \mathcal{U}_{0} $ and hence each $ \mathcal{U}_{k} $ is finite.  
By Corollary  \ref{Satakeuppercoro}, $ f_{1} : =  \mathscr{S} \big ( q ^{- \langle \lambda, \delta \rangle }  ( K \varpi ^{\lambda} K  )    \big )  - e^{W\lambda}  \in  \mathcal{R} [  \Lambda   ]  ^ { W  } $ equals a sum $ \sum c_{\lambda}(\mu   )    e^{W \mu } $ where $ \mu $ runs over the set $   \mathcal{U}_{1}   =  \left \{  \mu \in \Lambda^{+} \, | \, \mu \precneq \lambda \right \} $  and $ c_{\lambda}(\mu)  \in   \mathcal{R} $. 
By  Corollary \ref{Satakeuppercoro} again,  $$  f_{2} :=  \mathscr{S} \bigg (    q ^ { - \langle  \lambda , \delta  \rangle  } ( K \varpi^{\lambda } K )  -  \sum   \nolimits       _{ \mu \in \max \mathcal{U}_{1} } q ^ { -  \langle  \mu   ,  \delta   \rangle   } c_{\lambda} ( \mu )  ( K \varpi ^ { \mu } K )   \bigg   ) - e^{W \lambda } $$
is a linear combination of $ e ^ { W \mu }  \in  \mathcal{R}[\Lambda]^{W}  $ for  $  \mu  \in  \mathcal{U}_{2} $.     Continuing this process, we obtain a sequence of elements $ f_{k} \in \mathcal{R}[\Lambda]^{W} $ for $ k \geq 1 $ that are supported on $ \mathcal{U}_{k}  $ and such that $ e^{W \lambda} + f_{k} $ lies in the image of $ \mathscr{S} $.   Since $ \mathcal{U}_{k} $ are eventually empty, $ f_{k} $ are eventually zero and  we obtain the  desired   claim.    
\end{proof}

\begin{corollary}        Suppose $ w_{\circ} = -1 $ as an element of $ W $. Then the transposition  operation    $ \mathcal{H}_{R}(K  \backslash G / K ) $ corresponds under Satake transform to the negation of cochracters  on  $ \mathcal{R}[\Lambda]   ^ {  W   }  $.    
\end{corollary}     
\begin{proof}  For $ \lambda \in \Lambda^{+} $, $ \ch (K \varpi^{\lambda} K ) ^{t}   = \ch ( K \varpi^{\kappa}    K ) $  where $ \kappa := -\lambda^{\mathrm{opp}} \in \Lambda^{+}   $.  By    Lemma \ref{wcircnegativeresult},   $ \kappa = \lambda +  \lambda_{0} $ for some $ \lambda_{0} \in X_{*}(\mathbf{A})_{0} $. Since $ \varpi^{\lambda_{0} }   $ is central in $ G $,   $  (K\varpi^{\kappa}K  )=  (K \varpi^{\lambda}K ) * (K\varpi^{\lambda_{0}} K) $  and  as    $     W \lambda_{0} =  \lambda_{0} $, $$  \mathscr{S}( K \varpi^{\kappa} K ) = \mathscr{S}( K \varpi^{\lambda} K )  e^{\lambda_{0}}   .   $$  
Now for any $ \mu \in \Lambda   $ such that $ a_{\lambda}(\mu) \neq 0 $, we have $ \lambda \succeq \mu   $  by   Proposition  \ref{Satakeupperprop}   and   Lemma  \ref{dominateWeylorbit}.  Thus             $  - \mu^{\mathrm{opp}}  =  \mu +   \lambda_{0}   $  by  Lemma  \ref{wcircnegativeresult}.  The result now follows since $ e^{W\mu} \cdot e^{  \lambda_{0} }  =  e^{W (\mu + \lambda_{0} ) }  =  e^{ W  ( - \mu   )  } $.           
\end{proof} 
\begin{definition}   \label{Satakeleadingtermdefi}              For $ \lambda \in \Lambda ^ { + }  $,  we   call  the element    $ q^{\langle \lambda, \delta \rangle } e^{W \lambda }  \in  \mathcal{R}[\Lambda]^{W} $ the \emph{leading term} of the Satake  transform  of    $ ( K \varpi^{\lambda} K ) $ and the number $ q^{ \langle \lambda, \delta \rangle } $ its  \emph{leading coefficient}.      If $ gK \subset K \varpi^{\lambda} K $ is a coset, we call the unique cocharacter $ \mu \in  \Lambda $  such that $ gK = \varpi^{\mu} n K $ for some  $ n \in  N $  the    \emph{shape} of the coset $ gK  $. The shape $ \mu $ of any $ g K \subset K \varpi^{\lambda} K $ for $ \lambda \in  \Lambda^{   +   }    $      satisfies $ \lambda \succeq       \mu  $ by the  results  above.     
\end{definition}
\begin{remark} 
Proposition \ref{Satakeupperprop} and  most of  its corollaries may be found in several places in literature,  
though the exact versions           we needed     are    harder to  locate.    We have chosen to include proofs primarily to illustrate our conventions, which  will also be useful in computations in Part II.    
Cf.\ \cite[\S 3.2]{PilloniFakharuddin}. 

\end{remark} 
\begin{remark}  
One can strengthen  Proposition  \ref{Satakeupperprop}  to   $   a_{\lambda}(\mu) \neq 0  \Longleftrightarrow  \lambda  \succeq 
\mu $. See \cite[Theorem 1.1]{Rapoportpositive}.   
\end{remark}

\subsection{Examples}   \label{Satakeexamples}         
In this subsection, we provide a few examples  of Satake transform computations for $ \GL_{2} $ to illustrate  our  conventions     in a simple setting.    

Let $  \mathbf{G} = \GL_{2 , F } $, $ \mathbf{A} = \GG_{m} \times \GG_{m } \hookrightarrow \mathbf{G} $ be the standard diagonal torus and $ K = \GL_{2} (  \Oscr_{F}  )   $.    For $ i = 1, 2 $, let $ e _{i} :  \mathbf{A} \to   \GG_{m}   $ for $ i =1  , 2 $ be the characters given  by  $  \mathrm{diag}(u_{1}, u_{2}) \mapsto u_{i} $, $ i = 1 , 2 $ and $  f  _{i} :  \GG _{m} \to \mathbf{A} $ be the cocharacters that insert $ u $ into the $ i $-th component. Then $ \Phi = \left \{ \pm ( e_{1} - e_{2} ) \right \} $ and $ \Lambda = \ZZ f_{1} \oplus \ZZ f_{2}$. We will denote $ \lambda =  a_{1} f_{1} + a_{2} f_{2} \in \Lambda $ by $ (a_{1} ,a_{2} )$.  We take $  \chi  : = e_{1} - e_{2}  \in X^{*}(\mathbf{A}) $ as the positive root, so that $ \delta = \frac{\chi}{2} $ and  $ \Lambda^{+} $ is the set $ (a_{1},   a  _{2} )$ such that $ a_{1} \geq a_{2} $. Let $ \alpha   :    = e^{f_{1}} $, $ \beta   :   = e^{f_{2}} $ considered as elements of the group algebra $ \ZZ [ \Lambda ] $.   Then   $ \mathcal{R} [\Lambda] ^{W} = \mathcal{R} [\alpha^{\pm}, \beta^{\pm} ] ^{S_{2}} $ where the non-trivial element of $ S_{2} $ acts via $ \alpha \leftrightarrow \beta $. 

\begin{example}    
Let $ \lambda = f_{1} \in \Lambda^{+} $. Then  
$ \lambda^{\mathrm{opp}} = f_{2}   $. As is well-known,  $$     K \varpi^{\lambda} K = \begin{pmatrix} 1 &  \\ &  \varpi  \end{pmatrix} K  \,   \sqcup   \,  \bigsqcup _ { \kappa \in [\kay]  }  \begin{pmatrix}  \varpi &  \kappa   \\ &   1 \end{pmatrix}   K   .$$
In this decomposition, there is $ 1 $ coset of shape $ f_{2} $ and $ q $ cosets of shape $ f_{1} $. Therefore, we  obtain     
$$ \mathscr{S} ( K \varpi^{\lambda} K )  
=  q^{\frac{1}{2}} \beta + q \cdot  q^{ -\frac{1}{2} } \alpha =  q^{\frac{1}{2} }  ( \alpha    +   \beta   )  \in \mathcal{R} [ \Lambda]^{W} . $$ 
\end{example}

\begin{example} \label{SatakeexampleII}     Let $ \lambda = 2f_{1} \in \Lambda^{+} $. Then   
$    \lambda^{\mathrm{opp}} =  2f_{2} $. It is easy to see that   $$        K  \varpi^{\lambda} K = \begin{pmatrix}    1 &  \\ &  \varpi^{2}   \end{pmatrix} K  \,    \sqcup  \,  \bigsqcup _ { \kappa \in [\kay] \setminus  \left \{ 0  \right  \}     }  \begin{pmatrix}  \varpi &  \kappa   \\ &  \varpi  \end{pmatrix}  K   \,  \sqcup   \,    \bigsqcup  _ {  \kappa_{1}, \kappa_{2}  \in [\kay] }   \begin{pmatrix}   \varpi ^ { 2}  & \kappa_{1} + \varpi \kappa_{2}  \\  & 1       \end{pmatrix}   K   .$$
In this decomposition, there is  one coset   of shape $ 2f_{2} $, $ q-1 $ cosets of shape $ f_{1} +  f_{2} $  and $ q^{2} $ of shape    $  2f_{1}  $. So,        
\begin{align*} 
 \mathscr{S} ( K \varpi^{\lambda} K )   & = q \beta ^{2}  + (q-1)  \cdot \alpha \beta +    q^{2} \cdot   q^{-1}  \alpha ^{2} 
 \\   
 &    = q (\alpha ^{2}+  \beta ^{2} )   +  ( q - 1 ) \alpha  \beta    \in   \mathcal{R} [   \Lambda  ] ^ { W  }   .  
\end{align*}  
\end{example}            
\begin{remark}  One can  in fact  write an explicit formula for $ \mathscr{S} ( K \varpi^{\lambda} K) $ for  any  $ \lambda \in \Lambda $. See   \cite[\S 2 p.20]{CasselmanIran} for a formula   
in terms of  $ \mathcal{R} $-basis $  \alpha ^{m} \beta ^{n} $ of $ \mathcal{R}[\Lambda]  $.    
\end{remark} 

\subsection{Macdonald's formula}   

The Satake transform is not  explicit in the sense that the    coefficients of the non-leading terms are not explicit. In general, the coefficients can be quite cumbersome expressions in $ q $.  There is however the following formula   due to I.G.\ Macdonald \cite{Macdonaldspherical} (see also \cite[Theorem 5.6.1]{KHPIwahori}).
   
\begin{theorem}[Macdonald]   \label{macdonald}        Suppose $ \mathbf{G} $ is split and $ \Phi_{\bar{F}} = \Phi_{F} $ is  irreducible. Then  for  any    $ \lambda \in \Lambda ^{+ } $, 

$$ \mathscr{S} ( K \varpi^{\lambda} K ) =  \frac{q^{ \langle  \lambda ,  \delta \rangle } } { W_{\lambda} (q^{-1} )      }     \sum_{ w \in W }  \prod _ {   \alpha \in    \Phi^{ +}  }
   e ^ { w \lambda }   \cdot      \frac { 1 -  q ^{ - 1}   e ^ { -  w \alpha ^ { \vee }  }  } { 1 - e ^ { -  w \alpha^{\vee} } }  $$
   where $ W_{\lambda}(x) : = \sum_{w \in W^{\lambda} } x^{\ell(w)}    $ denotes the Poincar\'{e}  polynomial of the stabilizer $ W^{\lambda}   \subset  W  $ of $ \lambda  $.   

\end{theorem} 
For arbitrary reductive groups, there is a similar but slightly more complicated expression as it takes into account divisible/multipliable  roots and different contributions of root group filtrations.  We refer the reader to  \cite[Theorem 4.2]{Casselmanunramified} and     \cite[\S 3.7]{CartierSatake} for details.  These formulas however will not be needed.        

\begin{example} Retain the notations of \S  \ref{Satakeexamples}. We have $  e^{-\chi^{\vee} } = \alpha^{-1} \beta $ and 
$$ 
\frac{ 1 -  q ^ { - 1 }  e ^ {- \chi ^ {\vee} }   }     { 1 - e^{-\chi^{\vee} }   }      =  \frac{ \alpha - q^{-1} \beta }  { \alpha - \beta } ,  \quad  \quad   \quad        \frac  {     1 -   q ^ { - 1  }    e^{\chi^{\vee} } } { 1 - e ^{ \chi^{\vee} } }   =  \frac{   \beta -  q ^{-1}  \alpha }  {  \beta - \alpha }    . $$ 
For  $ \lambda = 2  f_{1} $, we  compute    \begin{align*}   \mathscr{S} ( K \varpi^{\lambda} K ) &   =  q \bigg (    \alpha ^ { 2 }  \cdot  \frac{\alpha - q^{-1} \beta}  { \alpha - \beta  }  +    \beta ^{2} \cdot \frac{ \beta - q ^{-1} \alpha }   { \beta  -   \alpha }    \bigg )    \\        
& =  q  ( \alpha^{2} + \alpha \beta + \beta ^{2} ) - \alpha \beta   \\
&    =  q ( \alpha^{2}  + \beta^{2}  ) + (q-1) \alpha  \beta    
\end{align*} 
which agrees   with    Example \ref{SatakeexampleII}. 
\end{example}

\subsection{Representations of Langlands dual}   

\label{RepsofLanglandsdual}    

Let $ \hat{\mathbf{G}} $ denote the dual group of  $ \mathbf{G} $ considered as a split reductive group over $ \QQ $. Let $ \hat{ \mathbf{M}} \subset \hat{\mathbf{G}} $ denote the maximal torus such that $ X_{*}(\hat{\mathbf{M}} ) = X^{*}(\mathbf{M}) $. We let $ \hat{\mathbf{P}} $ be the Borel subgroup of $ \hat{\mathbf{G}}     $ corresponding to $ \hat{\Phi}_{\bar{F}}^{+}  :=   (\Phi_{\bar{F}} ^ { \vee})^{+}    \subset X^{*}(\hat{\mathbf{M}}) = X_{*}(\mathbf{M})$. The action of $ \Gamma $ on based root datum of $ \mathbf{G} $ together with a choice of pinning determines an action of $ \Gamma $ on $ \hat{\mathbf{G}} $ which is  unique up to an  inner automorphism by $ \hat{\mathbf{M}} $.  We define the \emph{Langlands dual} to be $ {} ^{L} \mathbf{G} =  { } ^ {  L     } \mathbf{G}_{F}   : =  \hat{\mathbf{G}}  \rtimes  \Gamma $ considered as a disconnected locally  algebraic group over $ \QQ $. We  refer the reader to \cite[Ch.\ I-III] {Borelautomorphic} for a detailed treatment of this group.      See also \cite[\S  1]{Blasius-Rogawski}.     

\begin{remark} The subscript $ F $ in the notation $ {}^{L} \mathbf{G}_{F} $  is not meant to suggest base change of algebraic groups but rather the fixed field for the Galois group $ \Gamma $.  If $ E / F $ is an unramified field extension, and ${} ^ {L} \mathbf{G}_{E} $ denotes the subgroup $ \hat{\mathbf{G}} \rtimes \Gamma_{E} $ of $ {}^{L} \mathbf{G} _ {F } $.      
\end{remark} 

Since the weights of algebraic  representations of $ \hat{\mathbf{G}} $ are elements of  $  X^{*}(\hat{\mathbf{M}} ) =  X_{*} ( \mathbf{M} )$, we  also refer to  elements of $ X_{*}(\mathbf{M}) $ as \emph{coweights}. 
For each dominant coweight $ \lambda \in X_{*}(\mathbf{M})^{+}  $,  there exists a simple representation $   \big  (\pi, V_{\lambda}  \big  ) $   of     $ \hat{\mathbf{G}} $ unique up to isomorphism such that   $ \lambda  \succeq_{M}  \mu $ for any coweight $ \mu $ appearing in $ V_{\lambda}    $ (\cite[Theorem 22.2]{MilneAlgebraic}). Since $ \hat{\mathbf{G}} $ is defined over $ \QQ $, so is the representation $ V_{\lambda} $  (\cite[\S 22.5]{MilneAlgebraic}). For $  \mu $ is a coweight of  $ V_{\lambda} $,  we  denote by $ V_{\lambda}^{\mu} $ the  corresponding    coweight space.

Let   $ \varphi : \hat{\mathbf{G}} \to  \hat{\mathbf{G} } $ be  an   endomorphism  that sends $ \hat{\mathbf{P}} $,  $ \hat{\mathbf{M} } $ to themselves and preserves $ \lambda $ i.e.,  $ \lambda   \circ  \varphi  = \lambda $ as maps $ \hat{\mathbf{M}}  \to  \GG_{m}  $.      
Then  the  representation  of $ \hat{\mathbf{G}} $ obtained via the  composition $ \pi \circ  \varphi $ also has   dominant coweight $ \lambda $ and is therefore  isomorphic  to  $ V_{\lambda} $.  Since $ \mathrm{End}(V_{\lambda} ) \simeq \QQ $ (\cite[\S  22.4]{MilneAlgebraic}), there is a unique  isomorphism   $$  T_{\varphi} : (\pi,  V_{\lambda} )  \xrightarrow{\sim}  ( \pi \circ \varphi , V_{\lambda}     )     $$     of   $  \hat{\mathbf{G}} $-representations     such that $  T  _  { \varphi  }  $ is identity on  the highest  weight space $  V _{\lambda}   ^{\lambda}  $.  In other words, 
 $ T_{\varphi} : V_{\lambda}  \to  V_{\lambda} $  is  determined by the conditions that    $   T_{\varphi}( gv) = \varphi(g) T_{\varphi} (v)   $    for all $ g \in \hat{\mathbf{G}}(\bar{\QQ}) $, $ v \in V_{\lambda} $ and that  $ T_{\varphi} : V_{\lambda}^{\lambda} \to  V_{\lambda}^{\lambda}  $ is the  identity  map. Let us define $ (g, \varphi ) : V_{\lambda} \to V_{\lambda} $ to be  the mapping $  v \mapsto  g  \cdot     T_{\varphi} ( v )   $ for any $ g \in  \hat{\Gb}(\bar{\QQ} ) $, $ v \in V_{\lambda} $.    If $ \psi : \hat{\Gb} \to \hat{\Gb} $ is another such  automorphism, it is easily  seen by the characterizing property of these maps that 
$  T_{\psi} \circ T_{\varphi} = T_{ \psi \circ \varphi} $, so  that           $$ (h, \psi) \big ( ( g ,  \varphi    ) (v) \big )  =  ( h \psi(g) ,  \psi \circ   \varphi  ) (v)  $$
for all $ h, g \in \hat{\mathbf{G}} $, $ v \in V_{\lambda} $.  
Thus if $ \Xi \subset \mathrm{Aut}( \hat{ \mathbf{G}   }        )$ is a subgroup of  automorphisms preserving $ \hat{\mathbf{P}} $, $  \hat{\mathbf{M}} $ and $ \lambda $,  then the construction just  described  determines an action of $ \mathbf{G} \rtimes  \Xi $ on $ V_{\lambda} $ extending that of $ \hat{\mathbf{G}} $.   
Now suppose that the coweight  $  \lambda $  lies  in $  \Lambda^{+} =  X_{*}(\mathbf{A}) ^{+} \hookrightarrow   X_{*}(\mathbf{M})^{+} $ i.e.,  $ \lambda $ is $ \Gamma $-invariant.    Since the action of $ \Gamma $  on $ \hat{\mathbf{G}} $  preserves  $ \hat{\mathbf{M}} $, $ \hat{\mathbf{P}} $ by definition,  one can extend the action of $ \hat{\mathbf{G}} $ on $ V_{\lambda} $  to  an action of $ {}^{L} \mathbf{G}_{F}   $ on $ V_{\lambda} $ by taking $ \Xi = \Gamma   $  in     the discussion  above. Thus for $ \lambda \in \Lambda^{+} $, $ (\pi , V_{\lambda} ) $ is naturally a representation of $  {}^{L} \mathbf{G} $. 
\begin{remark} Note that the action of $ \Gamma $ on $ V_{\lambda} $ may not be trivial,  even  though  it is required to be so  on the highest weight space.     See \cite{EulerGU22} or \cite{Kim} for an example.    
\end{remark}

Let $ \gamma $ denote the Frobenius element in $ \Gamma $. Recall that the \emph{trace} of a finite dimensional algebraic $ \QQ $-representation $ (\rho, V) $ of $ {}^{L}\hat{\mathbf{G}}_{F}  $ is defined to be the  map $$ \mathrm{tr}_{\rho} : \hat{\mathbf{M}}(\bar{\QQ}) \to \bar{\QQ} \quad \quad (\hat{m}, \gamma)  \mapsto \mathrm{tr}   \big (   \rho(\hat{m}, \gamma) \big ) $$
where $ \hat{m} \in \hat{\mathbf{M}}(\bar{\QQ}) $.  By  \cite[Proposition 6.7]{Borelautomorphic} and its proof,  $\mathrm{tr}_{\rho} $   is naturally an element of $ \bar{\QQ} [\Lambda]^{W} $.  Since the weight spaces $ V^{\mu} $ of $ V $ are defined over $ \QQ $ and $ (1, \gamma) $ acts on these spaces by finite order rational matrices, the trace of $ \rho(1,\gamma) $ on $ V^{\mu} $ is necessarily integral. Hence   the trace of $ \rho(\hat{m},\gamma) = \rho(\hat{m},1) \rho(1,\gamma) $  restricted to  $ V^{\mu} $ is an integral multiple of $  \mu(\hat{m}) $ for any $ \hat{m}  \in  \mathbf{\hat{M}} (\bar{\QQ} )  $. It follows    that     $ \mathrm{tr}_{\rho} $ belongs to the sub-algebra $  \ZZ [ \Lambda] ^{W} $ of $ \bar{\QQ} [ \Lambda ] ^{W}  $.    In particular, 
 the trace of $ \bigwedge^{i} V_{\lambda}  $ for  any  
   $ \lambda \in \Lambda^{+} $ lies in $ \ZZ[\Lambda]^{W} $ for all $ i $. Cf.\
\cite[Lemma 3.1]{PilloniFakharuddin}.       

\begin{definition}   \label{Satakepolydefi}  Let $ \lambda   \in  \Lambda ^ { + }  $. The  \emph{Satake polynomial} $ \mathfrak{S}_{\lambda}(X)  \in \ZZ  [  \Lambda ]^{W }  [ X   ]    $ is defined to be the  reverse  characteristic polynomial of $ \hat{\mathbf{M}}   \rtimes   \gamma     $ acting on $ V _ { \lambda   }    $.     For $ s  \in \frac{1}{2} \ZZ $, the  \emph{Hecke   polynomial $  \mathfrak{H}_{\lambda, s}(X) \in \mathcal{H}_{\mathcal{R}}(K \backslash G / K )  [X]  $  centered at $ s  $}  is defined to be the unique polynomial  that satisfies   $   \mathscr{S} ( \mathfrak{H}_{ \lambda, s} ( X) )  =  \mathfrak{S}_{\lambda}( q ^ { -  s  }    X)  \in    \mathcal{R} [ \Lambda  ]  ^ { W } $. 
\end{definition}      
In other  words, $ \mathfrak{S}_{\lambda}(X) \in \ZZ [ \Lambda ]  ^ { W } [ X  ]  $ is  the   polynomial of degree $ d = \dim   _ {  \QQ  }  V _ { \lambda  }     $ such that  the coefficient of $ X^{k} $  in   $ \mathfrak{S}_{\lambda}(X)  $   is $ (-1)^{k} $ times  the trace of $   \hat {  \mathbf{M} }     \rtimes  \gamma $  on $  \bigwedge^{ k }   V (\lambda )   $ and   $   \mathfrak{H}_{\lambda, s} $ is the polynomial such that the    Satake  transform of the   coefficient of $ X^{k}$ in $ \mathfrak{H}_{\lambda, s }(X) $ is $ q^{- k s  } $ times the coefficient of $ X^{k} $   in $  \mathfrak{S}_{  \lambda } (  X    )   $.

\begin{remark}   \label{basechangeremark}      The  coweight   we   are interested  in for a given Shimura variety for a   reductive    group $ \Gb $ over $ \QQ $ arise  out of the natural cocharacter $ \mu_{h} : \GG_{m, \bar{\QQ}} \to \Gb_{\bar{\QQ}}  $ associated with the Shimura datum for $ \Gb  $.    The $ \Gb(\CC) $-conjugacy class of this  cocharacter is defined over a number field $ E $, known as the reflex  field of the datum.    At a rational prime $ \ell $ where the group $ \Gb $ is unramified, choose a prime $ v $ of $ E $ above it. Then $ E_{v} / \QQ_{\ell} $ is unramified and the  orbit of $ \mu_{h} $ under the (absolute) Weyl group of $ \Gb_{E_{v}} $ is stable under the action of the unramified Galois group of $ \Gamma_{E_{v}} $ of  $ E_{v} $. By \cite[Lemma 1.3.1]{Kottwitz}, we can 
pick a unique dominant cocharacter $ \lambda $ (with respect to a Borel defined over $ E_{v}$) of the maximal split torus in $ \Gb_{E_{v}} $ whose (relative) Weyl group orbit is identified with  the $ \Gamma_{E_{v}} $-stable absolute Weyl group orbit of cocharacters $ \mu_{h} $. This $ \lambda $ is the coweight  whose  associated   representation  we  are  interested  in.  
In the situation above,   $ F $ is intended to  be    $ E_{v} $.    

If $ E_{v} \neq \QQ_{\ell} $, the Satake polynomial  corresponds to a polynomial over the   Hecke algebra of $ \Gb(E_{v}) $ whereas the Hecke operators that act on the cohomology of Shimura variety need to be in the Hecke algebra of $ \Gb(\QQ_{\ell}) $. This is remedied by considering traces of  $  (\hat{\mathbf{M}}  \rtimes \gamma )^{[E_{v}: \QQ_{\ell}]} $ instead. This makes  sure that the traces on $ \bigwedge ^{k} V_{\lambda} $ belong to $ \ZZ [ \Lambda_{\QQ_{\ell}}  ] ^{W_{\QQ_{\ell}}} $ where $ \Lambda_{\QQ_{\ell}} $, $ W_{\QQ_{\ell} } $ are defined relatively for $ \Gb $ over  $\QQ_{\ell} $.      The exponentiation by $   [E_{v} : \QQ_{\ell}] $  here  can then 
be interpreted as a  \emph{base change} morphism from Hecke algebra of $ \Gb_{E_{v}} $ to the Hecke algebra  of $ \Gb_{\QQ_{\ell} } $.  
The Hecke polynomial of \S \ref{GU4Lfactorsection} is obtained in this manner.
\end{remark}

\subsection{Minuscule  coweights}   \label{minusculesec}    
The  representations of $ {}^{L} \mathbf{G}_{F }$  that will  be interested   in    will be   associated to  certain dominant cocharacters that arise out of a Shimura  data.  Such cocharacters  satisfy the special condition of being `minuscule'. In this  subsection, we recall this notion   and  record some  results scattered over several exercises of \cite[Ch.\ VI \S1-2]{Bourbaki}. 
The  reader may consult 
\cite[Chapter VI \S 1 n$^\circ $ 6-9]{Bourbaki} and  
\cite[Ch.\ VIII \S 7 n$^\circ 3$]{BoubakiII} 
for general reference of the material provided here. Cf.\  \cite[\S  2.3]{Kottwitz}.     

It will also  be  convenient to record our results in terms of  abstract  root data. Fix $ \Psi $ an abstract  root datum $  (X, \Phi, X^{\vee}, \Phi^{\vee} ) $  and retain the notations introduced in \S \ref{orderings}  before Lemma  \ref{dominateWeylorbit}.  We  assume throughout that $ \Phi   $  is  reduced.    

\begin{definition} Let $ \lambda $ be an element in  $ X^{\vee} $ or $ P^{\vee} $.    We say that   $  \lambda   $    is \emph{minuscule} if $ \langle \lambda , \alpha    \rangle  \in  \left \{ 1 , 0 , - 1  \right \}   $ for all $ \alpha \in  \Phi  $.

A subset $ S $ of $ X^{\vee} $ or  $ P^{\vee} $   is said to be \emph{saturated} or $ \Phi $-\emph{saturated} if for all $  x  \in S $, $ \alpha \in \Phi $ and integers $ i $  lying between $ 0 $ and $ \langle x , \alpha  \rangle $, we  have     $  x - i \alpha   ^  { \vee    }     \in  S   $.    For $ \lambda $ in  $ X^{\vee} $ (resp., $  P  ^  {  \vee   }     $), we define $ S(\lambda) $ to be the smallest saturated subset of $ X^{\vee} $ (resp.,  $ P ^ { \vee} ) $ containing $ \lambda $ i.e.,  $ S ( \lambda )$ is the intersection of all saturated subsets 
in $ X^{\vee} $  (resp.,  $P ^{\vee} $) that contain   $ \lambda $.  
\end{definition}   
Given $ \lambda \in X^{\vee} $, we will  denote  its reduction modulo $ X_{0} $ in $ P^{\vee} $ by $ \bar{\lambda} $. Similarly given a set $ S \subset X^{\vee} $, we denote the set of reductions of its  elements  by  $ \bar{S} $.   It  is then easy to see that $ \lambda \in X^{\vee}  $ is minuscule  iff $ \bar{\lambda} $ and $ S  \subset X $ is saturated only if $ \bar{S} $ is.  Moreover if  $   \lambda \in  X^{\vee} $,  the reduction of  $ S(\lambda) $ equals $ S( \bar{\lambda}) $.        
If a subset $ S $ of $ X^{\vee }$ or $ P ^{\vee} $ is saturated, then $ s_{\alpha} (x ) = x - \langle x , \alpha \rangle  \alpha ^{\vee} $ belongs to $ S  $ for  all  $ x \in S $, $ \alpha \in \Phi  $.  Thus any saturated set is $  W_{\Psi}  $-stable.  In particular, the orbit $ W _ { \Psi   }   \lambda $ is contained in $ S(\lambda) $  for  any  $ \lambda $ in $ X^{\vee} $ or $ P^{\vee} $.

\begin{proposition}   \label{minusculesaturated}                            A  dominant   $ \lambda $ in $ P^{\vee }$ or $ X^{\vee} $ is minuscule if and only if $ S(\lambda ) =  W_{\Psi}  \lambda  $.    
\end{proposition}    
\begin{proof} 

Let $ \lambda \in X^{\vee} $.  Then $ \lambda $ is minuscule if and only if $ \bar{\lambda} $ is and $ S(\lambda)$ equals $  W_{\Psi}   \lambda $ if and only if $ 
 \overline{S(\lambda) } = S(\bar{\lambda}) $ equals  $ W  _ { \Psi  }   
  \bar{\lambda} $.   It  therefore suffices to establish the claim  for $ \lambda \in (P^{\vee})^{+} $.      Denote  $ V^{\vee} = P^{\vee} \otimes \QQ $, $ V = Q \otimes \QQ   $.    Then $ P^{\vee} \subset V^{\vee} $,  $ Q \subset V $ are dual lattices under $ \langle - , - \rangle $.    
Let $  ( - ,  -   )  :   V^{\vee}   \times V ^{\vee}    \to    \RR  $  
be a $ W _ { \Psi   }   $-invariant pairing. Then $ V $ is identified with $ V^{\vee} $,  $  \langle - ,  -   \rangle $ with $ ( - , -) $, $ Q  $ with 
 $    P^{\vee}  $    and $ \alpha \in \Phi_{\bar{F}} $ with $ 2  \alpha ^ { \vee}  /  (   \alpha ^ { \vee}  ,   \alpha ^ { \vee}  )   $.  In  particular,    $$  (  \lambda  ,  \alpha^{\vee} ) =    \frac{   \langle    \lambda, \alpha  \rangle }      {      2 }  \cdot  ( \alpha^{\vee} , \alpha^{\vee}   )  . $$ Note that $ \langle      \lambda  ,  \alpha \rangle $ and therefore  $ (\lambda , \alpha^{\vee} ) $ are non-negative for $  \alpha \in  \Phi    $ as  $ \lambda $ is  dominant. \\

\noindent  ($\impliedby$)  Suppose $ S(\lambda) = W  _ {  \Psi    }        \lambda $ and  suppose  moreover  for   the    sake of contradiction that $ \lambda $ is not  minuscule. Then there exists  $ \alpha \in   \Phi ^ { +   }    
  $ such that $  k : =  \langle  \lambda , \alpha  \rangle > 1 $.   Then $ (  \lambda  , \alpha^{\vee} ) =  \frac{k}{2}     (\alpha^{\vee}, \alpha^{\vee} )  $. Set $  \mu :  = \lambda - \alpha^{\vee}   \in   P ^ { \vee }  $. Then $ \mu \in S(\lambda) $ by definition. Now  $$ (\mu , \mu )  = (\lambda   , \lambda ) -  k ( \alpha  ^ { \vee } ,  \alpha^{\vee}  )  +  ( \alpha^{\vee} , \alpha^{\vee})  < (\lambda, \lambda ) . $$
Since elements of $ W_{\Psi   }    \lambda $ must have the same length with respect to $ ( - , - ) $, $ \mu  \notin W_{ \Psi   }    \lambda = S(\lambda) $, a contradiction.  Therefore $ k \in \left \{0, 1 \right \} $ and  we deduce that  $ \lambda $ is minuscule.    \\ 

\noindent ($\implies $)  Suppose  that  $ \lambda $ is minuscule.  For all $ w \in    W _ { \Psi   }  
   $, $      \langle   w \lambda ,  \alpha   \rangle  = \langle  \lambda , w ^{-1}  \alpha  \rangle   \in  \left \{  1,  0 , - 1 \right \}$ which implies that  $ w \lambda -  i  
   \alpha   ^  { \vee   }     \in  \left \{ w \lambda ,  s_{\alpha} ( w \lambda  ) \right \}  $ for integers $ i $ lying between $ 0$ and $ \langle  w \lambda  ,  \alpha   \rangle     $.   Thus $ W _ { \Psi  }  \lambda $ is saturated and  therefore $ W   _ {  \Psi   }     \lambda  =  S ( \lambda  )  $.         
\end{proof}

\begin{corollary}      \label{saturatedcontminuscule}                       Every  non-empty  saturated subset of the coweight lattice  contains a minuscule element. 
\end{corollary}
\begin{proof}  Retain the notations in the proof of  Proposition  \ref{minusculesaturated}. Let $ S \subset  P^{\vee} $ be a saturated subset. Let $ \lambda \in S $ be the shortest element i.e.,  $ \| \lambda \| := (\lambda, \lambda)^{\frac{1}{2} } $ is minimal possible for $ \lambda \in   S   $. We claim that $ \lambda  $ is minuscule. Suppose on the contrary that there exist $ \alpha \in \Phi   $ such that $ \langle \lambda, \alpha \rangle \notin \left \{ 1,0,-1 \right \} $. Replacing $ \alpha $ with $ - \alpha $ if necessary, we may assume that $ \langle \lambda , \alpha \rangle > 1 $. Then $ \lambda - \alpha   ^ {  \vee  }        \in   S    $ by definition and the  length calculation in the proof of \ref{minusculesaturated} shows that $ \lambda - \alpha $ is a shorter element.   
\end{proof} 
Under additional  assumptions, one  can describe the  minuscule   elements of    $ X^  { \vee  }  $    more  explicitly. Let $ \Delta   = \left \{  \alpha_{1} , \ldots, \alpha_{n} \right \} $ and let $ \bar{\omega} _{1}, \ldots,  \bar{\omega}_   {n} \in P  ^ { \vee   }        $ denote the basis dual to the basis $ \Delta    $ of $ Q $.  The   elements   $     \bar{\omega}_{i} $  are referred to as   the \emph{fundamental coweights} of $ \Phi    $.    If $ \Phi $ is irreducible,  there exists a  highest root (\cite[Ch. VI \S1  n$^\circ$8]{Bourbaki}) $$ \tilde{\alpha} = \sum_{j=1}^{n} m_{\alpha_{j}}  \alpha_{j}   \in   \Phi ^ { + }    $$
where $ m_{\alpha_{j}} \geq  1   $ are integers. Let $ J \subset  \left \{1, \ldots, n \right \} $ be the subset of indices $ j  $ such that $ m_{\alpha_{j}} = 1 $.  
\begin{lemma}   \label{minusfund}  For  irreducible  $ \Phi $,   $ \left \{ \bar{\omega}_{j}  \right\}_{j\in J} $ is the set of all  non-zero minuscule elements in $ (P^{\vee})^{+} $.  These elements form a system of representatives  for non-zero classes in    $  P^{\vee}  /  Q^{\vee}   $. 
\end{lemma}    
\begin{proof}  Let $ \lambda \in  ( P   ^  {  \vee  } )   ^ {  +    }   $ be non-zero.  Since    $  \bar{\omega} _{1}, \ldots, \bar{\omega}_{n} $ is a basis of $ P^{\vee} $, we can write $  \lambda =  a_{1} \bar{\omega}_{1}  + \ldots  +   a_{n}  \bar{\omega}_{n} $ uniquely. Since $ \lambda $ is dominant and non-zero, we have   $ a_{1}, \ldots, a_{n} \geq 0 $ and at least one of  these  is  positive, say $ a_{k} $. Now  $ \lambda $ is minuscule only if $ a_{1} m_{\alpha_{1}   } +  \ldots  + a_{n} m _  {  \alpha_{n} } =   \langle \lambda , \tilde{\alpha}  \rangle = 1  $ as both $ a_{k} , m_{\alpha_{k}}  \geq 1  $. But  this   can  only   occur  if   $ a_{k} = 1 $, $ k \in J $ and $ a_{i} = 0 $ for $ i \neq j  $.   Thus minuscule elements of $  P ^{\vee}     - \left \{ 0 \right \} $ are contained in the  set  $ \left \{ \bar{\omega}_{j}  \right \}_{j \in J  }  $.  Since $ \tilde{\alpha} $ is highest, any root $ \sum_{j=1}^{n} p_{\alpha_{j}} \alpha_{j} \in \Phi_{F} $ satisfies $ m_{\alpha_{j}} \geq p_{\alpha_{j}} $ and one easily sees that   all  $ \bar{\omega}_{j} $ for $ j \in J $ are    minuscule.  The second  claim  follows   by     Corollary of Proposition 6 in \cite[Ch.\ VI \S 2 n$^\circ$3]{Bourbaki}   
\end{proof}

For  $  \lambda \in  X^{\vee }    $, set  $ \Sigma ( \lambda )  : =  \left \{ \mu \in  X^{\vee} \, |  \lambda \succeq  w  \mu   \text{ for all  } w \in W _ { \Psi    }     \right \} $. Similarly define $ \Sigma(\lambda) \subset P^{\vee}  $ for  $ \lambda \in (P^{\vee})^{+} $. Then $  \lambda  \in  \Sigma(\lambda)  $  by   Lemma   \ref{dominateWeylorbit} and $ \Sigma(\lambda) $ is  easily seen to be saturated.  Therefore       $ W_{\Psi} \lambda  \subset S(\lambda )  \subset  \Sigma(\lambda)   $.

\begin{corollary}   \label{minuscriteria}            A dominant   $ \lambda $ in $ X^{\vee}  $ or $ P^{\vee} $ is minuscule  if  $ \Sigma(\lambda ) =  W_{\Psi}  \lambda   $. The converse holds if $ \Phi $ is irreducible. 
\end{corollary}  
\begin{proof}  It is clear that $ \Sigma(\lambda) = W_{\Psi} \lambda  $   is  equivalent to $ \Sigma( \bar{\lambda} ) = W_{\Psi}   \bar{\lambda}   $ so it suffices to prove these claims for $ \lambda \in (  P  ^ { \vee} )  ^ { +}  $.  \\ 

\noindent   (${\impliedby}$) Suppose $ \Sigma (\lambda)  =   W_{\Psi}    \lambda $.     As, $ W  _ {  \Psi }    \lambda \subset S(\lambda)  \subset  \Sigma ( \lambda ) $, the equality $ \Sigma(\lambda)  =  W_{  \Psi   }       \lambda $ implies that $ S(\lambda)   =    W_{ \Psi } \lambda  $ which by Proposition \ref{minusculesaturated} implies that $ \lambda $ is minuscule. \\

\noindent (${\implies}$)  Suppose $ \Phi   $ is irreducible and $ \lambda \in  ( P^{\vee} ) ^ { + } $   is   minuscule. Then $ \lambda \in Q^{\vee} $ implies $ \lambda $ is zero, since  the  only non-zero  dominant minuscule  elements  in   $ P^{\vee} $ are those fundamental coweights which by Lemma \ref{saturatedcontminuscule} form representatives of non-zero  elements in $ P^{\vee}/Q^{\vee}  $. So is suffice to prove the claim for $ \lambda \notin Q^{\vee} $.   
Suppose now on the contrary  that  there exists a  $  \mu \in  \Sigma(\lambda) - W_{\Psi} \lambda   $.   
We may assume $ \mu $ is dominant since $ \Sigma(\lambda) - W_{\Psi}     \lambda    $ is   stable under $ W_{\Psi}  $ and $ W_{\Psi} \mu $ contains a dominant element.    Since $ \Sigma ( \lambda ) $ is saturated and contains $ \mu $,  $ \Sigma ( \lambda ) \supset  S(\mu ) $.  By Corollary   \ref{saturatedcontminuscule}    $ S(\mu) $ contains a minuscule element $ \lambda_{1} $. Since all elements of $ W_{ \Psi 
  }  \lambda_{1} $ are minuscule and $ S(\mu) $ is $ W_{  \Psi    }    $-stable, we may  take $ \lambda_{1} $ to   be  dominant. Since $ S(\mu) \subset \Sigma(\lambda) $, $ \lambda \succeq_{\flat}     \lambda _{1} $. In particular, $ \lambda - \lambda_{1} \in Q  ^ { \vee  }   $. Since $ \lambda \notin Q ^ { \vee }$, $ \lambda $ and $ \lambda_{1} $ are   distinct    non-zero dominant coweights that  represent  the  same non-zero  class in $ P ^ { \vee } / Q ^ { \vee } $. But this 
  contradicts the second part of Lemma \ref{minusfund}. Hence  $ \Sigma(\lambda) $ must equal $ W_{\Psi}   \lambda $.  The final claim is immediate.   
\end{proof}   
Now resume the notations of \S \ref{RepsofLanglandsdual}. Fix $ \lambda \in  \Lambda^{+} $ and let $  V_{\lambda}   $ be the irreducible representation of $  {}^{L}\hat{\Gb} $ of highest weight $ \lambda  $.    For each $ \mu \in X_{*}(\mathbf{M})^{+} $ with $ \mu \preceq  \lambda   $,     the dimension (as a vector space over $ \QQ $) of the coweight space $ V_{\lambda}^{\mu}   $ is called the \emph{multiplicity} of $ \mu $ in $ V_{  \lambda} $. Corollary \ref{minuscriteria} implies that when $ \lambda $ is minuscule and $ \Phi _{  \bar{F}} $ is irreducible, the set of  coweights in $ V_{\lambda} $  is  just the Weyl orbit $ W_{M} \lambda $. Since $ W_{M}  $ permutes the weights spaces, the  multiplicities of all coweights  are $  1 $. If $ \Gb $ is split, then the action of $ \Gamma $ on $ \hat{\Gb} $ is trivial and so is its  action on the coweight spaces of $ V_{\lambda} $. We therefore get the following  result.     
\begin{corollary}  Suppose $ \Gb $ is split and $ \Phi_{ F } = \Phi_{\bar{F} } $ is  irreducible.  Then for all  minuscule  $ \lambda  \in \Lambda^{+}  $,    $     \mathfrak{S}_{\lambda}(X)   =   \prod_{ \mu \in W \lambda } ( 1 - e^{\mu} X ) \in \mathcal{R}[\Lambda]^{W} $.  
\end{corollary}

\begin{remark}  The content of this  subsection is developed in Exercises 23-24   of \S1 and Exercise 5 of \S2 in \cite[Ch.\ VI]{Bourbaki}.  While the results are well-known, the  version  we need and their 
written proofs  seem    harder to  find.      We have included proofs here  for  future  reference.      
\end{remark} 

\subsection{Kazhdan-Lusztig theory}

\label{Kazhdanlusztigsection}  We  finish this section  by  recording  an  important property of the coefficients of Satake transform when taken  modulo   $ q - 1 $. We assume for all of this  section that  $ \mathbf{G} $ is split   and  $ \Phi_{\bar{F} }  = \Phi_{F} $ is  irreducible.   We refer the reader to \cite[\S 7.9]{coxeterHumphreys}, \cite[\S 7]{KHPIwahori} and   \cite{Kato} for the material  presented  here.  
See  also  \cite{Knop} for a   generalization to non-split case.

The \emph{Hecke algebra}  $ \mathcal{H}_{\mathcal{R}}(W_{I}) $ of $ W_{I} $  is the unital associative   $ \mathcal{R} $-algebra with $ \mathcal{R}$-basis $  \left \{  T_{w }   \right \}_{ w \in W _{I}  }  $ subject to the relations  
\begin{alignat*}{4} & T_{s}^{2}   =  (q-1) T_{s}  +  q T_{e}  &\quad   \quad   &  \text{ for $ s \in S_{\mathrm{aff}} $} \\ 
&T_{w} T_{w'}   = T_{ww'} &  \quad   \quad    & \text{ if $ \ell(w) + \ell(w') = \ell(ww')  $} 
\end{alignat*} 
Each element $ T_{w} $ possesses an inverse in  $ \mathcal{H}_{\mathcal{R}} ( W_ {I} ) $. Explicitly,  $ T_{s}^{-1} =  q^{-1} T_{s}    - ( 1 - q ^{-1} ) T_{e} $.  The $ \ZZ $-linear map $ \iota :  \mathcal{H}_{\mathcal{R}}(W_{I})    \to   \mathcal{H}_{\mathcal{R}}(W_{I}) $ induced by $ T_{w} \mapsto (T_{w^{-1}})^{-1} $ and $ q^{\frac{1}{2}}  \mapsto  q^ {  -  \frac{1}{2} }  $ induces a ring automorphism of order two known as the  \emph{Kazhdan-Lusztig involution}.      
\begin{definition} For each $ y , w \in W_{I} $ such that $ x \leq w $ in   (strong)  Bruhat ordering, the \emph{Kazhdan-Lusztig polynomial} $ P_{x,w}  (q)  \in  \ZZ[q] $ (considering $ q $  as   an   indeterminate) are uniquely characterized by the following  three  properties:   
\begin{itemize}  \setstretch{1}   
\item  $ \iota \big ( q ^ { - \ell ( w ) / 2 }   \sum _ { x   \leq w } P_{x,w}(q) T_{x} \big ) =   q ^ {  \ell ( w ) / 2 }  \sum  _ { x   \leq  w }  P_{x,w} ( q ) T_{x}  $,  
\item  $ P_{x,w}(q) $ is a polynomial of degree at most $    (  \ell(w) - \ell(x) - 1  ) /  2   $ if $ x  \lneq  w $, 
\item   $ P _  { w ,  w }   ( q )  = 1  $.  
\end{itemize}  If $ x \nleq w $, we extend the definition of these polynomials  by  setting   $ P_{x,w}(q) = 0 $.  We will refer to $ P_{x,w} $ for any $ x,w \in W_{I}$ as KL-polynomials.
\end{definition}

For any $ \lambda \in \Lambda $, there is a unique element denoted $ w_{\lambda} $ which has the longest possible length in the  double coset $ W t(\lambda) W  \subset W_{I} $. When $ \lambda \in \Lambda^{+} $, this element is $ t(\lambda) w_{\circ}  $ and $ \ell(t(\lambda) w_{\circ} ) = \ell ( t(\lambda) + \ell ( w_{\circ} ) =  2 \langle \lambda ,  \delta  \rangle  +  \ell ( w_{\circ}  )  $. For  any $  \lambda,  \mu \in \Lambda^{+} $, we have  $    \lambda  \succeq  \mu  $  (\S \ref{orderings}) iff $ w_{ \lambda   }   \geq  w _ { \mu  }  $.

\begin{theorem}[Kato-Lusztig]   Let $ \lambda \in \Lambda^{+} $ and $ \chi_{ \lambda  }  \in  \ZZ [ \Lambda ] ^ {W }   $  denote  the  trace of $ \hat{\mathbf{M}}$ on $ V_{ \lambda   }    $.  Then $$      \displaystyle{ \chi_{\lambda} =  \sum   \nolimits     _ { \mu   \preceq \lambda } q ^ { - \langle \lambda ,   \delta  \rangle  } P _ { w_{\mu } ,  w_{\lambda} } (q)   \mathscr{ S } ( K  \varpi^{\mu} K  )  }  $$ where the sum runs over $ \mu \in \Lambda^{+} $ with $ \mu  \preceq  \lambda  $.     \label{katolusztigformula}
\end{theorem}   
\begin{proof} See \cite[\S 7]{KHPIwahori}. We  also  note that the proof provided in  \cite{Kato} carries over with minor changes.  
\end{proof}   
\begin{corollary}
\label{Kostantweight}  $   \displaystyle{ \chi_{\lambda} = \sum   \nolimits    _{\mu \preceq \lambda }  P_{w_{\mu} , w_{\lambda} } (1) e ^ { W \lambda } }   $. 
\end{corollary}
\begin{proof} Using Macdonald's formula (Theorem \ref{macdonald})  for the expression $ \mathscr{S}(K \varpi^{\mu} K)  $  in the Kato-Lusztig formula  \ref{katolusztigformula},   we obtain an expression for $ \chi_{\lambda }$ as a linear combination in $ e^{W \mu } $ which has coefficients in $ \mathcal{R}_{q} $ (see \cite[Theorem 1.5]{Kato}). Since the $ \chi_{\lambda} $ is independent of $ q $, we can formally replace $ q $ with $ 1 $ which yields the expression above.   
\end{proof} 
Let   $ \mathcal{I} =  \mathcal{I}_{q}  \subset \mathcal{R} _{q} $ denote the ideal generated by $ q ^{\frac{1}{2} } - 1 $ and let $  \mathcal{S} =  \mathcal{S}_{q} :  = \mathcal{R}/\mathcal{I}  $.    For $ f \in \mathcal{R}[\Lambda] ^{W} $, we let $ [f]  \in  \mathcal{S}[\Lambda]^{W} $ denote the image of $ f $.  Similarly, for $ \xi \in \mathcal{H}_{\mathcal{R}}(K \backslash G / K ) $,  we let  $ [\xi]  \in  \mathcal{H}_{\mathcal{S}} ( K \backslash G / K ) $ denote the class of $ \xi $. For $ f = \sum _ { \mu \in \Lambda^{+}} c_{ \mu   } e^{W   \mu   }  \in  \mathcal{R}[\Lambda]^{W} $, let $$ \xi_{f}  :  = \sum 
  \nolimits    _{ \mu \in \Lambda ^ { + } } c_{   \mu    } (K \varpi^{  \mu    } K  )  \in \mathcal{H}_{\mathcal{R}}(K  \backslash G /  K ) . $$ 
\begin{corollary} Let $ f \in \mathcal{R}[\Lambda]^{W} $ and $ \xi =  \mathscr{S}^{-1}(f) $.  Then $ [\xi] = [\xi_{f} ]  $. \label{modqSatake}    
\end{corollary} 

\begin{proof} Since $ \chi_{\lambda}$ form a $ \ZZ $-basis for $ \ZZ [ \Lambda]^{W} $, it suffices to establish the claim for $ f =  \chi_{\lambda} $.  But this follows by Kato-Lusztig formula and Corollary  \ref{Kostantweight}.    
\end{proof}

%% file: Decompositions.tex
\section{Decompositions  of  
double  cosets}   
\label{titsdecosec}    
In this section, we   derive using the elementary theory of Tits systems a  recipe   for  decomposing  certain double cosets into their  constituent  left cosets. Invoking the existence of a such a system on the universal covering of the derived group of a 
   reductive group over a local field, we   obtain a recipe   for     decomposing Hecke operators  arising out of double cosets of what are  known 
 as   \emph{parahoric subgroups}  
  of unramified reductive groups.    
The method used here for decomposing such double cosets is  based on the one  introduced in \cite{Lansky} in the setting of split Chevalley groups. 
Theorem   \ref{BNrecipe}, the main result of this section,  will be our primary tool for executing the machinery of \S   \ref{abstractsec}  in concrete  situations.
\subsection{Motivation}             
\label{titsmotivation}    
To motivate what kind of decomposition we are looking for,  let us take a look at the case    of  decomposing $ K \sigma K  $ where $ K = \GL_{n}(\ZZ_{v}) $ for $ v $ a rational prime and $ \sigma = \mathrm{diag}(v,\ldots, v,  1, \ldots, 1 ) $ where there are $ k $ number of $ 1 $'s. Let $ G $ denote $ \GL_{n}(\QQ_{v}) $.    There is a natural $ G $-equivariant bijection between $ G / K $ and the set of $ \ZZ_{v} $-lattices in $ \QQ_{v}^{n} $ where $ K $ is mapped to the standard lattice. Then $ \sigma K $  corresponds to the  lattice generated by the basis where the first $   n -  k $ standard vectors are replaced by multiples of the uniformizer $ v $. Thus $ K \sigma K 
/ K    $ corresponds to the $ K $-orbit of this lattice. It is clear that any such lattice lies between the standard lattice $  \ZZ_{v}     ^{n} $ and $ v \ZZ_{v} ^{n} $.   Reducing modulo $ v $ therefore gives a bijection between  $  K  \sigma K /     K  $ and  the   $ \mathbb{F}_{v} $-points of the Grassmannian $ \mathrm{Gr}(k,n) $ of $ k $-dimensional subspaces in an $ n $-dimensional vector space.     Since $ \mathrm{Gr}(k,n)(\mathbb{F}_{v}) $ admits a stratification by Schubert cells, one obtains an explicit description of $ K  \sigma  K  /   K 
  $  by  taking $  \ZZ_{v}   $ lifts of their $ \mathbb{F}_{v} $ points.  See Example  \ref{GL4Schubert} that illustrates   this for $ n = 4   $,   $ k =  2   $.          We would like a similar recipe for more general reductive groups and arbitrary cocharacters.

\subsection{Coxeter systems}    \label{Coxetersection} Throughout this subsection,  $ (W,S) $  denotes   a Coxeter system. Given $ X  \subset S $,  we  let $ W_{X} \subset W $ be the group generated by $ X $. Then $ (W_{X}, X) $ is a  Coxeter system itself  and  $ W_{X} \cap S = X $. We refer to groups  obtained  in this manner as  \emph{standard parabolic subgroups}  of $ (W , S ) $.   Let $ \ell : W \to \ZZ $ denote the length function. Then $ \ell_{\mid W_{X} } $ is the length function on $ W_{X} $. Given $ X  ,  Y \subset W $ and $ a \in W $, consider  an element $ w \in  W_{X}  a  W_{Y} $ of minimal  possible length. The deletion condition for Coxeter groups implies that  any $ w ' \in W _ { X }  a  W_{Y     }     $ can be written as $ w' = x w y $ for some  $ x \in W_{X} $, $ y \in  W_{Y}  $  such  that  $$  \ell ( w'  ) = \ell ( x ) +  \ell (  w )  +  \ell  ( y )  .   $$  It follows that $ w \in W_{X} a W_{Y} $ is the unique element of minimal possible  length.    We refer to $ w $ as the $ 
(X,Y) $-\emph{reduced element}   of    $ W_{X} a W_{Y}  $ and denote the set  of   $ (X,Y)$-reduced elements in  $ W $ by $ [ W_{X}  \backslash W /  W_{Y} ] $.   If $ w \in W $ is $ (X,   \varnothing ) $-reduced, then we have the stronger property that   $  \ell ( x w ) = \ell ( x )  + \ell(w) $ for \emph{all}  elements $ x \in  W_{X} $. An arbitrary $ \sigma   \in  W $ can be written  uniquely    as $ \sigma   =  xw  $ for some  $ x \in W _{X} $ and $ w \in W $ a $ (X, 
  \varnothing)  $-reduced   element.      Similarly for $ (\varnothing , Y ) $-reduced  elements. An element in $ W  $  is $ (X,Y) $-reduced iff  it is $ (X,   \varnothing     )$-reduced and $ (  \varnothing,   Y) $-reduced.

The   stronger properties of minimal length  representatives for   one-sided cosets of parabolic subgroups can be generalized to double cosets as follows.   Let  $ \sigma \in W $ be  $ (X, Y ) $-reduced. Then   $ W_{X} \cap  \sigma W_{Y} \sigma ^{-1} $ is a standard parabolic subgroup of $ (W_{X}, X)  $ generated by $  Z:= X \cap (W_{X} \cap \sigma W_{Y} \sigma^{-1}) $  and 
\begin{equation}   \label{lanslengthformula}     \ell ( \tau \sigma \upsilon )  =  \ell ( \tau \sigma) + \ell (  \upsilon ) =    \ell ( \tau ) + \ell ( \sigma ) + \ell (  \upsilon )  
\end{equation} 
for  any $ \tau \in [ W_{X}  /  W_ { X} \cap  \sigma W_{Y }  \sigma  ^ { - 1}  ] $, $ \upsilon   \in    W_{Y} $. In other words,   the equality above holds for any $ (X,Y) $-reduced element $ \sigma \in W $,  any $ (\varnothing, Z) $-reduced element $ \tau \in W_{X} $     and arbitrary $  \upsilon  \in W_{Y} $.         

There is a generalization of these facts to a  slightly  larger class of groups. Let $ \Omega $ be a group and $ \Omega \times W \to W $ be  a left action that restricts to an action on $ \Omega \times S  \to  S   $.     We refer to elements of $ \Omega $ as automorphisms of the system $ (W,S) $. Since such automorphisms are length preserving, we may form the extension  $ \tilde{W}  :  = W  \rtimes \Omega  $ and extend the length function $ \ell : \tilde{W} \to  \ZZ  $ by declaring $ \ell( \sigma \rho ) = \ell ( \sigma ) $ for $ \sigma \in W $, $ \rho \in \Omega $. We refer to elements of $ \Omega  \subset  \tilde{W} $ as \emph{length zero  elements}.    Given $ A \subset W $, we denote by $   A  ^  { \rho  } $ the set $  \rho A \rho ^{-1} \subset  W$. Then $ \rho W_{X} \rho ^{-1} = W _ { X ^{ \rho} } \subset    W     $ for   any     $ X \subset S $.    Given $ X , Y \subset S $, $  b = a  \rho  \in  \tilde{W} $ where $  a \in  W   $, $ \rho \in \Omega $, there is again a unique element $  w  \in  W_{X}  b    W_{Y} $ of minimal possible length given by $  w =  \sigma \rho $ where $ \sigma $ is the $ (X, Y  ^ {
\rho  } )$-reduced element in $ W _{X} a W_{  Y ^ { \rho  }   } $.   Moreover  $ W_{X} \cap   w W_{Y} w^{-1} =  W_{X} \cap \sigma (W_{  Y ^ { \rho } }  )    \sigma^{-1} $ is   still   a standard  parabolic subgroup  of $ W_{X} $ with respect to $ X $ and  the length  formula (\ref{lanslengthformula})    continues to hold  when $ \sigma $ is replaced with $  w  = \sigma \rho $. 
We continue to call  the unique  element $ \sigma \rho $ as the $ (X,Y) $-reduced element of $  W_{X} b W_{Y}  $ and  denote the  collection obtained over all  double  cosets  by $ [ W_{X} \backslash \tilde{W} / W_{Y} ] $. If  $ w \in \tilde{W} $ is $ (X, \varnothing) $-reduced, we  again  have $ \ell(xw) = \ell(x) +     \ell(w) $ for all $ x \in W_{X} $.

\begin{remark} The  result on $ (X,Y)$-reduced  elements in the first paragraph above   appear in  \cite[Ch.\ 4, \S1 Ex. 3]{Bourbaki} from which we have also borrowed its terminology.  See   also      \cite[\S 1.10, \S 5.12]{coxeterHumphreys}.  
Detailed proofs of all claims in the second paragraph can  be found in \cite[Proposition 1]{Hombergh} or \cite[\S 4]{Lansky}. 
Groups $ \tilde{W} $ as above  are sometimes called quasi-Coxeter groups.    
\end{remark}

\subsection{Tits Systems}   \label{Titssection} 
     
\begin{definition} A \emph{Tits system}   $   \mathcal{T}   $     is  a quadruple $ (G, B , N,S) $ where $ G $ is a group, $  B  ,  N $ are  two  subgroups  of  $  G  $     and $ S $ is a subset of $  N / (  B     \cap N ) $ such that the following  conditions  are  satisfied:       
\begin{itemize}     [before = \vspace{\smallskipamount}, after =  \vspace{\smallskipamount}]    \setlength\itemsep{0.3em}       
\item [(T1)] $  B \cup N $ generates  $ G $ and $ T =   B    \cap N $ is a normal subgroup of $ N $ 
\item [(T2)] $ S $ generates  the  group  $   W =     N /   T   $ and consists of elements of order $ 2 $ 
 
\item [(T3)] $ s B w  \subset   B   w   B   \cup  B  s w   B      $ for all $ s \in S $, $ w  \in W  $.
\item  [(T4)] $  s  B   s   \neq   B    $ for all $ s \in S $
\end{itemize}    
We call $ W $ the  \emph{Weyl group} of the system and let $ \nu : N \to W $ denote the natural map.  
\end{definition}

\begin{remark}   For any $ v, w  \in W $, the products $ w B $, $ Bvw , vBw $ etc   are  well-defined since if, say, $ n_{w} \in N $ is a representative of $ w $, then any other is given by $ n_{w} t $ for $ t \in  T  \subset  B $ and one has   $ n_{w} t = t ' n_{w} $ for some $ t' \in  T    $ by normality of $ T $ in $ B \cap N $. 
\end{remark}    

For any such system, the pair $(W,S) $ forms a  Coxeter  system. We denote by $ \ell : W \to \ZZ $ the  corresponding length function.  The set $ S $ equals the set of non-trivial elements $ w \in W  $ such that $ B \cup  B  w  B   
$ is a group. Hence $ S $ is  uniquely determined by the groups $G, B , N $ and the axioms (T1)-(T4).  We therefore  also say that $ (G,B,N) $ is a Tits system or that $ (B,N) $ constitutes a Tits system for  $ G  $.     The axiom (T3) is equivalent to $$ BsBwB \subset  BwB \cup BswB.  $$
Since   $ BsBwB $ is a union of double cosets, it must equal either $ BswB $ or $BwB  \cup Bs wB   $  and the  two  cases  correspond to whether $ \ell(sw) $ equals $\ell(w) + 1 $ or $ \ell(w)-1 $.   
In particular, $ Bs Bs B $  equals  $ B \cup Bs B $ by (T4). 
The subsets $ BwB \subset G  $ for $ w \in W $ are called \emph{Bruhat cells} which provide a decomposition   
\begin{equation}   \label{BTdecomposition} G = \bigsqcup     
_{w \in W } Bw B  
\end{equation}    
called  the \emph{Bruhat-Tits   decomposition}. If $ w = s_{ 1} \cdots   s_{ \ell(w) } $ is a  reduced decomposition of $ W $, then $ BwB = Bs_{1} B \cdot B s_{2} B \cdots B s_{\ell(w)}  B $. A subgroup of $ G $ that contains $ B $ is called a \emph{standard parabolic}.    There is a bijection between such subgroups of $ G $ and subsets $ X $ of $ S $ given  in one  direction  as follows: given $ X  \subset   S  $, we let $ K_{X} : = B W_{X} B \supset B $  where $ W_{X}  \subset  W    $ is the group generated by $ X $. Then $ K_{X} $ is the standard parabolic subgroup associated with $ X $.    In particular, $ K_{\varnothing} = B $, $ K_{S} = G   $.    Any  standard   parabolic  subgroup  of $ G   $   equals  its  own   normalizer  in  $   G   $.      
If $ N '  $ is a subgroup of $ N $ such that $ \nu(N') = W_{X} $, then $ (K_{X}, B, N', X) $ is a Tits system  itself.   
If $ X, Y \subset  S $, the bijection $ B \backslash G / B \xrightarrow{\sim}  W $ induces a bijection 
\begin{equation}   \label{refinedBT}       K_{X} \backslash G / K_{Y}    \cong W_{X} \backslash W / W_{Y}   
\end{equation}    
given by sending $ K_{X} w K_{Y} \mapsto W_{X} w W_{Y} $.   For $ Z   $  a normal subgroup of $ G $ contained in $ B $,  denote  $ G ' = G/Z $ and let $ B' = B/Z $, $  N' = N /   ( N   \cap    Z )    $ denote the images of $ B , N $ in $ G ' $. Set  $ W ' = B' /  (B ' \cap N ') $ and $ S ' $ the image of $ S $ under $ W \to W ' $. Then $ (W,S) \to (W ', S')  $ is an isomorphism  of Coxeter groups and  $ (G', B', N', S' ) $ is a Tits system which is said to be  \emph{induced}   by $ (G,B,N,S) $.

\begin{definition}  Let $ (G,B,N,S) $ be a Tits system. We say that the system is  \emph{commensurable}   if $  B s B / B $ is finite for all $ s \in S $. Then  $ Bw B/  B   $ is finite for all $ w \in W $.  We let $ q_{w}  $  denote  the  quantity $ |  B w B / B  |   = [ B : B \cap w B w^{-1} ] $.
\end{definition}

\begin{lemma}   [Braid  Relations]     \label{braid}              Let $ (G,B,N,S) $ be   a  commensurable  Tits system.  For any  $  \sigma  ,  \tau   \in W  $ such that  $ \ell( \sigma ) + \ell (  \tau  ) =  \ell ( \sigma   \tau     )  $,   $ q _{ \tau \sigma  } =  q_{ \tau }  q_{\sigma  }   $ 
\end{lemma}

\begin{proof} Since $ B w B / B $ is finite for all $ w \in W $, one may form the    convolution algebra $ \mathcal{H}_{\ZZ} ( B \backslash  G  /  B ) $ with product $ \ch ( B w B )  *   \ch (B v   B) $ given as in  \S  \ref{heckeoperatorsubsection}. The linear map  $ \mathrm{ind} : \mathcal{H}_{\ZZ} ( B \backslash G  /B ) \to \ZZ $ given by $ \ch ( B w B ) \mapsto  q_{w }$ is   then    a homomorphism of rings. If $ s  \in S $, $ w \in W $, we have  $$  \ch ( Bw B) * \ch  ( B s B ) =  \sum_{ u \in W } c_{w, s} ^{u} \ch ( B u B )   $$   where $ c_{w,s}^{u} = | (  B w B  \cap u BsB )/ B  |    $ (see eq.\ (\ref{Heckeformula})). Note that  $ c^{u}_{w,s} \neq   0    $ if and only    if $   BuB \subset  B w B s B  $.   Suppose that $ \ell ( w ) + \ell ( s) =  \ell (w s ) $, so that $  BwBsB = Bws B $. This implies that   $ c_{w,s}^{ u } = 0 $ for $ u \neq ws $ and   that     $ wBsB \subset B ws B $.  Since $ BsBsB = B \cup BsB $, we have   $$ ws BsB \subset w  ( B s B s B  )  = w ( B \cup BsB)  \subset   wB \cup BwsB   . $$        
Using the above inclusion, we see that     $$ w B  \subset    BwB \cap ( ws BsB )  \subset BwB \cap ( wB \cup BwsB ) = wB $$ where the last equality follows  by disjointness of $ BwB $, $ BwsB $. It follows that  $ BwB \cap ws BsB = w B $ and therefore  $  c^{ w  s   } _{ w, s  }  =  1 $. Combining everything together, we see that   $  \ch (  B w B )  *   \ch (   B s B  )    =   \ch ( B ws B  )  $. Repeating this argument by writing $ \sigma = w s  = w' s' s  $, we see that $ \ch ( B \sigma B ) =  \ch ( B s_{1} B )   *  \cdots     *    \ch ( B s_{\ell(w)} B ) $ where $ \sigma   = s_{1}  \cdots  s_{ \ell(w) } $ is a reduced decomposition. Since $ \mathrm{ind} $ is a homomorphism, we see that $  q_{\sigma}  =  q_{s_{1} } \cdots  q_{s_{\ell(\sigma    ) }} $ and similarly for $ q_{\tau } $, $ q_{\tau \sigma } $.   The claim    follows since the product of two reduced expressions for $ \sigma $, $ \tau $ in that order  is  a  reduced 
 word    expression for $  \sigma  \tau   $.          
 \end{proof}   
\begin{definition}   \label{adapted}      Let $ (G,B,N,S) $ be a Tits system  and   $    \varphi : G \to \tilde{G} $ be a homomorphism of groups. Then $ \varphi $ is said to be  \emph{$ (B,N) $-adapted}  if 
\begin{itemize}   \setstretch{1.5}    
\item [(i)]  $  \ker  \varphi  \subset B $, 
\item [(ii)]  for all $ g \in  \tilde{G} $, there is $   h \in  G  $ such that $  g \varphi(B) g ^{-1} =   \varphi ( h B  h ^{-1}   )   $ and $  g \varphi(N) g^{-1} = \varphi ( h N h^{-1} ) $.      
\end{itemize}
For any such map, $ \varphi(G) \triangleleft \tilde{G} $ and the induced map $ G / \ker(\varphi ) \hookrightarrow \tilde{G} $ is adapted with respect to the induced Tits system on $ G / \ker \varphi $. 
\end{definition}   
Let $ \varphi : G  \to  \tilde{G}  $ be a $ (B,N)$-adapted  injection  and consider $ G $ as a (necessarily normal) subgroup of $ \tilde{G} $.   Denote by $ T = B \cap N $, $ W = N / T 
$ as above  and  set  $ \Omega =  \tilde{G} / G $.        Let $ \hat{B} $, $ \hat{N} $ denote respectively  the normalizers of $ B $, $ N $ in $  \tilde{G} $ and  set   $ \Gamma =  \hat{B} \cap  \hat{N} $.  Since    every $ g \in \tilde{G }$ has a  $ h \in G $ such that $   g^{-1} h \in \Gamma $, we  see  that     $ \tilde{G} = \Gamma G    $. If $ g $ is taken to be in $ \hat{B} $, $ h $ is forced to lie in $  B $ as   $ B $ equals its own normalizer in $ G  $.  Therefore    $ \hat{B} = \Gamma   B     $.   That $ N_{G} (B) = B $   also    implies  that $   \Gamma  \cap  G =   \Gamma
 \cap   B   $ from which  it follows   that    $ \Gamma / \Gamma \cap B $ and $ \hat{B}/ B $ are   both    canonically isomorphic to $ \Omega $.

 Define  $ \tilde{N} = N \Gamma $  and  $ \tilde{T} = \tilde{N} \cap B   $.  As $ \Gamma $ normalizes $ N $, $ \tilde{N} = N \Gamma = \Gamma N $ is a group and therefore   so is $   \tilde{T}   
 $.      Invoking $ N_{G}(B) = B $ again, we see that $ \tilde{T} =   N \Gamma \cap B = T (\Gamma \cap B ). $  Since $ \Gamma $ normalizes both $ B $ and $ N$,  it  normalizes the intersections  $ T = B \cap  N$ and $ \Gamma \cap B $. Thus  $ \Gamma $ normalizes  the  product  $ \tilde{T} = T ( \Gamma  \cap  B  )$.  If $ n \in N $, $ b \in \Gamma \cap B $, there exist $ n' \in N $ such that $ bn = n '  b   $. The decomposition  (\ref{BTdecomposition})    implies that $   n' $, $ n $ represent the same class in $ W $, and so $ n 'n^{-1} = b n b ^{-1} n ^{-1} \in T $. This implies that $  n b ^{-1}  n ^{-1} $ lies in $ \tilde{T} $. It  follows from this  that $ N $  also 
 normalizes $ \tilde{T} $.   Consequently,   $ \tilde{T} $ is a normal subgroup  of $  \tilde{N } $. 
We  let $$  \tilde{W}  =  \tilde{N} /  \tilde{T   } . $$      
Since $ N   $   contains   $ T  $,       $ N \cap \tilde{T}  =  N \cap T ( \Gamma \cap B ) = T   ( N \cap \Gamma \cap B ) = T .  $    
 Similarly $ \Gamma \cap \tilde{T} =  \Gamma \cap  B $. 
Thus      
the   inclusion of $ N   $ (resp.\ $ \Gamma $) in $ \tilde{N} $ allows us
to identify $ W $ (resp.  $   \Omega $) as a subgroup of $   \tilde{ W  }    $.   Since  $  \Gamma   $ normalizes $ N $, $ \Omega $ normalizes $ W $.  Since $ \Gamma \cap N = \Gamma \cap B  \subset \tilde{T} $,  $ W \cap \Omega $ is trivial in $ \tilde{W} $. 
It follows that $$ \tilde{W}  =  W  \rtimes  \Omega . $$ Since $ \gamma (B \cap B w B )  \gamma^{-1} = B \cap B \gamma w \gamma^{-1} B $ for any $ w \in W $, $ \gamma \in  \Gamma  $ and   since    $ B \cap B u B $ is a group for $ u $ non-trivial   if and only if $ u \in S $, we see that $ \Omega $ normalizes  $ S  $.   Consequently $ \Omega $ acts  on $ (W,S) $ by automorphisms and we may extend the length function from $ W $ to $ \tilde{W} $. From the decomposition  (\ref{BTdecomposition})    and the normalizing properties of $ \Gamma $, we obtain a \emph{generalized Bruhat-Tits     decomposition}  $$    \tilde{G}  = \bigsqcup   
_{ w \in \tilde{W} }  B  w    B  . $$ 
Similarly, if $  X, Y \subset S $, $ K_{X}, K_{Y} \subset G $ denote the corresponding groups, we obtain from  \ref{refinedBT}  a    decomposition     $ K_{X} \backslash  \tilde{G}  / K_{Y}   \cong  W_{X} \backslash  \tilde { W  } / W_{Y }   $.

\begin{remark} For the general theory of Tits systems, we refer the reader to \cite[Ch.\ 4 \S 2]{Bourbaki}.  
The material on $(B,N)$-adapted morphisms   and commensurable Tits systems is developed in Exercises 2, 8, 22, 23, 24 of \emph{op.\ cit.} and we have  included their proofs here.   
The terminology of Definition \ref{adapted} is taken from \cite[Ch.\ I \S 2.13]{TitsBruhat}. This  notion      is referred to as  \emph{generalized Tits systems} in  \cite{BruhatIwahori}.  
\end{remark}

\subsection{Decompositions}    Assume for all  of this subsection   that  $ (G,B,N,S) $ is a commensurable Tits system and  $ \varphi : G  \hookrightarrow \tilde{G }$ is a $ (B,N)$-adapted  inclusion.    Retain also the notations $ W $, $ \tilde{W} $, $ \Omega $ and $ q_{w} $ for $ w \in W $ introduced above.  For each $ s \in S $, let $ \kay_{s} \subset G  $   denote  a   set   of    representatives of $ B /  ( B \cap s  B s^{-1}   )  $ (so $ |\kay_{s}| = q_{s} $) and  let $  \tilde{s} $ denote a  lift of $ s $ to $ N $ under $ \nu  $    (so that $ \nu( \tilde{s} ) = s $). Define $$ g_{s} :  \kay_{s}  \to   G , \quad  \kay_{s} \ni  \kappa \mapsto  \kappa \tilde{s}  $$
considered as a map of sets.  
Fix a $ w = \sigma \rho  \in \tilde{W} $ where $ \sigma \in W $, $ \rho \in  \Omega $ and let $\tilde{\rho} \in \Gamma $ denote a lift of $ \rho $. Then $ \tilde{\rho} B = B \tilde{\rho }     $ is independent of the choice of the lift and we may  therefore   denote $ \tilde{\rho} B $ 
 simply  as    $ \rho B $.     Let  $ m = m_{w}   : = \ell  ( \sigma ) $ denote the length of $ \sigma$ and let    $ r( \sigma )  = (s_{1},  \ldots , s_{m}) $ denote  a  fixed  reduced word decomposition of $ \sigma  $. Denote by $ \kay_{r(\sigma)  } $ the product $  \kay _ { s_{1} }      \times  \kay_{s_{2} }   \times   \cdots   \times        \kay_{s_{m} }  . $
\begin{lemma}   \label{Bwbdecompose}     $ B w   B    =    \bigsqcup _ { \vec \kappa \in  \kay_{r(\sigma )} }   g_{s_{1} } (  \kappa_{1}  ) \cdots  g_{s_{m} } ( \kappa_{m} )  \rho   B   $
where $ \kappa_{i} $ denotes the $ i $-th  component of $ \vec{\kappa}  $.      
\end{lemma}    
\begin{proof}  We have $  B  w   B = B  \sigma  B  \rho    = B s_{1} B \cdots B s_{m} B   \rho   $. Now 
\begin{align*} B  \sigma  B     &   =   \bigcup _{ \kappa_{1} \in \kay_{1} } g_{s_{1} }( \kappa_{1} ) B s_{2} B \cdots B s_{m} B    \\
& = \bigcup _ {  \substack { ( \kappa_{1}, \kappa_{2}  )   \\
\in    \kay_{1}  \times    \kay_{2}  }  }  g_{s_{1}} ( \kappa_{1} ) g_{s_{2}}( \kappa_{2} ) B s_{3} B  \cdots B s_{m} B   =  \cdots  = \bigcup _ { \vec{\kappa} \in \kay_{r(w)} }  g_{s_{1} }( \kappa_{1} ) \cdots   g_{s_{m} } ( \kappa _{m}   )    B      
\end{align*}   
As $   |  B \sigma   B / B | =   q_{  \sigma  }    = q_{s_{1} } \cdots   q_{s_  {  m   }    }   $  by  Lemma   \ref{braid},    the union above  is  necessarily      disjoint.  Multiplying each coset in the decomposition above on the right by $  \rho   $ and moving it inside next to $ \sigma $ on the left hand side,    we get the desired decomposition of $ B w  B    $.        
\end{proof}

Retain the notations   $ w, \sigma, \rho $, $ m $.  We 
   define  $ \mathcal{X}_{r(\sigma) ,  \rho }   : \kay _ {   r ( \sigma )   } \to \tilde{G}       / B   $ to be the  map $  \vec{\kappa}  \mapsto  g_{s_{1}} ( \kappa_{1} )    \cdots  g_{s_{m} }(  \kappa_{m} ) \tilde{\rho} B $.   
   (where we have suppressed the dependency on the choices of lifts). 
   Then in this notation,  $$ BwB = \bigsqcup_{\vec{\kappa} \in \kay_{r(\sigma)}  }  \mathcal{X}_{r(\sigma) , \rho} (\vec{\kappa}) . $$ 
In particular, the image of $ \mathcal{X}_{r( \sigma )   ,     \rho } $ in $   \tilde{G}      / B $  is independent of all the choices involved. Since we will only be interested in the image  of    $ \mathcal{X}_{r(\sigma),   \rho } $ modulo subgroups of $ G $ containing $ B $,   we will abuse our notation to  denote  this map   simply  as   $\mathcal{X}_{w }  $. Moreover we will consider   $ 
  \mathcal{X}    _{w} $ as taking values in $  \tilde{G}  $ as opposed to $ \tilde{G}/B $,   if it is understood that these are representatives of left cosets for some fixed subgroup that  contains $ B  $.    
Similarly   we     denote $ \kay _{r(\sigma) } $ by $ \kay_{ w  } $.

\begin{theorem}   \label{BNrecipe}          Let $ X  , Y \subset S $, $ W_{X} $, $ W_{Y} $ be the subgroups  of $ W $ generated by $ X $, $ Y$  respectively and let $ K_{X} = B W_{X} B $, $ K_{Y} = B W_{Y} B $. For  any   $  w   \in [ W_{X} \backslash \tilde{W}     /  W_{Y } ]   $,   we   have                               $$ K_{X}  w    K_{Y} =   \bigsqcup _ {  \tau   }    \bigsqcup _ { \vec{\kappa} \in \kay_{\tau  w }  }  \,  \mathcal{X}_{\tau w }    ( \vec{\kappa} )  K_{Y }  $$     
where $ \tau $ runs over $ [ W_{X} /   (  W_{X } \cap   w W_{Y}  w ^ { - 1 }   )   ]   $.  In particular, $ | K_{X} w K_{Y} / K_{Y} | = 
{ \sum 
  \nolimits    _{\tau} | \kay _{\tau w} | }$.    
\end{theorem}   

\begin{proof}  First note that  $   K_{X} w K_{Y } =  \bigcup_{ x \in W _{X} }  B  x   B  w B K_{Y}  =  \bigcup_{x \in W_{X} }  B x  w  K_{Y } $ where the second equality follows  
since $ \ell ( x  w ) = \ell ( x ) + \ell ( w ) $ for all $ x \in W_{X}  $ (see \S \ref{Coxetersection}).   Since    $ B \backslash  \tilde{G} / K _  { Y }    $ is in bijection with    $ \tilde{W} / W_{Y} $,   we infer  that    $ Bx w K_{Y} = B x' w K_{Y} $   for  $ x ,  x '  \in   W_{X}   $     if and only if $  xw W_{Y} = x'w W_{Y} $.  It follows that  $$ K_{X} w K_{Y}   = \bigsqcup _ { \tau } B \tau w K_{Y} $$   where   $  \tau   $ runs 
   over a set of  representatives of  $  W_{X} /  ( W_{X} \cap w W_{Y} w^{-1}  )  $ and    which we are free to take  from  the   set      $ A  := [W_{X}  / ( W_{X} \cap  w  W_{Y }  w ^ { -1} ) ]     \subset W_{X} $. Fix  a $ \tau \in  A     $.  We have   
   \begin{equation}   \label{deco}                 B \tau w K_{Y} =  B  \tau w B K_{Y} =  \bigcup _  {    \kappa \in  \kay _ { \tau   w  } }    \mathcal{X}_{\tau w  } (  \vec{\kappa} ) K_{Y} 
   \end{equation}  by  Lemma    \ref{Bwbdecompose}.      Say $  \vec{\kappa}_{1} , \vec{\kappa}_{2} \in \kay_{\tau w } $ are such that $ g_{1} K _{Y} =  g_{2}  K _{Y }$  where $ g_{ i  }  : =  \mathcal{X}_{\tau w } ( \vec{\kappa}_{i} ) \in \tilde{G}   $  for $ i = 1 , 2 $.    As $ K_{Y}  = B W_{Y} B $, we have     $$  g_{1} K_{Y} =  \bigsqcup _ {  y  \in W_{Y} }  \bigsqcup _ { \vec{\kappa} \in \kay _{  y } }   g_{1}  \mathcal{X} _ {  y } (  \vec{ \kappa  } )  B  $$  
by Lemma  \ref{Bwbdecompose} again.    As $ g_{2} B \subset  g_{2} K_{Y} = g_{1} K_{Y}    $,  there exists   $  y  \in   W _{Y} $ and  $ \vec{\kappa}_{y} \in \kay_{y} $  such that  $ g_{1}  \mathcal{X}_{y} ( \vec{ \kappa }_{y} )  B  =   g_{2}   B . $ Now observe that  $$    B  g_{1}  \mathcal{X}_{y}   ( \vec {  \kappa   }_{y}  )     B   \subset B g_{1}  B   \mathcal{X}_{y} ( \vec{\kappa} _{y}  ) B   =    B \tau w B   y    B  $$  and $ B \tau w   B y B = B \tau w y B   $    since $ \ell ( \tau w  y )  =   \ell ( \tau w ) + \ell (  y    ) $ by 
  (\ref{lanslengthformula}). Therefore, $   g_{2} B  =   g_{1}  \mathcal{X} _{  y  } ( \kappa _{ y   }    ) B   \subset   B \tau w y      B   $.    Since $ g_{2} B $ is also contained in $ B \tau w B $, we see that $  B \tau w B = B \tau w y B . $ This can only happen if $ y = 1_{W_{Y}} $ which in particular means that $ \kay_{y} $ is a singleton and $ \mathcal{X}_{y} ( \vec{ \kappa  }    _{y} ) B = B $. We therefore  have   $  g_{1} B = g_{1} \mathcal{X}_{y}(\vec{\kappa} ) B = g_{2}  B    $ which in turn implies  that  $ \kappa_{1} =  \kappa_{2} $.  The upshot is that  the right hand side of (\ref{deco}) is a disjoint  union  for each fixed  $ \tau \in A $.  Thus         $$ K _{  X}  w K _ { Y } =   \bigsqcup _ { \tau \in  A }  B \tau w K_{Y}    =  \bigsqcup _ { \tau  \in  A }   \bigsqcup _ {  \vec{\kappa} \in \kay _ { \tau w } }   \mathcal{X}   _ { \tau w  , K_{Y} } (   \vec{\kappa} ) $$
which completes the proof. 
\end{proof}

\begin{remark}  The  proof of  Theorem     \ref{BNrecipe} is inspired by  \cite[Theorem 5.2]{Lansky}. 
\end{remark}     

\subsection{Reductive Groups}
\label{reductivetitssection}
In this subsection, we recall the relevant results from the theory of Bruhat-Tits buildings. We primarily follow \cite[\S 1]{Casselmanunramified} in our exposition    and  refer  the reader to  book   \cite{BTbook}  for additional  details and  background.    

Retain the notations  introduced in     \S  \ref{rootdatasection}  and  \S  \ref{Weylgroupsection}. 
In particular, $ \Gb $ denotes an unramified reductive group over $  F $  and $ G $ its group of $ F $ points.    Additionally, we let $ \tilde  { \mathbf{G} }   $ be the simply  connected  covering of the derived  group  $ \mathbf{G} ^ { \mathrm{der}} $ of  $   \mathbf{G}   $ and let $  \psi :  \tilde{ \mathbf{G} }  \to   \mathbf{G} $ denote the resulting map.  For a group $ \mathbf{H}  \subset \mathbf{G} $, we denote by $ \tilde{\mathbf{H}  }  \subset \tilde{ \mathbf{G} } $ the pre-image of $  \mathbf{H} $ under $ \psi $.

Let $  \mathscr{B} $ be the Bruhat-Tits  building  of  $ \tilde{G}   :  =   \tilde{\mathbf{G}}(F) $ and let $ \mathscr{A}    \subset   \mathscr{B}    $ be the   apartment   stabilized (as a subset) by   $ \tilde{A} : = \tilde{\mathbf{A}}(F)$.  By  definition $ \mathscr{A} $ is an affine space under the  real   vector space $ \tilde{V } :=  X_{*}( \tilde{\mathbf{A}}  ) \otimes \RR $.  Let   $  \tilde{M}  : =  \tilde{\mathbf{M} } ( F)  $. There is a unique homomorphism $ \nu : \tilde{M} \to \tilde{ V   } $ determined by the  condition $$  \chi (  \nu ( m )  ) =   -  \mathrm{ord} \big  ( \chi ( m)  \big  ) $$ for all $  m  \in \tilde{M} $, $ \chi  $  a $ F $-rational cocharacter of $  \tilde{\mathbf { A  }  } $.    The kernel of $ \nu $ is a maximal  compact open subgroup $ \tilde{M} ^{\circ} $ of  $ \tilde{M} $.  Set $ \tilde{A}^{\circ}   :  =  \tilde{A} \cap  \tilde{M}^{\circ} $. Then $  \tilde{A} / \tilde{A}^{\circ} = \tilde{M} / \tilde{M}^{\circ} $ via the inclusion $ \tilde{A} \hookrightarrow \tilde{M} $ and the image $ \nu(\tilde{M}) \subset \tilde{V} $  is identified   with     $ X_{*}(\tilde{\mathbf{A}})   $.   Let $ \tilde{\mathcal{N}} $ denote the stabilizer of $ \mathscr{A} $ (as a subset of $ \mathscr{B} $).  The  map $ \nu $ admits a unique extension $ \tilde{\mathcal{N}} \to \mathrm{Aut}(\mathscr{A} ) $ where $ \mathrm{Aut}(\mathscr{A}) $ denotes the group of  affine automorphisms of $ \mathscr{A} $. The action of  $ \tilde{G} $ on $ \mathscr{B} $ is then   uniquely  determined by this extension.

Fix $ x_{0} \in \mathscr{A} $ a hyperspecial point via which we identify $ \tilde{V} $ with $ \mathscr{A} $.   Then $ \nu $ identifies $   \tilde{\mathcal{N}} /   \tilde{M}^{\circ} $ with $ W_{\mathrm{aff} } = \tilde{\Lambda} \rtimes W $.  Let $ C   \subset   \mathscr{A}        $  be   an alcove (affine Weyl chamber) containing $ x_{0}  $ such that   the  set    $ S_{\mathrm{aff}} $ chosen in 
\S \ref{LfactorHeckesection} is identified with the set of reflections in the walls of $ C $.  Let    $ \tilde{B} $ be the (pointwise) stabilizer of $ C $ in $ \tilde{G} $.   Then $ (\tilde{G}, \tilde{B}, \tilde{\mathcal{N}} ) $ 
is a Tits system with Weyl group $ W_{\mathrm{aff}} $  and  the morphism $ \psi : \tilde{G} \to G $ is $ (\tilde{B} , \tilde{\mathcal{N}} ) $-adapted.  The action of $ G $ on $ \tilde{G} $ induced by the natural map $ \mathbf{G} \to \mathrm{Aut}(\tilde{\mathbf{G}}) $ determines an action of $ G $ on $  \mathscr{B} $.  The  stabilizer $ \mathcal{N}  \subset G $ of the action of $ G $ on  $ \mathscr{A} $  equals  the   normalizer     $ N_{G}(A) $ of $ A $ in $ G $.     If we denote by $ \nu : \mathcal{N} \to  \mathrm{Aut}(\mathscr{A}) $ the canonical morphism, the inverse image of translations coincides with $ M = \mathbf{M}(F) $ and $ \tilde{\mathcal{N}}/ \tilde{M} =  \mathcal{N} / M = W $.     By the discussion in \S\ref{Titssection},  the quotient $ G / \psi(\tilde{G}) $  acts naturally on $ (W_{\mathrm{aff}  }    , S_{\mathrm{aff}}) $. There is thus an induced    map $ \xi : G  \to  \mathrm{Aut}(\mathscr{A} ) $ such that each $ \xi(g) $ for $ g \in G $ sends $ C $ to itself.  Let   $$ G ^{1}   :    = \left \{ g \in G \, | \, \chi(g) = 1 \text{ for }   \chi : \mathbf{G} \to  \GG_{m}   \right \} .$$   Then $ M^{\circ} = M  \cap  G^{1}  $ and $ \psi(\tilde{G})  \subset G ^{1} $.      Let   $ B \subset G^{1} $ be the set of elements that stabilizes $ C $ (as a subset of $\mathscr{B} $) and  $ K \subset G^{1} $ the sub-group of elements stabilizing $ x_{0} $. 
Then $ B $ is a Iwahori subgroup of $ G $ and $ K $ a hyperspecial subgroup. In particular, $ K = \bigsqcup B w B $ for $ w \in W $.  We will  assume that the group scheme $ \mathscr{G} $ in \S \ref{Weylgroupsection} is  chosen so that $ \mathscr{G} (\Oscr_{F} ) = K $. 

Finally, let $ G^{0} = G^{1} \cap \ker \xi $  and    let    $ \mathcal{N}^{0} =   G^{0} \cap  \mathcal{N} $.   Since $ G 
 = \psi(\tilde{G}) M $ 
 and  $ \psi(\tilde{G}) \trianglelefteq  G $, we infer that    $ G^{0} =  \psi(\tilde{G}) M^{\circ} $, $ B = \psi(\tilde{B}) M^{\circ} $, $  \mathcal{N}^{0} = ( \psi(\tilde{G}) \cap  \mathcal{N}    ) M^{\circ}  =  \psi(\tilde{\mathcal{N}}) M^{\circ} $.

It is then elementary to see  $ (G^{0} , B , \mathcal{N}^{0} ) $ is a Tits system with Weyl group $ W_{\mathrm{aff} } $ (see   \cite[Lemma 1.4.12]{BTbook})   
and that $ G^{0} \hookrightarrow G $ is  a  $ (B,  \mathcal{N} ^{0}) $-adapted   whose  extended  Weyl group is the Iwahori Weyl group $ W_{I}  $.    
One may therefore apply  the  result of Proposition  \ref{BNrecipe} to the inclusion $  G^{0}  
\to G $ to obtain decompositions of double cosets in $ K_{1} \backslash G / K_{2}  $ where $ K_{1}, K_{2} \subset G^{0} $ are subgroups containing $ B $.   

\begin{remark}  
If  $ s \in W_{\mathrm{aff} }   $ denotes the reflection in a wall of the alcove, $ B  / (  B \cap s B s )    $ has cardinality $ q^{d(s)} $ for some $ d(s) \in \ZZ $ and a set of representatives can be taken in the  $ F $ points of the root group $ U_{\alpha} $ where $ \alpha \in \Phi_{F}  $ is the vector part of the corresponding affine root associated with $ s $. The precise description of $ d(s) $ is  given in terms of the root group filtrations and is recorded on the corresponding \emph{local index} which is the Coxeter diagram of $ W_{\mathrm{aff}} $ with additional  data.  When $ \mathbf{G} $ is split, $ d(s) = 1 $. We refer to \cite{Tits}
for more details.   
\end{remark}    
\begin{remark}   
In the notations of \cite[\S 2.5 (c)]{BTbook}, we have  $ \mathbf{M}(F)^{1} = \mathbf{M}(F)^{0} = M^{\circ}  $ as  $ \mathbf{M} $ is split over an unramified  extension.    
The group $ G^{0} $ 
therefore coincides with \cite[Definition 2.6.23]{BTbook}. That $ (G^{0}, B, \mathcal{N}^{0} , S_{\mathrm{aff}} ) $ forms a Tits system is  established  in    Theorem 7.5.3 of \emph{op.cit.} 
\end{remark}   

\begin{convention}
In the sequel,  we will  denote the Iwahori subgroup  $ B \subset G $ by the letter  $ I $. 
\end{convention}

\subsection{Decompositions    for  $  \GL_{2} $}  \label{GL2decompositions}    
Retain the notations  introduced in \S  \ref{Satakeexamples}. Let $ \chi^{\vee} = f_{1} - f_{2} $ denote the coroot associated with $ \chi $ and  $ s = s_{\chi} $ denote the unique non-trivial element in  $ W $. Let $$       w_{0} =  \begin{pmatrix} &    \frac{1}{ \varpi} \\  \varpi &  \end{pmatrix}  ,  \quad   w_{1} = \begin{pmatrix} &  1 \\ 1 &   \end{pmatrix} ,  \quad   \rho  =  \begin{pmatrix} & 1 \\  \varpi   \end{pmatrix} $$ 
Then $ w_{0}, w_{1}, \rho $ normalize $ A $ and $ \rho w_{0} \rho^{-1} = w_{1} $, $ \rho w_{1} \rho^{-1} = w_{0}  $.  Under the conventions introduced, the   matrices $ w_{0} $, $ w_{1} $ represent the  two simple reflections $ S_{\mathrm{aff}} = \left \{ t(\chi^{\vee} ) s , s \right \} $ of the affine Weyl group $ \ZZ \langle f_{1} -f_{2} \rangle \rtimes W  $.  The element $ \rho $ represents $ t(-f_{2})  s_{\chi} \in \Lambda \rtimes W   =  W_{I}   
$ and is a generator of $ \Omega = W_{I} / W_{\mathrm{aff} } $.  The action of $ \rho $ on $ \mathscr{B} $ preserves the alcove $ C $ and permutes the two walls corresponding to $ w_{0}, w_{1} $. We say that  $ \rho $ induces an  automorphism of the   Coxeter-Dynkin diagram  \vspace{-0.1em}  
\begin{center}   \dynkin[extended, labels = {0, 1}, edge length = 1cm]A{1}   \vspace{-0.1em}     
\end{center}
given by switching the two nodes.  
Let $ I $ denote the Iwahori subgroup corresponding to the set of affine roots $ \chi $ and $ -\chi+1 $ (considered as functions on the space $ \Lambda \otimes \RR $). Then $ I $ is the usual Iwahori subgroup of $ \GL_{2}(\Oscr_{F}) $ given by matrices that reduce to upper triangular matrices modulo $ \varpi $. Let $ x_{0}, x_{1}  : \GG_{a} \to \GL_{2} $ denote the following `root group'  maps $$         x_{0} : u \mapsto  \begin{pmatrix} 1 & \\ \varpi u  &  1 \end{pmatrix},    \quad  \quad     \, x_{1} :   u    \mapsto    \begin{pmatrix}  1 & x \\ &  1    \end{pmatrix}  . $$ 
and let $ [\kay] \subset \Oscr_{F} $ denote a set of representatives of $ \kay $. Then $ x_{i} ( [\kappa] ) $   constitute a set of representatives for $ I / ( I \cap w_{i} I w_{i} ) $ for $ i = 0 , 1 $.   
Let $ g_{ w_{i}  } : [\kay] \to G $ be the maps  $ \kappa \mapsto  x_{i}(\kappa) w_{i} $. For $ w = s_{w, 1}  \cdots s_{w, \ell(w)} \rho_{w}  \in W_{I}   $ a reduced word decomposition (where  $ s_{w,i} \in S_{\mathrm{aff}}   
$ ,
$  \rho_{w} \in \Omega =  \rho  ^ {  \ZZ   }  $) such that $ w $ is shorter of the two elements in $ w W $,   define    
\begin{align*} \mathcal{X}_{w} : [\kay]^{\ell(w)}  & \to G/ K   \\
 (\kappa_{1} ,\ldots, \kappa_{\ell(w) } )   &  \mapsto  g_{s_{w,1}}   (    \kappa_{1} ) \cdots g_{s_{w, \ell(w) }    }(\kappa_{\ell(w)})   \rho  _ { w  }        K  . 
\end{align*} 
The maps $ \mathcal{X}_{w} $ may be thought of as parameterizing $ \Oscr_{F} $ lifts of certain    \emph{Schubert cells}\footnote{See \S  \ref{GLnminusculesection} that makes the connection with classical  Schubert cells of Grassmannians   more     precise.} 
and will be referred to as such.    
Proposition  \ref{BNrecipe}    provides a decomposition of double cosets $ K \varpi ^ { \lambda  } K $ for $ \lambda \in \Lambda $  in terms of these  maps. Let us illustrate  this  decomposition  with a few simple examples.        
\begin{example}     \label{GL2rhoexample}     Let $ \lambda = f_{1}  $. Then $ K \varpi ^{\lambda} K =  K \varpi^{\lambda^{\mathrm{opp}}} K = K  \rho K $.  Clearly $ \rho \in [W \backslash W_{I} / W   ]  $ and $  [ W / (W \cap \rho W \rho^{-1} ) ] = W $.    The decomposition therefore   reads  $$ K \varpi^{\lambda} K   /    K   =  \mathrm{im} ( \mathcal{X}_{\rho} ) \sqcup  \mathrm{im} ( \mathcal{X}_{w_{1}   \rho}   )   . $$  
Explicitly,  we have    
\begin{align*}  
\mathrm{im}( \mathcal{X}_{\rho} ) =  \Set* {    
\begin{pmatrix} 1 & \\ & \varpi \end{pmatrix} K }  \quad  \text{ and }     \quad      
 \mathrm{im}(\mathcal{X}_{w_{1} \rho} )  =  \Set* {  \begin{pmatrix}    \varpi & \kappa  \\ & 1    \end{pmatrix}   K \given  \kappa  \in  [   \kay   ]     }.
\end{align*}   
There are a total of $ q+1 $ left cosets contained in $ K \varpi^{\lambda} K $.     
\end{example}

\begin{example}   \label{GL2second}    Let $ \lambda = 2f_{2} $. Then $  K \varpi ^ { \lambda} K   = K w_{0}  \rho ^{2} K $ and $ w      :     =  w_{0} \rho^{2}  \in [W \backslash  W_{I} / W ] $ and $ [ W / W \cap w W w^{-1}] = W $.    
The decomposition therefore reads $$ K \varpi^{\lambda}   K  /   K   = \mathrm{im} ( \mathcal{X}_{ w   }     )     \sqcup  \mathrm{im} ( \mathcal{X} _{  w_{1}  w    }    ) .   $$  
Explicitly, we have    
\begin{align*}   
\mathrm{im}(\mathcal{X}_{  w    } )  =  \Set*{ 
\begin{pmatrix}  1 & \\    \kappa   \varpi   &  \varpi^{2}  \end{pmatrix}    K   \given  \kappa       \in[ \kay] \, }   \quad  \text{ and }  \quad   
 \mathrm{im}(\mathcal{X}_{w_{1}  w } )  = \Set* { 
 \begin{pmatrix}  \varpi^{2}   &  \kappa _{1}  \varpi +  \kappa_{2}   \\  &   1  \end{pmatrix}   K  \given  \kappa_{1} , \kappa_{2}  \in [ \kay ]  } .
 \end{align*}
There are $ q 
( q + 1 ) $ cosets contained in $ K \varpi^{\lambda } K $. 
 Cf.\    Example \ref{SatakeexampleII}.     
\end{example}
As seen from the examples, the Schubert cell maps $ \mathcal{X}_{w} $ are recursive in nature and going  from    one Schubert cell to  the    `next'    amounts to applying a reflection operation on rows and adding   a  multiple of one row to another.   We also  note that the actual product of matrices in $ \mathcal{X}_{w} $ in the example above may not necessarily be upper or lower triangular  as displayed e.g.,  with the choices above,  $ \mathcal{X}_{w_{0}  \rho^{2} } ( \kappa ) =  g_{w_{0}}(\kappa) \rho^{2} $ equals $$   
\begin{pmatrix}  & 1 \\ \varpi^{2}  &  \kappa  \varpi    \end{pmatrix} . $$
However, since  we are only interested left $ K$-coset representatives, we  can   replace   $ \mathcal{X}_{w_{0} \rho^{2} } (\kappa ) $ with $ \mathcal{X}_{w_{0} \rho^{2}} (\kappa ) \gamma $ for any $ \gamma \in K $.  
In general,  multiplying by a reflection matrix on the left has the effect of `jumbling up' the diagonal entries of the matrix.  While performing these computations,  it is desirable to  keep  the `cocharacter'  entries on the diagonal and one may do so  by  applying a corresponding reflection operation on columns  using  elements of $ K   $. In the computations done in Part II,   
this will be done without any comment.           

\begin{remark}    
In computing $ \mathcal{X}_{w} $, one can   often      establish certain `rules' specific to the group at hand  that dictate where the entries of the a particular cell are supposed to be written depending on the permutation of $ \lambda $ described by the word. For instance, the rule of filling a Schubert cell      
$$     
\begin{pmatrix}   \varpi^{a}  &  \square  \\
\scalebox{0.9}{$\bigcirc$}  &  \varpi^{b} 
\end{pmatrix} $$
as displayed above is as   follows:     
\begin{itemize}   
\item  if $ a \geq b $, the  \scalebox{0.9}{$\bigcirc$}-entry is zero and the $\square$-entry runs over a set of representatives of $ \varpi^{a} \Oscr_{F} /  \varpi^{b} \Oscr_{F} $ 
\item if $ a < b $, then $ \square  $-entry is zero, and the   $  \scalebox{0.9}{$  \bigcirc$} $-entry runs over representatives of $ \varpi^{b} \Oscr_{F} / \varpi^{a+1} \Oscr_{F} $.   
\end{itemize}   
\end{remark}    

\subsection{Reduced words} 
Retain the notations introduced   \S   \ref{rootdatasection}  and  \ref{Weylgroupsection}. Fix a $ \lambda \in \Lambda  ^ { +  }$. The recipe  of   Proposition      \ref{BNrecipe} requires writing the reduced decomposition of the word $ w \in W_{I} $ of minimal possible length  such that $ K \varpi^{\lambda} K = K w K $. This is of course the same for $ K \varpi^{\lambda^{\mathrm{opp}} }K $.  We may equivalently think of $ W_{I} $ as $ t(\Lambda) \rtimes W $  via the morphism (\ref{vmap})    and the length we seek is  the minimal possible length of  elements in  $ W t ( - \lambda^{\mathrm{opp}} ) W  \subset t(\Lambda) \rtimes W $.  For any $ \mu \in \Lambda $, we denote the  minimal possible length in $ W t(\mu) W $ 
by 
$ \ell_{\mathrm{min}}(t(\mu)) $.    
Let $  \Psi =  \Phi_{F}^{\mathrm{red}} \subset \Phi_{F} $ denote the subset of  indivisible roots and let $ \Psi^{+} = \Psi \cap \Phi_{F}^{+} $. 
\begin{lemma}   \label{Iwahorilengthlemma}        For any $ \lambda \in \Lambda $, the minimal possible length of elements in $ t(\lambda) W  \subset W_{I} $ is achieved by a unique element.  If $ \Phi _  {  F  }     $ is irreducible,  the length of this element is given by   \vspace{-0.1em}   $$  \sum _ { \alpha \in \Psi_{\lambda}}  | \langle  \lambda , \alpha \rangle |  +  \sum _ { \alpha \in \Psi^{\lambda} }  (  \langle   \lambda , \alpha  \rangle    - 1  )   \vspace{-0.1em}     $$ 
where $ \Psi _{\lambda}  = \left \{ \alpha \in \Psi ^ { +  }  \,  | \,  \langle \lambda , \alpha  \rangle \leq 0  \right \}  $, $ \Psi^{\lambda} = \left \{ \alpha \in \Psi^{+} \, | \,  \langle  \lambda,  \alpha  \rangle > 0  \right \}  $.    
If  $ \lambda \in \Lambda^{+} $, the minimal length in $ t(\lambda) W $ also equals $ \ell_{\mathrm{min}} ( t( \lambda ) ) =  \ell_{\mathrm{min}}(t (-\lambda^{\mathrm{opp}})) $.  
\end{lemma} 
\begin{proof} The first claim holds generally for any  Coxeter group (\S \ref{Coxetersection}). Assume $ \Phi_{F} $ is irreducible. It is clear that $ P_{F}^{\vee} = P(\Phi_{F}^{\vee} ) $ is the weight lattice associated with the irreducible reduced root system $ \Psi $. By \cite[\S 1.7]{Iwahori},  $ P^{\vee}_{F} \rtimes W $ is an extension of the Coxeter group $ W_{\mathrm{aff}} =  Q^{\vee}_{F} \rtimes W $ by $  \Omega'  = P^{\vee}_{F} / Q^{\vee}_{F}  $ which acts on $ W_{\mathrm{aff}} $ by automorphisms. Thus the length function on $ W_{\mathrm{aff}} $ can be extended to $ P^{\vee} \rtimes W $. Let $ \varphi :  \Lambda \rtimes W \to P^{\vee}_{F}  \rtimes W $ be the map given by $ (\lambda, w) \mapsto  ( \bar{\lambda} , w) $ where $ \bar{\lambda} = \lambda  \pmod  {X_{0}^{\vee}  }    $. The $ \varphi $ factorizes as $ \Lambda  \rtimes W \to ( \Lambda / X_{0} ) \rtimes W \to  P^{\vee}_{F} \rtimes W $. As both maps in this composition are length preserving, we see that $ \varphi $ is length preserving. The  second  claim then follows  by   \cite[Proposition 1.25]{Iwahori}. Since the sum is maximized for dominant $ \lambda $ and is the same for both $ \lambda $ and $ - \lambda^{\mathrm{opp}} $,  we obtain the last  claim.     
\end{proof}

\begin{example} Retain the notation of \S   \ref{Satakeexamples}.    Let  $  \lambda = 5f_{1} \in \Lambda^{+} $. Then $$ \ell_{\mathrm{min}}(t(\lambda)) = \langle 5f_{1} , e_{1} - e_{2} \rangle - 1 =  5 - 1 =   4 . $$  Say  $ w \in W_{I} $ is of length $ 4 $ and $ K \varpi^{\lambda} K = K w K $. Since $ \det ( \varpi^{\lambda} ) = 5 $, we may assume that $ w = v \rho^{5} $ where $ v $ is a word on $ S_{\mathrm{aff}} = \left \{ w_{0} ,  w_{1}  \right \} $. Now the final letter of $ v $ cannot be $ w_{0} $, since $ \rho w_{0} \rho^{-1} = w_{1} \in K $. Thus we may assume that $ v = v' w_{1} $. Since we can only place $ w_{0} $ next to $ w_{1} $ for a reduced word, we see that   the  only  possible  choice  is         $  w = w_{0} w_{1} w_{0} w_{1} \rho^{3}  $.    
\end{example} 
\subsection{Weyl orbit diagrams}   
\label{Weylorbitdiagramsubsection} 
Retain the notations introduced in \S  \ref{rootdatasection} and   \ref{Weylgroupsection}.    Besides the usual Bruhat order $ \geq $  on the  Weyl group $ W $, there is another partial order that will be useful to us.  We say that $ w \succeq  x $ for $ w,x\in W $ if there exists a reduced word decomposition for $ x $ which appears as a consecutive string on the left of some reduced word for $ w $.   The pair $ (W, \succeq ) $ is then a graded lattice \cite[Chapter 3]{Bjorner} and is known as the  \emph{weak (left) Bruhat order}.   
\begin{definition} For $ \lambda \in \Lambda $, let $ W^{\lambda} $ denote the stabilizer of $ \lambda  $ in $ W $.  The \emph{Weyl orbit diagram} of $ \lambda $ is the Hasse diagram on the set of representatives of  $ W / W^{\lambda} $ of minimal possible length with respect to $ \succeq   $.  As $ W / W^{\lambda} = W \lambda $,  the nodes of such a diagram  can be labelled by elements of $ W \lambda $.      
\end{definition}

Assume  that  $  \Phi_{F}   $   is   irreducible.     Let $ \lambda \in \Lambda^{+} $ and let $ w_{\lambda} \in W_{I} $ be the unique element of minimal possible length such that   $ K \varpi^{\lambda} K = K w_{\lambda} K $.   By Proposition   \ref{Iwahorilengthlemma},  we see that  $ w _ { \lambda } = \varpi^{\lambda^{\mathrm{opp}}}  \sigma_{\lambda} $ for a unique $ \sigma_{\lambda} \in W $ and $ W \cap w_{\lambda} W  w_{\lambda}^{-1} $ is just the   stabilizer of $ -  \lambda^{\mathrm{opp}}  $ (equivalently $ \lambda^{\mathrm{opp}} $)  in $ W $.   So we can make the  identification  $$    [ W / ( W \cap  w_{\lambda}  W  w_{\lambda}^{-1}   )     ]  \simeq  [ W / W ^{\lambda^{\mathrm{opp}}}   ]   .   $$
Thus the decomposition of $ K \varpi^{\lambda} K / K $ as described by   Proposition \ref{BNrecipe} can   be viewed as a collection of Schubert cells $ \mathcal{X}_{\mu} $, one for each node $ \mu \in W \lambda^{\mathrm{opp}} = W \lambda $ of the Weyl orbit diagram of $ \lambda $ (though note that $ \mathcal{X}_{\mu}$  is an abuse of notation). See the proof of Proposition \ref{GU4satake} which illustrates this  point.

In the following, we adapt the convention of drawing the Weyl orbit diagrams of $ \lambda \in \Lambda^{+} $ from left to right, starting from the anti-dominant cocharacter $ \lambda^{\mathrm{opp}}  $  and ending in $ \lambda $. The permutation of $ \lambda $  corresponding to the node then `appears' in the matrices of the corresponding  Schubert cell.       For example, in the notations of  \S  \ref{GL2decompositions},  the Weyl orbit diagram of $ f_{1} $  is $$ f_{2}  \xrightarrow{ s_{\chi}    }  f_{1}  $$  
and the matrices in $ \mathrm{im} ( \mathcal{X}_{\rho} ) $, $   \mathrm{im } (   \mathcal{X}_{w_{1} \rho}   ) $  in Example  \ref{GL2rhoexample}  have  `diagonal entries'  given by $  \varpi^{f_{2}} $, $ \varpi^{f_{1}} $   respectively.    We will often  omit the explicit cocharacters on the nodes  in these diagrams and only display the labels of the arrows. See also \cite[Example 7.1]{Siegel1}.     

\begin{remark}  Observe  that  the shape (in the sense of Definition \ref{Satakeleadingtermdefi}) of the matrices in these cells  may not match the corresponding cocharacter. In Example    \ref{GL2decompositions}, the shape of the matrices that appear in the decomposition of $ K \varpi^{2f_{1}} K $  can be $ 2f_{1} $, $ 2f_{2} $  or $ f_{1}+f_{2}          $   when converted to upper  triangular   matrices.       
\end{remark}

\subsection{Miscellaneous  results}   \label{mixedcosetstructures}   
In this subsection, we record  assortment of results that are useful in determining the structure of mixed double cosets in practice.

\begin{lemma}  \label{volumemeansdistinct}  Suppose $ G $ is a group, $ X , Y \subset G $ are subgroups. Then for $  \sigma ,  \tau  \in G $,  $ X \sigma  Y  = X  \tau  Y  $  only if $ X \cap  \sigma  Y \sigma  ^{-1} $ and $ X  \cap   \tau   Y   \tau^{-1} $ and $ X   $-conjugate.   
\end{lemma}
\begin{proof} $  X  \sigma   Y   = X  \tau  Y  \Longleftrightarrow \sigma = x \tau   y  $ for $  x \in X  $, $ y  \in  Y   \implies   X  \cap \sigma Y   \sigma   ^{-1} =  x   ( X  \cap \tau  Y  \tau    ^{-1} ) x ^{-1 } $.  
\end{proof}    

\begin{lemma}   \label{distinctgen}         Let $  \iota :  H   \hookrightarrow  G $ be an inclusion of groups, $ K \subset G $ a subgroup and $  U = K \cap H $. Then for any $ h_{1} ,h_{2} \in H $,  $  g \in G $,  $ Uh_{1} g K = U h_{2} g  K $ if and only if $ U h_{1} H_{g} =  U h_{2} H_{g} $ where $ H_{g} $ denotes $  H  \cap  g K  g^{-1}   $. Moreover  for any $ h \in H $,  the index $ [ H_{hg} : U  \cap   hg   K  ( hg ) ^{-1} ] $ is equal to $ [H_{g} : H_{g} \cap  hU h^{-1} ] $.      
\end{lemma}

\begin{proof}  The map (of sets) $ H  \twoheadrightarrow  H g K / K $,  $     h \mapsto  hgK $ induces a $ H $-equivariant bijection $ H / H_{g} \xrightarrow{\sim} H g K/ K $ where $ H $ acts by left multiplication.  Thus the orbits of $ U $ on the two coset spaces  are identified i.e.,  $ U  \backslash H /  H_{g}    \xrightarrow { \sim }  U    \backslash H g K / K  $ which proves the first claim.  For any $ h \in H $, $ H_{hg} = h H_{g} h^{-1} $ and $ H_{g} \cap  hUh^{-1} = h ( U \cap  g K g )h^{-1}    $ which proves the second  claim. 
\end{proof} 

The  next   result is  helpful  in  describing   the   structure of double cosets associated with  certain non-parahoric  subgroups.  It is needed in \cite[\S 9]{Siegel1}.   
\begin{lemma}  \label{ohsoclever}               Let $ H $ be a group, $ \sigma \in H $ an element and  $  U, U_{1} ,X $ be subgroups of $ H $ such that $   U_{1} \sigma U / U   $, $ X U_{1}  / U_{1}  $ are finite sets and $ U_{2} = XU_{1} $ is a group.    Then  $ U_{2} \sigma U / U $ is finite and  $$ e \cdot  \ch (  U_{2}  \sigma  U ) =   \sum      \nolimits        _{ \delta } \ch ( \delta   U_{1} \sigma U )  $$ 
where $ \ch(Y) : H \to \ZZ $ denotes the characteristic of $ Y \subset   H $, $ \delta \in X $ run over  representatives of $ X / ( X  \cap  U_{1 }  ) $  and      $  e   =    [U_{2} \cap \sigma U \sigma^{-1} : U_{1} \cap  \sigma  U  \sigma ^{-1} ] $. If $ U_{2}  \cap \sigma U \sigma^{-1} $ is equal to the product of    $ X  \cap \sigma  U  \sigma ^{-1} $ and $ U_{1} \cap  \sigma U \sigma ^{-1} $,    then $ e = [ X \cap \sigma U  \sigma ^{-1} : X \cap  U_{1}  \cap   \sigma U \sigma ^{-1} ] $.  
\end{lemma}

\begin{proof} Let $ W_{i}   :  = U_{i} \cap  \sigma  U  \sigma  ^{-1}  $ for $ i = 1, 2  $,   $     Z  : =  X \cap   U_{1}  $ and let  $ \gamma_{1}, \ldots, \gamma_{m} \in U_{1}  $  be representatives of $ U_{1} / W_{1}    $,  $ \delta_{1} ,  \cdots,   \delta_{n} \in X $ be representatives of $ X / Z $.    We first show that $ \delta_{j} \gamma_{i} $  form a complete set of  distinct representatives of the coset space $ U_{2} / W_{1} $. Let $ x \in X $, $  u \in U_{1} $. Then there exists a $ z \in Z $, $  w   \in  W _{1} $ and (necessarily unique)     integers  $ i $, $ j $ such that $ xz  = \delta_{j} $, $ z^{-1} u  w     = \gamma_{i} $. In other words, $ x  u  W_{1}  =  ( x z) (z^{-1} u w ) W_{1} =  \delta_{j}  \gamma_{i} W_{1}  $.   Therefore, every element of $ U_{2} / W  _  {   1   }     $ is of the form $ \delta_{j} \gamma_{i} W   _ { 1  }      $ and so   $$ U _{2}    =  \bigcup_{j = 1 } ^{ n }  \bigcup _{  i = 1 } ^ { m } \delta_{j} \gamma_{i} W_{1}   $$
We claim that this  union is disjoint. Suppose $  x, y  \in X $, $ u , v \in U_{1} $ are such that $  xu  W_{1}  =  y v  W_{1}  $.    Then $ v^{-1} y^{-1}  x u \in W _{1}  $. Since $ U_{2} $ is a group containing both $  v^{-1} \in  U_{1} $ and  $  y^{-1}   x   \in     X   $, $  v^{-1} y ^{-1}  x \in U_{2} $. Since  $ U_{2} $ is equal to $ X \cdot  U_{1} $,  there exists     $ x_{1} \in X $, $ u_{1} \in  U_{1} $ such that $ v^{-1} y^{-1} x = x_{1} u  _ {1} $  or  equivalently, $ y^{-1}  x  = v x_{1} u _{1}    $.      Now   \begin{align*} 
v^{-1} y^{-1} x u \in W   _ { 1 } &  \implies  x_{1} u_{1} u \in W   _ { 1 }   \subset U_{1}  \\   & \implies x_{1} \in U_{1}    \\ &  \implies   y^{-1} x    =   v  x_{1}  u_{1}  \in  U_{1}    
\implies    x Z = y Z 
\end{align*}    
Thus if $ x , y $ are distinct modulo $ Z   $, $ xu W _{ 1}  $, $ y v W _{1}  $ are distinct left $ W_{1} $-cosets for  \emph{any} $ u, v \in U_{1} $. Thus, in the union above, different $ j $  correspond to necessarily distinct   
$ W_{1} $-cosets. It is clear that $ \delta_{j} \gamma_{i_{1}}  W _ { 1 }  = \delta_{j} \gamma_{i_{2}}  W_{1} $ iff $ i_{1} = i_{2} $.   Thus the  union above is disjoint as both $ \delta_{j} $ and $ \gamma_{i} $ vary.   

Now  we prove the first  claim.   Let  $  p : U_{2} / W_{1} \to U_{2} / W_{2}  $ be the natural  projection  map. Since $ U_{2} / W_{1} $ is finite, so is    $ U_{2} / W_{2}$ and therefore $ U_{2} \sigma U / U $. Moreover, as $ W_{2} / W_{1} \hookrightarrow U_{2} / W_{1} $, $ e = [ W_{2} :  W_{1} ] $ is   finite.   Let   $ y =        a W_{2} \in U_{2} / W_{2} $ be a $ W_{2}$-coset of $ U_{2} $.  Then $ p^{-1}( y) = \left \{ a w W_{1} | w \in W_{2}  \right \} $  and  we have  $$ | p^{-1}(y) | = p^{-1}(W_{2}) = [  W_{2} : W_{1}  ]    =   e .  $$   Thus  in     the list of  $ mn $ left  $ W_{2} $-cosets given by   $  \delta_{1}  \gamma_{1} W_{2} $, $  \delta_{1} \gamma_{2} W_{2}  , \ldots ,    \delta_{n} \gamma_{m}  W_{2} $,     each element of $ U_{2} / W_{2} $ appears exactly $ e   $ times. Equivalently, among the $ mn $ left  $U$-cosets $   \delta_{1} \gamma_{1} \sigma U $,  $  \delta_{1} \gamma_{2} \sigma U , \ldots ,   \delta _ { m }   \gamma _{ n }   \sigma   U  $,   
each element of $ U_{2} \sigma U/U $ appears exactly $ e $ times. Since $ U_{1}   \sigma U = \bigsqcup _ { i = 1 } ^{ m } \gamma_{i} \sigma U $, we see that  $$   e \cdot \ch(U_{2} \sigma U ) = \sum   \nolimits  _ {i ,j } \ch (  \delta_{j}   \gamma_{i}    \sigma U )  = \sum  \nolimits _{j} \ch ( \delta_{j} U_{1}  \sigma U )   $$    
and the   first   claim  is  proved.  The second claim follows since $ W_{2} = (X \cap \sigma U  \sigma^{-1}) W_{1} $ implies that  $ W_{2}/ W_{1} = (  X \cap \sigma U \sigma^{-1} ) W_{1} / W_{1} = ( X \cap \sigma U \sigma ^{-1} ) / (  X \cap W_{1} )  $.   
\end{proof}

%% file: Examples.tex
\vspace{-1em}   
\partdivider{}{Part 2. Examples}
 \vspace{-1em}    
\section{Arithmetic considerations}   

In this section, we record two embeddings of Shimura-Deligne  varieties   that are of arithmetic interest  from the perspective of Euler systems. Our goal here is only to motivate the local zeta element problems  arising from these scenarios, cast them in the axiomatic framework of \S \ref{basicsetup} and justify  various  choices of data in order to align these problems with the actual arithmetic situation.  In particular, we will  make no attempt to study the arithmetic implications of these problems.   
In the sections that follow, we solve the resulting combinatorial problems using techniques developed in Part I.  These examples are meant to test our machinery in situations where the computations are relatively straightforward in comparison to, for instance, \cite{Siegel1}. For a concrete  arithmetic application of  such  combinatorial  results  to  Euler system constructions,     
 we refer the reader to \cite{Siegel2}.

\subsection{Unitary Shimura varieties} 
Let $ E \subset \CC $ be an imaginary quadratic number field  and  $    \gamma \in \Gal(E/\QQ ) $ denote the non-trivial automorphism. Let $ J = \mathrm{diag}(
1,\ldots, 1
, 
-1,\ldots,-1
) $ be the diagonal matrix  where there number of $ 1 $'s is $ p $ and the number of $ -1 $'s is $ q $.  Clearly $  \gamma(J)^{t}  = J $ i.e.,    $ J $ is $ E /\QQ $-hermitian. Let  $     \mathrm{GU}_{p,q} $ denote the algebraic group over $ \QQ $ whose $ R $ points for a $ \QQ $-algebra $ R $ are given by $$ \mathrm{GU}_{p,q}(R)  : =  \left \{ g \in \GL_{p,q}(R) \, | \,  \gamma(g)^{t}J \gamma(g) = \mathrm{sim}(g) J \, \text{ for some } \mathrm{sim}(g) \in R^{\times} \right \} . $$
The  resulting map 
$ \mathrm{sim} : \mathrm{GU}_{p,q} \to \GG_{m} $ is a character called the \emph{similitude}. 
 Let \begin{align*}  h :  \delT     &    \to \mathbf{G}_{\RR}    \\    
z    &   \mapsto  \mathrm{diag}(
z,\ldots,z
, 
\bar{z}, \ldots, \bar{z}
)  
\end{align*} 
and  let $  \mathcal{X }    $ be the $ \mathbf{G}(\RR) $-conjugacy class of $ h $. Then $ (\mathbf{G}, X) $ constitutes a Shimura-Deligne data   that satisfies (SD3) if $ p, q \neq 0 $   (see   \cite[Appendix B]{Anticyclo} for terminology).        The dimension of the associated Shimura varieties    is     $ p q $.  There is an identification $  \mathbf{G}_{E} \simeq \GG_{m,E} \times \GL_{p+q, E} $ induced by the isomorphism of $ E $-algebras $ E \otimes R \simeq R^{\times} \times R^{\times} $, $ (e,r) \mapsto (er , \gamma(e)r) $ for any $ E $-algebra $ R $. The cocharacter $    \mu_{h} : \GG_{m, \CC} \to \GG_{m, \CC} \times \GL_{p+q, \CC} $ associated with $ h $ is given by $ z \mapsto  \big   ( z ,  \mathrm{diag}(
z,\ldots,z
,
1,\ldots,1
)  \big )    $. The reflex field is then easily seen to be $ E $ if $ p \neq q $ and $ \QQ $ otherwise.

\label{embeddinganticyclosec}

For $ m \geq 1 $ an integer, let $ \Gb : = \mathrm{GU}_{1,2m-1} $. Then the so-called arithmetic middle degree\footnote{one plus the dimension of the variety} of the Shimura varieties  of $ \Gb $ is $  2m   $. Thus one  construct classes in this degree   by  taking   pushforwards   of      special cycles of codimension $ m  $. One such choice is given by the  fundamental cycles of Shimura varieties of  $$ \mathbf{H} : = \mathrm{GU}_{1,m-1   }    \times _ { \GG_{m}  }     \mathrm{GU}_{0,m}. $$ 
where the fiber product is over the similitude map.    There is  
a natural embedding $ \mathbf{H} \hookrightarrow \Gb $ which constitutes a  morphism of SD data and gives an embedding of    varieties is  over $ E $.    
We note that $ \mu_{h} $ for $ \Gb $ corresponds to the representation of $ {}^{L} \mathbf{G}_{E} $ which is trivial on the factor $ \mathscr{W}_{E} $ and  which  is the standard representation on $ \widehat{\Gb} = \GG_{m} \times \GL_{n} $. Thus at a choice of a split prime  $ \lambda $ of $ E $  above $ \ell $, we are interested in the Hecke polynomial  of the standard representation of $ \GL_{n} \times \GG_{m} $.  This case is studied \S \ref{GLnLfactorsection}.  When $ \ell $ is inert, we are interested in the \emph{base change}    of the standard $ L $-factor (Remark  \ref{basechangeremark}).  
This setup is the studied  \S \ref{GU4Lfactorsection} for the case $ m = 2  $.   As  we  are  pushing  fundamental  cycles of the  Shimura varieties of $ \mathbf{H} $, we are led to consider the trivial functor that models the distribution relations of these cycles. See \cite[Theorem 6.4]{Anticyclo} for a description of the relevant Galois representations which the resulting  norm relations are geared towards.

To construct classes that go up a tower of number fields, we need to specify a choice of torus $ \mathbf{T} $ and a map $ \nu : \mathbf{H} \to \mathbf{T} $, so that the Shimura set associated with $ \mathbf{T }$ corresponds to non-trivial abelian extensions of the base field $ E  $. We can    then construct classes in towers by considering the diagonal embedding $ \Hb \hookrightarrow \Gb \times \Tb  $. 
One such choice is $  \mathbf{T} :=  \mathbf{U}_{1} $, the torus of norm one elements in $ E $. It is considered as a quotient of $ \mathbf{H} $ via $$ \nu :  \mathbf{H} \to  \mathbf{T}    \quad \quad  \quad  (h_{1}, h_{2})  \mapsto   \det h_{2} / \det h_{1}     . $$
The extensions determined by the associated reciprocity law are anticyclotomic i.e.,  the natural action of $ \Gal(E/\QQ) $ on them is by inversion. The behaviour of arithmetic Frobenius $ \mathrm{Frob}_{\lambda} $ at a prime $ \lambda $ of $ E $ in an unramified extensions contained in such towers is rather  special. Let $ \ell $ be  the rational prime of $ \QQ $ below $ \lambda $. When $ \ell $ is split, we denote by $ \bar{\lambda} $ the other prime above $ \ell $. Then $ \mathrm{Frob}_{\lambda} $ is trivial if $ \ell $ is inert and $ \mathrm{Frob}_{\lambda} = \mathrm{Frob}_{\bar{\lambda}}^{-1} $ if $ \ell $ is split. If $ \ell $ is split, the choice of $ \lambda $ above $ \ell $ allows us to pick  identifications $  \mathbf{H}_{\QQ_{\ell}} \simeq \GG_{m}  \times \GL_{m} \times \GL_{m} $ and $ \mathbf{T}_{\QQ_{\ell}} \simeq \GG_{m} $, so that $ \nu $ is identified with the map $ (c, h_{1}, h_{2}) \mapsto \det h_{2} / \det h_{1} $. With these conventions, the  induced map $ \nu \circ \mu_{h} $ sends the uniformizer  at $ \lambda $ in $ \in E_{\lambda} ^{\times} $  to   $ 1 \in \mathbf{T}(\QQ_{\ell}) $ if $ \ell $ is inert and to $ \ell^{-1} \in \QQ_{\ell} \simeq  \mathbf{T}(\QQ_{\ell}) $ if $ \ell $ is spilt.    The group $ \mathbf{T}(\QQ_{\ell}) $ has  a    compact open subgroup of index $ \ell +1 $ (resp., $\ell-1$) if $ \ell $ is inert (resp., split). These groups provide the `layer extensions'   for our zeta element problem.

\begin{remark} The choice of $ \nu $ is made to match that in \cite{Anticyclo}. One equivalently work with $ \nu' $ that sends $ (h_{1}, h_{2}) \mapsto \det h_{1} /\det h_{2} $ in which case $ \lambda $ is sent to $ \ell \in \Tb(\QQ_{\ell}) $ for $ \ell $ split. The Shimura varieties we have written also admit certain CM versions, and the local zeta element problem studied in \S \ref{GLnLfactorsection} apply to these more general versions    too.   
\end{remark}

\begin{remark}  That the resulting Euler system is 
non-trivial is the subject of a forthcoming work. This particular embedding of Shimura varieties is motivated by a unitary analogue of the period integral of Friedberg-Jacquet \cite{ZhangXiao}, \cite{GanChen}.     A first step towards interpolating these  periods and the construction of a suitable $ p $-adic $L$-function is taken in \cite{Graham1}, \cite{Graham2}.    
\end{remark}

\begin{remark} The inert case of the situation above studied in \S \ref{GU4Lfactorsection}  also   serves as a precursor for a   slightly 
 more involved calculation performed in \cite{EulerGU22} for the twisted exterior square representation.   
\end{remark}

\subsection{Symplectic  threefolds}{}    
\label{embeddingGsp4sec}  
Let $ \mathbf{G} : = \mathrm{GSp}_{4} $ 
and $  \mathbf{H} = \GL_{2} \times_{\GG_{m}} \GL_{2} $    where the fiber product is over the determinant map. We have an embedding $  \iota : \Hb  \hookrightarrow  \Gb  $ 
obtained by considering the automorphisms of the two  orthogonal  
sub-spaces of the standard symplectic vector space $ V $ spanned  by $ e_{1} , e_{3} $ and $ e_{2}, e_{4} $ where $ e_{i} $ are the standard bases vectors.  Let    \begin{align*} 
h :  \delT \to \mathbf{G}_{\RR}  \quad \quad 
(a+b \sqrt{-1}) \mapsto   \begin{psmallmatrix} a  & &  b  \\  & a & & b \\  -b &  &  a \\ & -b & & a    \end{psmallmatrix}.
\end{align*}   
Note that $ h $ factors through $ \iota $. Let  $ \mathcal{X}_{\Hb} $ (resp., $  \mathcal{X}_{\Gb}   $)    denote 
the $ \Hb(\RR) $ (resp.,$  \Gb(\RR) $) conjugacy class of $ h  $. Then $ (\Hb_{\QQ} ,  \mathcal{X}_{\Hb}) $,  $ (\mathbf{G}_{\QQ}, \mathcal{X}_{\Gb})$  satisfies axioms SV1-SV6 of \cite{MilneShimura} and in particular,  constitutes a Shimura data. These  Shimura varieties are respectively the fibered product of two modular curves and the Siegel modular threefold that parametrizes abelian surfaces with polarization and  certain level structures.  The reflex fields of both of these varieties is  $ \QQ $. The cocharacter $ \mu_{h} $ associated to $ h $  corresponds to the four dimensional spin representation of $ {}^{L} \mathrm{GSp}_{4} = \mathrm{GSpin}_{5} \times \mathscr{W}_{\QQ} $ and we are thus interested in establishing norm relations involving the Hecke polynomial associated to the spinor representation. See \cite[Theorem 10.1.3]{LSZ} for a description of the relevant four dimensional  Galois representations to which such norm relations are geared  towards.

As the codimension of the two families of Shimura varieties is $ 1 $, one needs to push classes from $ \mathrm{H}^{2}_{\et} $ of the source variety to be able construct classes in arithmetic middle degree of the target Shimura variety.    As first proposed by Lemma in \cite{Lemma}, one can take (integral) linear combinations of the cup products of two Eisenstein classes in the $ \mathrm{H}^{1} _ { \et}      $ of each modular curve for this purpose. The distribution relations of such cup products  can then be modelled via the tensor product of two CoMack functors associated to Schwartz spaces of functions on $ 2 \times 1 $ adelic column vectors minus the origin. This tensor product is then itself    a   Schwartz  space over a four dimensional adelic vector space (minus two planes that avoid the origin) which then becomes  our (global) source functor. 
The local source bottom class (\S \ref{basicsetup}) is then  the such a  characteristic  function.  
\begin{remark}  Apriori, one can only define Eisenstein classes integrally by taking integral linear combinations of torsion sections determined by the level structure of the modular curves.  The main result of \cite{distpolylog} upgrades this association to all integral Schwartz functions, which justifies  our use of these function spaces as source functors for the zeta element  problem.     
\end{remark}
To construct classes in a tower, we  can consider the torus $ \mathbf{T} = \GG_{m} $ which admits a map $ \nu : \mathbf{H} \to \mathbf{T} $ given by sending a pair of matrices to their  common  determinant. 
As above, we consider the embedding $ \mathbf{H}  \hookrightarrow  \Gb \times \mathbf{T} $ which in this case also factors through $ \mathbf{H} \hookrightarrow  \Gb $.  With this choice, the map induced by $ \mu_{h} : \GG_{m} \to \Tb $ is identity i.e., locally at a prime $ \ell $, the pullback action of $ \ell \in \Tb(\QQ_{\ell})    $  corresponds to the action of geometric Frobenius. Then   $ 1 + \ell \ZZ_{\ell} \subset \Tb(\ZZ_{\ell}) $  provides us with a  `layer extension' of degree $ \ell - 1 $.   Under  the  reciprocity  law  for  $ \Tb $, these layer extensions correspond to the ray class extensions of $ \QQ   $ of degree $ \ell - 1  $.  
\begin{remark}  Although the zeta element problem is only of interest over $ \QQ_{\ell} $, we have chosen to work with an arbitrary local field for consistency of notation.  
\end{remark}  
\begin{remark}  The question that this  construction  leads  to   a non-trivial   Euler system is  addressed  
in \cite{LZBK}. 
\end{remark}

\begin{remark} As the arithmetic middle degree is even, one may ask if interesting classes can be constructed in this degree via special cycles. Such a setup was proposed in \cite[\S 5.1]{Ruishen} which  allows one to construct  classes  over an imaginary quadratic field.  It would be interesting to see if this  construction indeed sees the behaviour of an  $ L$-function.  
\end{remark}

\section{Standard \texorpdfstring{$ L $}{}-factor of  \texorpdfstring{$\GL_{2\lowercase{m}}$}{}}

In this section,   we study the  zeta  element problem for the split case of the embedding discussed in \S \ref{embeddinganticyclosec}. 
\label{GLnLfactorsection}    
\begin{notation*} The symbols $F $, $ \mathscr{O}_{F} $, $ \varpi $, $ \kay $, $q$ and $ [\kay] $  have the same meaning as in   Notation \ref{LfactorHeckesectionnota}. 
The letter $\mathbf{G}$ will denote the group scheme $\mathbb{G}_{m} \times \mathrm{GL}_{n}$ over $ \Oscr_{F} $ where $ n $ is a positive integer and is assumed to be even    from   \S  \ref{mixedcosetglnsection}     onwards.  We will denote $ G : =  \mathbf{G}  (F)$ and  $K : =  \mathbf{G} \left(\mathscr{O}_{F}\right)$. For a ring $R$, we let $\mathcal{H}_{R}  =    \mathcal{H}_{R}(K \backslash G / K)$ denote the Hecke algebra of $G$ of level $K$ with coefficients in $R$ with respect to a Haar measure $\mu_{G}$ such that $\mu_{G}(K)=1$.    For simplicity, we will often  denote  $ \ch(K \sigma K )  \in  \mathcal{H}_{R} $   simply       as $ (K \sigma K )$.    
\end{notation*}

\subsection{Desiderata}   \label{GLndesiderata}             
Let $  \mathbf{A} = \GG_{m}^{n+1} $ and $ \mathrm{dis} :   \mathbf{A} \to \mathbf{G} $ be the embedding given by $$   (u_{0}, u_{1}, \ldots, u_{n} ) \mapsto  \big ( u_{0},  \mathrm{diag}( u_{1} ,  \ldots ,  u_{n } )  \big  ) . $$
Then $ \mathrm{dis} $
identifies $ \mathbf{A} $ with a maximal    torus in $ \mathbf{G} $. We denote $ A : = \mathbf{A}(F) $ the $ F $-points of $ A $  and $ A ^{\circ}  :  = A \cap  K $ the unique maximal compact subgroup.  For $ i = 0, \ldots n $, let   $ e_{i} : \mathbf{A} \to \GG_{m} $ be the projection on the $ i $-th component  and   $ f_{i} : \GG_{m} \to \mathbf{A} $ be the cocharacter inserting $ u $ in the $ i $-th component of $  \mathbf{A} $. We will denote by $ \Lambda $ the cocharacter lattice $  \ZZ f_{0} \oplus \cdots \oplus \ZZ f_{n}   $. The   element $  a_{0} f_{0}  + \ldots + a_{n} f_{n}  \in \Lambda  $ will  also  be  denoted as $ (a_{0}, \ldots, a_{n} ) $.  The set $ \Phi  \subset X^{*}(\mathbf{A} ) $ of roots of $ \mathbf{G} $ are  $ \pm (   e_{i} - e_{j}  )   $ for $ 1 \leq i < j \leq  n   $      
which constitutes an irreducible root system of type $ A_{n-1} $. We let $ \Delta  =  \left  \{ \alpha_{1}, \ldots, \alpha_{n-1} \right \}  \subset \Phi $ where $$   \alpha_{1} = e_{1} - e_{2}, \quad  \quad  \alpha_{2} = e_{2} - e_{3}, \,\,\,\,  \ldots,  \,\,\,\,   \alpha_{n} =  e_{n-1} - e_{n} .  $$Then $ \Delta $  constitutes a base for $ \Phi $. We let $ \Phi ^ { + } \subset \Phi  $ denote the set of  resulting positive roots.   The  half  sum  of  positive  roots  is  then
\begin{equation} \delta : = \frac{1}{2}  \sum_{k=1}^{n} (n-2k+1) e_{k}  \label{halfsumGLn} 
\end{equation}  
With respect to the ordering induced by $ \Delta $, the  highest root is $ \alpha_{0} = e_{1} - e_{2n} $. We let $ I = I_{G}  $  be the standard Iwahori subgroup  of $ G $, which  corresponds to the alcove determined by the simple affine roots $ \alpha_{1}  + 0  $, $ \alpha_{2} + 0 $, $  \ldots $,  $ \alpha_{n-1}   +       0     $, $     - \alpha_{0} + 1  $.    The coroots corresponding to $ \alpha_{i} $ are  $$ \alpha_{0}  ^ { \vee } =  f_{1} - f_{n}    ,  \quad \alpha_{1}^{\vee} = f_{1} -  f_{2},   \quad   \alpha_{2}^{\vee} =   f_{2}  - f_{3},   \,  \,  \, \ldots,  \, \, \,  \alpha_{n-1}^{\vee} =  f_{n-1} - f_{n}  $$   and their $ \ZZ $ span in $ \Lambda $ is denoted by $ Q ^{\vee} $.   An  element    $  \lambda = (a_{0} , \ldots, a_{n} )  \in \Lambda $ is     dominant iff $ a_{1} \geq a_{2} \geq  \ldots \geq a_{n} $ and anti-dominant if all these inequalities hold in reverse. We denote the set  of dominant cocharacters by $ \Lambda^{+ } $. The translation action of $ \lambda \in  \Lambda $ on $ \Lambda \otimes \RR $ via $ x \mapsto x + \lambda $ is denoted  by $ t(\lambda) $. We denote $ \varpi^{\lambda } \in A $ the element $ \lambda(\varpi )  $ for $ \lambda \in \Lambda  $ and      $ v : A / A ^{\circ}  \to \Lambda $ be the inverse of the map $ \Lambda \to A / A ^{\circ} $, $ \lambda  \mapsto \varpi^{-\lambda} A^{\circ} $.  Let $ s_{i} $ be the reflection associated with  $ \alpha_{i} $ for $  i = 0, \ldots,  n   $.  The action of $  s_{i} $ on $ \Lambda $ is given explicitly  as  follows:          
\begin{itemize}    [before = \vspace{\smallskipamount}, after =  \vspace{\smallskipamount}]    \setlength\itemsep{0.2em}     
\item   $   s_{i} $ acts by the transposition $ f_{i} \leftrightarrow f_{i+1} $ for $ i = 1,2, \ldots, n   -  1    $ 

\item $ s_{0} 
$ acts by transposition  $ f_{1}   \leftrightarrow  f_{n} $. 
\end{itemize}
For $ \lambda \in \Lambda $, we let $ e ^ { \lambda } \in \ZZ [ \Lambda ] $ denote the element corresponding to $ \lambda  $ and $ e ^ { W \lambda } \in \ZZ [ \Lambda ] $ denote the element obtained by   taking     the formal  sum of elements in the orbit $ W \lambda $.
Let  $S_{\mathrm{aff} } = \left \{ s_{1} , s_{2} , \ldots , s_{n-1}     ,      t ( \alpha _ {  0 }  ^ { \vee }    )    s_{0}  \right \} $ and $ W   $, $ W_{\mathrm{aff}} $, $ W_{I}    $ be the Weyl, affine Weyl and Iwahori Weyl groups respectively determined by $ A $.  We consider $ W_{\mathrm{aff}} $ as a subgroup of affine  transformations of $ \Lambda \otimes \RR $.  We have   
\begin{itemize}    [before = \vspace{\smallskipamount}, after =  \vspace{\smallskipamount}]    \setlength\itemsep{0.2em}   
\item $ W  =   \langle  s_{1}  , \ldots, s_{n-1}   \rangle  \cong   S_{n-1}    $,
\item $ W_{\mathrm{aff}} = t(Q^{\vee} ) \rtimes W $
\item $ W_{I} =   N_{G}(A) / A^{\circ} =  A/A^{\circ} \rtimes W  \stackrel{v}{\simeq}     \Lambda \rtimes W $
\end{itemize} 
where $ v $ is the map (\ref{vmap}).    
The pair $ (W_{\mathrm{aff} }, S_{\mathrm{aff}} ) $  forms a  Coxeter system of type $ \tilde{A}_{n-1}   $.   We   consider  $ W_{\mathrm{aff}}   $    a   subgroup of  $  W_{I } $ via  $ W_{\mathrm{aff}} \simeq Q^{\vee}      \rtimes W  \hookrightarrow \Lambda \rtimes W   \stackrel{v}{\simeq}    W_{I} $.  
The natural  action of $  W_{\mathrm{aff}} $ on $ \Lambda \otimes \RR $ then extends to $ W_{I} $ with $  \lambda \in \Lambda $ acting as a translation  $ t(\lambda)   $.  We set    $ \Omega : = W_{I} /  W_{\mathrm{aff}} $,   which is a free abelian group on two generators  and we  have $ W_{I} \cong  W_{\mathrm{aff}}  \rtimes \Omega $.   
We  let   $ \ell :  W_{I}  \to \ZZ $ denote the induced length function with respect $ S_{\mathrm{aff}} $. Given $ \lambda \in  \Lambda $, the minimal length $ \ell_{\mathrm{min}}(t(\lambda))$ of elements in the coset $ t  ( \lambda ) W $ is  achieved by a unique element. This length can be  computed using Lemma \ref{Iwahorilengthlemma}.    We  let      $$     w_{1} : = 
    \scalebox{0.9}{$\begin{pmatrix}
 0&  1&  &&  & \\ 
 1&  0&  &  &  & \\ 
 &  &  1&  &  & \\ 
&  &  &   \ddots &  & \\ 
 &  &  &  & 1 & \\ 
 &  &  &  &  & 1
\end{pmatrix}$}, \quad   
w_{2}   :    =    \scalebox{0.9}{$\begin{pmatrix}
 1&  &  &&  & \\ 
 &   0 & 1 &  &  & \\ 
 &  1 &  0 &  &  & \\ 
&  &  & 1  &   & \\ 
 &  &  &   &  \ddots   &  \\ 
 &  &  &  &    & 1  
\end{pmatrix}$},   \ldots , \,\, \,  w_{n-1}  : =  \scalebox{0.9}{$   \begin{pmatrix}
 1&  &  &  &  & \\ 
 &   1  &  &  &  & \\ 
 &    &  \ddots    &  &  & \\ 
&  &  &   1  &  & \\ 
 &  &  &   &   0 &   1\\ 
 &  &  &  &  1&   0  
\end{pmatrix} $}, $$    
$$  w_  {  0   }   : = \scalebox{0.9}{$ \begin{pmatrix}
 0&  &  &  &  &    
 \scalebox{1.09}{$\frac{1}{\varpi}$} 
   \\ 
 &  \,  1&  &  \\ 
 &  &  \, 1 
 \\ 
 &  &  &  \ddots     \\ 
 &  &  &  &    1  \\ 
 \varpi &  &  &  &  &   0
\end{pmatrix} $}, \quad  \rho   =  
\scalebox{0.9}{$
\begin{pmatrix}
0 &  \!  1  &  & &  &   \\ 
  & \!    0 & 1  &  &    & \\ 
 &   &    0  
 &  1  
 &  &\\ 
&     &   & \ddots &  \ddots  & \\ 
 &   &  &   &0  & 1  \\ 
 \varpi \! &  & &   &    & 0 
\end{pmatrix}$}  $$ 
which we consider as elements   of $ N_{G}(A) $ (the normalizer of $ A $ in $ G $) whose component in $ \GG_{m} $ is $ 1 $.  The   classes of $ w_{0} , w_{1} , \ldots, w_{n-1}  $ in $ W_{I} $   represent $ t(\alpha_{0} ^{\vee} ) s_{0},   s_{1} , \ldots , s_{n-1} $ respectively and    the class of  $ \rho $  is  a   generator of  $ \Omega / \langle t( f_{0} ) \rangle  $.    The reflection $ s_{0} $ in $ \alpha_{0} $ is  then  represented  by  $ w_{\alpha_{0}} :  =  \varpi^{f_{1}}  w_{0}    $.    We will   henceforth    use the   letters    $ w_{i} $, $ \rho $ to denote both the matrices and the their classes in $ W_{I} $ if no confusion can arise. We note that  conjugation by $   \rho  
$ on $ W   _ { I  }     $  acts by  cycling  the (classes of) generators via  $ w_{n-1} \to w_{n-2} \to \ldots  \to w_{1} \to w_{0} \to  w_{n-1} $, thereby  inducing  an   automorphism of the extended   Coxeter-Dynkin  diagram  
 $$  
\dynkin[extended,Coxeter,
edge length= 1cm,
labels={,1,2,n-2,n-1}, labels*={0}]
A[1]{}$$
where the labels below the vertices correspond to the index of $ w_{i } $. Note also that $ \rho^{n} = \varpi^{(1,1,\ldots,1)}  \in A  $ is central. 
For $ i = 0,1, \ldots, n - 1 $, let 
$ x_{i} : \GG_{a} \to \mathbf{G} $ be the root group maps defined by  
\begin{align*}   
    x_{1} & : u  \mapsto            
  \scalebox{0.9}{$\begin{pmatrix}
 1&   u  &  && \\ 
 &   1 &  &  &   \\ 
 &   &   \ddots  &  & \\ 
&  &  & 1  &  \\ 
 &  &  &  &  1    
\end{pmatrix}$}, \, x_{2 }  : u   \mapsto  
 \scalebox{0.9}{$\begin{pmatrix}
  1&  &  & &   \\ 
 &   1&  u &  & \\ 
 &  &   1&  &   \\ 
 &  &  &   \ddots & \\ 
 &  &  &   &    1  \\
\end{pmatrix}$},
\, \, \,  \ldots,  \, \,  \, 
x_{n-1} :  u  \mapsto  
\scalebox{0.9}{$\begin{pmatrix}
 1&  &  &  & \\ 
 &    1 &  &  & \\ 
 &  &    \ddots  &  &    \\ 
 &  &  &  1    & u  \\ 
 &  &  &  &  1     \\ 
\end{pmatrix}$}
, \\ 
 x_{0} &  :   u \mapsto  
\scalebox{0.9}{$\begin{pmatrix}  
 1 &  &  & &  \\ 
  &  1 &  &  & \\ 
 &  &  \ddots &  &  \\ 
 &  & &  1   & \\ 
 \varpi u &    &  &   &1 \\ 
 \end{pmatrix}$}
\end{align*}
where again  the matrices are considered as elements of $ \Gb $ with $ 1  $ in the $ \GG_{m}  $ component.  Let  $ g_{w_{i}}  : [ \kay ] \to G $ be the maps $ \kappa \mapsto x_{i}(\kappa ) w_{i} $. Then $ I w_{i} I = \bigsqcup_{\kappa \in [\kay]}  g_{w_{i}}(\kappa) I  $.  For $ w \in  W_{I}   $ such that $ w$ is the unique minimal length element in the   coset   $ w W $,  choose a reduced word decomposition  $ w = s_{w,1} s_{w,2}\cdots s_{w, \ell(w) }  \rho_{w} $  where $   s_{w,i}  \in  S_{\mathrm{aff}} $, $ \rho_{w} \in  \Omega $. Define  
\begin{align}   
\label{GLnXw}  \mathcal{X}_{w} :  [  \kay ] ^ { \ell(w) } &  \to  G / K \\ 
 (\kappa_{1} ,   \ldots,   \kappa_{\ell(w)}  )& \mapsto    g_{s_{w,1} } ( \kappa_{1} )  \cdots  g_{s_{w}, \ell(w) } ( \kappa_{\ell (w) }   )   \rho_{w}  K   \notag 
\end{align}
where we have suppressed the dependence on the  decomposition chosen in the notation.   By   Theorem  
\ref{BNrecipe},  the    image of $ \mathcal{X}_{w} $ is independent of the choice of   decomposition and    $ \#   \,        \mathrm{im} ( \mathcal{X}_{w} )    =  q ^ { \ell(w) } $. 
We  note that   $  \ell(w) =    \ell_{\min}( t(-\lambda_{w} )  )    $ where $  \lambda_{w}    \in \Lambda $ is the   unique     cocharacter   such that $ w K  =  \varpi^{\lambda_{w}} K $.  
\begin{remark} Cf. the matrices in  \cite[p.\ 75]{BruhatIwahori}. 
\end{remark}

\subsection{Standard Hecke  polynomial} 
\label{GLnHeckesection}     Let $\mathcal{R}=\mathcal{R}_{q}$ denote the ring $\mathbb{Z} [ q^{\pm \frac{1}{2}}  ] $ and let $ y_{i}:=e^{f_{i}} \in \mathcal{R}[\Lambda]$ the element corresponding to $f_{i}$. Then $\mathcal{R}[\Lambda]=\mathcal{R}[y_{0}^{\pm}, \cdots, y_{n}^{\pm}]   $.   We are interested in the characteristic polynomial of the standard representation of the dual group $  \widehat{\mathbf{G}}_{F} =   \mathbb{G}_{m} \times \mathrm{GL}_{n} $ whose highest  coweights are $\mu_{\mathrm{std}}  =f_{0}+f_{1}$. Note that $ \mu_{\mathrm{std}} $  is the cocharacter obtained from the  Shimura     data in \ref{embeddinganticyclosec}.  
Since $ \mu_{\mathrm{std}} $ is minuscule, 
the (co)weights of the associated representation are the elements in the Weyl orbit of $\mu_{\mathrm{std}}$.  These are $ f_{0}+f_{1} $, $  f_{0}+f_{2} , \ldots , f_{0} + f_{n}  $.     
The Satake polynomial (see Definition \ref{Satakepolydefi}) for $\mu_{\mathrm{std}} $ is therefore
\begin{align*}
&\mathfrak{S}_{\mathrm{std}}(X)=\left(1-y_{0} y_{1} X\right)\left(1-y_{0} y_{2} X\right) \cdots   \left (  1  -   y_{0}  y_{n} X   \right     )    \in  \ZZ [ \Lambda ] ^ { W } [X] 
\end{align*} 
As in \S  \ref{Sataketransormsection}, we let $\mathscr{S}: \mathcal{H}_{\mathcal{R}}  \rightarrow \mathcal{R}[\Lambda]^{W}$ denote the Satake isomorphism.
\begin{definition} The polynomial $\mathfrak{H}_{\mathrm{std}, c}(X) \in \mathcal{H}_{\mathcal{R}}[ X]$ is defined so that $\mathscr{S}\left(\mathfrak{H}_{\mathrm{std}, c}(X)\right)=\mathfrak{S}_{\mathrm{std}}\left(q^{-\frac{c}{2}} X\right)$
for any $c \in \mathbb{Z} $. 
\end{definition}

\begin{proposition}[Tamagawa]        Let  $ \varrho = \varpi^{f_{0}} \rho \in N_{G}(A) $.  \label{GLnHeckePolynomialprop}   Then  

$$ \mathfrak{H}_{\mathrm{std}, c}(X)= \sum_{k=0}^{n}(-1)^{k} q^{-k(n-k+c)/2} \left(K \varrho^{k} K\right) X^{k} . $$
In particular if $ n $ is even and $ c $ is odd, $ \mathfrak{H}_{\mathrm{std},c}(X) \in \mathcal{H}_{\ZZ[q^{-1}]}[X] $.

\end{proposition} 
\begin{proof} Let  $ p_{k} =  p_{k}(y_{1}, \ldots, y_{n}) \in \ZZ[\Lambda]^{W} $ denote the $ k $-th elementary symmetric polynomial in $ y_{1} , \ldots, y _{n} $. 
Then $ \mathfrak{S}_{\mathrm{std}}(X) = \sum_{k=0}^{n} (-1)^{k}  x_{0}^{k}  p_{k}   X^{k}  $. So it suffices to  establish that $$  \mathscr{S} ( K \varrho^{k}  K ) =  q^{k(n-k)/2} x_{0}^{k}  p_{k}  . $$
For $ k \geq 1 $,  set $ \mu _{k} : = f_{0} + f_{1} + \ldots + f_{k}  \in \Lambda   ^ { + }   $. Then,  $ K   \varrho ^{k }   K     = K \varpi^{ \mu_{ k} } K $ as double cosets. But  $  \mu  _{k}$ are  themselves minuscule. Therefore,   Corollary     \ref{Satakeuppercoro}   and the second part of Corollary \ref{minuscriteria}  together  imply   that  $ \mathscr{S} ( K \varpi^{ \mu _{  k }  }    K ) $ is supported on $ x_{0}^{k} p_{k} $ and  that  the coefficient of $ x_{0} p_{k} $ is $ q^{\langle    \mu _ { k }   ,   \delta    \rangle   }  $ where $ \delta $ is as in (\ref{halfsumGLn}).  One  easily calculates that   $  \langle \mu _{k}, \delta \rangle = k(n-k)/2 $.      
\end{proof}

\begin{remark}   \label{Heckepolyhistory}     
 
 The formula for $ \mathfrak{H}_{\mathrm{std},c} $  was first obtained by Tamagawa \cite[Theorem 3]{TamagawaSatake} and the case $ n = 2 $ is due to Hecke \cite{HeckeHecke},   hence the  terminology `Hecke polynomial' -- see the  note at the bottom of \cite[p.62]{ArithmeticShimura} and  the historical commentary in \S4, \S8 of  \cite{CasselmanIran}.   Cf.  \cite[eq.\ (3.14)]{Gross}. 
\end{remark}

\begin{remark}
 An alternate proof of Proposition  \ref{GLnHeckePolynomialprop}   that does not use  Corollary      \ref{Satakeuppercoro} may be obtained using the decomposition of $ K \rho ^ { k }  K $  described in  Proposition  \ref{GLnHeckedecompositions} which is closer in  spirit to the proof by  Tamagawa.     
 \end{remark}

\subsection{Decomposition of minuscule operators}               In this section, we  study the decomposition of Hecke operators $ K   \varrho   ^ {  k  }      
K  $ for $ k \in \left \{ 1, \ldots, n \right \}  $ into individual  left cosets. Here  $ \varrho  = \varpi^{f_{0}} \rho $ as above.   Since  $ (\varpi^{k},1) \in G $ is central, it suffices to describe the decomposition $ K  \rho ^{k}    K $, so that  the left coset  representatives $ \gamma $  will have $ 1 $ in the $ \GG_{m} $-component.

\label{GLnminusculesection}    

\begin{definition} Let $  k $ be an integer satisfying $ 1  \leq  k \leq n $. A \emph{Schubert symbol} of length $ k  $   is a $ k $-element subset $ \mathbf{j} $ of $ [n] : = \left \{ 1 , \ldots, n \right \} $. We write the elements of $ \mathbf{j} = \left \{ j_{1}, \ldots, j_{k} \right \} $ such that $ j_{1} < \cdots < j_{n} $.   
The   \emph{dimension} of $ \mathbf{j} $ is defined to be $ \| \mathbf{j} \| = j_{1} +  \ldots + j_{k} - \binom{k+1}{2} $.    The set of Schubert symbols of length $ k $ is denoted by $  J _ { k }   $.  We have $ | J _{k} |  = \binom{n}{k} $.  
 \end{definition}
We define a  partial order $ \preceq $ on $ J_{k} $ by declaring $ \mathbf{j} \preceq \mathbf{j}' $ for    symbols   $ \mathbf{j} = \left \{ j_{1}, \ldots, j_{k} \right  \} $, $ \mathbf{j}' = \left \{ j_{1}', \ldots, j_{k}' \right \} $  if $ j_{i} \leq j_{i}' $ for all $ i = 1 , \ldots, k $.  Then $ (J_{k}, \preceq) $ is a   lattice (in the sense of order theory). The smallest 
and the largest elements of $ J_{k} $ are $ \left \{1, \ldots, k\right \} $ and $ \left \{n-k+1, \ldots,  n  \right  \} $   respectively. We assign a grading to $ J_{k} $ so that the smallest element has length is $ 0  $.

\begin{definition}  For  $ \mathbf{j} \in J_{k}  $,  the   \emph{Schubert cell $ \mathcal{C}_{\mathbf{j}} $}     is the finite subset of $  \mathrm{Mat}_{n \times k} ( F   ) $ consisting of all $ n \times k $ matrices $ C $ such that   
 \begin{itemize}    [before = \vspace{\smallskipamount}, after =  \vspace{\smallskipamount}]    \setlength\itemsep{0.2em}     
 \item $ M $ has $ 1 $ in $ (  j_{i} , i )  $-entry, which are referred to as  \emph{pivots}. 
 \item the entries of $ M $ that are  below or to the right of a pivot are zero, 
 \item $ M $ has entries in $ [\kay] \subset \Oscr_{F} $ elsewhere. 
 \end{itemize}   
 Then $ | \mathcal{C}_{\mathbf{j}} | =  q ^ {  \| \mathbf{j} \|   } $.  Given $ C \in \mathcal{C}_{\mathbf{j}} $, we let $  \varphi_{\mathbf{j}}(C) \in \GL_{n}(\Oscr_{F})  $ be the $ n \times n $ matrix obtained by  inserting the $ i $-th  column of $ \mathcal{C}_{\mathbf{j}} $ in the $ j_{i} $-th column of $ \varphi_{\mathbf{j}}(C) $,  making the rest of the diagonal entries $ \varpi $ and inserting zeros elsewhere.

We let $ \mathcal{X}_{\mathbf{j}}  \subset  \GL_{n}(F)  $ denote the image of $ \varphi_{\mathbf{j}}(\mathcal{C}_{\mathbf{j}}) $
and consider $  \mathcal{X}_{\mathbf{j}} \subset G $ by taking $ 1 $ in the $ \GG_{m}$-component.              
\end{definition}

\begin{example}   \label{GL4Schubert}      Let $ n =   4  $, $, k   = 2 $.  Then the Schubert cells are    $$   
\mathcal{C}_{\left\{ 1 ,  2 \right \}  }  =    \scalebox{0.9}{$\begin{pmatrix}    1  &    \\  &   1  \\  &  \\   &      \end{pmatrix}$},  \quad  \quad   
\mathcal{C}_ {  \left  \{ 1,3  \right  \}   }  =   \scalebox{0.9}{$ \begin{pmatrix}   1  &   \\   &  * \\   &  1      \\   &  \end{pmatrix}$} ,   \quad   \quad     
\mathcal{C}_{\left\{2,3 \right \}  }  =       \scalebox{0.9}{$\begin{pmatrix}    *  &   *  \\ 1 & \\  &  1  \\   &         \end{pmatrix}$}      $$
$$  
\mathcal{C}  _ {  \left  \{  1 , 4 \right  \}   }     =       \scalebox{0.9}{$\begin{pmatrix}   1 &   \\   & *  \\   & *   \\ &   1      \end{pmatrix}$},   \quad \quad  
  \mathcal{C}_ {  \left  \{ 2 , 4 \right  \}   }  =    \scalebox{0.9}{$\begin{pmatrix}  *  &   *  \\   1 &    \\    &  *       \\   &  1  \end{pmatrix}$} ,    \quad  \quad    
\mathcal{C}  _ {  \left  \{ 3,  4   \right  \}   }     =   \scalebox{0.9}{$\begin{pmatrix}   *  &  * \\  * &  * \\  1 &   \\  &  1    \end{pmatrix}$} 
$$       
where the star entries are elements of $ [\kay] $ and zeros are omitted.  The corresponding   collections $ \mathcal{X}_ {\mathbf{j}} $     are   
$$  
\mathcal{X}_{\left \{1,2\right\}}
=\scalebox{0.9}{$\begin{pmatrix}
{1}& & &  \\
& {1}& & \\
& &  \varpi & \\
& & & \varpi   
\end{pmatrix}$}, 
\, \, \,          
\mathcal{X}_{\left \{1,3\right \}}
= 
  \scalebox{0.9}{$\begin{pmatrix} 
{1}&  &  & \\     
& \varpi & * & \\ 
& & {1} & \\
&  &  & \varpi  
\end{pmatrix}$},  
\, \, \,
\mathcal{X}_{\left\{2,3 \right\}}
=      
  \scalebox{0.9}{$\begin{pmatrix} \varpi &*&*&\\     
&{1}& &\\   
& &  {1}   & \\ 
& & &  \varpi    
\end{pmatrix}$}, $$     
$$
\mathcal{X}_{\left\{1,4\right\}}     
=   
  \scalebox{0.9}{$\begin{pmatrix}
{1}& & & \\
&\varpi & &     * \\
 & &  \varpi    &   *   \\ 
&  & &  {   1       }         \end{pmatrix}$},   
\,  \,  \,          
\mathcal{X}_{\left \{2, 4\right \}}   
=   
  \scalebox{0.9}{$\begin{pmatrix}    
\varpi &  * & & *\\ 
&{1} & &\\
 &  & \varpi   & * \\   
& & &{1}  
\end{pmatrix}$}, 
\, \,  \, 
\mathcal{X}_{\left\{3,4\right \}}     
=      
  \scalebox{0.9}{$\begin{pmatrix} 
\varpi & & * & *\\    
& \varpi & * & *\\   
& &{1} &   \\
& & &{1}   
\end{pmatrix}$}    $$   We have a total of   $  1 +  q +  q ^ { 2    } +  q ^ {  2   }  +   q ^ {  3  }  +  q ^ {   4  } $ matrices  in  these  six  sets.     
\end{example}

\begin{proposition}   \label{GLnHeckedecompositions}  For $ 1 \leq k \leq n $, $    \displaystyle {   K \rho^{k} K =  \bigsqcup _ {  \mathbf{j } \in J_{k}}  \bigsqcup _ { \gamma \in \mathcal{X}_{\mathbf{j}} } \gamma K .    }    $   
\end{proposition}

\begin{proof}    Let $ \lambda _ { k } =    \sum_{i=1}^{k} f_{ n - k+i} \in \Lambda ^{-} $. We   have        $ \rho^{k} W \rho^{-k} =  \langle  S_{\mathrm{aff}}    \setminus  w_{n -  k}  \rangle    $ and therefore 
 $ W \cap \rho ^{k} W \rho^{-k} =    \mathrm{Stab}_{W} (  \lambda _ { k } ) $.   
By Theorem     \ref{BNrecipe}, $$ K\rho^{k} K = \bigsqcup _ { w \in [ W  / W  ^{\lambda_{k}} ] }  \mathrm{im}   \big (  \mathcal{X}_{w \rho   ^ { k  }     }   \big )  .       $$
where $ W^{\lambda_{k}} :=  \mathrm{Stab}_{W}(\lambda_{k}) $ and $ [ W / W^{\lambda_{k}} ] $ denotes the set of representatives in $ W $ of $ W / W^{\lambda_{k}} $ of minimal  possible length.  
For $ \lambda  \in W \lambda_{k} $, let $ \mathbf{j}(\lambda )  \in  J_{n-k}  $ be the  Schubert symbol consisting of integers $ 1 \leq  j_{1} <  \ldots < j_{n-k} \leq n  $ such that the coefficient $ f_{j_{i}} $ in $ \lambda $ is   $  0   $. If $ w \in [W / W ^{\lambda_{k}  }    ] $ and $ \lambda = w  \lambda_{k }   \in W  \lambda_{k}       $, we let $ \mathbf{j}(w) : =  \mathbf{j}( w  \lambda_{k} ) $. We let $  \preceq $ denote the left (weak) Bruhat order on $ W $ with respect to $ S $. Then $ (W,  \preceq) $ is a graded lattice with grading given by length.     \\      

\noindent   \textit{Claim 1. 
The   map $ w \mapsto  \mathbf{j}(w)  $ sets up an order preserving  bijection   $ [W / W^{\lambda_{k}} ] \xrightarrow{\sim}  J_{n-k}  $.} \\[0.4em]       
\noindent    The set $ W / W ^{\lambda_{k}} $ is in one-to-one correspondence with the orbit $ W \lambda_{k}    \subset    \Lambda     $. The orbit consists of the $ \binom{n}{k} $ permutations of the cocharacter $ \lambda_{k} = f_{n-k+1} +  \cdots + f_{n} $. Picking a permutation of $ \lambda_{k} $ in turn is the same  thing as choosing $ n-k $ integers $ 1 \leq  j_{1} <  \ldots <  j_{n-k} \leq n  $ such that $ f_{j_{1} } , \ldots, f_{j_{n-k}} $ have coefficient zero in the permutation of $ \lambda_{k} $. This establishes the bijectivity of $ w \mapsto \mathbf{j}(w) $.  The identity element is mapped to $ \left \{1 , \ldots, k \right \} $ and one establishes by induction on the length that the mapping preserves the orders.     \\

\noindent   \textit{Claim 2. For all  $ w \in [W / W  ^  { \lambda _ { k } } ] $,  
$ \mathrm{im} ( \mathcal{X}_{w \rho } )  =  \left \{  \gamma K \, | \,  \gamma \in \mathcal{X}_{\mathbf{j}(w)    }  \right \}  $.}   \\[0.4em]   
\noindent  We proceed by induction on the the length of $ w $. If $ w $ is of length $ 0 $, then $ w $ is the identity element  and $ \mathbf{j} =  \mathbf{j}_{\lambda_{k}}  =      \left \{ 1, \ldots, n-k \right \} $.  Now  $  \mathrm{im}(\mathcal{X}_{\rho^{k}} ) = \left \{ \rho  ^ { k }      K \right \} $ is a singleton and  $ \mathcal{X}_{\mathbf{j}} = \varpi^{\lambda_{k}} K $. As $ \rho  ^ { k }      K  =  \varpi ^ { \lambda_{k}}  K    $, the base case holds. Now suppose that the claim holds for all $ w \in [W / W ^ { \lambda _{k} } ] $ of length $ m $.  Let $ v=  s w $ where $ s \in \left \{ s_{1} , \ldots , s_{n-1}   \right \} $, $ w \in [W / W ^{\lambda_{k}}] $ such that $ \ell(v)  = \ell(w) + 1 $ and $ \ell(w) = m $.  Let $ \mathbf{j}_{v} $, $  \mathbf{j}_{w} $ be the Schubert symbols corresponding to $v $, $ w $ respectively. By Claim 1,    
there  exists a unique  $ j \in \left \{ 1, \ldots, n-1 \right \} $ such that $ i \in \mathbf{j}_{w} $, $ j +1 \in \mathbf{j}_{v} $ and $ \mathbf{j}_{w} \setminus \left \{ j + 1  \right \} = \mathbf{j}_{v}  \setminus \left \{ j  \right \}   $. If $ \sigma K  \in  \mathcal{X}_{ w \rho }   $, then $ \sigma  K =  \varphi_{\mathbf{j}(w)}(C) K $ for some $ C \in  \mathcal{C}_{\mathbf{j}_{w}} $ by  by induction hypothesis.    Denote $  \tau    := \varphi_{\mathbf{j}(w)}(C)  $.  By definition, $ \tau ( j , j   )    = 1 $, $ \tau ( j+1,j+1 ) = \varpi $ and $ \tau ( j,j_{1} ) = \tau ( j_{2}, j  ) =  \tau ( j+1, j_{3} )  =  0 $ for $ j_{1} , j_{2} > j $, $ j_{3} \neq j + 1 $. 
$$ 
\tau =   
\begin{pNiceMatrix}[name = matrix]  
\ddots 
\\ &   &  &   \\  
  &  &  1 \,  & 0 & \cdots   &  0 \\ 
  &  & 0 \,  &  \varpi  \\ 
   & & \vdots & &    \\
   &  &  0  &   &   & \ddots 
   \CodeAfter 
   \tikz \draw[decorate, decoration= calligraphic brace,transform canvas = {yshift=-2em}, thick] (row-6-|col-4.east) -- node[midway, below=1pt]{\scalebox{0.8}{$j$}} (row-6-|col-3.west)   ;
   \tikz \draw[decorate, decoration= calligraphic brace,transform canvas = {yshift=0em,xshift=2.7em}, thick] (row-3-|col-6) -- node[midway, right=1pt]{\scalebox{0.8}{$j$}} (row-4-|col-6)   ; 
   \end{pNiceMatrix} 
 $$\\[0.7em]    
Then $ g_{w_{j}}(\kappa)  \tau K = x_{j}(\kappa) w_{j}   \tau   w_{j}   K $ i.e., the the effect of multiplying $  \tau   K $ by $ g_{w_{j}}x_{j}$ is to switch the rows and columns in indices  $ j $ and $ j + 1 $ and then adding $ \kappa $ times the $ j+1$-st row to  the $ j $-th row.  Clearly, $   x_{j}(\kappa)   w_{j} \tau  w_{j} \in   \mathcal{X}_{\mathbf{j}(v)} $. Since $ g K $ was arbitrary,  we  see that $ \mathrm{im}(\mathcal{X}_{w \rho^{k}}   )     = \left \{ \gamma K \, | \,  \gamma  \in  \mathcal{X} _ { \mathbf{j}(w)    }       \right \}    $ for $ w \in [W / W ^ { \lambda_{k} } ] $ with $ \ell(w) = m + 1 $.  By induction, we  get the claim.         
\end{proof}  
\begin{remark}  This can also  be proved  directly  by appealing  to the stratification of the Grassmannian that parametrizes  $ n-k $-dimensional subspaces in an $ n $-dimensional vector space over a finite field.  
\end{remark} 
\subsection{Mixed decompositions}

\label{mixedcosetglnsection}    
From now on,  let $ n = 2m $ be even.    If $ g \in \GL_{2m}(F) $, we will denote by $ A_{g} , B_{g}, C_{g}, D_{g} \in \mathrm{Mat}_{m \times m } (F) $  so that  $$ g  =    \begin{pmatrix}{ A_{g}   } &   B_{g}  \\  C_{g}    &   D_{g}       \end{pmatrix} . $$
If $ g \in G $, then $ A_{g} $, $ B_{g} $, $ C_{g} $, $ D_{g} $ denote the matrices associated with the $ \GL_{2m}(F) $ component of $ G $.  Moreover, we adapt the following      
\begin{convention} An element of $ \GL_{2m}(F) $ is considered as an element of $ G $ via the embedding $ \GL_{2m}(F) \hookrightarrow G $ in the second component.    
\end{convention}    

Let $ \iota :  \mathbf{H} \hookrightarrow \mathbf{G}  $ be the subgroup generated $ \mathbf{A} $ and root groups of $ \Delta  \setminus   \left    \{      \alpha_{m}  \right \}    $.     Then $ \mathbf{H} \simeq \GG_{m} \times \GL_{m} \times \GL_{m} $ embedding block diagonally in $ \Gb $.  We  denote $   H       =  \mathbf{H}(F) $,  $   U = H \cap K $ and $ H_{1} = H_{2} \simeq  \GL_{n} $ the two components so that $ H = F^{\times } \times H_{1} \times H_{2}  $.     If $ h \in H $, we denote by $ h_{1}, h_{2} $ the components of $ H $ in $ H_{1} $, $ H_{2} $  respectively.     We let $ W _ { H } \simeq S_{m} \times S_{m}  $ denote  the Weyl  group of $ W $ which we consider as a subgroup of $ W $ generated by $ s_{1} , \ldots, s_{m-1}, s_{m+1}  ,  \ldots, s_{2m}  $.   The roots of $ H $ are denoted by $ \Phi_{H} $. These are $ \pm (e_{i} - e_{j} )$ for $ 1 \leq i, j \leq m $ and for $ m+1 \leq i, j , \leq 2m $ and we have a partition $ \Phi_{H} =  \Phi_{H_{1}} \sqcup \Phi_{H_{2}} $ into a union of two root systems isomorphic to $ A_{m-1} $.    For $  \alpha = e_{i} - e_{j}  \in    \Phi_{H} $ and $ k \in \ZZ $,  we let $ U_{\alpha, k} $ denote the unipotent subgroup of $ H $ with $ 1 $'s on diagonal  and    zeros elsewhere except for the $ (i,j) $ entry, which is required to have $ \varpi$-adic  valuation less than or equal to $ k $.

For  $ k = 0, \ldots , 2m $,    let $ P_{k} $  denote the set of pairs $ (k_{1}, k_{2} )$  of   non-negative  integers   such that $ k_{1} + k_{2} = k $ and $ k_{1},k_{2} \leq m $. For $  \kappa =  ( k_{1}, k_{2} ) \in P_{k} $,  denote $ l(\kappa) :  =   \min   (k_{1} , m - k _{2}  ) $ and  let      
\begin{equation}   
\displaystyle { \lambda _{  \kappa   } : =  \sum_{ i = 1 }^{k_{1} } f_{i}  +  \sum_{ j = m-k_{2}  + 1  }^{ m}  f_{m+j} }  \in  \Lambda  \end{equation}   

For $ i = 0, \ldots, m $,  let   
 $ t_{i}  :=   \mathrm{diag}( \underbrace{ \varpi^{-1}  \ldots,  \varpi^{-1}  }_{ i }  ,    \underbrace{0 , \ldots, 0}   _ {m - i }  ) \in \mathrm{Mat}_{m \times m}(F) $ and  
\begin{equation} \label{tauiGLn}  
\tau_{i}  : =  \begin{pmatrix} 1_{m}    &    t_{i}      \\[1em]   
&   1_{m}       
\end{pmatrix} \in \GL_{2m}  (F).
\end{equation}
Set   $ H_{\tau_{i}} :  = H \cap \tau_{i} K \tau_{i}^{-1} $.  For $ g \in G $, let $  U \varpi^{\Lambda} g K $ denote the set of all double cosets $ U \varpi ^ { \lambda } g K $ for $ \lambda   \in   \Lambda   $.

\begin{lemma}   \label{GLndistincttaui}         For $ i = 0, \ldots , m $, the collections  $  U  \varpi ^ { \Lambda } \tau_{i} K $ are disjoint.      
\end{lemma}
\begin{proof}  It suffices to show that $ H \tau_{i} K $ are distinct double cosets. Suppose for the sake of contradiction that that $  \tau_{i} \in  H \tau_{j} K $ for $ i \neq j $. Then $  \tau_{i} ^ { -1} h \tau_{j} \in K $.  Say $ h = (u, h_{1} , h_{2} ) $.  Now      $$  \tau_{i} ^{-1}  (h_{1}, h_{2} ) \tau_{j}    =    \begin{pmatrix}  h_{1}  &   h_{1} t_{j}   -  t_{i} h_{2}  \\%[em]
&   h_{2}   \end{pmatrix} $$
and therefore $ \tau_{i} ^{-1}  h \tau_{j}  \in K  $ implies that $ h_{1}, h_{2} \in  \GL_{m} (    \Oscr_{F} )$ and $ h_{1} t_{j}  - t_{i }    h_{2}  \in   \mathrm{  Mat}   _{m \times m } ( \Oscr_{F} ) $.  But the second condition implies that the reduction  modulo   $ \varpi   $     of one of $ h_{1} $, $ h_{2} $ is singular (the determinant vanishes modulo $ \varpi $), which contradicts  the  first  condition.  
\end{proof}

\begin{proposition}   \label{GLnmixeddecompositions}            For each $ k = 0, 1, \ldots, 2m $,  
$  \ch(K \rho ^{k} K)  =     \displaystyle {  \sum_{  \kappa \in  P_{k}  }  
\sum _ { i = 0 } ^ { l  (\kappa)  } \ch ( U \varpi ^ { \lambda_{ \kappa }   }  \tau_{i}  K)    }   . $    \end{proposition}

\begin{proof} We first claim that for each  $ k = 0 , 1 \ldots, 2m $, the double  cosets  $  U \varpi^{\lambda_{\kappa}   }     \tau_{i}    K $ for distinct choices of $ \kappa \in P_{k}  $ and $ i = 0 , 1, \ldots, l (\kappa)  $.   By Lemma   \ref{GLndistincttaui}, two such cosets are disjoint for distinct $ i $, so it suffices to distinguish the cosets for different $ \kappa $ but  fixed $ i $.    By  Lemma \ref{distinctgen}, it suffices to show that $ U \varpi^{\lambda_{\kappa} } H_{\tau_{i}} $ are  pairwise disjoint  for $ \kappa  \in P_{k} $.   Since $ H_{\tau_{i}}  \subset U $, it in turn suffices to show that $ U \varpi^{\lambda_{\kappa }    }      U $ are  pairwise disjoint  for  $   \kappa  \in  P _ { k}  $.     But this follows by Cartan    decomposition for  $ H $.  

Fix a $ k $.  For $ \kappa = (k_{1} , k_{2} ) \in  P _ { k} $,  let $  \mathbf{j} = \left \{1,\ldots, k_{1} \right \} \cup \left \{ m -  k_{2}  + 1, \ldots, 2m \right \} $. From the description  of  the Schubert cell $ \mathcal{X}_{\mathbf{j}} $ and Proposition \ref{GLnHeckedecompositions}, it is easy to see that   $  \varpi^{\lambda_{\kappa}   }     \tau_{i} K \subset K   \rho ^ { k   }   
K $ (and therefore $  \varpi^{\lambda_{\kappa}} \tau_{i} K \subset  \tau_{i} U \varpi^{\lambda_{\kappa} } \tau_{i} K $)   for all  $ \kappa \in  P_{k}    $, $ 0\leq  i  \leq l(\kappa) $. So to prove the claim at hand,  it  suffices to show that for any $ \gamma \in G $ such that $ \gamma K \subset K \rho^{k} K $, there exist $  \kappa  $ and $ i $ such that $  U \gamma  K   =     U \varpi^{\lambda_{\kappa}   }     \tau_{i} K $.  By Proposition \ref{GLnHeckedecompositions}, it suffices to restrict attention to $ \gamma \in  \mathcal{X}_{\mathbf{j}} $ for some Schubert symbol $ \mathbf{j} \in J_{k}  $.  Furthermore, since any  $  \gamma  \in  \mathcal{X}_{\mathbf{j}} $  has non-zero   non-diagonal     entries  only   above a pivot and these entries are in $ \Oscr_{F} $, we can replace $ \gamma $ by an element $ \gamma ' $ such that $ A_{\gamma'} $, $ D_{\gamma'}  $ are diagonal matrices and $ U \gamma K =  U \gamma'    K  $.  Let us define a  set $ \mathcal{Y}_{\mathbf{j}} \subset \GL_{n}(\Oscr_{F} ) $ that contains all such $ \gamma' $ as follows. An element $ g \in G $ lies in $ \mathcal{Y}_{\mathbf{j}} $  if        
\begin{itemize}     [before = \vspace{\smallskipamount}, after =  \vspace{\smallskipamount}]    \setlength\itemsep{0.3em}     
\item the diagonal of $ g $ has $ 1 $ (referred to as pivots) in positions $ (j,j) $ for $ j \in  \mathbf{j} $ and $ \varpi $ if $ j  \notin \mathbf{j} $,
\item $ A_{g} $, $ D_{g} $ are diagonal matrices  and $ C_{g} = 0 $, 
\item $ B_{g} $ has non-zero entries only in  columns of $ H $   that contain a pivot and rows that  do  not.    
\end{itemize} 
For any $ \mathbf{j} \in  J_{k} $, let $ \mathbf{j}_{1} $ (resp., $ \mathbf{j}_{2}     )        $ denote the   subset of  elements not  greater   than $ m $ (resp.,   strictly   greater than $ m $) and let $ \kappa(  \mathbf{j} )  :  =  ( | \, \mathbf{j}_{1}  | ,  | \,  \mathbf{j}_{2} | )  \in  P_{k}        $. It suffices to establish the following.   \\

\noindent   \textit{Claim.  For any Schubert symbol  $ \mathbf{j} \in  J_{k} $ and  any $  \gamma \in \mathcal{Y}_{\mathbf{j}} $ there exists an integer   $ i  \in \left \{0, 1, \ldots,   l(\kappa (  \mathbf{j} )     )  \right \} $  such that $ U \gamma K = U \varpi^{\lambda_{\kappa}} \tau_{i} K $.} \\[0.4em]       
\noindent We prove this by induction on $ m $. The case $ m = 1 $ is   straightforward.  Assume the truth of the claim for some positive integer $ m -  1 \geq 1 $. If $ \mathbf{j}_{1} = \varnothing $, then $ A_{\gamma} = I_{m} $, $ B_{\gamma} = C_{\gamma} = 0 $ and $ D_{\gamma} $  is diagonal.     Since $ w_{m+1}, \ldots, w_{2m} $ lie in both $ U $ and $ K $,   one can put all the $ k \leq m $ pivots in the top diagonal entries of $ D_{\gamma} $ and we are done.  We can similarly rule out the case $ \mathbf{j}_{2} = \left \{ m+1, \ldots, 2m \right \} $.  Finally, if $ B_{\gamma} = 0 $, we can again  use  reflections in $ H $   to  rearrange  the   $ A_{\gamma }$ and $ D_{\gamma} $ diagonal  entries to match $ \varpi^{\lambda_{l,k}} $.

So suppose that $  k  _ { 1  }  :  =    | \,  \mathbf{j}_{1}| > 0   $,     
$    k_{2} : =  |\,    \mathbf{j}  _{2}    | < m $ and $ B_{\gamma } \neq 0 $.  Pick $ j_{1} \in \mathbf{j}_{1} $ such that the $ j_{1} $-th  row of  $ B_{\gamma} $ is non-zero    and let $ j_{2}  \notin  \mathbf{j}_{2} $, $ m+1 \leq j_{2}  \leq 2m $ be such that the $ (j_{1}, j_{2} ) $ entry of $ \gamma $ in $ B_{\gamma}$ is not $ 0  $.    If $ j_{1} \neq 1 $, then using row and columns operations, one can switch the first and $ j  _     {  1   }       $-th row and columns to obtain a  new matrix $ \gamma ' $.  Clearly,  $ \gamma ' $  is an  element of $ \mathcal{Y}_{\mathbf{j} ' } $ for some new $ \mathbf{j}' $, $ U \gamma K = U \gamma ' K $ and the $ (1,j_{2}) $ entry of  $ \gamma ' $ is non-zero.      Similarly if $  j_{2}    \neq m+1 $, we can produce  a  matrix using row and columns operations so that $ (m+1, m+1) $ diagonal entry of the new matrix is $ 1 $ and the class of this matrix in $ U \backslash G / K $ is the same as $ \gamma $.  The upshot is  that we may  safely  assume  that $ j_{1} = 1 $, $  j_{2}   =    m+1 $ (so in particular,   $ 1 \in \mathbf{j} $,  $ m+1 \notin \mathbf{j}$).           

$$ 
 \gamma  =       
\begin{pNiceArray}{cccc|cccc}[margin]
   \encircled{\,\,\,}        &  & &  & & &   \\
&   \ddots  &  & & & &      \\ 
&   & \! \!  \varpi \! \!  &  &   &  &    \!   *      \!      &      \\ 
 &  &  &   \ddots     &     &    & \\   \hline
  &  & &  &  \!  \square        &  &  \\
   &  & & & &  \ddots  \\ 
   &  &   &    & &    &  \! 1  \!           &   \\
   &  &  &   &    &   &    &     \ddots    
   \CodeAfter 
   \tikz \draw[decorate, decoration= calligraphic brace,transform canvas = {yshift=-2em}, thick] (row-8-|col-4.south east) -- node[midway, below=1pt]{\scalebox{0.9}{$j_{1}$}} (row-8-|col-3.south west)   ;
   \tikz \draw[decorate, decoration= calligraphic brace,transform canvas = {yshift=-2em}, thick] (row-8-|col-8.east) -- node[midway, below=1pt]{\scalebox{0.9}{$j_{2}$}} (row-8-|col-7.west)   ;
   \tikz \draw[decorate, decoration= calligraphic brace,transform canvas = {yshift=0em,xshift= 3 em}, thick] (row-3-|col-8) -- node[midway, right=1pt]{\scalebox{0.9}{$j_{1}$}} (row-4-|col-8)   ; 
\end{pNiceArray}     
\quad \quad  \rightsquigarrow \quad     \quad 
\gamma'   =   \begin{pNiceArray}{cccc|cccc}[margin]
\varpi &   & &  &  *\\
&   \ddots  \\ 
&    &    \! \!  \encircled{\,\,\,}  \!\!   &  &      &  &         &      \\ 
 &  &  &   \ddots     &   &  & \\   \hline
  &  & &  &1  &  &  \\
   &  & & & &  \ddots  \\ 
   &  &   &    & &    &   \!   \!    \square  \!  \!                    &   \\
   &  &  &   &    &   &    &     \ddots    
   \CodeAfter 
   \tikz \draw[decorate, decoration= calligraphic brace,transform canvas = {yshift=-2em}, thick] (row-8-|col-4.south east) -- node[midway, below=1pt]{\scalebox{0.9}{$j_{1}$}} (row-8-|col-3.south west)   ;
   \tikz \draw[decorate, decoration= calligraphic brace,transform canvas = {yshift=-2em}, thick] (row-8-|col-8.east) -- node[midway, below=1pt]{\scalebox{0.9}{$j_{2}$}} (row-8-|col-7.west)   ;
    \tikz \draw[decorate, decoration= calligraphic brace,transform canvas = {yshift=0em,xshift= 3 em}, thick] (row-3-|col-8) -- node[midway, right=1pt]{\scalebox{0.9}{$j_{1}$}} (row-4-|col-8)   ; 
   \end{pNiceArray}    
 $$\\[0.5em]  
Since the top  left     diagonal entry of $ B_{\gamma} $ is non-zero,   we  can  use   elementary   operations for   rows and  columns  with labels in $ \mathbf{j}_{2}   $\footnote{the   non-zero  columns of $ B_{\gamma} $ are above a pivot of $ \gamma $}   to make all the other entries of the first row of $ B_{\gamma } $ zero and keep $ D_{\gamma} $ a diagonal matrix.    The  column  operations  may   change the other rows of $ B_{\gamma} $ but the new matrix still  belongs  to    $ \mathcal{Y}_{\mathbf{j}} $  and
has same class in $ U \backslash G  / K $.  Similarly, we can use  elementary operations for rows and columns with labels in $ \left \{1, \ldots, m \right \}    \setminus \mathbf{j}_{1} $ to make all the  entries below $ (1,m+1) $ in $ B_{\gamma} $ equal to zero,   while keeping $ A_{\gamma} $ a diagonal  matrix.      
Finally, conjugating by an   appropriate        element of  the compact diagonal   $ A^{\circ}   \subset U $, we can  also  assume that the top left entry of $ B_{\gamma}  $ is  $  1  $.

In summary,  we have arrived at a matrix    that has the same class in $ U \backslash G / K $ as the original $ \gamma $  and   has    zeros in rows and columns  labeled $ 1 $, $ m + 1 $ except for the diagonal entries  in  positions  $   (1,1) $, $ (m+1, m+1) $, $  (1,m+1) $ which are $ \varpi $, $ 1 $,   $ 1 $ respectively.  The  submatrix  obtained by deleting  the first and $ (m+1) $-th  rows  and columns is a  $ (2m-2) \times (2m-2) $ matrix in $ \mathcal{Y}_{\mathbf{j}'} $ for some $ \mathbf{j}' $ of cardinality $ k-1  $.  By induction, this matrix can be put into the desired  form  using the groups $ U $ and $ K $ associated with $ \GG_{m}  \times \GL_{2m  - 2 }   $. The  possible value of $ i $  that can  appear  from this   submatrix have to be at most  $ \max(k_{1}  -     1, m-1+k_{2})  $  by  induction  hypothesis  and therefore the bound for possible   $ i $ holds for $ m  $ as well.    This  completes  the    proof.          
\end{proof}

\begin{example} Suppose 
  $ m = 2 $ and $ k = 2 $, so that $ P_{k} = \left \{  (2,0), (1,1) , (0,2)  \right \} $.  Proposition \ref{GLnmixeddecompositions} says that \\[-0.5em]  
\begin{align*}  
\ch \, K \begin{psmallmatrix}   \varpi  & \\ &  \varpi  \\ &  & 1  \\  & & & 1   \end{psmallmatrix}  K    &   =  \ch \,   U \begin{psmallmatrix}  \varpi \\ &  \varpi  \\ & & 1  \\ & & & 1    \end{psmallmatrix}   K  
\,  +  \,  \ch \,  U 
 \begin{psmallmatrix}   \varpi  &  & 1  \\ &  \varpi   \\ & & 1  \\ & & & 1  
\end{psmallmatrix}  K   \,  +  \,    
\ch \,  U  \begin{psmallmatrix}   \varpi  & &  1 \\ &  \varpi  &   &  1   \\  & & 1 \\  & & &  1   \end{psmallmatrix}   K \\
&  +   \ch  \,    U  \begin{psmallmatrix}  \varpi   \\ & 1 \\  & & 1 \\ & &  &   \varpi    \end{psmallmatrix} K  
\, + \,   \ch \,   U  \begin{psmallmatrix}  \varpi  &  &  1  \\ &  1 \\  &  &  1  & \\  & & &  &  \varpi  \end{psmallmatrix}  K   \,   + \, \ch \,    U  \begin{psmallmatrix}    1 & \\ &    1 \\ & &   \varpi  \\  & & &   \varpi   \end{psmallmatrix}   K          
\end{align*} 
\vspace{-0.5em} 
\end{example}

\subsection{Mixed degrees}      \label{GLnmixeddegreessec}          
For $ 1 \leq  r   \leq m $, let $ \mathscr{X}_{r} : = \GL_{r}(F) $. We have inclusions $ \mathscr{X}_{1} \hookrightarrow \mathscr{X}_{2}  \hookrightarrow \ldots  \hookrightarrow  \mathscr{X}_{m} $ obtained by a considering a matrix $ \sigma  \in  \mathscr{X}_{r} $ as a  $ (r+1) \times (r+1) $ matrix whose top left $ r \times r $   submatrix is $ \sigma $, has $ 1 $ in last  diagonal entry and zeros elsewhere.  For each $ r $, let  
let
$$ 
\jmath _ { r }     :  \mathscr{X}_{ r } \to G  \quad \quad    \sigma  \mapsto  \iota( \sigma , \sigma )   =   \begin{pmatrix} \sigma &  \\   &  \sigma   \end{pmatrix}    \in G   $$ 
where $ \sigma $ is considered as an element of $  H_{1} $, $ H_{2} $  as above, so that $  j_{ r } $ factorizes as $  \mathscr{X}_{ r }   \hookrightarrow  \mathscr{X}_{m}  \xrightarrow{j_{m}}   G  $.     We henceforth consider all $ \mathscr{X}_{ r } $ as  subgroups of $ G $ and omit $ \jmath_{r}    $ unless necessary.  We  denote $ \mathscr{X}_{r } ^{\circ}  =  \mathscr{X}  \cap K  \simeq  \GL_{ r }(\Oscr_{F} ) $.   

For $ \alpha = e_{i} - e_{j} \in \Phi_{H} $, $ k \in \ZZ $, let $ U_{\alpha, k} $ be the unipotent subgroup of matrices  $ h \in H $ such that the diagonal entries  of $ h $ are $ 1 $, the $ (i,j) $ entry of  $ h $ has valuation at least $ k $   and all other   entries are $ 0 $.  For each $ r \geq 1 $, let $ \psi_{r } : \Phi_{H}  \to  \ZZ  $   be the function $$  
\psi_{s }(\alpha ) =  \begin{cases}  1  &  \text{ if }  \alpha  \in   \left \{  e_{i}  - e_{j} \in \Phi_{H}   \,  |   \,     \text{either } 1 \leq j \leq   r  
\text{ or } m + 1  \leq i  \leq  m +  r \right \}  \\
0   & \text { otherwise}     \end{cases}   $$
and let $ H _ {\psi_{ r }  } $  be the subgroup generated by $ U_{\alpha, \psi_{ r }(\alpha)}   $ and $ A \cap \tau_{r } K \tau_{ r }^{-1} $   
More explicitly, $ H_{\psi_{ r }}   $ 
is the subgroup of  elements $ (v,h_{1}, h_{2}) \in U $ satisfying the three   
conditions below:   
\begin{itemize}    [before = \vspace{\smallskipamount}, after =  \vspace{\smallskipamount}]    \setlength\itemsep{0.2em}     
\item all the non-diagonal entries in  the first $ r $ columns of $ h_{1} $ are divisible by $ \varpi $,
\item all non-diagonal  entries in the  first $ r  $ rows of $ h_{2} $ are divisible by $ \varpi $,
\item the difference of the $ j  $ and $ j+ m - $th  diagonal entries of $ h = ( h_{1} , h_{2} ) \in G  $ is divisible by $ \varpi $ for all $ j = 1, \ldots, r   $.  
\end{itemize}

\begin{lemma}   \label{GLnIwahoridoublecoset}         $  H_{\tau_{ r }   }  =    \mathscr{X}_{ r }   ^ { \circ}     H_{\psi_{ r }  }    =  H_{\psi_{ r  }  }      \mathscr{X}_{  r  }  ^ {  \circ }         $  for  $ r   = 1, \ldots ,  m  $.        
\end{lemma}

\begin{proof}  The $ \GG_{m} $ component on both sides are $ \Oscr_{F}^{\times} $ and we may therefore ignore  it.    Let $ h  = ( h_{1} , h_{2} ) \in H $.   Then   $ h \in H_{\tau_{ r  }   } $ and if and only if $ h \in U $ and $$ h_{1}   t _{ r }  -  t _{ r }   h_{2}  \in  \mathrm{Mat}_{ m \times  m  }     (   \varpi      \Oscr_{F} )   $$  
(see  the calculation in  Lemma  \ref{GLndistincttaui}). It is then clear that $  H_{\tau_{r }}  \supset   \mathscr{X}_{r}^{\circ} \cdot H_{\psi_{ r }} $.  Let $ h   =  (h_{1} ,  h_{2} )  \in H_{\tau_{ r }} $. From the description of $ H_{\tau_{r }} $, we see that  the   $ r \times  r $  submatrix $ \sigma  $  formed by first $ r $ rows and columns of $ h_{1} $ must be invertible (and similarly for $ h_{2} $). Then $  \jmath_{r }(\sigma^{-1}) \cdot h $ has the top $ r  \times  r    $ block equal the identity matrix. Since this matrix lies in $ H_{\tau_{ r }} $, we see  again from the 
description of elements of $ H_{\tau_{ r }} $   that $ \jmath _ { r } ( \sigma ^{-1} ) h  \in  H_{\psi_{r } } $. This    implies the reverse inclusion $ H_{\tau_{r }}  \subset  \mathscr{X}^{\circ} _ { r }  H_{\psi_{ r }}$. Since the product of $  \mathscr{X}_{r }^{\circ} $ and $ H_{\psi_{ r } } $ is a group, $ \mathscr{X}_{r }^{\circ}  H_{\psi_{r }}  =  H_{\psi_{r }}  \mathscr{X}_{i}^{\circ}  $.    
\end{proof}  
Recall that $ \Phi_{H} = \Phi_{H_{1}} \sqcup  \Phi_{H_{2}}   $.  
Declare   $    \alpha_{1} , \ldots , \alpha_{m} \in \Phi_{H_{1}} $ and $ - \alpha_{m+1} , \ldots,  -  \alpha_{2m}  \in \Phi_{H_{2}} $ to be the set of positive roots of $  \Phi_{H}   $.    Then $ \alpha_{1,0} : = e_{1} - e_{m} \in \Phi_{H_{1}} $, $ \alpha_{2, 0}   =  e_{2m }  - e_{m+1}  \in \Phi_{H_{2}} $ are the highest roots.   Let $ s_{1, 0} $, $ s_{2, 0}  \in W_{H}  $ denote the reflections associated with $ \alpha_{1 ,   0} $, $ \alpha_{2,   0} $  respectively. Then the  affine Weyl group $ W_{H, \mathrm{aff} } $ (as a subgroup of $ W_{\mathrm{aff}} $)  is  generated by $$ S_{H, \mathrm{aff}}  = \left \{ t( \alpha_{1,   0 } ^ { \vee}  )  s_{1,0} ,  s_{1}, \ldots, s_{m-1}  \right \} \sqcup \left \{ t(  \alpha_{2, 0 } ^ { \vee }  )    s_{2, 0} ,   s_{m+1}, \ldots,  s_{2m-1}   \right \}  $$     and $ (W_{H_{\mathrm{aff}} } , S_{H, \mathrm{aff}} ) $ is a Coxeter  system of type $ \tilde{A}_{m-1} \times \tilde{A}_{m-1} $.         We denote by $ \ell_{H}  : W_{H} \to \ZZ $ the resulting  length    function. The  extended Coxeter-Dynkin   has two components
\begin{align} \label{GLncoxeterdynkinforH} 
\dynkin[extended,Coxeter,
edge length= 1cm,
labels={,1,2,m-2,m-1}, labels*={0_{1}}]
A[1]{}   \quad \quad   \quad   
\dynkin[extended , Coxeter,
edge length= 1cm,
labels={,m+1,m+2,2m-2,2m-1}, labels*={0_{2}}]
A[1]{}
\end{align}
where the labels $ 0_{1} $, $ 0_{2} $ correspond to the two affine reflections corresponding to $ \alpha_{0,1} $,  $ \alpha_{0,2} $.   

Now let   $ I _ { H_{1}} $  (resp.,  $ I_{H_{2}} $)   be the Iwahori subgroup of $ H  _ { 1  } $ (resp.,   $ H_{2} )      $  consisting    of   integral  matrices  that reduce modulo $ \varpi $ to upper triangular (resp.,  lower triangular)  matrices and set   $   I_{H} :  = \Oscr_{F}^{\times}  \,    \times    \,    I_{H_{1}} \times I_{H_{2}}  $. Then $ I_{H} $ is the Iwahori subgroup associated with alcove determined by $ S_{H,   \mathrm{aff}  }  $. We    let    
$$    
\rho_{1} : =     \scalebox{0.9}{$     \begin{pmatrix}  
0 & 1  &  & &  &   \\ 
  &    0 & 1  &  &    & \\ 
 &   &    0  
 &  1  
 &  &\\ 
&     &   & \ddots &  \ddots  & \\ 
 &   &  &   &0  & 1  \\ 
 \varpi &  & &   &    & 0 
\end{pmatrix}$}  \in H_{1} \quad   \quad      \rho_{2}   : =    \scalebox{0.9}{$      
\begin{pmatrix}
0 & &  & &  &    \varpi      \\ 
 1  & 0 & &  &    & \\ 
 & 1 &    0  \\ 
&     &     1 & 0       &  & \\ 
 &   &  &  \ddots  & \ddots   &   \\ 
&  & &   &    1  & 0 
\end{pmatrix}$}  \in   H_{2}     $$
(so we have    $ \rho_{1} = \rho_{2}^{t} $).  Both   $ \rho_{1} $, $ \rho_{2}  $  normalize $ I_{H } $ and the effect of conjugation $ w \mapsto  \rho_{1} w \rho_{1}^{-1} $ (resp.,  $ w \mapsto \rho_{2} w \rho_{2}^{-1} $) is by cycling in  clockwise (resp.,   counterclockwise)  direction     the left (resp.,  right) component of the diagram      displayed   in          (\ref{GLncoxeterdynkinforH}).  We   set $ \rho_{H} :  = (\rho_{1}, \rho_{2} ) \in H $ and  for   $ \kappa =  (k_{1} , k_{2} )  \in \ZZ ^ { 2 } $, we denote by  $ \rho_{H}^{\kappa} $ the element $ ( \rho_{1} ^{k_{1}} , \rho^{k_{2}}   _ { 2  }       )   \in H   $.  We will denote   by    $ - \kappa $ the pair $ (-k_{1},  - k _{2} ) $.             

\begin{definition}   \label{GLnHIwahoridefi}              
For $ r = 0 , \ldots ,     m    $, let    $  I_{H, r } $ denote the subgroup of $ H $ which contains $ I_{H} $ and whose Weyl group $ W_{ H ,  r   }  \subset W_{H} $  is generated by   $ S_{H,  r} :   =   \left  \{     s_{r }  \ldots , s_{m - 1 }  ,    s_{m+r }, \ldots ,  s_{2m  - 1  }   \right  \}     $.   More  explicitly,  $ I_{H,r} $ is the subgroup of $  U  $ consisting of all   matrices  as  below  \\ 
$$   
\begin{pNiceArray}{cccc|cccc}[margin] 
 & \, \, \,   & \, \, \,   & \, \, \,    & \, \, \,   & \, \,  \,   & \, \,  \,   & \, \, \,     \\
 & & & & &  &  & \\
 & &       &   & & &  & \\  
 & & &   & &  & & \\   \hline 
 & & & &  &  &  &\\ 
 & & & & & & & \\ 
 & & & & &  &  &    \\ 
 & & & & & & & 
 \CodeAfter
 \tikz \draw% [transform canvas = {yshift = - 0.2em, xshift = 0.2 em }] 
 (row-1-|col-1) -- (row-4-|col-4) ;
     \tikz \draw%[transform canvas = {yshift = - 0.2em, xshift = 0.2 em }]
     (row-1-|col-1) -- (row-4-|col-1) ;
 \tikz \draw%[transform canvas = {yshift = - 0.2em, xshift = 0.2 em }]   
 (row-4-|col-1) -- (row-4-|col-4) ;
% \tikz \draw   (row-5-|col-5) -- (row-5-|col-8);
  \tikz \draw %[transform canvas  = {yshift = -0.2 em  % }   , xshift = 0.2 em} ]
  (row-5-|col-5) -- (row-5-|col-8);
    \tikz \draw% [transform canvas = {yshift = - 0.2em, xshift = 0.2 em }]
    (row-5-|col-5) -- (row-8-|col-8);
      \tikz \draw  %  [transform canvas = {yshift = - 0.2 em, xshift = 0.2 em}]
      (row-5-|col-8) -- (row-8-|col-8); 
\tikz\draw[decorate, decoration= calligraphic brace,transform canvas = {yshift=0.2em}, thick] (row-1-|col-1) -- node[midway, above=1pt]{\scalebox{0.9}{$r$}}     (row-1-|col-4);
\tikz\draw[decorate, decoration= calligraphic brace,transform canvas = {yshift=-0.3  em,   xshift =0.1  em}, thick] (row-8-|col-8) -- node[midway,
below=1pt]{\scalebox{0.8}{$r$}}     (row-8-|col-5);
\tikz\draw[decorate, decoration= calligraphic brace,transform canvas = {yshift=-1.3em, xshift = 1.0 em}, thick] (row-7-|col-4) -- node[midway, left =1pt]{\scalebox{0.8}{$r$}}     (row-4-|col-4);    
\tikz\draw[decorate, decoration= calligraphic brace,transform canvas = {yshift  =  0 em,    xshift =  -1.1    em}, thick] (row-1-|col-5) -- node[midway,  right    =1pt]{\scalebox{0.8}{$r$}}     (row-4-|col-5);  
\end{pNiceArray}
$$
such that the   non-diagonal  entries inside the  two  triangles are divisible by $ \varpi  $.     
\end{definition}     
\begin{lemma}  For any $ k = 0, \ldots, 2m $,  $ \kappa \in  P_{k} $ and $ r   = 0,\ldots,  l ( \kappa) $,   we have    $  H _ { \tau_{ r }  }  \varpi^{-\lambda_{\kappa}} U =  I _ { H ,  r  }    \,      \rho_{H} ^ { -  \kappa }    U   $.  \end{lemma}     

\begin{proof}  Since $ r  \leq l (\kappa)  $, $ \varpi^{-\lambda_{\kappa} } $ commutes with $ \mathscr{X}_{r} ^{\circ} $ and therefore $ H_{\tau_{r}} \varpi^{-\lambda_{k}} U =  H_{\psi_{r}} \varpi^{-\lambda_{\kappa} } U $. It is also easily seen that $$  I_{H , r }   =    H_{\psi_{r}}  \cdot    A^{\circ} \cdot\prod_{\alpha \in  \Phi_{H,r}^{+}  }  U_{\alpha, 0}   $$ 
where $  \Phi_{H,r}^{+}    :  = \left \{ e_{i} - e_{j} \in  \Phi_{H} ^{ + } \, | \,      \text{either } 1 \leq  j \leq   r 
\text{ or } m + 1  \leq i  \leq  m + r     \right \} $.  Since   $ \varpi^{\lambda_{\kappa} } $ commutes with $ A^{\circ}  $ and     $ U_{\alpha, 0} $ for $ \alpha \in   S_{r}      $, we see that $ H_{\psi_{r}} \varpi^{-\lambda_{\kappa} }  U  $.  Since $ \varpi^{\lambda_{\kappa} } U =  ( \rho_{1}^{k_{1}}, \rho_{2}^{-k_{2}}   )     U =    \rho_{H}^{-\kappa}  U  $,      the  claim  follows.   
\end{proof}

For  $  \kappa = (k_{1} , k_{2} )   \in  P_{k} $, $ r = 0, \ldots, l (\kappa) $, let  $ W_{\kappa, r} \subset W_{H,r} $ denote the subgroup generated by $ S_{H,r} \setminus \left \{ s_{k_{1}} ,  s_{2m - k_{2}}  \right \} $. Then $ W_{\kappa, r} $ is a Coxeter subgroup of $ W_{H, r} $.  Let     $$  P_{  \kappa  ,  r   }        : =  \sum_ { w \in [ W_{H,r} / W_{\kappa, r} ] }  q ^ { \ell  _ { H  }     (w) }   $$
denote the Poincar\'{e} polynomial of $  [ W_{H,r} / W_{\kappa, r}  ] \subset W_{H}      $.

\begin{proposition}    For any $ k $, $ \kappa   \in  P_{k} $ and   $ r   = 0,\ldots,  l ( \kappa) $,  we have   $ \deg   \,     [U \varpi^{\lambda_{\kappa} } \tau_{ r }    K ]_{*} =  P_{ \kappa,  r }(q)   $.   
\end{proposition}

\begin{proof}   We have  $  \deg  \,  [U \varpi^{\lambda_{\kappa}}   \tau_{r}     K ]_{*} =  \deg [ H_{\tau_{   r   }   }     \varpi^{-\lambda_{\kappa}  } U ] $ which is by definition the cardinality of $ H _ { \tau_{ r } }  \varpi^{-\lambda_{\kappa} } U / U  $.  By  Lemma  \ref{GLnIwahoridoublecoset},  $ H_{\tau_{r}} \varpi^{-\lambda_{\kappa} }  U / U  =  I_{H,r} \,   \rho_{H} ^ { -  \kappa    }  U / U $.   Theorem    \ref{BNrecipe}  therefore  implies that $ \deg  \,   [ U  \varpi^{\lambda_{\kappa}}   \tau_{r}     K ]  _ { * }  $ is  the Poincar\'{e} polynomial of $   [ W_{H,r} / ( W_{H, r} \cap \rho^{-\kappa} W_{H} \rho^{\kappa}  )  ] $.    Now $ \rho ^ { - \kappa }  W_{H} \rho^{ \kappa } $ is the subgroup of $ W_{I,H} $ generated by     $$ S_{H, \mathrm{aff} }   \setminus
\rho_{H}^{-\kappa}\left \{ s_{1,0}, s_{2,0}  \right \}   \rho _ { H } ^ {    \kappa  }   =  S_   { H   ,    \mathrm{aff} }  \setminus  \left \{  s_{k_{1}}  ,  s_{2m-k_{2}}   \right \}  $$ 
where the equality  follows  since $ \rho_{1}^{-1} {s_{1,0}}    \rho_{1} = s_{1} $ and  $ \rho_{2} ^{-1} s_{2,0} \rho_{2}      =    s_{2m} $ (see  above for the description of the action of $ \rho_{1} ,\rho_{2} $ on      (\ref{GLncoxeterdynkinforH})).  Thus we have    $$ W_{H,r } \cap  \rho^{-k}   W_{H}  \rho^{k}    =  W_{\kappa, r }  $$  and  the   claim   follows.     
\end{proof}

\begin{corollary}   \label{GLnmixeddegree} With notation as above,        $ \deg \,    [ U \varpi^{\lambda_{\kappa} }  \tau  _  {  r    } K ] _{*}  \equiv  \scalebox{0.9}{$\displaystyle{
\binom{m- r}{m-k_{1}}  \binom{m-r}{k_{2}}}$}    \pmod {q-1}  $. 
\end{corollary}

\begin{proof}  $ | W_{H, r} | = ( m- r) ! \cdot ( m- r)! $ since $ W_{H,r} $ is the product of the groups generated $ s_{r}, \ldots , s_{m-1 } $ and $ s_{m+r} , \ldots , s_{2m-1} $, each of which  have   cardinality  $ (m-r)! $.  Similarly,  $ W_{ \kappa,  r} $ is the product of four groups generated by four sets of reflections labeled\vspace{-0.2em} $$ r + 1 , \ldots, k_{1}-1,  \quad     \quad   k_{1} + 1,   \ldots,  m-1 ,  \quad \quad  m+r +  1 , \ldots ,   2 m - k _{2} - 1 ,  \quad     \quad       2m - k _{2} + 1 ,  \ldots  ,   2m-1    \vspace{-0.2em}      $$
which have sizes $( k_{1} - r ) ! $, $ ( m - k_{1} ) ! $, $ (  m -   k_{2} - r ) ! $   and   $  k_{2} !  $   respectively.    
\end{proof}

\subsection{Zeta  elements} 

We now formulate the zeta element problem relevant to the situation of \S   \ref{embeddinganticyclosec} and show that one exists using the work done above.       Let  $    T : =   F  ^ { \times  }   $, $ C  =  \Oscr_{F}^{\times} \subset T $ the unique maximal compact subgroup,  $ D = 1 +  \varpi \Oscr_{F} $ a subgroup of index $ q - 1 $ and  $ \nu   :   
H \to T $ be the map given by $  (h_{1} , h_{2} ) \mapsto    \det(h_{2} ) / \det h_{1}   $.   Let   $ \mathcal{O} $ be any  integral  domain  containing  $ \ZZ[ 
q ^ { - 1}  ]    $.      Set     
\begin{itemize}     [before = \vspace{\smallskipamount}, after =  \vspace{\smallskipamount}]    \setlength\itemsep{0.2em}    
\item $  \tilde{G}  = G \times T $,   
\item $ \tilde{\iota} = \iota \times \nu : H  \to \tilde{G} $, 
\item $ U \subset H$ and  $ \tilde{K}     :   = K \times C \subset \tilde {G} $ as bottom levels
\item $ M_{H, \OO } = M_{H, \OO , \mathrm{triv}} $ the trivial functor, 
\item $ x_{U}  = 1_{\OO} \in M_{H, \OO}(U) $ the  source bottom class, 
\item $ \tilde{L}  = K \times D $ the  layer  extension 
of degree $ q - 1 $, 
\item $  \tilde{\mathfrak{H}}_{c} = \mathfrak{H}_{\mathrm{std},  c}  (\mathrm{Frob})  \in  \mathcal{C}_{ \OO    }  ( \tilde{K} \backslash   \tilde{G}   /  \tilde{K}   ) $ where $ \mathrm{Frob} : = \ch( \varpi^{-1} C) $.
\end{itemize} 
\begin{remark} This setup generalizes the one studied in \S \ref{toyexamples}. 
\end{remark}

\begin{theorem}  \label{GLnzeta}   There  exists a   zeta  element for   
$ (x_{U} , \tilde {  \mathfrak{H}   } _{c},  \tilde{L}  ) $ 
for all $ c \in \ZZ \setminus  2 \ZZ $.         
\end{theorem}    

\begin{proof}  For  each $  k_{2}  =  0 , \ldots,  m   $ and $ i $ an integer such that $  0 \leq i \leq m -  k_{2}  $,  let $ g_{i,   k_{2} }  :  = (1 , \tau_{i}, \varpi^{-2  k_{2}  } ) \in  \tilde{G}  $ and $ J_{i,  k_{2} }  :   = \left \{  ( k_{1}  ,  k_{2} )   \, | \, i \leq  k_{1} \leq  m ,  k_{1}  \in \ZZ   \right \}  $.   For each $ i, k_{2} $ as above, let $$ d_{i,k_{2}} :  = [ H \cap g_{i,k_{2}} \tilde{K}  g_{i,k_{2}}^{-1} :  H \cap g_{i,k_{2}} \tilde{L} g_{i  ,  k_{2}    }^{-1} ] .   $$   By Lemma  \ref{Heckefrobzetalemma}(iii),  $ d_{i,k_{2}} = [  H   _ { \tau_{i}  }      \cap \tau_{i} K \tau_{i}^{-1} : \nu^{-1}(D) ]  $. We therefore write $ d_{i} $ for $ d_{i,k_{2}} $.   Since $ \nu (  H \cap \tau_{i} K \tau_{i}^{-1}  ) =  C $ for $ i =0,\ldots, m-1 $,   we  have      $ d_{i} = q - 1 $. Now if $ (h_{1} , h_{2} ) \in H_{\tau_{m}} $, then $ h_{1} - h_{2} \in \varpi  \cdot  \mathrm{Mat}_{m\times m}(\Oscr_{F} ) $. Thus, $ \nu ( H_{\tau_{m}} ) \subset D $ (if fact, equal) and $ H_{\tau_{m}} = \nu^{-1}(D) $. This implies that  $ d_{m} = 1 $. To summarize, $$ d_{0} = \ldots = d_{m-1} = q - 1, \quad \quad  d_{m} = 1 . $$    
Next, for  each $ (i,  k_{2}  ) $ as above and  $ j = (k_{1},k_{2} ) \in J_{i,k_{2}} $, denote  $ h_{ j } :  =  (  \varpi^{ k f_{0} } ,  \varpi^{\lambda_{j }   }  )  \in H $ and $  
\sigma_{j} =  \iota_{\nu} ( h_{j}  ) \cdot  g_{i,k_{2}} =  ( \varpi^{k f_{0} } ,  \varpi^{ \lambda_{j} }  ,  \varpi^{- k }  )   \in  \tilde{G}  $    
where $ k $ in these    expressions denotes $ k_{1} + k_{2}   $.      
Denote by $ J $ the  disjoint union of $ J_{i,v} $ for all possible $ i, v$ as above.     By   Proposition  \ref{GLnHeckePolynomialprop}, Proposition  \ref{GLnmixeddecompositions} and   Lemma \ref{Heckefrobzetalemma}(a),  $$ \tilde{\mathfrak{H}}_{c}     =   \sum _{ j \in J }  b_{j}   \,     \ch ( U \sigma_{j}  \tilde{K}   ) $$    
where $ b_{j}   \in  \ZZ[q^{-1}]   $ for $ j = (  k_{1} , k _{2}  ) \in J_{i,k_{2}} $ is given by $ ( - 1 ) ^ {k}  q ^ { - k (2m-k +c)/2}    $ and  $ k = k_{1} +  k _{2} $ as before.     
In particular, $ b _{j}  \equiv (-1)^{k}  \pmod{q-1}   $. It is then clear that $$ H \backslash H \cdot \supp(\tilde{\mathfrak{H}}_{c} )  /  \tilde{K} = \left \{ g_{i,k_{2}} \, | \, 0 \leq k_{2} \leq m   ,   \, \,  0 \leq i \leq m - k_{2} \right \}. $$  
Let $ \mathfrak{h}_{i,k_{2}} $ denote the $ ( H , g_{i,k_{2}} )$-restriction of $ \tilde{\mathfrak{H}}_{c} $.   By  Corollary  \ref{GLnmixeddegree} and  Lemma  \ref{Heckefrobzetalemma} (ii),
\begin{align*}   \deg  (  \mathfrak {h}_{i,k_{2}}^{t})   &  = \sum_{ j \in J_{i,k_{2}}   }  c_{j} \deg   \,      [  U \sigma_{j} \tilde{K} ]  _ { * }  \\   &  \equiv  \sum_{ k_{1} = i }^{m}  (-1)^{k_{1}+k_{2}} \binom{m  - i } { m - k_{1} }  \binom{m-i}{    k_{2}}   \pmod{q-1}     \\   &     =  (-1)^{k_{2}} \binom{m-i}{ k_{2}} \cdot (-1)^{i}  (1-1)^{m-i} =   0 
\end{align*}    
for all $ i, k_{2} $ as above   such that $ i < m $. 
Since $ d_{m} = 1 $,  the  criteria of  Corollary \ref{easyzeta1} is satisfied.      
\end{proof}    

\begin{remark}     
For $ m = 2 $, the coefficients  $ \sum_{j \in J_{i,k_{2}} }    c_{j}  \deg   \,     [U \sigma   \tilde{K} ] _{ * } $ as follows
\begin{itemize}   [before = \vspace{\smallskipamount-2pt}, after =  \vspace{\smallskipamount-2pt}]    \setlength\itemsep{0.2em}     
\item  $1-q^{-\frac{c+3}{2}}(q+1)+q^{-(c+2)}$ for $g_{0,0}$, 
\item $q^{-(c+2)}-q^{-\frac{3 c+1}{2}}$ for $g_{1,0}$,
\item $q^{-(c+2)}$ for $g_{2,0}$
\item $(q+1)\left(q^{-(c+2)}(q+1)-q^{-\frac{3}{2}(c+1)}-q^{\frac{-(c+3)}{2}}\right)$ for $g_{0,1}$,
\item $q^{-(c+2)}-q^{-\frac{3}{2}(c+1)}(q+1)+q^{-2 c}$ for $g_{1,2}$,
\item $q^{-(c+2)}-q^{-\frac{3}{2}(c+1)}$ for $g_{1,1}$.
\end{itemize}   When  $c=1$, the  sets   $g_{0,1}\tilde{K}, g_{1,2}\tilde{K}, g_{1,1}\tilde{K} $ do not contribute to the  support of the  zeta element, since their corresponding coefficients all vanish.  An induction argument shows that for $ c=  1 $, the zeta element is only supported on $ g_{i,0} \tilde{K}   $.   
\label{phantomtwist}    

The normalization at $ c= 1 $ is relevant for the setting   \cite[\S 7]{Anticyclo} (corresponding to the $ L $-value at $ s = \frac{1}{2} $), and the  coefficients of the zeta element we obtain match \emph{exactly} with those of  test vector specified  Theorem 7.1  of  \emph{loc.cit}.  
More precisely, the  coefficient  denoted `$ b_{i} $' in Theorem 7.1 (2) of  \emph{loc.cit.}\  is the coefficient for $ g_{i,0} $ computed in the proof  above multiplied with $ \frac{q}{q -1}  \cdot  \mu_{H}(U) /\mu_{H}(V_{i}) $ (after replacing $ \ell $ in \emph{loc.cit.}\    with $ q $).  Note also that what we denote by $ V_{i} $ here  is denoted `$ V_{1,i} $' in \emph{loc.cit}.   One of the chief advantages of the approach here is that one does not need to compute the measures $ \mu_{H}(V_{i}) $ in Definition \ref{abstract zeta} which seem to have far more complicated  formulas.     
\label{sametestvector}               
\end{remark}  

\begin{remark}     \label{anticyclobehavior}   Notice that the   $g_{m,0}$ (equivalently, $ \tau_{m} $ in the decomposition  Proposition \ref{GLnmixeddecompositions}) only arises from a single Hecke operator $K \varrho^{m} K$. By  Corollary  \ref{easyzeta1}, we see that a zeta element exists only if the degree $ d_{m} $ is $ 1 $. 
In Theorem \ref{GLnzeta}, this was guaranteed by the choice of $ \nu $ and $ T $.  If   say, $ \nu $ is replaced by the product of determinants of $ H_{1} $, $ H_{2} $, then no zeta elements exist.  So in a sense, one  can only hope to make `anticyclotomic'  zeta elements in this  setting.        
\end{remark}

\section{Base  change \texorpdfstring{$L$}{}-factor of  \texorpdfstring{$\mathrm{GU}_{4}$}{}}     
\label{GU4Lfactorsection}
In this  section,  
we study the inert  case of the embedding discussed in \S \ref{embeddinganticyclosec}.  We first collect some generalities on the unitary group $ \mathrm{GU}_{4} $. Let $E / F$ be  separable extension of of degree $2$, $ \Gamma : =\operatorname{Gal}(E / F)$,  $\gamma \in \Gamma$ the    non-trivial element. Let 
\begin{equation}  
\label{hermitian} 
 J=\begin{pmatrix} 
& 1_{2} \\
1_{2} &
\end{pmatrix} 
\end{equation} 
where $1_{2}$ denotes the the $2 \times 2$ identity matrix. Then $ J  =  \gamma(J)^{t} $ is Hermitian.  We let $\mathbf{G}   =   \mathrm{GU}_{4}  $ be the reductive group over $F$ given whose $ R $ points for a $ F $-algebra $ R $ are given by  $$ \mathbf{G} (R) =\left\{g \in \mathrm{GL}_{4}(E \otimes R) \mid  \gamma({}^{t}g)   J g =   \mathrm{sim} (g) J \text{ where }    \mathrm{sim} (g)   \in R ^ { \times }  \right\}.  $$   Then $\mathbf{G}$ is the unique quasi-split unitary similitude group of split rank $3$ (see \cite[3.2.1]{Minguez}).  It's derived group is a special unitary group whose  Tits index 
is  ${ }^{2} A_{3,2}^{(1)}$  (see \cite{Tits-Semisimple}). The mapping $ \mathbf{G} \to \GG_{m} $, $ g \mapsto  \mathrm{sim} (g) $ is  referred to as the   similitude.  The   determinant  map $ \det : \mathbf{G} \to \mathrm{Res}_{E/F} \GG_{m} $ then satisfies $ \gamma \circ \det \,  \cdot \, \det =  \mathrm{sim}^{4}    $. For $R$ an  $E$-algebra, we let $$ \gamma_{R}: E \otimes R \rightarrow E \otimes R, x \otimes r \mapsto \gamma(x) \otimes r $$ the map induced by $\gamma$ and $$i_{R}: E \otimes R \rightarrow  R \times R $$ the isomorphism $x \otimes r \mapsto(x r, \gamma(x) r)$, where $x \in E, r \in R$. We let $\pi_{1}, \pi_{2}: E \otimes R \rightarrow R$ the projections of $i_{R}$ to the first and second component respectively. We have an induced action $\gamma_{R}: \mathrm{GL}_{4}(E \otimes R) \rightarrow$ $\mathrm{GL}_{4}(E \otimes R)$ and an induced isomorphism $ i_{R}: \mathrm{GL}_{4}(E \otimes R) \rightarrow \mathrm{GL}_{4}(R) \times \mathrm{GL}_{4}(R)$ given  by  $  \left(g_{i, j}\right) \mapsto\left(\pi_{1}\left(g_{i, j}\right), \pi_{2}\left(g_{i, j}\right)\right. $ 
Under the identification $i_{R}$, the group $\mathbf{G}(R) \subset \mathrm{GL}_{4}(E \otimes R)$ is identified with the subgroup of elements $(g, h) \in \mathrm{GL}_{4}(R) \times \mathrm{GL}_{4}(R)$ such that $$ \left({ }^{t} h,{ }^{t} g\right) \cdot(J, J) \cdot(g, h)=(r J, r J).  $$  We thus have functorial isomorphisms
$\psi_{R}:\left(c_{R}: \mathrm{pr}_{1} \circ i_{R}\right): \mathbf{G}(R) \stackrel{\sim}{\longrightarrow} \mathbb{G}_{m} \times \mathrm{GL}_{4}(R)$ via which we identify $\mathbf{G}_{E} \xrightarrow{\sim}    \mathbb{G}_{m} \times \mathrm{GL}_{4}$ (as group schemes
over $E$) canonically.   
\begin{notation*} The symbols $ F,  \Oscr_{F} , \varpi , \kay = \kay_{F}  , q = q_{F} $ have the same meaning as in \S \ref{LfactorHeckesection}. We let $ E / F $ denote an unramified quadratic extension  and  set  $ q_{E} = | \kay |_{E} =  q^{2} $ where $ \kay_{E} $ is the residue field of $ E $.  We denote by $ [\kay_{F} ] $, $ [\kay_{E}] $   a fixed choice of representatives in $ \Oscr_{F} $, $ \Oscr_{E} $ of elements of $ \kay_{F} $, $ \kay_{E} $   respectively.  We let   $ \mathbf{G} $ be the group defined above and denote  $$G=\mathbf{G}(F) , \quad  \quad   G_{E}=\mathbf{G}(E) \stackrel{\psi}{=}   E^{\times} \times \mathrm{GL}_{n}(E),  \quad \quad  K_{E}  \stackrel{\psi}{=}   \Oscr_{E} ^ { \times } \times \GL_{4} ( \Oscr_{E} ), \quad  \quad    K =  K_{E} \cap \mathbf{G}(F).  $$  For a ring $ R $, we let $\mathcal{H}_{R} $,
$\mathcal{H}_{R, E } $ denote the Hecke algebras $\mathcal{H}_{R}\left(K \backslash G / K\right), \mathcal{H}_{R}(K_{E} \backslash G_{E} / K_{E})$ over $ R $  respectively.   For  simplicity, we will denote
$ \ch(K \sigma K ) \in \mathcal{H}_{R} $ simply  by $ (K \sigma K )$.  Similarly for $ \mathcal{H}_{R,E} $.          
\end{notation*}

\subsection{Desiderata}   \label{GU4desiderata}      Let $\mathbf{A} =\mathbb{G}_{m}^{3}$, $ \mathrm{dis} : \mathbf{A} \rightarrow \mathbf{G}$ be the map
$$  
\left(u_{0}, u_{1}, u_{2}\right) \mapsto\begin{pmatrix}
u_{1} & & & \\
& u_{2} & & \\
& & \frac{u_{0}}{u_{1}} & \\
& & & \frac{u_{0}}{u_{2}}
\end{pmatrix}
$$
which identifies $ \mathbf{A} $ with the maximal split torus of $\mathbf{G}$. Let $\mathbf{M} $ be the normalizer of $ \mathbf{A} $.  Then $  \psi : \mathbf{M}_{E} \xrightarrow{\sim} \mathbb{G}_{m, E}^{5}$ and we consider $ \GG_{m,E}^{5} $ as a maximal torus  of $ \mathbf{G}_{E} \stackrel{\psi}{=} \mathbb{G}_{m, E} \times \mathrm{GL}_{4, E}$ via $\left(u_{0}, \ldots, u_{4}\right) \mapsto\left(u_{0}, \operatorname{diag}\left(u_{1}, \ldots, u_{4}\right)\right) $. We will denote $ A : = \mathbf{A}(F) $, $ M : = \mathbf{M}(F) $.      We have $X^{*}(\mathbf{M})=\mathbb{Z} e_{0} \oplus \cdots \oplus \mathbb{Z} e_{4}, X_{*}(\mathbf{M})=$
$\mathbb{Z} f_{0} \oplus \cdots \oplus \mathbb{Z} f_{4}$, where $f_{i}, e_{i}$ are as in $\S   \ref{GLndesiderata} $. The Galois action $\Gamma$ on $X_{*}(\mathbf{M}), X^{*}(\mathbf{M})$, is as follows:
$$   
\gamma \cdot e_{i}=       \begin{cases} 
{ e _ { 0 } } & { \text { if } i = 0 } \\
{ e _ { 0 } - e _ { i + 2 } } & { \text { if } i = 1 , \ldots , 4 }
\end{cases} \quad   \quad      \gamma    \cdot   f_{i}   =  \begin{cases}
f_{0}+\cdots+f_{4} & \text { if } i=0 \\
-f_{i+2} & \text { if } i=1, \ldots, 4
\end{cases}
$$
where $e_{i}=e_{i-4}, f_{i}=f_{i-4}$ if $i>4$. For $i=0,1,2$, let
\begin{itemize}     [before = \vspace{\smallskipamount+1pt}, after =  \vspace{\smallskipamount+1pt}]     \setlength\itemsep{0.2em}  
\item $\phi_{i}: \mathbb{G}_{m} \rightarrow \mathbf{A}$ by sending $u$ to the $j$-th component,
\item $\varepsilon_{i}: \mathbb{G}_{m} \rightarrow \mathbf{A}, \operatorname{dis}\left(u_{0}, u_{1}, u_{2}\right) \mapsto u_{i}$
\end{itemize} 
Then $X^{*}(\mathbf{A})=\mathbb{Z} \varepsilon_{0} \oplus \mathbb{Z} \varepsilon_{1} \oplus \mathbb{Z} \varepsilon_{2}$, $X_{*}(\mathbf{A})=\mathbb{Z} \phi_{0} \oplus \mathbb{Z} \phi_{1} \oplus \mathbb{Z} \phi_{2}   $.   Let  $ \mathrm{res}  :   X^{*}(\mathbf{M}) \rightarrow X^{*}(\mathbf{A})$, $ \mathrm{cores} :  X_{*}(\mathbf{A}) \rightarrow X_{*}(\mathbf{M})$ be the maps obtained by restriction and inclusion respectively. Then
$$   
\operatorname{res}\left(e_{i}\right)= \begin{cases}

\varepsilon_{i} & \text { if } i=0, 1,2 \\
\varepsilon_{0}-\varepsilon_{i-2} & \text { if } i=3,4
\end{cases} \quad  \quad     \operatorname{cores}\left(\phi_{i}\right)= \begin{cases}f_{0}+f_{3}+f_{4} & \text { if } j=0 \\
f_{i}-f_{i+2} & \text { if } j=1,2\end{cases}
$$
We let $\Phi_{E}$ denote the set of absolute roots of $\mathbf{G}_{E}$ as in \S \ref{GLndesiderata} for $ n = 4 $ and $\Phi_{F}$ denote the set of relative roots obtained as restrictions of $\Phi_{E} $ to $ \mathbf{A} $.  Then $ \Phi_{F}=\left\{\pm\left(\varepsilon_{1}-\varepsilon_{2}\right), \pm\left(\varepsilon_{1}+\varepsilon_{2}-\varepsilon_{0}\right), \pm\left(2 \varepsilon_{1}-\varepsilon_{0}\right), \pm\left(2 \varepsilon_{2}-\varepsilon_{0}\right)\right\}  $,  which   constitutes a
root system of type $C_{2}$. 
We choose $ \beta_{1}=e_{1}-e_{2} $. $  \beta_{2}=e_{2}-e_{4} $ and $   \beta_{3}=e_{4}-e_{3}$
as simple roots and let $ \Delta_{E}=\left\{\beta_{1}, \beta_{2}, \beta_{3}\right\} $. In this ordering, the half sum of positive roots is  \begin{equation}  \label{halfsumgu4} \delta =  \frac{1}{2}( 3 e_{1} + e_{2} - e_{4} - 3 e_{3} )
\end{equation}  and 
$\beta_{0}=e_{1}-e_{3}$ is the highest root. The set $\Delta_{E}$ and $\beta_{0}$ are invariant under $\Gamma$, and the labeling is chosen so that (absolute) local Dynkin diagram (with the bar showing the Galois orbits) is the diagram on the left 
$$\dynkin[fold, scale  =  1.7, extended,    edge length  = 0.7,  labels={0,1,2,3}]A{ooo}   \quad  \quad   \quad     \dynkin[scale = 1.7,  labels*={1,2,1}, extended,    edge length = 0.7,     labels={0,1,2}]C{oo}       
$$ 
The set of corresponding relative simple roots is therefore $\Delta_{F}=\left\{\alpha_{1}, \alpha_{2}\right\}$ where   $ \alpha_{1}=\varepsilon_{1}-\varepsilon_{2}$, $  \alpha_{2} =  2  \varepsilon_{2} -  \varepsilon_{0} $. 
With this ordering, the highest root is $\alpha_{0}=2 \varepsilon_{1}-\varepsilon_{0}  $.  The  associated simple coroots are  $ \alpha_{0} ^ { \vee } = \phi_{1} $, $ \alpha_{1}^{\vee }  =   \phi_{1}  -  \phi_{2} $, $ \alpha_{2}^{\vee}   =  \phi_{2}  $ and we denote by $ Q^{\vee} $ their  span in $ \Lambda $.  Executing the recipe provided in \S 1.11  of  \cite{Tits} on the absolute diagram above, we find that the  \emph{local index} or \emph{relative local Dynkin diagram} (see \S 4 of \emph{op.cit.}) is the diagram on the right above.    Here,  the  indices below the diagram correspond to the affine roots $-\alpha_{0}+1, \alpha_{1}, \alpha_{2}$ and the indices above the diagram are half the number of  roots of a  semi-simple  group of  relative  rank $ 1 $ whose  absolute Dynkin-diagram is the corresponding Galois orbit in the diagram on the left.   The endpoints of the diagram on the right, and in particular the one labelled $ 0 $, are hyperspecial and hence so is the subgroup $K$ by construction. The diagrams above can be found in the fourth row of the  table on p.\ 62 of \emph{op.\ cit}. 
\begin{remark} For $ \lambda =  a_{0} \phi_{0} + a_{1} \phi_{1} + a_{2}  \phi_{2}  \in X_{*}(\mathbf{A}) $,   $  \langle  \lambda , \delta  \rangle $ can be computed by pairing $  \lambda $ with  $ \mathrm{res}(\delta) = -2  \varepsilon _{0}  + 3 \varepsilon_{1} +  \varepsilon_{2} $ and equals $  - 2 a_{0} + 3 a_{1} + a_{2} $. Note also that $$ 2 \cdot  \mathrm{res}(\delta) = 2 (\varepsilon_{1}  -   \varepsilon_{2}    )  +  2  (  \varepsilon_{1}  +  \varepsilon_{2}   -   \varepsilon_{0}  )  +  ( 2   \varepsilon _ { 1   }   -   \varepsilon_{0} )    +   ( 2  \varepsilon_{2}  -   \varepsilon_{0}  )   $$
is a weighted  sum  of  the positive roots in $ \Delta_{F} $, with the weights given by   the  degree of  the splitting field of the  corresponding  root. 
\end{remark} 

From now on, we denote by $ \Lambda $ the cocharacter lattice $ X^{*}(\mathbf{A}) $ and denote by $ t $ the translation action of $ \Lambda $ on $ \Lambda \otimes  \RR   $.      An  element $ \lambda =  a_{0} \phi_{0} + a_{1} \phi_{1} + a_{2} \phi_{2} \in  \Lambda $ will be denoted by $ (a_{0} ,a_{1} , a_{2} ) $ and $ \varpi ^ { \lambda }$ denotes the element $ \lambda(\varpi)  \in  A $. Let $ s_{i} $, $ i = 0,1,2 $ denote the simple reflections associated $ \alpha_{i} $. 
The action of $ s_{i} $ on $ \Lambda $ is given explicitly as follows: 
\begin{itemize}       [before = \vspace{\smallskipamount}, after =  \vspace{\smallskipamount}]    \setlength\itemsep{0.2em}      
\item $ s_{1} $ acts as a transposition $\phi_{1} \leftrightarrow \phi_{2}$,
\item $s_{2}$ acts by sending $\phi_{0} \mapsto \phi_{0}+\phi_{2}, \phi_{1} \mapsto \phi_{1}, \phi_{2} \mapsto - \phi_{2}$
\item $s_{0}=s_{1} s_{2} s_{1}$ acts by sending $ \phi_{0} \mapsto \phi_{0} + \phi_{1} $, $ \phi_{1} \mapsto -  \phi_{1} $, $ \phi_{2}  \mapsto  \phi_{2} $. 
\end{itemize}   
As before, we let $ e^{\lambda} $ (resp.,  $ e^{W \lambda } )$ denote the element in the group algebra $ \ZZ [ \Lambda ] $ corresponding to $ \lambda $ (resp.,  the formal sum over $ W \lambda ) $.  Let  $ S_{\mathrm{aff} } =  \left \{ s_{1} , s_{2} ,  t( \alpha_{0} ^{\vee} )  s_{0}    \right \} $ and  $ W $, $ W_{\mathrm{aff}} $ and $ W_{I} $ denote the Weyl, affine Weyl, Iwahori Weyl groups respectively. We  consider    $ W_{\mathrm{aff}} $ as a group of affine transformations of $ \Lambda \otimes \RR $. We have   
\begin{itemize}      [before = \vspace{\smallskipamount}]    \setlength\itemsep{0.2em}       
\item  $  W  \cong  ( \ZZ / 2 \ZZ ) ^{2}   \rtimes  S_{2} $, 
\item $W_{\mathrm{aff}}=  t ( Q  )     ^{\vee}   \rtimes     W$ the affine Weyl group
\item $ W_{I} = A / A^{\circ}  \rtimes  W  \xrightarrow[\sim]{v}   \Lambda  \rtimes  W $,     
\end{itemize} 
The pair $ (W_{\mathrm{aff}}, S_{\mathrm{aff}}) $ is a Coxeter system of type $ \tilde{C}_{2}$ and we consider   $ W_{\mathrm{aff} } \subset  W_{I} $ via $ v $.  Then $ W_{I}  =  W_{\mathrm{aff}}  \rtimes \Omega $.  Given $ \lambda \in \Lambda $, the   minimal possible length of elements in $ t(\lambda) W $ is obtained by a unique element. This length is given by 
\begin{equation}   \label{GU4minlength}        \ell_{\mathrm{min}}(\lambda) = \sum_{ \lambda \in \Phi_{\lambda}^{1}} |  \langle \lambda , \alpha \rangle   |     +  \sum_{   \alpha   \in  \Phi_{\lambda}^{2}}  ( \langle  \lambda  ,   \alpha  \rangle - 1  ) 
\end{equation}   
where $ \Phi_{\lambda}^{1} = \left \{ \alpha \in \Phi_{F}^{+} \, | \,  \langle  \lambda , \alpha  \rangle  \leq 0  \right \} $,  $ \Phi_{\lambda} ^{2}  =  \left \{   \alpha     \rangle   \in  \Phi_{F}^{+} \, | \,  \langle  \lambda , \alpha  >  0  \right \} $.  When $ \lambda $ is dominant, the first sum is zero, and the length is then also minimal among  elements of $ W t(\lambda) W  $.  Consider the following elements in the normalizer $ N_{G}(A) $:  
$$  
w_{0}=  \scalebox{0.9}{$\begin{pmatrix} &   & \frac{1}{\varpi} & \\  & 1 & & \\ \varpi & &   & \\ & & &  1  \end{pmatrix}$},   \quad   w_{1}=  \scalebox{0.9}{$\begin{pmatrix} & 1 & & \\ 1 & & & \\ & &  & 1 \\ & & 1  & \end{pmatrix}$},  \quad 
  w_{2}=  \scalebox{0.9}{$\begin{pmatrix} 1 & & & \\ &  &  & 1  \\ &  & 1 & \\ &  1 & & \end{pmatrix}$}, \quad  \rho =  \scalebox{0.9}{$\begin{pmatrix} & & &  1 \\ & &  1 & \\ &    \varpi  &   &    \\  \varpi  & &  & \end{pmatrix}$} .  $$ 
The  classes of $w_{0}, w_{1}, w_{2}$ represent $t\left(\alpha_{0}^{\vee}\right) s_{0}, s_{1}, s_{2}$ in $W_{I}$ and $\rho$ represents $ t\left(-\phi_{0}\right) s_{2} s_{1}s_{2} $, which is a generator of $\Omega \cong \mathbb{Z}$.   The  conjugation action of $\rho$ switches $w_{0}, w_{2}$ and  keeps $w_{1}$ fixed, inducing an automorphism of the extended Coxeter diagram $ 
\dynkin[Coxeter,   labels = {0,1,2}   ,  extended,]C2 $.  

Let $ \xi  \in  \Oscr_{E} ^ { \times }$ be an element of trace $ 0 $ i.e.,  $ \xi + \gamma (\xi ) =  0  $.
Let  $ x_{1} : \mathrm{Res}_{E/F}   \GG_{a}  \to   \mathbf{G} $ and $ x_{i } : \GG_{a} \to \mathbf{G} $ for $ i = 0 , 2 $ be the root group maps 
$$   
x_{0}: u \mapsto  \scalebox{0.9}{$\begin{pmatrix}
1 & & & \\
& 1 & & \\
\varpi  \xi u& & 1 & \\
 & & & 1
\end{pmatrix}$},  \, \, \,     x_{1}: u \mapsto  \scalebox{0.9}{$ \begin{pmatrix}
1 & u & & \\
& 1 & & \\
& & 1 & \\
& & -  \bar{  u }  & 1
\end{pmatrix}$},  \, \, \,       x_{2}: u \mapsto   \scalebox{0.9}{$\begin{pmatrix}
1 & & & \\
& 1 &  & \xi u  \\
& & 1 & \\
& & & 1
\end{pmatrix}$}, 
$$
where $ \bar{u} : = \gamma(u) $. We let $\kay_{w_{i}} = \kay_{w_{2}} :=  \kay_{F}$, $  \kay_{w_{1}} :  = \kay_{E}   $  and for $ i = 0,1,2 $,  we denote by  $ g_{w_{i} } : [ \kay_{w_{i}} ]  \to     G     $ the map $ u \mapsto  x_{i} (  u )  w_{i}   $. If $ I $ denotes the Iwahori subgroup\footnote{note that $K$ is hyperspecial, i.e., its Weyl group equals $W$} of $ K $ whose reduction modulo $ \varpi $ lies in the Borel of $ \Gb(\kay) $ determined by $ \Delta_{F} $,  then $ I w_{i} I  = \bigsqcup_{\kappa \in [\kay_{w_{i}}]}  g_{w_{i}} ( \kappa )  .  $     
For  $ w \in W_{I }$ such that $ w $ is the unique  minimal length element in $ w W $,  choose a reduced word decomposition   $ w  =   s_{w,1} s_{w,2}   \cdots  s_{w,\ell(w)}  \rho _{w} $ where $ s_{w  , i }   \in  S_{\mathrm{aff}} $ and  $  \rho_{w} \in \Omega  $.  Define  \begin{align}   \label{GU4Xw}       \mathcal{X}_{w} :   \prod_{ i  = 1 } ^ { \ell(w)}   [\kay _{w_{i} }]   &  \to   G \\
   (\kappa_{1}, \ldots,   \kappa_{\ell(w)} )   &   \mapsto   g_{s_{w,1} }   ( \kappa_{1} )  \cdots  g_{s_{w,\ell(w) } }   ( \kappa_{\ell(w)} )   \rho_{w}    \notag 
\end{align}   
where we have suppressed the dependence on the the  decomposition of $ w $  in the notation.   By   Theorem  
\ref{BNrecipe},  the    image of $ \mathcal{X}_{w} $ is independent of the choice of   decomposition.     

\subsection{Base change Hecke  polynomial}  Let  $ y_{i} :  = e^{\phi_{i}} \in \ZZ [ \Lambda ]  $, so that $ \ZZ [ \Lambda ] = \ZZ [ y_{0} ^ { \pm } , y_{1} ^ { \pm } ,  y_{2}   ^  { \pm }  ] $ and  let  $  \mathcal{R}_{q} : = \ZZ [ q ^{\pm \frac{1}{2}}] $, $ \mathcal{R}_{q^{2}}   :     = \ZZ [ q^{-1} ] $. The abelian group homomorphism $1+\gamma: X_{*}(\mathbf{M}) \rightarrow X_{*}(\mathbf{M})$ given by $f \mapsto f+\gamma \cdot f$ has image in $\Lambda=X_{*}(\mathbf{M})^{\Gamma}$ and hence induces a  map $e^{1+\gamma}$ on $\mathcal{R}_{q^{2}}$-algebras
\begin{center} 
\begin{tikzcd}
\mathcal{H}_{\mathcal{R}_{q^{2}}}\left(  G_{E} \right)  \arrow[r, "{\mathscr{S}_{E}}"]   \arrow[d, "{\mathrm{BC}}", swap]    & \mathcal{R}_{q^{2}}\left[X_{*}(\mathbf{M})\right]^{W_{E}}   \arrow[d, "{{e^{1+\gamma}}}"]        \\
\mathcal{H}_{\mathcal{R}_{q}}(    G   ) \arrow[r, "{\mathscr{S}_{F}}"] &   \mathcal{R}_{q}[\Lambda]^{W_{F}}
\end{tikzcd}  
\end{center} 
corresponding to which we have what is called the \emph{base change map}    $$  \mathrm{BC}   :  \mathcal{H}_{\mathcal{R}_{q^{2}}} ( G_{E} )  \to   \mathcal{H}_{\mathcal{R}_{q}}(G) .  $$     
The Satake polynomial that we need to  consider here is the base change of the Satake polynomial of $\mathbf{G}_{E}$ associated with the standard representation considered in $\S  \ref{GLnHeckesection} $.
This polynomial is
$$
\mathfrak{S}_{\mathrm{bc}}(X)=\left(1-y y_{1} X\right)\left(1-y y_{1}^{-1} X\right)\left(1-y y_{2} X\right)\left(1-y y_{2}^{-1} X\right) \in \mathbb{Z}[\Lambda]^{W_{F}}[X] .
$$       
where $  y =  y_{0  } ^ {   2 }  y _ { 1 }   y _  { 2    }    $.    
\begin{remark} We have a componentwise embedding $ {} ^ { L } \mathbf{G} _{E}  \hookrightarrow  { } ^ { L } \mathbf{G}_{F} $.     
Given  an unramified $L$-parameter $ \varphi : \mathscr{W}_{F} \to {} ^ { L } \mathbf{G}_{F} $, let  $ \hat{t}   \rtimes \mathrm{Frob}_{F} ^{-1} : =  \varphi( \mathrm{Frob}_{F} ^ {-1}  )  $, where $ \mathrm{Frob}_{F} \in \mathscr{W}_{F} $ denotes  a lift of the arithmetic  Frobenius, we have  $$  \varphi(\mathrm{Frob}_{E} ^{-1} ) = (  \hat{t}  \rtimes   \mathrm{Frob}_{F} ^{-1}     )   ^{2}  =  \hat{t}  \gamma(\hat{t})   \rtimes   \mathrm{Frob}_{E} ^{-1}   \in  { } ^ { L } \mathbf{G}_{E} .  $$    If   we think of $ \hat{t} $ as the Satake parameters of an unramified representation $ \pi_{F} $  of $ \mathbf{G}(F) $, then  $ \hat{t} \gamma( \hat{t}  )  $ are the Satake parameters of an unramified representation $ \pi_{E} $ of $ \mathbf{G}(E) $ which is called the \emph{base change} of $ \pi_{F} $. The base change map BC above can then   also    be characterized as in \cite[\S 2.2]{Kottwitz}.       
\end{remark}

\label{GU4Hecesectino}      
\begin{definition} 
\label{GU4Heckepolydefinition} We define     $\mathfrak{H}_{\mathrm{bc}, c}(X) \in \mathcal{H}_{\mathcal{R}}[X] $ to be the image of $\mathfrak{H}_{\mathrm{std}, c}(X)$ under the map $\mathrm{BC}$ for $c$ any integer.   Equivalently, $ \mathfrak{H}_{\mathrm{bc},c}(X) $ is the unqiue polynomial such  that $ \mathscr{S}_{F}   \left   ( \mathfrak{H}_{\mathrm{bc},c}(X)  \right  ) = \mathfrak{S}_{\mathrm{bc}} (q^{-c} X ) $.  
\end{definition}

\begin{proposition}   \label{GU4satake}      We have         
\begin{enumerate}      [label = \normalfont (\alph*),  before = \vspace{\smallskipamount}, after =  \vspace{\smallskipamount}]    \setlength\itemsep{0.2em}   
\item  $ \mathscr{S}_{F}(K \varpi^{(2,2,1)} K) =   q  ^{3} e ^{W(2,2,1)} +  ( q -1  ) ( q^{2} + 1 )  e ^{(2,1,1)}  ,$ 
\item $ \mathscr{S}_{F}(K \varpi^{(4,3,3)}K) =  q^{4} \, e^{W(4,3,3)} +  q ^ 3 ( q - 1 ) \,  e ^{W(4,3,2)}  +  q ( q - 1 ) ( 1 + q + 2q^2) \,  e ^{(4,2,2)} .    $    

\end{enumerate}  
\end{proposition} 
\begin{proof}  Note that since $ \res(\mathrm{\delta})  \in X^{*}(\mathbf{A})  $, the Satake transform of $ (K \varpi^{\lambda} K ) $ for any $ \lambda \in \Lambda $  all have coefficients in $ \ZZ[q^{-1}][\Lambda]^{W}  $. The leading coefficients are obtained by Corollary 
\ref{Satakeuppercoro} which also shows that the support of these transforms is on Weyl orbits of cocharacters that are succeeded by $ \lambda $ under $ \succeq $.  \\ 

\noindent (a) Since $ (2,2,1) - (2,1,1) = \alpha_{1}^{\vee} + \alpha_{2} ^ {\vee} $, $ (2,2,1) \succeq (2,1,1) $ and it is easily seen that  $ (2,1,1) $ is the only dominant cocharacter which $ (2,2,1) $ succeeds. 
Thus  $$ \mathscr{S}( K \varpi^{(2,2,1)} K )  = q^{3} e^{W(2,2,1)} + b  \, e ^{(2,1,1)}  $$  
for some $ b \in \ZZ[q^{-1}] $.  To obtain the value $ b $, we use the decomposition recipe of Theorem \ref{BNrecipe}. Note that $ \ell_{\min}(2,2,1) = 1 $   and that  $ K w K =  K \varpi ^ {(2,2,1)} K $ where $ w=  w_{0} \rho^{2} $. So we see  from the the Weyl orbit diagram   
$$   
\begin{tikzcd}
{(2,0,1)} \arrow[r, "s_{1}"] & {(2,1,0)} \arrow[r, "s_{2}"] & {(2,1,2)} \arrow[r, "s_{1}"] & {(2,2,1)}
\end{tikzcd} $$ 
that  $  | K w_{0} \rho^{2} K / K| =   q   + q ^{3} + q ^{4} + q^{6} $. Of these, the number of cosets of shape a permutation of $ (2,2,1) $ is $ \sum _ { \mu \in  W (2,2,1)  }  q ^ {  \langle  \lambda + \mu , \delta \rangle } $  $  =  1   +  q^{2} + q^{4} + q ^ {6 } $  by $ W $-invariance of $ \mathscr{S} $ (see Corollary \ref{Satakeuppercoro}). Thus the number of cosets of shape $ (2,1,1) $  is $$ q + q^{3} + q^{4}  + q^{6}  - ( 1 + q^{2} + q^{4} + q^{6} )  = ( q -1  ) ( q^{2} + 1 ) . $$   
Since $ \langle (2,1,1) , \delta \rangle = 0  $, the claim follows. \\

\noindent (b)  Arguing as in part (a), we have  
$$  \mathscr{S}(K \varpi^{(4,3,3)}K) = q^{4} e^{W(4,3,3)} + b_{1} \,  e ^{W (4,3,2)} + b_{2}  \,   e ^{(4,2,2) } $$ for some $ b_{1} , b_{2} \in \ZZ[q^{-1}]  $.  Here we need a more  explicit  description of the Schubert cells in order to find $ b_{1} $, $ b_{2} $.    Observe that $ \ell_{\mathrm{min}}(4,3,3) = 3 $  and that $ K \varpi^{(4,3,3)} K = Kw K $ where $ w  = w_{0} w_{1} w_{0}  \rho^{4} $.   The Weyl orbit diagram of $ (4,3,3) $ is  \vspace{-0.2em}     
\begin{center}
\begin{tikzcd}  (4,1,1)  \arrow[r,"s_{2}"] & (4,1,3)  \arrow[r,"s_{1}"] & (4,3,1)  \arrow[r,"s_{2}"] &  (4,3,3) .
\end{tikzcd}
\vspace{-0.2em} 
\end{center}
By Theorem \ref{BNrecipe},  $   K \varpi^{(4,3,3)}  K /  K = \bigsqcup_{ i = 1 }^{3}  \mathrm{im} ( \mathcal{X}_{\sigma_{i}}) $ where $ \sigma_{0} = w $, $ \sigma _{1} = w_{2} \sigma_{0} $, $ \sigma _{2} = w_{1} \sigma_{1} $, $ \sigma _{3} = w_{2} \sigma_{2}  $.  
Explicitly,   
\begin{align*}   
\mathrm{im} ( \Xcal_{\sigma_{0}} ) &  =   \scalebox{0.85}{$ \Set*{ \left(\begin{array}{cccc}
\varpi & 
\\
& \varpi  & & \\[0.1em] 
x_{1}  \varpi^{2}& a  \varpi^{2}    & \varpi^{3}  & \\[0.1em] 
-\bar{a} \varpi^{2}  &   
x  \varpi^{2}  &  
& \varpi^3 
\end{array}\right) K 
\given  
\begin{aligned} &  a   \in [\kay_{E}]   ,  \\ &    x, x_{1}    \in \xi [\kay_{F}]   \end{aligned}  
}$}  , 
\\[0.1em]   
\mathrm{im} ( \Xcal_{\sigma_{1}} ) &  =   \scalebox{0.85}{$ \Set*{ \left(\begin{array}{cccc}
\varpi  &  &  &  \\[0.1em]  
-\varpi^2 \,\overline{a}  & \varpi^3  &  & x\, \varpi^2 + y\,\varpi \\[0.1em] 
x_{1}\, \varpi^2    &  & \varpi^3  & a\,\varpi^2 \\[0.1em] 
 &  &  & \varpi 
\end{array}\right) K \given  \begin{aligned} &   a   \in [\kay_{E}]   ,  \\ &    x, x_{1} , y   \in \xi [\kay_{F}]   \end{aligned} 
}$}  \\[0.1em]  
\mathrm{im} ( \Xcal_{\sigma_{2}} )  &  =   \scalebox{0.85}{$ \Set*{ \left(\begin{array}{cccc}
\varpi^3  & a_1 \,\varpi -\varpi^2 \,\bar{a}  & x\, \varpi^2 +  \,y\,\varpi  & \\
& \varpi  & & \\
&  & \varpi  & \\
& \varpi^2 \,x_1  & a\,\varpi^2 -\varpi \,\bar{a}_{1}  & \varpi^3 
\end{array}\right) K 
\given    \begin{aligned} &  a ,  a_{1} \in [\kay_{E}]   ,  \\ &    x, x_{1},   y   \in \xi [\kay_{F}]   \end{aligned} }$},    \\[0.1em] 
\mathrm{im} ( \Xcal_{\sigma_{3}} )  &  =   \scalebox{0.85}{$ \Set*{ \left(\begin{array}{cccc}
\varpi^3  & & x\, \varpi^2 + y\,\varpi  & a_1 \,\varpi - \overline{a} \,  \varpi^{2} \\[0.1em] 
& \varpi^3  & a\,\varpi^2 - \overline{a_1 } \, \varpi   & x_1 \,\varpi^2 + y_1 \,\varpi \\
& & \varpi  & \\
& & & \varpi 
\end{array}\right) K \given    \begin{aligned} &  a ,  a_{1} \in [\kay_{E}]   ,  \\ &    x, x_{1},  y , y_{1}    \in \xi [\kay_{F}]   \end{aligned} }$}. 
\end{align*} 

From the cells above, it is not hard  to see that the shape of any  coset in
\begin{itemize} 
\item $   \mathrm{im}(\mathcal{X}_{\sigma_{0}}) $ is \begin{itemize} \item 
$ (4,1,1) $ if $  x_{1} = x = a =  0 $, 
\item $ (4,1,2) $ if $ x_{1} = a = 0 $, $ x \neq 0 $, 
\item $ (4,2,1) $ if $ x_{1} \neq 0 $ and  $ a \bar{a} + x x_{1}\xi^{2} \in \varpi \Oscr_{F} $, 
\item $ (4,2,2) $ if either $ x_{1} = 0 $, $ a \neq 0 $ or $  x_{1} \neq 0 , a \bar{a} + x x_{1} \xi^{2} \notin  \varpi  \Oscr_{F}$  
\end{itemize}  
\item $ \mathrm{im}(\mathcal{X}_{\sigma_{1}}) $ is $ (4,1,3) $ if $ x_{1} = a = 0 $,  
$ (4,2,2) $ if $ x_{1} = 0 $, $ a \neq 0 $ and $(4,2,3) $ 
if $ x_{1} \neq 0 $,   
\item $ \mathrm{im}( \mathcal{X}_{\sigma_{2}}) $ is $ (4,3,1) $ if $ x_{1} = 0 $ and  $ (4,3,2) $ if $ x_{1} \neq 0 $,   
\item $ \mathrm{im}(\mathcal{X}_{\sigma_{3}} )  $ is $ (4,3,3) $.  
\end{itemize}
So in $ K \varpi^{(4,3,3)} K / K $,  there are exactly   
$ q^{6}(q-1) $ cosets of shape $ (4,3,2) $. Since $ \mathscr{I}(\varpi^{(4,3,2)}K) = q^{-3} e^{(4,3,2)} $, 
$$ b_{1}  = q^{-3} \cdot q^{6}(q-1) = q^{3}(q-1) $$ by $ W $-invariance of $ \mathscr{S} $.  Thus the number of cosets in $K  \varpi^{(4,3,3)} K / K $ whose shape is in $ W (4,3,2) $ is $ \sum_{ \mu \in W (4,3,2)} q^{\langle \mu  , \delta  \rangle }  q^{3}(q-1)  =   (q-1)( 1 + q^{2}  + q^{4} + q^{6} )       $.  
Since the nubmer of cosets of shape in $ W( 4,3,3) $ is $ \sum_{ \mu \in W (4,3,3)}  q ^{  \langle \mu , \delta  \rangle } = 1 + q^{2} + q^{6} + q^{8} $ and $ | K \varpi^{(4,3,3)} K / K | = q^{4} + q^{5} + q ^{7} + q^{8} $, we see that \begin{align*} b_{2}  &   =    q^{4} + q^{5} + q ^{7} + q^{8}  - ( 1 + q^{2} + q^{6} + q^{8} ) -  (q-1)( 1 + q^{2}  + q^{4} + q^{6} )  \\
& =     q ( q - 1 ) ( 1 + q + 2q ^2     )  \qedhere 
\end{align*}  
\end{proof}

\begin{corollary}   \label{GU4Heckepolynomial}    We have
\begin{align*}
\mathfrak{H}_{\mathrm{bc,c}}(X) &=( K ) \\
&-q^{-(c+3)}\left((K w_{0} \rho^{2} K ) + (q^{2}+1)(1-q)(K \rho^{2} K)\right) X \\
&+q^{-( 2c+4)}\left( (K w_{0} w_{1} w_{0}    \rho^{4} K ) + ( 1 -q ) ( K w_{0} \rho ^{4}  K )     + (q^{2}+1)(1 - q + q^{2}) ( K \rho^{4} K ) \right) X^{2} \\
&-q^{-(3c+3)}\left((K w_{0} \rho^{6} K  )  + ( q ^{2} + 1 ) (  1  - q  )   ( K  \rho    ^{6} K   )    \right) X^{3} \\
&+q^{-4 c}   ( K \rho^{8} K  )   X^{4}  \in \mathcal{H}_{\ZZ[q^{-1}]}[X] 
\end{align*} 
where the words appearing in each Hecke operator are of minimal possible length. 
\end{corollary}   
\begin{proof}   Since  
$$    \mathfrak{ S } _ { \mathrm{bc} , c }  (  X )  =   1 - e^{W (2,2,1  )   }  X       + \left ( e ^{W(4,3,3) }  + 2  e ^ {  W ( 4,2,2)  }   \right ) X ^{2}  -   e ^ {   W    ( 6,4,3) }  X^{3}  +  e ^ { (8,4,4) }  X^{4}  , $$
the result follows from  Proposition  \ref{GU4satake}.   
\end{proof}   

\subsection{Mixed   coset   structures}
Let $ \mathbf{H} $ be the subgroup of $ \mathbf{G}$ generated by the maximal torus $ \mathbf{M}  $, and the root groups corresponding to $ \pm \alpha_{0} $, $ \pm \alpha_{2} $. Then $ \mathbf{H}  =  \mathrm{GU}_{2} \times_{ \mu }   \mathrm{GU}_{2} $. Here $ \mathrm{GU}_{2} $ is the 
reductive group over $ F $ whose $ R $ points for a $ F $-algebra $ R $ are given by $$   \mathrm{GU}_{2}(R) = \left \{  g \in \GL_{2}(E \otimes R) \, | \, \gamma ( {} ^ { t } g ) J_{2} g = \mu(g)  J_{2}  , \mu(g) \in   R ^ { \times }   \right \} $$   where $ J_{2} =         \begin{psmallmatrix}  & 1 \\ 1 &  \end{psmallmatrix} $ and the fiber product in $ \mathbf{H }$ is over the similitude character of the two copies of $ \mathrm{GU}_{2} $. Explicitly, we get the embedding 
\begin{align}  \label{gu4embed}  \iota :  \Hb 
\to \Gb \quad \quad \quad  
\left ( \left (  \begin{smallmatrix} a & b \\ c & d  \end{smallmatrix}  \right )   ,  \left (   \begin{smallmatrix}  a_ {  1 } & b_{1} \\ c _ {1 } &  d_{1}  \end{smallmatrix}  \right ) \right )  
\mapsto  \left (    \begin{smallmatrix} a & & b \\   &   a_{1}    & &  b_{1} \\   c &  & d  \\ &  c_{1} & &  d_{1}   \end{smallmatrix}   \right ) 
\end{align}   
We let  $   H  =   \mathbf{H} (F )$,   $ U = H \cap K = \Hb(\Oscr_{F})  $. The Weyl group of $W_{H} $ of $ H$ can be identified with the subgroup of $ W $ generated by $ s_{0}, s_{2} $ and is isomorphic to $ S_{2} \times S_{2}   $. 
\begin{remark} 
This embedding is isomorphic to the one obtained by localizing the global one in \S \ref{embeddinganticyclosec} by a local change of variables that sends $ J $ in  (\ref{hermitian}) to $ \mathrm{diag}(1,-1,-1,-1) $, which can be explicitly written by the formula given in \cite[p. 249]{Lewis}.  
\end{remark} 
To describe the twisted restrictions arising from the Hecke polynomial $ \mathfrak{H}_{bc,c}(X) $, we define the  elements $ \tau_{0} = 1_{G} $ and $$   \tau_{1}  =  
  \scalebox{0.9}{$\begin{pmatrix}   \varpi &  &  & - 1  \\
&    \varpi &  1   \\[0.1em]   
& &     1 & \\     
& & & 1 \end{pmatrix}$}, \quad  \quad    \tau_{2} =   
  \scalebox{0.9}{$\begin{pmatrix}   \varpi^{2} & &  &  
 - 1 \\    
&  \varpi^{2} &   1   \\[0.1em]   
 & &      1 &  \\ 
 &  & & 1   
 \end{pmatrix}$},  \quad    \quad   \tau_{3} =    \scalebox{0.9}{$\begin{pmatrix}  
 \varpi^{2} &  \varpi &  1 & - \varpi \\   
  &  \varpi   & 1   & \\[0.1em]       
  &  &     1 &  \\ 
  &  &   \!   \!    - 1    &  \,\,  \varpi  
 \end{pmatrix}$}   .  $$ 

\begin{lemma}  \label{distinctgu4}   If $ 2 \in \Oscr_{F}^{\times} $, then   $ H\tau_{i} K $ are pairwise disjoint  for $ i = 0 , 1, 2, 3 $.  
\end{lemma}  
\begin{proof}  If $ H\tau_{i} K = H \tau_{j} K $, there exists an $ h \in H $ such that $ \tau_{i}^{-1} h \tau_{j} \in K $. Writing $ h $ as  in  (\ref{gu4embed}), 
the matrices  $ h \tau_{1} $, $ h \tau_{2} $, $ h \tau_{3} $, $ \tau_{1}^{-1} h \tau_{2} $ respectively have the form     
 \begin{alignat*}{6}  
\scalebox{0.9 }{$\begin{pmatrix}  a\varpi  &  & * & -  a \\  & * &  * &  * \\  c\varpi  & &  *  &  -c  \\  &   *  & * &  *  \end{pmatrix}$}, & \quad   \quad  
\scalebox{0.9}{$\begin{pmatrix}   a \varpi ^ { 2  }    &  & * & - a \\  & * &  * &   *  \\   c \varpi ^{2}   & &  * &  -c  \\  &   *  &  *  &  * 
\end{pmatrix}$},  & \quad  \quad  
\scalebox{0.9}{$\begin{pmatrix}  a  \varpi^{2}  & * &  * & - a \varpi    \\ &  * & * & *    \\   c \varpi^{2} &  * & *  &  -  c  \varpi   \\  & * & *  & *  \end{pmatrix},$} \quad \quad  \scalebox{0.9}{$\begin{pmatrix}
a  \varpi  &  *  & * &    \mfrac{d_{1} - a } { \varpi }       \\[0.2em]  
- c \varpi &   * & *   &  * \\  c  \varpi   ^ {  2 }    & & * &   -  c   \\   &  *  &  * & d_{1} 
\end{pmatrix}$}   
\end{alignat*} 
where $ * $ denotes an expression in the entries of $ h $ and an empty space means zero. It is then easily seen that first column in each of these matrices becomes an integral multiple of $ \varpi $ if we require it to lie in $ K $, which is a contradiction. Moreover 
\begin{align*} 
\tau ^{-1}_{1} h  \tau  _{3}   =   
\scalebox{0.9}{$ \left(\begin{array}{cccc}
a\,\varpi  & a+c_1  & \mfrac{a + b + c_{1} - d_{1} }{\varpi } & d_1 -a\\[0.3em]
-c\,\varpi  & a_1 -c & \mfrac{a_1 - b_{1} - c - d  }{\varpi } & b_1 +c\\[0.3em]
c\,\varpi^2  & c\,\varpi  & c+d & -c\,\varpi \\
& c_1 \,\varpi  & c_1 -d_1  & d_1 \,\varpi 
\end{array}\right) $}   \quad \text{ and }  \quad  
\tau_{2}^{-1}   h  \tau  _{3} &  =   
\scalebox{0.9}{$ \left(\begin{array}{cccc}
 a  &  \mfrac{a+c_{1}}{\varpi}& * &  \mfrac{d_{1} - a}{\varpi}  \\[0.3em] 
-c & *  
& * 
& * 
\\
c\,\varpi^2  & c\,\varpi  & * & -c\,\varpi \\
& c_1 \,\varpi  & *  & d_1 \,\varpi 
\end{array}\right)  $} .
\end{align*}\\[-0.5em]  
If $  \tau_{1} h \tau_{3} \in K $, then  $ a_{1}  - c $, $  a_{1} - b_{1} - c - d $, $  b_{1} + c \in \Oscr_{F} $ and this implies that $ c - d \in \Oscr_{F} $. Since $ c + d \in \Oscr_{F} $ as well and $ 2 \in \Oscr_{F}^{\times} $,  we have $ c $, $ d \in \Oscr_{F }$. Similarly  we can deduce that $ c_{1} $, $  d_{1} \in \Oscr_{F} $. This forces all entries of $ h $ to be integral.   But then the first column is an integral multiple of $ \varpi $, a contradiction. 
Finally, note that if $ \tau_{2} h \tau_{3}^{-1} \in K $, then $ a , c , c_{1}, d_{1} \in \Oscr_{F} $ and column expansion along the fourth row forces $ \det ( \tau_{2}^{-1} h \tau_{3}) \in \varpi \Oscr_{F} $, a contradiction.  
\end{proof} 
For $ w \in W_{I} $, let $ \mathscr{R} (w) $ denote $ U \backslash K w K / K $.  When writing elements of $ \mathscr{R}(w) $, we will only write the corresponding representative elements in $ G $ and it will be understood that these form a complete system of representatives.  For $ g \in G $, we denote $ H \cap g K g^{-1} $ simply by $ H_{g} $. Observe that $$   \begin{psmallmatrix} &  & 1 \\  &  &  & 1    \\   - 1 \\  & -  1    \end{psmallmatrix} \in     H _{\tau_{1}}   $$   
is a lift of $ s_{0} s_{2} \in W_{H} $. Therefore 
 $ U \varpi^{\lambda} \tau_{1} K = U \varpi^{s_{0}s_{2}(\lambda)} \tau_{1}K.$  
\begin{proposition}          \label{GU4decompose}  If $ 2 \in \Oscr_{F}^{\times} $, then  
\begin{itemize}       [before = \vspace{\smallskipamount}, 
]    \setlength\itemsep{0.3em}      
\item $ \mathscr{R} (w_{0} \rho^{2}) = \left \{  \varpi^{(2,2,1)}, \, \varpi^{(2,1,2)}, \,   \varpi^{(1,1,0)}\tau_{1}, \, \tau_{3} \right \} ,  $   
\item  $  \mathscr{R} ( w_{0} w_{1} w_{0} \rho^{4} ) =   \left \{ \varpi^{(4,3,3)}, \,   \varpi ^ { (3,2,2) } \tau_{1},  \,  \varpi^{(2,1,1)} \tau_{2 }  \right \}   $.  
\end{itemize}   

\end{proposition}   
\begin{proof} Note that  Lemma \ref{distinctgu4} implies that $ U \varpi^{\lambda} \tau_{i} K \neq U  \varpi^{\mu} \tau_{j} K $ for any $ \lambda , \mu \in  \Lambda $ if  $ i \neq j $. Lemma \ref{distinctgen} implies that $ U\varpi^{\Lambda} K $ is in one-to-one correspondence with $ U \varpi^{\Lambda} U $, and Cartan decomposition for $ \Hb $ distinguishes $ \varpi^{(2,2,1)} $ and $ \varpi^{(2,1,2)} $ in $ \mathscr{R}(w_{0} \rho^{2}) $. Thus all the listed  elements represent distinct classes.  It remains to show that these also exhaust all of the classes.      \\ 

\noindent  $\bullet $ $ w = w_{0} \rho^{2} $. From the Weyl orbit diagram drawn in the proof of Proposition \ref{GU4Heckepolynomial} (and Theorem \ref{BNrecipe}), we see that $ Kw K / K   =  \mathrm{im}(\mathcal{X}_{w})  \sqcup   \mathrm{im} (\Xcal_{w_{1} w }  )   \sqcup   \mathrm{im} (   \mathcal{X}_{w_{2} w_{1} w  } )  \sqcup  \mathrm{im}  (  \Xcal_{w_{1} w_{2} w_{1} w } )  $.   
Thus to describe $  \mathscr{R}(w) $, it suffices to study the orbits of $ U $ on Schubert cells corresponding to the words $ \sigma _{0} : = w_{0} \rho^{2} $, $  \sigma _{1}  : =  w_{1} \sigma _{0} $ and $  \sigma _{2}   :  = w_{1} w_{2} \sigma _{1} $. These cells are 
$$ \mathrm{im}(\mathcal{X}_{\sigma_{0}})  =  \scalebox{0.9}{$\Set* {\left(\begin{array}{cccc} 1
\\%[0.1em] 
 & \varpi  & 
 &  \\%[0.1em]  
 x \varpi  & & \varpi^{2} &  \\
 &  & 
 & \varpi 
\end{array}\right) K  \given  x \in  \xi  [\kay_{F}]  
}$,}\quad \quad  \mathrm{im}(  \mathcal{X}_{\sigma_{1}  })  =    \scalebox{0.9}{$\Set* {\left(\begin{array}{cccc} \varpi  & a  
\\%[0.1em] 
 & 1  & 
 &  \\%[0.1em]  
  & & \varpi &  \\
 &  
 x  \varpi  & - \bar{a}  \varpi  & \varpi ^{2} 
\end{array}\right) K  \given \begin{aligned} & a \in [\kay_{E}]  ,  \\ &   x   \in  \xi   [\kay_{F}]   \end{aligned}   
}$,} $$ 
$$  \mathrm{im} ( \mathcal{X}_{\sigma_{2}}  )  = 
\scalebox{0.9}{$\Set* {\left(\begin{array}{cccc}
\varpi^2  & a_1 \,\varpi  & a\,a_1 + \,y+ x \,   \varpi   & -\varpi \,\bar{a} \\[0.1em] 
 & \varpi  & a &  \\[0.1em]  
 & & 1 &  \\
 &  & -\bar{a} _1   &   \varpi 
\end{array}\right) K  \given   \begin{aligned} &  a ,  a_{1} \in [\kay_{E}]   ,  \\ &  \,   x,  y   \in    \xi   [\kay_{F}]   \end{aligned}  
}$}   .  
$$
For the $  \sigma _{0} $-cell, one eliminates the entry $ x \varpi $ by a row operation and conjugates by $ w_{\alpha_{0}}  : = \varpi^{(0,1,0)}w_{0} $ to arrive at the representative  $ \varpi^{(2,2,1)} $.  For the  $ \sigma_{1} $-cell, one eliminates $ x  \varpi $ and conjugate by $ w_{2}$ to arrive at   $$  \scalebox{0.9}{$ 
\begin{pmatrix} \varpi &    &  &  a  \\
 &    \varpi^{2}  &  - \bar{a}   \varpi 
 &  \\ 
& &   \varpi &  \\ 
&   &    & 1 
 \end{pmatrix}$}   $$ 
If $ a = 0 $, we get the representative $ \varpi^{(2,1,2)} $. If not, then conjugating by $ \mathrm{diag}(-a^{-1},1,-\bar{a},1) \in M^{\circ}  $ leads us to $ \varpi^{(1,0,1)}\tau_{1} $ and we have $ U \varpi^{(1,0,1)} \tau_{1} K = U \varpi^{(1,1,0)} \tau_{1} K $.

As for the  $   \sigma _{2}$-cell, begin by eliminating $  y + \varpi x  $ in the third column using a row operation. 
If $ a_{1} = 0 , a = 0 $, then we obtain the representatives  $ \varpi^{(2,2,1)} $. If $ a_{1} = 0 $, $ a \neq 0 $, we can conjugate $  \mathrm{diag}( \bar{a}^{-1} , 1, a , 1 ) \in M ^ { \circ } $ to obtain the representative   $ \varpi^{(1,1,0)} \tau_{1}   $.  Finally, if $ a_{1} \neq 0 $, we can conjugate by $ \mathrm{diag}(a_{1} ^{-1} , 1,  \bar{a}_{1} , 1 )$ to arrive at the matrix  $$  \scalebox{0.9}{$ 
\begin{pmatrix} \varpi^{2}  &  \varpi   & \, \, \, \, u &  - \varpi  \bar{u}  \\
 &    \varpi &  \, \, \, \,  u  &  \\ 
& & \, \,  \, \, 1  &  \\ 
&   &  - 1  &  \, \, \,  \varpi  
 \end{pmatrix}$}   $$ 
where $ u =  a / \bar{a}_{1}  \in  \Oscr_{E}   $. We can assume $ u \in \Oscr_{F} $ by applying row and column operations. If $ u = 0 $ at this juncture, we can conjugate by $ w_{2} $ and $ \mathrm{diag}(1,1,-1,-1) $ to obtain the reprsentative $ \varpi^{(1,1,0) } \tau_{1} $, and if $u \neq 0 $, then conjugating by $ \mathrm{diag}(1,1,u,u)$ gives us the representative  $ 
\tau_{3} $. So altogether, we have $$ K w_{0}\rho^{2} K  =  U \varpi^{(2,2,1) } K \sqcup U  \varpi^{(2,1,2)} K \sqcup U  \varpi^{(1,1,0)} \tau_{1} K  \sqcup U \tau_{3} K.   $$ \\[-0.5em] 
\noindent  $\bullet $ $ w =  w_{0} w_{1} w_{0} \rho^{4} $. The Schubert cells for this word were all written in Proposition \ref{GU4satake}(b).  Here we have to analyze the $ U $-orbits cells corresponding to words $  \sigma_{0}$ and $ \sigma_{2}$-
in the notation used there. We record the reduction steps for the $ \sigma_{2}$-cell, leaving the other  case  for the reader. 

Begin by eliminating the entries $ x_{1} \, \varpi^{2}   
$ and $ x  
 \varpi^{2} + y \varpi $ using row operations. Conjugating by  $ w_{2}  $ makes the diagonal $ \varpi^{(4,3,3)} $ and puts the entry $ a_{1} \varpi - \varpi^{2} \bar{a} $ and its conjugate on the top  right anti-diagonal. A case analysis of whether $ a, a_{1} $ are zero or not   gives us $ \varpi^{(4,3,3)} $, $ \varpi^{(3,2,2)} \tau_{1} $ and $ \varpi^{(2,1,1)} \tau_{2} $ as possibilities.  
\end{proof} 
\subsection{Zeta  elements}
Let $ \mathrm{U}_{1} $ be the $ F $-torus whose $ R $ points over a $ F $-algebra $ R $ are  given by $  \mathrm{U}_{1} (R ) =  \left \{ z \in  E ^ { \times } \, | \,  z \gamma(z) = 1 \right \} $. Then $ \mathrm{U}_{1}(F)  \subset \Oscr_{E}^{\times }   $  is  compact.    There is a homomorphism of $ F $-tori given $ \mathscr{N} : \mathrm{Res}_{E/ F} \GG_{m}  \to  \mathrm{U}_{1} $ given by $ z \mapsto  z / \gamma(z)  $ with kernel $ \GG_{m}  $. An application  Hilbert's Theorem 90 gives us that $ \mathscr{N} $ is surjective, inducing isomorphism  $ \Oscr_{E}  ^ { \times } /  \Oscr_{F} ^ { \times  }   =  E ^ { \times } / F ^ { \times }   \xrightarrow{ \sim }   \mathrm{U}    _{1}(F)  $.      
Denote  $ T = C :  =   \mathrm { U } _{1}(F) $, $ D =  \mathscr{N} (  \Oscr_{F} ^ { \times }  +   \varpi   \Oscr_{E} )   $, and define  $$  \nu :  H \to T, \quad (h_{1} , h_  {2 } )  \mapsto \det h_{2} / \det h_{1}   . $$   Fix $   \OO  $ an  integral domain containing $ \ZZ [ q ^{-1} ] $.   For  the  zeta  element  problem, we take  
\begin{itemize}           [before = \vspace{\smallskipamount}, after =  \vspace{\smallskipamount}]    \setlength\itemsep{0.2em}      
\item $ \tilde{G} : = G \times T $ the target group,
\item $ \tilde{\iota}  : = \iota \times \nu : H \to  \tilde{G}  $,  

\item $ M_{H,\OO } =  M_{H, \mathcal{O}, \mathrm{triv}}  $ the trivial functor, 
\item $ U $ and $  \tilde{K} : = K \times C $ as bottom levels,   
\item  $ x_{ U }  = 1 _{ \OO }  \in  M_{H, \OO } ( U ) $ as the source bottom class, 
\item $ \tilde{L}  =  K \times D $ as the layer extension degree $ d = q + 1 $, 
\item $  \tilde{\mathfrak{H}  }  _{c} =  \mathfrak{H}_{\mathrm{bc},c}(\mathrm{Frob})  \in   \mathcal{C}_{\OO} (   \tilde{K}  \backslash  \tilde{G} /  \tilde{K}  ) $  the Hecke polynomial where $ \mathrm{Frob}= \ch(C) $.   
\end{itemize}

\begin{theorem}   \label{GU4zeta}           If $ 2 \in \Oscr_{F}^{\times} $,    there  exists a zeta element for $ (x_{U} ,  \tilde{\mathfrak{H}}_{c},  \tilde{L}  ) $ for all $ c \in \ZZ  \setminus   2 \ZZ $.                     
\end{theorem}   

\begin{proof} For $ i = 0 , 1, 2 , 3 $, let $ g_{i} =  ( \tau_{i} ,1_{T} ) \in \tilde{G}  $  span a zeta element. Using  centrality of $ \rho^{2} $ and that $ c $ is odd,  we see that  $$ \tilde{\mathfrak{H}}_{c} \equiv (1 - \rho^{2})^{4} (\tilde{K}) - (1  -  \rho^{2}) ^2 ( \tilde{K}w_{0}\rho^{2} \tilde{K}) +  (\tilde{K} w_{0} w_{1} w_{0} \rho^{4} \tilde{K})  \pmod{q+1}  
$$
where we view $ w_{i} , \rho $ etc., as elements of $ \tilde{G} $ with $ 1 $ in the $ T $-component. It follows from  Proposition \ref{GU4decompose} that $$  H \backslash H \cdot  \supp(\tilde{\mathfrak{H}}_{c}) / \tilde{K}= \big \{ H g_{i} \tilde{K} \, | \, i = 0,1,2,3 \big  \} . $$  For $ i = 0 ,1 , 2 ,3 $,  let $ \mathfrak{h}_{i} \in \mathcal{C}_{\ZZ[q^{-1}]}(U\backslash H / H_{g_{i}})  $ denote the $ (H, g_{i}) $-restriction for  $ \tilde{\mathfrak{H}}_{c} $ (where $ H_{g_{i}} = H \cap g_{i} \tilde{K}  g_{i}^{-1} $) and $ d_{i} = [H_{g_{i}} : H \cap g_{i} L g_{i}^{-1} ] $. Then 
\begin{align*}  \mathfrak{h}_{0}  &  \equiv (1 - \rho^{2})^{4} ( U ) - ( 1  -  \rho^{2})^{2}     \Big  ( U \varpi^{(2,2,1)} U )  + ( U \varpi^{(2,1,2)} U ) \Big )  +  ( U \varpi^{(4,3,3)} U )  ,  \\
\mathfrak{h}_{1} &  \equiv - ( 1-  \rho^{2}) ^2 ( \tilde{U} \varpi^{(1,1,0)} H_{g_{1}}) +  ( U \varpi^{(3,2,2)} H_{g_{1}}  )  ,   \\
\mathfrak{h}_{2} &  \equiv  (U \varpi^{(2,1,1)} H_{g_{2}})  ,   \\
\mathfrak{h}_{3} &  \equiv  - ( 1 - \rho^{2})  ^ {  2 }  ( U H_{g_{3}})   . 
\end{align*} 
modulo  $ q + 1 $. Since $ \rho^{2} $ is central, we see that \begin{align*} \deg (\mathfrak{h}_{0}^{t})  & \equiv  \deg  \,  [ U \varpi^{(4,3,3)} U]_{*}  = (q+1) ^{2} \equiv 0  \pmod{q+1}   \\ 
\deg(\mathfrak{h}_{3}^{t})   & \equiv  0   \pmod{q+1} .
\end{align*}  
Now since $ d_{i} | (q+1) $ for all $ i $,  we see that $ d_{0} |  \deg (  \mathfrak{h}_{0}^{t} )  $ and $ d_{3}  | \deg  (  \mathfrak{h}_{3}^{t}  )  $.  Next observe that  $ H_{g_{i}} = H \cap \tau_{i} K \tau_{i}^{-1} $ for all $ i $. If we write $ h  \in H $ as in (\ref{gu4embed}), we see that for $ i = 1 , 2 $,  $$ \tau_{i} h \tau_{i}^{-1} =  
\scalebox{0.9}{$
\left(\begin{array}{cccc}
a & c_1  & \mfrac{b + c_{1} }{\varpi^{i} }
& \mfrac{d_1  - a }{\varpi^{i} }
\\[0.2em] 
-c & a_1  & \mfrac{a_1  - d }{\varpi^{i} }
& \mfrac{b_1 + c }{\varpi^{i} } 
\\
c\,\varpi  &  & d & -c\\[0.2em]  
 & c_1 \,\varpi  & c_1  & d_1 
\end{array}\right)$}  
$$
If  now  $ h \in H_{\tau_{i}} $, then the matrix above lies in $ K $ and thus all its entries must be in $ \Oscr_{F} $.   It is then easily seen that $ H_{\tau_{1}}, H_{\tau_{2}} \subset U$ and that $  \nu(h) \in  1 + \varpi \Oscr_{E}  \subset  D $. So $ d_{1} = d_{2} = 1 $ and   $ d_{1} \mid \deg(\mathfrak{h}_{1}^{t}) $, $ d_{2} |  \deg  (\mathfrak{h}_{2}^{t}) $  holds  trivially.  We have therefore established that $$ d_{i} | \deg(\mathfrak{h}_{i}^{t})  \quad \text{ for } \quad  i = 0 , 1 , 2 , 3 $$ and the claim follows by  Corollary    \ref{easyzeta1}.   
\end{proof}  

\begin{remark}    The value of $ c $ in our normalization that is relevant to the setting of \cite{Anticyclo} is $ 1 $ since $ (q_{E})^{\frac{1}{2}}  = q $. Note that for even $ c $, no zeta element exists in this setup. 
\end{remark}

\section{Spinor   \texorpdfstring{$L$}{}-factor of \texorpdfstring{$\mathrm{GSp}_{4}$}{}}  \label{gsp4section} 
In the  final  section, we study the zeta element  problem for the embedding discussed  in \S \ref{embeddingGsp4sec}.

\begin{notation*} 
The symbols $F $, $ \mathscr{O}_{F} $, $ \varpi $, $ \kay $, $q$ and $ [\kay] $  have the same meaning as in   Notation \ref{LfactorHeckesectionnota}.     Let $ \mathbf{G} $ be the reductive over $ F $ whose $ R$ points for a $ F $-algebra $ R $ are $ \left \{  g \in \GL_{4} ( R ) \, | \,   {}^{t} g J_{4} g =  \mathrm{sim}(g) J_{4}  \text{ for } \mathrm{sim}(g) \in R ^ { \times }  \right \} $ where $  
J_{4} = \left ( \begin{smallmatrix}  &  1_ { 2 }  \\   - 1 _{2} &  \end{smallmatrix}  \right )   $ 
is the standard  symplectic matrix. The map $ g \mapsto \mathrm(g) $ is referred to as the similitude character. We let $$ G = \mathbf{G}(F) 
, \quad K = G \cap \GL_{4} (  \Oscr_{F} ) . $$ 
For a ring $R$, we let $\mathcal{H}_{R}  =    \mathcal{H}_{R}(K \backslash G / K)$ denote the Hecke algebra of $G$ of level $K$ with coefficients in $R$ with respect to a Haar measure $\mu_{G}$ such that $\mu_{G}(K)=1$.    For convenience, we will sometimes denote  $ \ch(K \sigma K )  \in  \mathcal{H}_{R} $  simply   as $ (K \sigma K )$.  
\end{notation*}

\subsection{Desiderata}

Let $ \mathbf{A} = \GG_{m}^{3}   $, $ \mathrm{dis} : \mathbf{A} \to \mathbf{G }$ be the map $ (u_{0} , u_{1} , u_{2} ) \mapsto  \mathrm{diag} ( u_{1} , u_{2} ,  u_{0} u_{1} ^{-1} ,  u_{0} u_{2} ^{-1} ) $. Then $ \mathrm{dis} $ identifies $ \mathbf{A} $ with the maximal torus in $ \mathbf{G}  $.   We  let $ A = \mathbf{A}  ( F ) $ and $  A ^ { \circ } = A \cap K $ denote the unique maximal compact  open   subgroup. For $ i = 0 , 1 , 2 $,  let $ \phi _{ i } $, $ \varepsilon_{i} $ be the maps defined in $ \S \ref{GU4desiderata} $. As before, we let  $$ \Lambda =   \ZZ \phi_{0} \oplus \ZZ \phi_{1} \oplus \ZZ \phi_{2}   $$   denote the cocharacter lattice. The conventions for writing elements of $ \Lambda $ as introduced in \S \ref{GU4desiderata} are     maintained. 
The set $ \Phi $ of roots of $ \Gb $ relative to $ \mathbf{A} $ is  the set denoted $ \Phi_{F} $ in \S \ref{GU4desiderata}.   The half sum of positive roots is  
\begin{align*}  \label{halfsumgsp4} \delta =   2 \varepsilon_{1} + \varepsilon_{2} -  \sfrac{3}{2}    \varepsilon_{0}    
\end{align*}  
We let $ \alpha_{1} = \varepsilon_{1} - \varepsilon_{2} $, $ \alpha_{2} = 2 \varepsilon_{2} - \varepsilon_{0} $ as our  choice of simple roots. Then  $ \alpha_{0} =  2 \varepsilon_{1} - \varepsilon_{0} $ the highest root.  The groups $ W $, $ W_{\mathrm{aff}} $, $ W_{I} $,  $ \Omega $,  the  set $  S _ { \mathrm{aff} }$  are analogous to the ones defined in \S  \ref{GU4desiderata} 
We let $ \ell :  W_{I} \to  \ZZ $ denote   the length function on $ W_{I} $  The minimal length of elements in $ t(\lambda) W  \subset \Lambda \rtimes W \simeq W_{I} $  can be computed   using the formula    
(\ref{GU4minlength}).  Set $$  
w_{0} =   \scalebox{0.9}{$\begin{pmatrix}  & &    \frac {  1 } { \varpi  }   \\  &  1 \\  \varpi   \\ & &  &  - 1  \end{pmatrix}$},  \, \, \,  w_{1} =      \scalebox{0.9}{$     \begin{pmatrix}  &  1  & &  \\ 1  & \\  &  &  & 1 \\   & &  1    \end{pmatrix}$} ,  \,  \, \,  w_{2} =      \scalebox{0.9}{$\begin{pmatrix} 1 & & & \\ &  &  & 1  \\ &  &   - 1 & \\ &  1 & & \end{pmatrix}$},  \,  \,  \,     \rho  =     
\scalebox{0.9}{$\begin{pmatrix} & & &  1 \\ & & 1 & \\ &    \varpi  &   &    \\  \varpi  & &  & \end{pmatrix}$}     .$$
These represent the elements $ t(\alpha_{0}^{\vee} ) s_{0} $, $ s_{1} $, $ s_{2} $, $ t(-\phi_{0}) s_{2} s_{1}  s_{2} $  in $ W_{I }  $.    
We  let $ w_{\alpha_{0}}  : =  w_{1} w_{2} w_{1}  = \varpi^{\phi_{1}} w_{0} \in N_{G}(A) $, which is a  matrix representing  the  reflection $ s_{\alpha_{0}}   $.        For $ i = 0 , 1 ,2 $, let  $ x_{i} : \GG_{a} \to  \mathbf{G} $  be the root group maps 
\begin{equation}    \label{gsp4rtgrpmap}     
x_{0}: u \mapsto      \scalebox{0.9}{$ \begin{pmatrix}
1 & & & \\
& 1 & & \\
\varpi u& & 1 & \\
 & & & 1
\end{pmatrix}$},  \,  \,   \,  \,   x_{1}: u \mapsto   \scalebox{0.9}{$\begin{pmatrix}
1 & u & & \\
& 1 & & \\
& & 1 & \\
& & -  u & 1
\end{pmatrix}$},  \,   \,  \,   \,       x_{2}: u \mapsto   \scalebox{0.9}{$\begin{pmatrix}
1 & & & \\
& 1 &  & u  \\
& & 1 & \\
& & & 1
\end{pmatrix}$}  
\end{equation}    
and let $ g_{i} :  [  \kay   ]   \to G $ be the map $ \kappa \mapsto x_{i} ( \kappa ) w_{i} $.   If $ I $ denotes the Iwahori subgroup of $ K $ whose reduction modulo $ \varpi $ lies in the Borel of $ \Gb(\kay) $ determined by $ \Delta = \left \{ \varepsilon_{1} - \varepsilon_{2}, 2 \varepsilon_{2} - \varepsilon_{0} \right \} $,  then $ I w_{i} I  = \bigsqcup_{\kappa \in [\kay_{w_{i}}]}  g_{w_{i}} ( \kappa ) .  $      For $  w \in W_{I }$ such that $ w $ is the unique minimal length element in $ w W $, choose a reduced word decomposition  $  w  =  s_{w, 1 }  \cdots s_{w, \ell(w) }  \rho_{w}  $, where $ s_{w,i} \in  S_{\mathrm{aff}} $, $ \rho_{w} \in \Omega $,  a reduced word  decomposition. As usual, define  
\begin{flalign}   \label{GSp4map}      \mathcal{X}_{w}  :   [ \kay   ] ^ { \ell(w)  }   &  \to    G / K    \\
 ( \kappa_{1} ,   \ldots,  \kappa_{\ell(w) }  )   &   \mapsto  g_{s_{w, 1} }     ( \kappa_{1} )  \cdots  g_{s_{w , \ell(w)}  }  ( \kappa_{\ell(w)} )   \rho_{w} K   \nonumber 
\end{flalign}   
Then $ \mathrm{im}(\mathcal{X}_{w} ) $ is independent of the choice of decomposition
of $ w $ by   Theorem \ref{BNrecipe}.

\subsection{Spinor  Hecke  polynomial}    The dual  group of  $ \mathbf{G} $ is $ \mathrm{GSpin}_{5} $ which has a four dimensional representation called the \emph{spin representation}. The highest coweight of this representation is $ \phi_{0}  + \phi_{1} + \phi_{2}   $ (see \S \ref{embeddingGsp4sec} for arithmetic motivation) which is minuscule. By Corollary \ref{minuscriteria}, the coweights are $ \frac{2\phi_{0} + \phi_{1} + \phi_{2}}{2}  \pm \frac{\phi_{1}}{2} \pm \frac{\phi_{2}}{2}  $. The Satake polynomial is therefore $$ \mathfrak{S}_{\mathrm{spin}}  ( X  ) =  ( 1 - y_{0} X ) ( 1 - y_{0} y_{1} X ) ( 1- y_{0} y_{2} X ) ( 1 - y_{0} y_{1}  y_{2 } X )   \in  \ZZ [ \Lambda ] ^ {W } [X]  $$
where $ y_{i} = e ^ { \phi_{i}} \in \ZZ [ \Lambda] $. Let $ \mathcal{R} =  \ZZ [ q ^ { \pm \frac{1}{2} } ] $, and  $ \mathscr{S}  : \mathcal{H}_{  \mathcal{R } } ( K \backslash G / K ) \to      \mathcal{R} [ \Lambda] ^{W} $ denote the Satake isomorphism (\ref{Satakeform}).  For $ c \in \ZZ  - 2 \ZZ $,  the  polynomial    $ \mathfrak{H}_{\mathrm{spin}, c} (X) $  is defined so that $ \mathscr{S}  \left  ( \mathfrak{H}_{\mathrm{spin},c} ( X )  \right  ) =  \mathfrak{S}_{\mathrm{spin}}(q^{-c/2} X ) $.
\begin{proposition}   \label{Gsp4Heckepolynomial}                           For  $ c \in  \ZZ \setminus 2 \ZZ $,  
\begin{align*}    \mathfrak{H}_{\mathrm{spin},c}(X) &  =   (K) -  q ^ {  - \frac{c+3}{2}  } (K \rho K ) X   \\
& +  q ^{-(c+2)}   \left  (  ( K w_{0} \rho ^{2} K )  + ( q^{2} + 1 ) (   K \rho^{2} K )    \right ) X^{2}  \\
& - q ^ { - \frac{3c+ 3}{2} }  ( K \rho^{3}K ) X^{3}  +   q^{-2c}  ( K \rho ^ {4} K )  X^{4}  \in  \mathcal{H}_{\ZZ[q^{-1}]}( K \backslash G    /  K  ) [X] .  
\end{align*}
where the words appearing in each Hecke operator are of minimal possible length. 

\end{proposition}

\begin{proof}  We have $$ \mathfrak{S}_{\mathrm{spin}}(X) =  1 - e^{W(1,1,1)} X +  \left (  e ^ { W(2,2,1)} + 2 e ^{W(2,1,1)} \right ) X ^{ 2}    -  e ^ { W(3,2,2)} X^{3} +  e ^ { (4,2,2) }  X^{4} .  $$
The lengths of the cocharacters appearing as  exponents in the coefficients of $ \mathfrak{S}_{\mathrm{spin}}(X) $  is computed using the formula (\ref{GU4minlength})  and  the  corresponding  words   are  easily  found.    The leading coefficient  (see Definition \ref{Satakeleadingtermdefi}) of $ K \varpi^{\lambda} K $ for $ \lambda \in \Lambda^{+} $ is   $ q ^ { -  \langle \lambda , \delta \rangle } $ (Corollary \ref{Satakeuppercoro}) shifted by an appropriate power of $ q^{-c/2}   $, which are  easily  computed.      The  coefficient of the non-leading term $  (K \rho^{2} K) $  in the monomial $ X^{2} $ is computed as follows. Consider the Weyl orbit diagram  
\begin{equation}  \label{gsp4weylorbitdiagram}    
\begin{tikzcd} 
{(2,0,1)} \arrow[r, "s_{1}"] & {(2,1,0)} \arrow[r, "s_{2}"] & {(2,1,2)} \arrow[r, "s_{1}"] & {(2,2,1)}
\end{tikzcd}
\end{equation}
of $ (2,2,1) $.    
From (\ref{gsp4weylorbitdiagram})
and Theorem \ref{BNrecipe},  we see that  $ | K w_{0} \rho^{2} K / K | =  q + q^{2} + q^{3} + q^{4}  $. 
Since the leading coefficient of the Satake transform   of  $ ( K w_{0} \rho^{2} K ) $ is $ q^{\langle (2,2,1), \delta \rangle   } = q^{2} $, the number of cosets in $ K w_{0} \rho^{2} K / K $ whose shape  lies in the $ W $ orbit of $  (2,2,1)    $ is $$ \sum   \nolimits  _{ \mu \in W (2,2,1)} q ^ { \langle     (2,2,1) + \mu , \delta    \rangle    }  =  1 + q + q ^{3} + q ^{4} .  $$ Thus the required coefficient is $ q^{-c} $ multiplied with  $ 2 - q^{-2}( q + q ^{2} + q^{3} + q^{4}  -  ( 1  + q + q^{3} + q^{4} ) ) = q^{-2} (  q^{2} + 1  ) $.
\end{proof}   
\begin{remark} The formula for $ \mathfrak{H}_{\mathrm{spin},c}$ is again  well known, e.g.,  see \cite[Lemm 3.5.4]{LSZ} or \cite[Proposition  3.3.35]{Andrianov} where $ c $ is taken to be $    - 3 $.  We  have however included a proof for  completeness and to provide a check on our computations.       
The  dual  group  of   $ \mathbf{G}  $ also has a $ 5 $ dimensional representation called the \emph{standard representation}. Its highest   coweight is $ \phi_{1} $ and it's  Satake polynomial is $$  \mathfrak{S}_{\mathrm{std}}(X) =  ( 1-  X )( 1 -  y_{1}^{-1} X ) (   1 -   y_{1} X ) ( 1 - y_{2}  ^{-1}  X  ) (  1-  y_{2} X ) .  $$
Cf.  the   polynomial   $   \mathfrak{S}_{\mathrm{bc}}(X) $ of \S \ref{GU4Hecesectino}.  See   \cite{Asgari} for a discussion of this $ L$-factor.   
\end{remark} 

\subsection{Mixed coset decompositions}    

Let $ \mathbf{H} $ be the subgroup of $ \mathbf{G} $ generated by $ \mathbf{A} $ and the root groups of $ \pm \alpha_{0} $, $   \pm  \alpha_{2 }$.  Then $ \mathbf{H} \cong \GL_{2} \times_{\det} \GL_{2} $, the fiber product being over the determinant map. 
Explicitly, we get an embedding   
\begin{align*} 
\iota :  \Hb  
\to \Gb  \quad \quad  \quad    
\left ( \left (  \begin{smallmatrix} a & b \\ c & d  \end{smallmatrix}  \right )   ,  \left (   \begin{smallmatrix}  a_ {  1 } & b_{1} \\ c _ {1 } &  d_{1}  \end{smallmatrix}  \right ) \right )  
\mapsto  \left (    \begin{smallmatrix} a & & b \\   &   a_{1}    & &  b_{1} \\   c &  & d  \\ &  c_{1} & &  d_{1}   \end{smallmatrix}   \right ) 
\end{align*}    
Set  $ H =  \mathbf{H}(F) $, $ U = H \cap K  $,  $ W_{H} = \langle s_{0} , s_{2}  \rangle \cong S_{2} \times S_{2} $ the Weyl group of $ H $ and $  \Phi_{H} := \left \{ \pm \alpha_{0},  \pm \alpha_{2} \right \} $ the set of roots of $ \mathbf{H} $. 
For convenience in referring to the components of $ H $, we let $ H_{1}, H_{2} $ denote $ \GL_{2}(F) $ (so that $ H = H_{1} \times_{F^{\times}   } H_{2} $) and $ \pr_{i} : H \to H_{i} $ for $ i =1,2 $ denote the natural projections  onto the two component groups of $ H $.   To describe the twisted $H$-restrictions of the spinor Hecke polynomial, we introduce the following elements in $ G  $: $$ \tau_{0} = \scalebox{0.9}{$\begin{pmatrix}   1 &  &  &   \\
&   1  &  \\    
& &  1 & \\     
& & & 1 \end{pmatrix}$}, \quad \quad  \quad    \tau_{1}    =   
\scalebox{0.9}{$\begin{pmatrix}   \varpi &  &  & 1   \\
&    \varpi &  1   \\[0.1em]   
& &     1 & \\     
& & & 1 \end{pmatrix}$},  $$ 
As  in  \S   \ref{GLnLfactorsection}, we will need to know the strucuture of $  H_{\tau_{1}} = H \cap \tau_{1} K \tau_{1}^{-1} $,  
Let $ \ess =   \begin{psmallmatrix} & 1 \\ 1 & \end{psmallmatrix} \in \GL_{2}(F) $ and define 
\begin{align}   \label{jmathmap}  \jmath : \GL_{2}(F)  & \hookrightarrow   H 
\\  \notag  h 
  & \mapsto ( h, \ess h   \ess 
) . 
\end{align}  
Let $ \mathscr{X}^{\circ} = \jmath(\GL_{2}(\Oscr_{F}) $, $ \mathscr{X} = \jmath(\GL_{2}(F) ) $ and 
$ J $   be the compact open subgroup of $  H_{\tau_{1}} $ whose reduction modulo $ \varpi $ lies in the diagonal torus of $ \mathbf{H}(\kay) $.      

\begin{lemma}   \label{GSp4Htau}   $  H _ {  \tau_{1} }  =  \mathscr{X} ^ {  \circ  }     J  \subsetneq U $. In particular, $ H K$ and $ H \tau_{1} K $ are disjoint. 
\end{lemma}  
\begin{proof}  Let  $ h =  ( h_{1} , h_{2}  )    \in H $ and say $   
h_{i} : =   \left (   \begin{smallmatrix}  a_{i} & b_{i}  \\ c_{i}   & d_{i} \end{smallmatrix}  \right )  $
where $ a_{i}, b_{i}, c_{i} , d_{i} \in F $. Then $ h \in H_{\tau_{1}} $ implies that    
$$ a_{1}, a_{2} , c_{1}, c_{2} , d_{1} , d_{2} \in \Oscr_{F}  \quad   \quad  \text{ and }   \quad   \quad a_{1} - d_{2}   ,   
a_{2} - d_{1}   , b_{1} - c_{2},  b_{2} - c_{1}   \in \varpi \Oscr_{F} . 
$$  
It follows that $  \mathscr{X} ,  \, J \subset   H_{\tau_{1}}  \subset  U $. 
In particular,  
$ H_{\tau_{1}  }  \supset  \mathscr{X} J $. 
For the reverse inclusion, say   $ h = (h_{1}, h_{2} ) \in H_{\tau_{1} } $. Since $ \jmath(h_{1}) \in \mathscr{X} \subset H_{\tau_{1}} $, we see that  $ (h_{1}', h_{2}') :=  \jmath(h_{1}^{-1}) \cdot  h $ lies in $  H_{\tau_{1} } $. By construction, we  have $ h_{1} ' = 1_{H_{1}} $. The conditions of the membership $ (1_{H_{1}} , h_{2} ' ) \in H_{\tau_{1}} $ force  are 
$ (1_{H_{1}} , h_{2}' ) \in  J   $. For the second claim, note that $ H_{\tau_{1}} \neq U $ since $ A^{\circ}  \not \subset H_{\tau_{1}} $ and  invoke  Lemma \ref{volumemeansdistinct}.  
\end{proof} 
For $ w \in W_{I} $,  let $ \mathscr{R}(w) $ denote the double coset $ U \backslash K w K / K $.  As before,   we l only write the representative elements when describing $ \mathscr{R}(w) $ and these  representatives are understood to be distinct.

\begin{proposition}    \label{GSp4decompose}    We have                     
\begin{itemize}      [before = \vspace{\smallskipamount}, after =  \vspace{\smallskipamount}]     \setlength\itemsep{0.2em}    
\item $ \mathscr{R} ( \rho )  =   \left  \{  \varpi^{(1,1,1)} ,  \,  
\tau_{1}  \right \} $, 
\item $  \mathscr{R}  (   w _{0}  \rho  ^ { 2 }  )   =   \left \{ \varpi ^{(2,2,1)} , \,   \varpi^{(2,1,2)}  ,  \,    \varpi^{ (1,1,0)} \tau_{1}    \right \}     $. 
\end{itemize}

\end{proposition}    

\begin{proof}  Since $ H K $ and $ H \tau_{1} K  $ are disjoint, $ H_{\tau_{1}} \subset U $ and  $ U \backslash H / U \simeq W_{H} \backslash \Lambda $,  the listed elements represent distinct classes in their respective  double coset spaces. To show that they represent all classes, we study the orbits on $ K w K / K $  using   Theorem \ref{BNrecipe}.  \\

\noindent  $ \bullet $ Let $ w =  \rho $. We have $ K w K  / K =   \bigsqcup  _ { \sigma }  \mathrm{im} ( \mathcal{X}_{\sigma} ) $ for $ \sigma \in \left \{ w, w_{2} w, w_{1} w_{2} w, w_{2} w_{1} w \right \} $.  To obtain the mixed  representatives, we need to analyze the $ U $-action on the cells corresponding to the words $ \sigma_{0} =  \rho $ and $ \sigma_{1} =  w_{1} w_{2} \rho $.  The first is a singleton and gives $ \varpi^{(1,1,1)} $ (after conjugating by $ w_{\alpha_{0}}  w_{2} $). As for $ \sigma_{1} $, we have    $$  
\mathrm{im}(  
\mathcal{X}_{ \sigma_{1} }     ) =        \scalebox{0.9}{$      \Set*  {\begin{pmatrix}   
\varpi  & a & y &   \\
& 1 & & \\
& & 1 & \\
& & -a & \varpi 
\end{pmatrix}     K  \given    a , y \in [\kay]   }$}    
$$ We can eliminate $ y $ by a row operation from $ U $, and conjugating by $ w_{2} $ gives us a matrix with diagonal $ \varpi ^{(1,1,1)} $. If $ a = 0 $, we obtain $ \varpi^{(1,1,1)} $ and  if   $ a \neq 0 $, we conjugate by $   \mathrm{diag}(1, 1, a , a)  $ to obtain $ \tau_{1} $. 
\\

\noindent   $ \bullet $   Let $ w =        w_{0}          \rho  ^{2}      $. From diagram (\ref{gsp4weylorbitdiagram}), we have $ K w K  / K = \bigsqcup_{\sigma}  \mathrm{im} ( \mathcal{X}_{w} ) $ for $ \sigma \in \left \{ w, w_{1} w, w_{2}w_{1} w, w_{1} w_{2} w_{1}w \right \}$ and it suffices to analyze the cells corresponding to $  \sigma_{0} = w, \sigma_{1} = w_{1} w $, $ \sigma_{2} = w_{1} w_{2} w_{1} w $. These cells are as follows:  
$$ \mathrm{im}(\mathcal{X}_{\sigma_{0}})  =  \scalebox{0.9}{$\Set* {\left(\begin{array}{cccc} 1
\\%[0.1em] 
 & \varpi  & 
 &  \\%[0.1em]  
 x \varpi  & & \varpi^{2} &  \\
 &  & 
 & \varpi 
\end{array}\right) K  \given  x \in \kay  }$,}\quad \quad  \mathrm{im}(  \mathcal{X}_{\sigma_{1}  })  =    \scalebox{0.9}{$\Set* {\left(\begin{array}{cccc} \varpi  & a  
\\%[0.1em] 
 & 1  & 
 &  \\%[0.1em]  
  & & \varpi &  \\
 &  
 x  \varpi  & - 
 a \, \varpi  & \varpi ^{2} 
\end{array}\right) K  \given 
a, x  \in [\kay]  ,  
}$,} $$ 
$$  \mathrm{im} ( \mathcal{X}_{\sigma_{2}}  )  = 
\scalebox{0.9}{$\Set* {\left(\begin{array}{cccc}
\varpi^2  & a_1 \,\varpi  & a\,a_1 + \,y+ x \,   \varpi   & a \,   \varpi 
\\[0.1em] 
 & \varpi  & a &  \\[0.1em]  
 & & 1 &  \\
 &  & - a_{1}  & \varpi 
\end{array}\right) K  \given  a , a_{1}, x , y \in [\kay]    
}$}   .  
$$
The $ \sigma_{0} $-cell obviously leads to $ \varpi^{(2,2,1)} $. For the $ \sigma_{1} $-cell we can eliminate $ x \varpi $, conjugate by $ w_{2} $. If $ a = 0 $, we have $ \varpi^{(2,1,2)} $ at our hands and if not, then conjugating by $ \mathrm{diag}(1,1,a,a) $ gives us $ \varpi^{(1,0,1)} \tau_{1} $. Now observe that since  $ \jmath(\ess) \in  H_{\tau_{1}} $ is a lift of $ s_{0} s_{2} $, we have $$ U \varpi^{(1,0,1)} \tau_{1} K = U \varpi^{(1,1,0)} \tau_{1} K . $$    
Finally for the $ \sigma_{2} $-cell, begin by eliminating $ aa_{1} + y  +    \varpi x $. Next note that conjugation by $ w_{2} $ swaps $ a_{1} $ and $ a $. Using row and column operations, we can assume that $ a_{1} = 0 $. If $ a = 0 $, we end up with $  \varpi^{(2,2,1)} $ and if not, then  conjugation by  $ \mathrm{diag}(1,1,a,a) $ gives us $ \varpi^{(1,1,0)} \tau_{1}   $.   
\end{proof}

\subsection{Schwartz space   computations}    
Let   $  X : = F^{2} \times F^{2} $ considered as a totally disonnected topological spaces.  We view elements of $ X $ as pairs of $ 2 \times 1 $  column   vectors. We  let  $ H_{1} \times H_{2} $ act on $ X $ on the right  via $$ ( \vec{  u  }  , \vec{  v  }   ) \cdot (h_{1}, h_{2})  \mapsto  (  h_{1} ^ { - 1  }  \vec {   u   }  ,   h _{ 2 }  ^ { - 1}   v)  \quad \quad  \vec{  u }  ,   \vec{ v    }   \in F^{2}, \,  h_{1} \in H_{1},  \,   h_{2} \in H_{2} . $$
Via the natural embedding $ H \hookrightarrow H_{1} \times H_{2} $, we obtain an action of $ H $ on $ X $.  
Let  $ \OO $ be an integral domain that contains $ \ZZ [ q ^{-1} ] $ and let $  \mathcal{S}_{X}  =  \mathcal{S}_{X,\mathcal{O}}  $ be $ \mathcal{O}$-module of 
all functions  $ \xi : X \to \OO $ which are locally constant and   compactly supported  on $ X$. Then $ \mathcal{S}_{X} $ has an induced  left action $ \mathcal{S} \times H \to \mathcal{S} $ via $ (h,  \xi    ) \mapsto  \xi ( ( - ) h  )  $,  which makes $ \mathcal{S} $ a smooth representation of $ H $.  Let $ \Upsilon_{H} $ be the set of all compact open subgroups of $ H $ and $$ M_{H,\OO} :  \mathcal{P}(H, \Upsilon _  {  H } 
  )  \to  \mathcal{O} \text{-Mod}$$ denote the functor $ V \mapsto \mathcal{S}_{X}^{V} $ associated with $ \mathcal{S}   $  (see Definition \ref{Mtrivdefinition}).  For $ u, v, w , x  \in \ZZ $, let $ Y_{u } := \varpi^{u } \Oscr_{F}  \subset F $,  
$ Y_{u,v} = Y_{u} \times Y_{v} \subset F^{2} $ and   $   X  _{u,v,w,x}  := Y_{u,v} \times Y_{w, x} \subset X  $. We denote  $$  \phi_{(u,v,w,x)} := \ch ( Y_{u,v,w,x} ) , \quad \quad  \quad   \bar{\phi}_{(u,v,w,x)} = \phi_{(-u,-v,-w,-x)}  $$ where $ \mathrm{ch}(Y) $ denotes the  characteristic     function of $ Y \subset X $.  These belong to $ \mathcal{S} $.  We  also denote   $ \phi : =  \phi_{(0,0,0,0)} $ for simplicity.     The element $ \phi $ will serve as the source bottom class of the zeta element.   

\begin{lemma}      \label{explicithecke}       We have 
\begin{enumerate} [ label = \normalfont (\alph*)   ,   before = \vspace{\smallskipamount}, after =  \vspace{\smallskipamount}]     \setlength\itemsep{0.3em}         
    \item  
    $ [ U \varpi^{(1,1,1)} U ]_{*} ( \phi 
    ) =   \bar{\phi}_{(1,1,1,1)} + q ( \bar{\phi}_{(1,1,0,0)} +  \bar{\phi}_{(0,0,1,1)} )  + q^{2}  \phi  $, 
    \item   
    $ [ U \varpi^{(2,2,1)}  U ]_{*}(  \phi 
    )  =   \bar{\phi}_{(2,2,1,1)} +  ( q - 1 )  \bar{\phi}_{(1,1,1,1)}  +  q^{2}  \bar{\phi}_{(0,0,1,1)}  $, 
    \item 
    $ [ U \varpi^{(2,1,2)}  U ] _{*}(  \phi 
    )  =  \bar{ \phi}  _{(1,1, 2,2)}  +  (  q - 1 )  \bar{\phi}_{(1,1,1,1)}   +  q ^{2} \bar{\phi}_{(1,1,0,0)} $.

   \end{enumerate}             
\end{lemma}

\begin{proof}     If we denote $ U_{1} = U_{2} := \GL_{2}( \Oscr_{F} )$ and pick any $ \lambda = (a_{0}, a_{1}, a_{2} ) \in \Lambda $, we have  
$$ [ U \varpi ^{\lambda} U ]_{*} \big (\phi_{(u,v,w,x)}\big ) = [U_{1} t_{1} U_{1} ]_{*} \big   (\phi_{(u,v)} \big  )  \, \otimes \,  [U_{1} t_{2} U_{2}]_{*} \big ( \phi_{(w,x)}  \big     )  $$ where $ t_{i} = \mathrm{diag} (\varpi^{ a_{i} },   \varpi ^ { a_{0} - a_{i} }     )  $ for $ i = 1, 2 $  and  $    \phi_{(a,b)} : F^{2} \to \mathcal{O} $ denote   the characteristic function of $ Y_{a} \times Y_{b}              $ for $ a,b, \in \ZZ $.  The resulting functions  can be computed using the decomposition recipe of Theorem \ref{BNrecipe}. See also 
 \cite[Lemma 9.1]{Siegel1} for a more general  result.  
\end{proof} 
To facilitate checking the trace criteria  for one of the twisted restrictions, we do a preliminary calculation.  Let $ \mathrm{Mat}_{2 \times 2 }(F) $ be the  $ F $-vector space $ 2 \times 2 $ matrices over $F $. We make the identification 
\begin{align}  \label{imath}  \imath : X  & \xrightarrow{\sim}     \mathrm{Mat}_{2 \times 2} ( F )  \quad \quad \quad  \scalebox{1.1}{$
\left ( \left ( \begin{smallmatrix} u_{1} \\ u_{2}  \end{smallmatrix} \right ),  \left (   \begin{smallmatrix} v_{1} \\ v_{2}  \end{smallmatrix}   \right   ) \right )  
\mapsto  \left ( \begin{smallmatrix} u_{1} &  v_{2}  \\   u_{2}   
&  v_{1}   \end{smallmatrix}  \right )$}  .
\end{align} 
and define a right action 
\begin{equation}
\label{*action}  \mathrm{Mat}_{2 \times 2}(F) \times \GL_{2}(F)  \to  \mathrm{Mat}_{2 \times 2}(F) \quad \quad  (h, M) \mapsto h^{-1} M 
\end{equation}  
Then for all  $ h \in \GL_{2}(F) $ and $ (\vec{u}, \vec{v} ) \in X $, $$ \imath \big ( (\vec{u} , \vec{v} ) \cdot \jmath ( h ) \big )  = \imath( \vec{ u } ,  \vec{ v   }  ) \cdot   h
$$
where $ \jmath $ is as in  (\ref{jmathmap}) and the action on the right hand side is (\ref{*action}).     Let $ \psi  \in   \mathcal{S}_{X} $ denote the function such that $  \psi \circ \imath^{-1} : \mathrm{Mat}_{2 \times 2 }(F) \to \mathcal{O}   $ is the characteristic function of $ \mathrm{diag}(\varpi, \varpi)^{-1} \cdot  \GL_{2} ( \Oscr_{F} ) $. 
Let
\begin{equation}    \mathfrak{h}_{1}'  : = q ( U  H_{\tau_{1}}  ) 
- ( U 
\varpi^{(1,1,0)}  H_{\tau_{1}} 
)  +   (   U \varpi ^ {( 2, 1,1) }  H_{\tau_{1}  }   )      \in   \mathcal{C}_{\ZZ}(U \backslash H / H_{\tau_{1}}   )  \label{gsp4frakhpoly}  \end{equation} 
\begin{lemma}   \label{gsp4frakh}         $ \mathfrak{h}_{1,*}'(\phi) = \psi $.    
\end{lemma}      
\begin{proof} By  Lemma \ref{GSp4Htau},  $ U H_{\tau_{1}} = U \mathscr{X}^{\circ}$, $
U \varpi^{(2,1,1)} H_{\tau_{1}} = U \varpi^{(2,1,1)} \mathscr{X}^{\circ}$ and $  
U \varpi^{(1,1,0)}  
H_{\tau_{1}} 
=    U \varpi^{(1,1,0)}  \mathscr{X} ^ { \circ }$,    
where  we used that $ \varpi^{(1,1,0)} J \varpi^{-(1,1,0)}  
\subset U $  in the last equality.  
Moreover  $ \varpi^{(1,1,0)} = \jmath \big (\mathrm{diag}(\varpi, 1)  \big  )  $ and $  \varpi^{(2,1,1)} = \jmath \big  ( \mathrm{diag}(\varpi , \varpi ) \big ) $ and $  U \cap  \mathscr{X} = \mathscr{X}^{\circ} $. A straightforward analogue of  Lemma \ref{distinctgen} implies that we have a bijection 
\begin{align*} \mathscr{X}^{\circ}  \backslash  \mathscr{X}^{\circ}  h \mathscr{X}^{\circ}  &  \xrightarrow{\sim}  U \backslash U \jmath(h) \mathscr{X}^{\circ}\\ 
\mathscr{X}^{\circ} \gamma  &   \mapsto  U \jmath(\gamma)  
 \end{align*} 
Therefore   $ \mathfrak{h}_{1,*} '( \phi ) = 
\left ( q \,  \jmath(\phi) - T_{\varpi} ^{t} \cdot  
\jmath(\phi) +  S_{\varpi}^{t}
\cdot  \jmath(\phi) \right ) \circ \imath  $ where $ T_{\varpi} $, $ S_{\varpi} $ are the Hecke operators of $ \GL_{2}(F) $ given by the characteristic functions of  $ \GL_{2} ( \Oscr_{F} ) $-double cosets of $ \mathrm{diag}(1,  \varpi ) $, $  \mathrm{diag}(\varpi, \varpi  )  $  respectively, $ T_{\varpi}^{t}  $, $ S_{\varpi}^{t} $ denotes their transposes and the action of these operators is via  (\ref{*action})    
Now $  \jmath(\phi) $ is just the characteristic function of $ \mathrm{Mat}_{2\times 2 }( \Oscr_{F}  ) $. A 
straightforward  computation shows that the function   $$ q \, \jmath(\phi) - T_{\varpi} ^{t} \cdot  \jmath(\phi) + S_{\varpi}^{t}  \cdot  \jmath(\phi)  $$ on $ \mathrm{Mat}_{2 \times 2}(F) $  vanishes on any  matrix whose entries are not in $ \varpi^{-1} \Oscr_{F} $ or whose determinant is not in $ \varpi^{-2} \Oscr_{F}^{\times} $. The claim  follows.     
\end{proof}    
\begin{remark} A very closely related computation  appears in \cite[Proposition 1.10]{Colmez2002-2003} in the context of Kato's Euler system, which is what inspired the choice of $ \mathfrak{h}_{1} $ above. 
\end{remark}

\subsection{Zeta   elements} Following the discusion in  \S  \ref{embeddingGsp4sec}, we introduce    $ T =   F  ^ { \times  }   $, $ C  =  \Oscr_{F}^{\times} $  and $ D = 1 +  \varpi \Oscr_{F} \subset C $. We let $  \nu =  \mathrm{sim} \circ  \iota  : H \to T $ be the map that sends $  (h_{1} , h_{2} ) $ to the common determinant of $ h_{1} , h_{2}  $. For  the zeta  element problem,  we set   
\begin{itemize}     [before = \vspace{\smallskipamount}, after =  \vspace{\smallskipamount}]    \setlength\itemsep{0.2em}    
\item $ \tilde{G} = G \times T $, 
\item $ \iota_{\nu} = \iota \times \nu : H  \to \tilde{G} $, 
\item $ U $ and $  \tilde{K}     :   = K \times C $ as bottom levels
\item $ x_{U}  =  \phi = \phi_{(0,0,0,0)}   \in M_{H,\OO}(U) $ as the  the source bottom class
\item   $ \tilde{L}  = K \times D  $ as the  layer extension  
of degree $ q - 1 $, 
\item $  \tilde{\mathfrak{H}}_{c} = \mathfrak{H}_{\mathrm{spin}, c}  (\mathrm{Frob} )  \in  \mathcal{C}_{ \OO    }  (   
\tilde{K}  \backslash  \tilde{G}   /  \tilde{K } ) $ as the Hecke  polynomial.    
\end{itemize} 
\begin{theorem}            There  exists  a     zeta  element     for $ (x_{U} ,   \tilde{ \mathfrak{H}} _{c} ,  \tilde{L} 
) $ for all $ c \in \ZZ - 2 \ZZ $.   \label{Gsp4zeta}       
\end{theorem} 
\begin{proof} Denote $  \varrho   = (\rho , \varpi) \in \tilde{G} $.   By Proposition \ref{Gsp4Heckepolynomial}, we see that  $$  \tilde{\mathfrak{H}}_{c}   \equiv  ( 1 + 2 \varrho^{2} +  \varrho^{4} )  ( \tilde{K} )  - ( 1 + \varrho^{2} ) ( \tilde{K}  \varrho \tilde{K}  )  + (\tilde{K} w_{0} \varrho^{2} \tilde{K})  \pmod{q-1} $$
where we view $ w_{0} \in \tilde{G}$ via $ 1_{G} \times \nu $.  
For $ i = 0, 1 $, let $ g_{i} = (\tau_{i}, 1_{T}) $. By  Proposition \ref{GSp4decompose}, we see that $ H \backslash  H \cdot \supp (  \tilde{ \mathfrak{H} }  _{c} ) / \tilde{K }  =  \big \{ H g_{0} \tilde{K} ,  H  g_{1} \tilde{K}  \big   \}   .  $
So it suffices to consider restrictions with respect to $ g_{0} $ and $ g_{1}$. 
Let $ \mathfrak{h}_{i}  $ denote the $ (H,g_{i})$-restriction of $ \tilde{\mathfrak{H}}_{c} $.   Observe that $$ H_{g_{i}} = H \cap g_{i} \tilde{K} g_{i}^{-1} = H \cap \tau_{i} K \tau_{i}^{-1} , $$  so that $ \mathfrak{h}_{i} \in \mathcal{C}_{\mathcal{O}}( U \backslash H / H _{\tau_{i}} ) $.     Let $ z = \left ( \left ( \begin{smallmatrix} \varpi \\ &  \varpi  \end{smallmatrix}  \right ) , \left (  \begin{smallmatrix}  \varpi \\ &  \varpi  \end{smallmatrix}  \right )  \right )  \in H  $.  Invoking  Proposition   \ref{GSp4decompose} again, we see that 
\begin{align*}   
\mathfrak{h}_{0} &  \equiv   ( 1 + 2 z + z^{2} ) ( U ) -  ( 1 + z ) ( U \varpi^{(1,1,1)} U)  + ( U \varpi^{(2,2,1)} U )  + (  U \varpi^{(2,1,2)}  U  )      \\
\mathfrak{h}_{1} & \equiv   \mathfrak{h}_{1}' 
\end{align*}  
modulo $ q- 1  $.  Note that the   action of $ z $ on $ \phi $ in the covariant convention is by its inverse. To avoid writing minus signs, let us denote $ z_{0} = z^{-1} $.  Then by Lemma \ref{explicithecke},  
\begin{align*} \mathfrak{h}_{0,*}(\phi) & \equiv  ( 1 + 2z_{0} +  z_{0}^{2}) \phi  - ( 1 +  z_{0}  ) \big ( z_{0}  \cdot  \phi +  \bar{\phi}_{(1,1,0,0)} + \bar{\phi}_{(0,0,1,1)} + \phi \big )  \, +  \\
 & \quad   
 \big (  
 z_{0} \cdot  \bar{\phi}_{(1,1,0,0)} + \bar{\phi}_{(0,0,1,1)} \big)  + \big (   z_{0} \cdot \bar{\phi} _{(0,0,1,1)}   +  \bar{\phi}_{(1,1,0,0)}  \big )  \\
&  = 0   \pmod{q-1} 
\end{align*} 
On the other hand, $ \mathfrak{h}_{1,*}(\phi) \equiv \mathfrak{h}_{1,*}'(\phi) = \psi $. It is easily seen that the stablizer of every point in $ \mathrm{supp}(\psi) $ in $ H_{\tau_{1}} $ reduces to identity modulo $ \varpi $.  In particular, these stabilizers are  contained in the subgroup $ H \cap g_{i} \tilde{L}  g_{i}^{-1} $ of $ H_{\tau_{1}} $.  So by  Theorem \ref{traceriteria}, $ \psi $ is in the image of the trace map $$      \pr_{*} : M_{H, \mathcal{O}} (H \cap  g_{i} \tilde{L} g_{i}   ^{-1}  ) 
\to M_{H, \mathcal{O}}(H_{\tau_{1}}) .  $$  We now invoke Corollary  \ref{torsionfreezeta}.   
\end{proof}

%% file: main.bbl
\providecommand{\bysame}{\leavevmode\hbox to3em{\hrulefill}\thinspace}
\providecommand{\MR}{\relax\ifhmode\unskip\space\fi MR }
\providecommand{\MRhref}[2]{%
  \href{http://www.ams.org/mathscinet-getitem?mr=#1}{#2}
}
\providecommand{\href}[2]{#2}
\begin{thebibliography}{LSZ22b}

\bibitem[AGV73]{SGA4}
Michael {Artin}, Alexander {Grothendieck}, and Jean-Louis {Verdier}, \emph{\href{https://doi.org/10.1007/BFb0070714}{Th\'{e}orie des topos et cohomologie \'{e}tale des sch\'{e}mas. {T}ome 3}}, Lecture Notes in Mathematics, Vol. 305, Springer-Verlag, Berlin-New York, 1973, S\'{e}minaire de G\'{e}om\'{e}trie Alg\'{e}brique du Bois-Marie 1963--1964 (SGA 4), Dirig\'{e} par M. Artin, A. Grothendieck et J. L. Verdier. Avec la collaboration de P. Deligne et B. Saint-Donat. \MR{0354654}

\bibitem[And87]{Andrianov}
Anatolij~N. Andrianov, \emph{Quadratic forms and {H}ecke operators}, Grundlehren der Mathematischen Wissenschaften [Fundamental Principles of Mathematical Sciences], vol. 286, Springer-Verlag, Berlin, 1987. \MR{884891}

\bibitem[AS01]{Asgari}
Mahdi Asgari and Ralf Schmidt, \emph{\href{https://doi.org/10.1007/PL00005869}{Siegel modular forms and representations}}, Manuscripta Math. \textbf{104} (2001), no.~2, 173--200. \MR{1821182}

\bibitem[BB04]{Bley}
Werner Bley and Robert Boltje, \emph{\href{https://doi.org/10.1016/j.jnt.2003.09.002}{Cohomological {M}ackey functors in number theory}}, J. Number Theory \textbf{105} (2004), no.~1, 1--37. \MR{2032439}

\bibitem[BB05]{Bjorner}
Anders Bj\"{o}rner and Francesco Brenti, \emph{\href{https://doi.org/10.1007/3-540-27596-7}{Combinatorics of {C}oxeter groups}}, Graduate Texts in Mathematics, vol. 231, Springer, New York, 2005. \MR{2133266}

\bibitem[BBJ20]{JetchevBoum}
R\'{e}da Boumasmoud, Ernest~Hunter Brooks, and Dimitar~P. Jetchev, \emph{\href{https://doi.org/10.1093/imrn/rny119}{Vertical distribution relations for special cycles on unitary {S}himura varieties}}, Int. Math. Res. Not. IMRN (2020), no.~13, 3902--3926. \MR{4120313}

\bibitem[Bor79]{Borelautomorphic}
A.~Borel, \emph{Automorphic {$L$}-functions}, Automorphic forms, representations and {$L$}-functions ({P}roc. {S}ympos. {P}ure {M}ath., {O}regon {S}tate {U}niv., {C}orvallis, {O}re., 1977), {P}art 2, Proc. Sympos. Pure Math., XXXIII, Amer. Math. Soc., Providence, R.I., 1979, pp.~27--61. \MR{546608}

\bibitem[Bor19]{Borel}
Armand Borel, \emph{\href{https://doi.org/10.1090/ulect/073}{Introduction to arithmetic groups}}, University Lecture Series, vol.~73, American Mathematical Society, Providence, RI, 2019, Translated from the 1969 French original MR 0244260 by Lam Laurent Pham, edited and with a preface by Dave Witte Morris. \MR{3970984}

\bibitem[Bou02]{Bourbaki}
Nicolas Bourbaki, \emph{Lie groups and {L}ie algebras. {C}hapters 4--6}, Elements of Mathematics (Berlin), Springer-Verlag, Berlin, 2002, Translated from the 1968 French original by Andrew Pressley. \MR{1890629}

\bibitem[Bou05]{BoubakiII}
\bysame, \emph{Lie groups and {L}ie algebras. {C}hapters 7--9}, Elements of Mathematics (Berlin), Springer-Verlag, Berlin, 2005, Translated from the 1975 and 1982 French originals by Andrew Pressley. \MR{2109105}

\bibitem[Bou21]{Reda}
R\'{e}da Boumasmoud, \emph{\href{https://arxiv.org/abs/2106.12540}{General horizontal norm compatible systems on unitary {S}himura varieties}}, 2021.

\bibitem[BP21]{highercoleman}
George Boxer and Vincent Pilloni, \emph{\href{https://arxiv.org/abs/2110.10251}{Higher Coleman theory}}, 2021.

\bibitem[BR94]{Blasius-Rogawski}
Don Blasius and Jonathan~D. Rogawski, \emph{Zeta functions of {S}himura varieties}, Motives ({S}eattle, {WA}, 1991), Proc. Sympos. Pure Math., vol.~55, Amer. Math. Soc., Providence, RI, 1994, pp.~525--571. \MR{1265563}

\bibitem[BT65]{Borel-Tits}
Armand Borel and Jacques Tits, \emph{\href{http://www.numdam.org/item?id=PMIHES_1965__27__55_0}{Groupes r\'{e}ductifs}}, Inst. Hautes \'{E}tudes Sci. Publ. Math. (1965), no.~27, 55--150. \MR{207712}

\bibitem[BT72]{TitsBruhat}
F.~Bruhat and J.~Tits, \emph{\href{http://www.numdam.org/item?id=PMIHES_1972__41__5_0}{Groupes r\'{e}ductifs sur un corps local}}, Inst. Hautes \'{E}tudes Sci. Publ. Math. (1972), no.~41, 5--251. \MR{327923}

\bibitem[Car79]{CartierSatake}
P.~Cartier, \emph{Representations of {$p$}-adic groups: a survey}, Automorphic forms, representations and {$L$}-functions ({P}roc. {S}ympos. {P}ure {M}ath., {O}regon {S}tate {U}niv., {C}orvallis, {O}re., 1977), {P}art 1, Proc. Sympos. Pure Math., XXXIII, Amer. Math. Soc., Providence, R.I., 1979, pp.~111--155. \MR{546593}

\bibitem[Cas80]{Casselmanunramified}
W.~Casselman, \emph{\href{http://www.numdam.org/item?id=CM_1980__40_3_387_0}{The unramified principal series of {$p$}-adic groups. {I}. {T}he spherical function}}, Compositio Math. \textbf{40} (1980), no.~3, 387--406. \MR{571057}

\bibitem[Cas17]{CasselmanIran}
B.~Casselman, \emph{\href{http://bims.iranjournals.ir/article_1154.html}{Symmetric powers and the {S}atake transform}}, Bull. Iranian Math. Soc. \textbf{43} (2017), no.~4, 17--54. \MR{3711821}

\bibitem[CG22]{GanChen}
Rui Chen and Wee~Teck Gan, \emph{\href{https://arxiv.org/abs/2108.04064}{Unitary {F}riedberg-{J}acquet periods}}, 2022.

\bibitem[CGP15]{CGP}
Brian Conrad, Ofer Gabber, and Gopal Prasad, \emph{\href{https://doi.org/10.1017/CBO9781316092439}{Pseudo-reductive groups}}, second ed., New Mathematical Monographs, vol.~26, Cambridge University Press, Cambridge, 2015. \MR{3362817}

\bibitem[CGS]{EulerGU22}
Antonio Cauchi, Andrew Graham, and Syed Waqar~Ali Shah, \emph{Euler systems for exterior square motives}, in preparation.

\bibitem[Col03]{Colmez2002-2003}
Pierre Colmez, \emph{\href{http://eudml.org/doc/252144}{La conjecture de {B}irch et {S}winnerton-{D}yer $p$-adique}}, Séminaire Bourbaki \textbf{45} (2002-2003), 251--320 (fre).

\bibitem[Cor18]{Cornut}
Cristophe Cornut, \emph{\href{https://webusers.imj-prg.fr/~christophe.cornut/papers/ESHT.pdf}{An {E}uler system of {H}eegner type}}, 2018.

\bibitem[CRJ20]{CJR}
Antonio Cauchi and Joaquín Rodrigues~Jacinto, \emph{\href{https://doi.org/10.25537/dm.2020v25.911-954}{Norm-compatible systems of {G}alois cohomology classes for {${\bf GSp}_6$}}}, Doc. Math. \textbf{25} (2020), 911--954. \MR{4151876}

\bibitem[CRJS]{EulerG2}
Antonio Cauchi, Joaquín Rodrigues~Jacinto, and Syed Waqar~Ali Shah, \emph{Euler systems for motives of {G}alois group of type {${G}_2$}}, in preparation.

\bibitem[Del71]{DeligneTS}
Pierre Deligne, \emph{\href{https://doi.org/10.1007/BFb0058700}{Travaux de Shimura}}, S\'{e}minaire Bourbaki : vol. 1970/71, expos\'{e}s 382-399, S\'eminaire Bourbaki, no.~13, Springer-Verlag, 1971, pp.~123--165.

\bibitem[Dre73]{Dress}
Andreas W.~M. Dress, \emph{\href{https://doi.org/10.1007/BFb0073727}{Contributions to the theory of induced representations}}, Algebraic {$K$}-theory, {II}: ``{C}lassical'' algebraic {$K$}-theory and connections with arithmetic ({P}roc. {C}onf., {B}attelle {M}emorial {I}nst., {S}eattle, {W}ash., 1972), 1973, pp.~183--240. Lecture Notes in Math., Vol. 342. \MR{0384917}

\bibitem[Fal05]{FaltingsEisenstein}
Gerd Faltings, \emph{\href{https://doi.org/10.1007/0-8176-4447-4_8}{Arithmetic {E}isenstein classes on the {S}iegel space: some computations}}, Number fields and function fields---two parallel worlds, Progr. Math., vol. 239, Birkh\"{a}user Boston, Boston, MA, 2005, pp.~133--166. \MR{2176590}

\bibitem[FP21]{PilloniFakharuddin}
Najmuddin Fakhruddin and Vincent Pilloni, \emph{\href{https://doi.org/10.1017/S1474748021000050}{{H}ecke operators and the coherent cohomology of {S}himura varieties}}, Journal of the Institute of Mathematics of Jussieu (2021), 1–69.

\bibitem[GM92]{may}
J.~P.~C. Greenlees and J.~P. May, \emph{\href{https://doi.org/10.2307/2159592}{Some remarks on the structure of {M}ackey functors}}, Proc. Amer. Math. Soc. \textbf{115} (1992), no.~1, 237--243. \MR{1076574}

\bibitem[Gra24a]{Graham2}
Andrew Graham, \emph{\href{https://arxiv.org/abs/2403.05960}{Unitary {F}riedberg--{J}acquet periods and anticyclotomic p-adic L-functions}}, 2024.

\bibitem[Gra24b]{Graham1}
Andrew Graham, \emph{\href{https://doi.org/10.2140/ant.2024.18.1117}{On the {$p$}-adic interpolation of unitary {F}riedberg-{J}acquet periods}}, Algebra Number Theory \textbf{18} (2024), no.~6, 1117--1188. \MR{4740094}

\bibitem[Gre71]{Green}
J.~A. Green, \emph{\href{https://doi.org/10.1016/0022-4049(71)90011-9}{Axiomatic representation theory for finite groups}}, J. Pure Appl. Algebra \textbf{1} (1971), no.~1, 41--77. \MR{279208}

\bibitem[Gro98]{Gross}
Benedict~H. Gross, \emph{\href{https://doi.org/10.1017/CBO9780511662010.006}{On the {S}atake isomorphism}}, Galois representations in arithmetic algebraic geometry ({D}urham, 1996), London Math. Soc. Lecture Note Ser., vol. 254, Cambridge Univ. Press, Cambridge, 1998, pp.~223--237. \MR{1696481}

\bibitem[GS23]{Anticyclo}
Andrew Graham and Syed Waqar~Ali Shah, \emph{\href{ https://doi.org/10.1112/plms.12566}{Anticyclotomic {E}uler systems for unitary groups}}, Proc. Lond. Math. Soc. (3) \textbf{127} (2023), no.~6, 1577--1680. \MR{4673434}

\bibitem[Hec37]{HeckeHecke}
E.~Hecke, \emph{\href{https://doi.org/10.1007/BF01594160}{\"{u}ber {m}odulfunktionen und die {d}irichletschen {r}eihen mit {e}ulerscher {p}roduktentwicklung. {i}}}, Math. Ann. \textbf{114} (1937), no.~1, 1--28. \MR{1513122}

\bibitem[HKP10]{KHPIwahori}
Thomas~J. Haines, Robert~E. Kottwitz, and Amritanshu Prasad, \emph{Iwahori-{H}ecke algebras}, J. Ramanujan Math. Soc. \textbf{25} (2010), no.~2, 113--145. \MR{2642451}

\bibitem[HR10]{HainesRostami}
Thomas~J. Haines and Sean Rostami, \emph{\href{https://doi.org/10.1090/S1088-4165-10-00370-5}{The {S}atake isomorphism for special maximal parahoric {H}ecke algebras}}, Represent. Theory \textbf{14} (2010), 264--284. \MR{2602034}

\bibitem[Hum78]{Humphreys}
James~E. Humphreys, \emph{Introduction to {L}ie algebras and representation theory}, Graduate Texts in Mathematics, vol.~9, Springer-Verlag, New York-Berlin, 1978, Second printing, revised. \MR{499562}

\bibitem[Hum90]{coxeterHumphreys}
\bysame, \emph{\href{https://doi.org/10.1017/CBO9780511623646}{Reflection groups and {C}oxeter groups}}, Cambridge Studies in Advanced Mathematics, vol.~29, Cambridge University Press, Cambridge, 1990. \MR{1066460}

\bibitem[IM65]{Iwahori}
N.~Iwahori and H.~Matsumoto, \emph{\href{https://doi.org/10.1007/BF02684396}{On some {B}ruhat decomposition and the structure of the {H}ecke rings of {${p}$}-adic {C}hevalley groups}}, Inst. Hautes \'{E}tudes Sci. Publ. Math. (1965), no.~25, 5--48. \MR{185016}

\bibitem[Iwa66]{BruhatIwahori}
Nagayoshi Iwahori, \emph{Generalized {T}its system ({B}ruhat decompostition) on {$p$}-adic semisimple groups}, Algebraic {G}roups and {D}iscontinuous {S}ubgroups ({P}roc. {S}ympos. {P}ure {M}ath., {B}oulder, {C}olo., 1965), Amer. Math. Soc., Providence, R.I., 1966, pp.~71--83. \MR{0215858}

\bibitem[Kat82]{Kato}
Shin-ichi Kato, \emph{\href{https://doi.org/10.1007/BF01389223}{Spherical functions and a {$q$}-analogue of {K}ostant's weight multiplicity formula}}, Invent. Math. \textbf{66} (1982), no.~3, 461--468. \MR{662602}

\bibitem[Kat04]{kkato}
Kazuya Kato, \emph{\href{http://www.numdam.org/item/AST_2004__295__117_0/}{$p$-adic {Hodge} theory and values of zeta functions of modular forms}}, Cohomologie $p$-adiques et applications arithm\'etiques (III) (Berthelot Pierre, Fontaine Jean-Marc, Illusie Luc, Kato Kazuya, and Rapoport Michael, eds.), Ast\'erisque, no. 295, Soci\'et\'e math\'ematique de France, 2004, pp.~117--290 (en). \MR{2104361}

\bibitem[KK08]{Kim}
Henry~H. Kim and Muthukrishnan Krishnamurthy, \emph{Twisted exterior square lift from {${\rm GU}(2,2)_{E/F}$} to {${\rm GL}_6/F$}}, J. Ramanujan Math. Soc. \textbf{23} (2008), no.~4, 381--412. \MR{2492574}

\bibitem[Kno05]{Knop}
Friedrich Knop, \emph{\href{https://doi.org/10.1090/S1088-4165-05-00237-2}{On the {K}azhdan-{L}usztig basis of a spherical {H}ecke algebra}}, Represent. Theory \textbf{9} (2005), 417--425. \MR{2142817}

\bibitem[Kot84]{Kottwitz}
Robert~E. Kottwitz, \emph{\href{http://eudml.org/doc/163936}{Shimura varieties and twisted orbital integrals}}, Mathematische Annalen \textbf{269} (1984), 287--300.

\bibitem[KP23]{BTbook}
Tasho Kaletha and Gopal Prasad, \emph{\href{ https://doi.org/10.1017/9781108933049}{Bruhat-{T}its theory---a new approach}}, New Mathematical Monographs, vol.~44, Cambridge University Press, Cambridge, 2023. \MR{4520154}

\bibitem[Kud97]{kudla}
Stephen~S. Kudla, \emph{\href{https://doi.org/10.1215/S0012-7094-97-08602-6}{Algebraic cycles on {S}himura varieties of orthogonal type}}, Duke Math. J. \textbf{86} (1997), no.~1, 39--78. \MR{1427845}

\bibitem[Lan79]{Langlandszeta}
R.~P. Langlands, \emph{\href{https://doi.org/10.4153/CJM-1979-102-1}{On the zeta functions of some simple {S}himura varieties}}, Canadian J. Math. \textbf{31} (1979), no.~6, 1121--1216. \MR{553157}

\bibitem[Lan01]{Lansky}
Joshua~M. Lansky, \emph{\href{https://doi.org/10.2140/pjm.2001.197.97}{Decomposition of double cosets in {$p$}-adic groups}}, Pacific J. Math. \textbf{197} (2001), no.~1, 97--117. \MR{1810210}

\bibitem[Lem10]{Lemma}
Francesco Lemma, \emph{\href{https://doi.org/10.4171/CMH/213}{A norm compatible system of {G}alois cohomology classes for {${\rm GSp}(4)$}}}, Comment. Math. Helv. \textbf{85} (2010), no.~4, 885--905. \MR{2718141}

\bibitem[Lew82]{Lewis}
D.~W. Lewis, \emph{\href{https://doi.org/10.1016/0024-3795(82)90258-0}{The isometry classification of {H}ermitian forms over division algebras}}, Linear Algebra Appl. \textbf{43} (1982), 245--272. \MR{656449}

\bibitem[Loe21]{loe}
David Loeffler, \emph{\href{https://doi.org/10.1007/s10884-020-09844-5}{Spherical varieties and norm relations in {I}wasawa theory}}, J. Th\'{e}or. Nombres Bordeaux \textbf{33} (2021), no.~3, part 2, 1021--1043. \MR{4402388}

\bibitem[LSZ22a]{gu21}
David Loeffler, Christopher Skinner, and Sarah~Livia Zerbes, \emph{\href{https://doi.org/10.1007/s00208-021-02224-4}{An {E}uler system for {$\rm GU(2,1)$}}}, Math. Ann. \textbf{382} (2022), no.~3-4, 1091--1141. \MR{4403219}

\bibitem[LSZ22b]{LSZ}
\bysame, \emph{\href{https://doi.org/10.4171/jems/1124}{Euler systems for {$\rm GSp(4)$}}}, J. Eur. Math. Soc. (JEMS) \textbf{24} (2022), no.~2, 669--733. \MR{4382481}

\bibitem[LZ20]{LZBK}
David Loeffler and Sarah~Livia Zerbes, \emph{\href{https://arxiv.org/abs/2003.05960}{On the {B}loch-{K}ato conjecture for {$\rm GSp(4)$}}}, 2020.

\bibitem[Mac71]{Macdonaldspherical}
I.~G. Macdonald, \emph{Spherical functions on a group of {$p$}-adic type}, Publications of the Ramanujan Institute, No. 2, University of Madras, Centre for Advanced Study in Mathematics, Ramanujan Institute, Madras, 1971. \MR{0435301}

\bibitem[Mat77]{Matsumoto}
Hideya Matsumoto, \emph{Analyse harmonique dans les syst\`emes de {T}its bornologiques de type affine}, Lecture Notes in Mathematics, Vol. 590, Springer-Verlag, Berlin-New York, 1977. \MR{0579177}

\bibitem[Mil03]{MilneShimura}
J.~S. Milne, \emph{{I}ntroduction to {S}himura {V}arieties}, {H}armonic {A}nalysis, {T}he {T}race {F}ormula, and {S}himura {V}arieties ({R}obert~{K}ottwitz {J}ames {A}rthur, {D}avid~{E}llwood, ed.), Clay Mathematics Proceedings, vol.~4, American Mathematical Society, 2003, pp.~265 -- 378.

\bibitem[Mil17]{MilneAlgebraic}
J.~S. Milne, \emph{\href{https://doi.org/10.1017/9781316711736}{Algebraic groups}}, Cambridge Studies in Advanced Mathematics, vol. 170, Cambridge University Press, Cambridge, 2017, The theory of group schemes of finite type over a field. \MR{3729270}

\bibitem[Mí11]{Minguez}
Alberto Mínguez, \emph{Unramified representations of unitary groups}, On the stabilization of the trace formula, Stab. Trace Formula Shimura Var. Arith. Appl., vol.~1, Int. Press, Somerville, MA, 2011, pp.~389--410. \MR{2856377}

\bibitem[NSW08]{Neu}
J\"{u}rgen Neukirch, Alexander Schmidt, and Kay Wingberg, \emph{\href{https://doi.org/10.1007/978-3-540-37889-1}{Cohomology of number fields}}, second ed., Grundlehren der Mathematischen Wissenschaften [Fundamental Principles of Mathematical Sciences], vol. 323, Springer-Verlag, Berlin, 2008. \MR{2392026}

\bibitem[Rap00]{Rapoportpositive}
Michael Rapoport, \emph{\href{https://doi.org/10.1007/s002290050010}{A positivity property of the {S}atake isomorphism}}, Manuscripta Math. \textbf{101} (2000), no.~2, 153--166. \MR{1742251}

\bibitem[Sat63]{Satakeoriginal}
Ichir\^{o} Satake, \emph{\href{http://www.numdam.org/item?id=PMIHES_1963__18__5_0}{Theory of spherical functions on reductive algebraic groups over {${p}$}-adic fields}}, Inst. Hautes \'{E}tudes Sci. Publ. Math. (1963), no.~18, 5--69. \MR{195863}

\bibitem[Sha22]{AESthesis}
Syed Waqar~Ali Shah, \emph{\href{https://dash.harvard.edu/handle/1/37372227?show=full}{On an approach to automorphic {E}uler systems}}, ProQuest LLC, Ann Arbor, MI, 2022, Thesis (Ph.D.)--Harvard University. \MR{4464230}

\bibitem[Sha23a]{distpolylog}
\bysame, \emph{\href{https://arxiv.org/abs/2309.10938}{On distribution relations of polylogarithmic Eisenstein classes}}, 2023.

\bibitem[Sha23b]{explicitdescent}
\bysame, \emph{\href{https://arxiv.org/abs/2310.01677}{Explicit {H}ecke descent for special cycles}}, 2023.

\bibitem[Sha23c]{Norm}
\bysame, \emph{\href{https://arxiv.org/abs/2310.14543}{Norm relations for CM points on modular curves}}, 2023.

\bibitem[Sha24a]{Siegel2}
\bysame, \emph{Euler systems for motives of {S}iegel modular sixfolds}, 2024, in preparation.

\bibitem[Sha24b]{Siegel1}
\bysame, \emph{Horizontal norm compatibility of cohomology classes for {$\rm GSp_6$}}, 2024.

\bibitem[Shi94]{ArithmeticShimura}
Goro Shimura, \emph{Introduction to the arithmetic theory of automorphic functions}, Publications of the Mathematical Society of Japan, vol.~11, Princeton University Press, Princeton, NJ, 1994, Reprint of the 1971 original, Kan\^{o} Memorial Lectures, 1. \MR{1291394}

\bibitem[Spr79]{SpringerBorel}
T.~A. Springer, \emph{Reductive groups}, Automorphic forms, representations and {$L$}-functions ({P}roc. {S}ympos. {P}ure {M}ath., {O}regon {S}tate {U}niv., {C}orvallis, {O}re., 1977), {P}art 1, Proc. Sympos. Pure Math., XXXIII, Amer. Math. Soc., Providence, R.I., 1979, pp.~3--27. \MR{546587}

\bibitem[Ste06]{Stembridge}
John~R. Stembridge, \emph{\href{http://www.combinatorics.org/Volume_11/Abstracts/v11i2r14.html}{Tight quotients and double quotients in the {B}ruhat order}}, Electron. J. Combin. \textbf{11} (2004/06), no.~2, Research Paper 14, 41. \MR{2120109}

\bibitem[SV24]{Vincentelli}
Marco~Antonio Sangiovanni~Vincentelli, \emph{\href{ http://arks.princeton.edu/ark:/88435/dsp01qf85nf67q}{Crafting {E}uler systems:\ beyond the motivic mold}}, 2024, Thesis (Ph.D.)--Princeton University.

\bibitem[Tam63]{TamagawaSatake}
Tsuneo Tamagawa, \emph{\href{https://doi.org/10.2307/1970221}{On the {$\zeta $}-functions of a division algebra}}, Ann. of Math. (2) \textbf{77} (1963), 387--405. \MR{144928}

\bibitem[Thi11]{Ulc}
Ulrich Thiel, \emph{\href{https://arxiv.org/abs/1111.0942}{Mackey functors and abelian class field theories}}, 2011.

\bibitem[Tit66]{Tits-Semisimple}
J.~Tits, \emph{Classification of algebraic semisimple groups}, Algebraic {G}roups and {D}iscontinuous {S}ubgroups ({P}roc. {S}ympos. {P}ure {M}ath., {B}oulder, {C}olo., 1965), Amer. Math. Soc., Providence, R.I., 1966, 1966, pp.~33--62. \MR{0224710}

\bibitem[Tit79]{Tits}
\bysame, \emph{Reductive groups over local fields}, Automorphic forms, representations and {$L$}-functions ({P}roc. {S}ympos. {P}ure {M}ath., {O}regon {S}tate {U}niv., {C}orvallis, {O}re., 1977), {P}art 1, Proc. Sympos. Pure Math., XXXIII, Amer. Math. Soc., Providence, R.I., 1979, pp.~29--69. \MR{546588}

\bibitem[TW95]{Webb}
Jacques Th\'{e}venaz and Peter Webb, \emph{\href{https://doi.org/10.2307/2154915}{The structure of {M}ackey functors}}, Trans. Amer. Math. Soc. \textbf{347} (1995), no.~6, 1865--1961. \MR{1261590}

\bibitem[Urb20]{UrbanII}
Eric Urban, \emph{\href{https://www.math.columbia.edu/~urban/EURP.html}{Euler systems and {E}isenstein congruences}}, 2020.

\bibitem[Urb21]{Urban}
\bysame, \emph{\href{ http://www.numdam.org/articles/10.5802/jtnb.1191/}{On {E}uler systems for adjoint {H}ilbert modular {G}alois representations}}, J. Th\'{e}or. Nombres Bordeaux \textbf{33} (2021), no.~3, part 2, 1115--1141. \MR{4402393}

\bibitem[vdH74]{Hombergh}
A.~van~den Hombergh, \emph{\href{https://doi.org/10.1016/1385-7258(74)90003-1}{About the automorphisms of the {B}ruhat-ordering in a {C}oxeter group}}, Nederl. Akad. Wetensch. Proc. Ser. A {\bf 77}=Indag. Math. \textbf{36} (1974), 125--131. \MR{0360857}

\bibitem[Vig89]{Vigneras}
Marie-France Vign\'{e}ras, \emph{\href{ http://www.numdam.org/item/CM_1989__72_1_33_0/}{Repr\'{e}sentations modulaires de {${\rm GL}(2,F)$} en caract\'{e}ristique {$l, F$}corps {$p$}-adique, {$p\neq l$}}}, Compositio Math. \textbf{72} (1989), no.~1, 33--66. \MR{1026328}

\bibitem[Yos83]{Yoshida}
Tomoyuki Yoshida, \emph{\href{https://doi.org/10.2969/jmsj/03510179}{On {$G$}-functors. {II}. {H}ecke operators and {$G$}-functors}}, J. Math. Soc. Japan \textbf{35} (1983), no.~1, 179--190. \MR{679083}

\bibitem[Zha21]{Ruishen}
Ruishen Zhao, \emph{\href{https://arxiv.org/abs/2111.07475}{Special cycles on orthogonal {S}himura varieties}}, 2021.

\bibitem[ZX]{ZhangXiao}
Wei Zhang and Jingwei Xiao, \emph{Unitary {F}riedberg--{J}acquet periods and their twists}, in preparation.

\end{thebibliography}
